%% file: main.tex
\documentclass[hidelinks,onefignum,onetabnum]{siamart220329}

\input{ex_shared}

\usepackage[english]{babel}
\usepackage[T1]{fontenc}
\usepackage[utf8]{inputenc}
\usepackage{amsmath}
\usepackage{amssymb}
\usepackage{mathrsfs}
\usepackage{graphicx,mathtools,tikz}
\usepackage{caption}
\usepackage{enumitem}
\usepackage{subcaption}
\usetikzlibrary{positioning}
\usepackage{mathtools}
\usepackage{algpseudocode}
\usepackage{xr}

\usepackage{xcolor}
\usepackage{hyperref}

\newtheorem{tm}{Theorem}[section]
\newtheorem{df}{Definition}[section]
\newtheorem{lm}{Lemma}[section]
\newtheorem{prop}{Proposition}[section]

\newtheorem{rk}{Remark}[section]

\numberwithin{equation}{section}

\newcommand*{\bX}{\ensuremath{\boldsymbol{X}}}
\newcommand*{\bP}{\ensuremath{\boldsymbol{P}}}

\begin{document}

\maketitle
\begin{abstract} 
We propose a supervised learning scheme for the first order Hamilton--Jacobi PDEs in high dimensions. The scheme is designed by using the geometric structure of Wasserstein Hamiltonian flows via a density coupling strategy. It is equivalently posed as a regression problem using the Bregman divergence, which provides the loss function in learning while the data is generated through the particle formulation of Wasserstein Hamiltonian flow. We prove a posterior estimate on $L^1$ residual of the proposed scheme based on the coupling density. Furthermore, the proposed scheme can be used to describe the behaviors of  Hamilton--Jacobi PDEs beyond the singularity formations on the support of coupling density.
Several numerical examples with different Hamiltonians are provided to support our findings.  
\end{abstract}

\begin{keywords}
Hamilton--Jacobi PDE, high dimension,  supervised learning, density coupling, Wasserstein Hamiltonian flow
\end{keywords}

\begin{MSCcodes}
65M75, 65P10, 49Q22, 68T07
\end{MSCcodes}

\section{Introduction}

In this paper, we are concerned with solving the following Hamilton--Jacobi equation numerically,
\begin{equation}
  \frac{\partial u(x,t)}{\partial t} + H(x,\nabla u(x,t)) = 0, \quad u(x,0)=g(x), \label{HJ}
\end{equation}
where $t\in [0,T], x\in \mathbb R^d$ with $d\in \mathbb N^+$, and the Hamiltonian $H$ is convex with respect to the second variable. 
Hamilton--Jacobi partial differential equations (HJ PDEs) \eqref{HJ}  arise in many areas of applications, including the
calculus of variations, control theory, and differential games \cite{MR3135343}. 
For instance, in classical mechanics, the solution of HJ PDE (also called as the generating function), when used with canonical transformations which preserve the structure of the Hamiltonian equations, can reveal hidden symmetries, and thus has been applied to design symplectic integrators for the motion of planets and other celestial bodies \cite{pmlr-v139-chen21r, MR2221614}.  HJ PDE is  often used  to determine the optimal policy in optimal control and density control problems \cite{MR3489825}, and has been applied to model the optimal investment strategies and pricing of derivatives in financial mathematics.
In the semiclassical approximation, HJ PDE is related to the Schrödinger equation and has been used in the WKB (Wentzel-Kramers-Brillouin) approximation to describe quantum systems and  tunnelling phenomena \cite{jin2011mathematical}.

However, obtaining their analytical solutions, if at all possible, is often challenging, especially in high dimensions. As indispensable tools, numerical methods such as finite difference \cite{MR744921,MR1111446}, fast sweeping \cite{MR2004194} and level set methods \cite{MR2069944, jin2005computing} have been developed and refined over the years to approximate the solutions and predict their longtime dynamics. Those traditional algorithms involve discretizing the equation on grids and approximating the derivatives by using either finite difference or finite element techniques, and thus their applicability is limited by the so-called curse of dimensionality, namely the computational cost grows exponentially with respect to the problem dimension $d$ \cite{MR134403}. 

In recent years, several strategies are proposed 
to mitigate the challenges caused by the curse of dimensionality when solving HJ PDEs numerically\footnote{For more related topics and research problems on high dimensional HJ equations, we refer to {\color{blue} \href{http://www.ipam.ucla.edu/programs/long-programs/high-dimensional-hamilton-jacobi-pdes/?tab=activities}{http://www.ipam.ucla.edu/programs/long-programs/high-dimensional-hamilton-jacobi-pdes/?tab=activities}}.}, including the optimization method \cite{darbon2016algorithms}, sparse grids \cite{MR3045704}, 
neural networks \cite{Darbon_2021,MR3847747}, etc. For instance, several approaches have been developed from a probabilistic perspective. The works in \cite{fahim2011probabilistic} and \cite{MR3847747} relies on the Feynman-Kac representation of the second-order Hamilton-Jacobi-Bellman (HJB) equations. In \cite{nusken2021solving}, the authors address high-dimensional HJB equations with quadratic kinetic energy by minimizing the discrepancy between the path measures of controlled and uncontrolled diffusion processes.

Another category of research builds on the Hopf–Lax formula, which characterizes the viscosity solution of the HJ equation. The work in \cite{darbon2016algorithms}, along with its extension in \cite{chow2019algorithm}, employed a Hopf–Lax-type formula and iterative optimization algorithms for the scalable computation of HJ equations. A deep learning-based method for stationary HJ equations on bounded domains was proposed in \cite{gang_yin_et_al_IEEE}. In \cite{Darbon_2021}, the authors designed a deep neural network inspired by the Hopf-Lax formula and optimized its parameters to solve the HJ equation.

Causality-free methods for solving HJB equations arising from optimal feedback control were introduced in \cite{nakamura2020causality, nakamura2021adaptive}. These approaches compute numerical solutions by minimizing the $L^2$ loss between the neural network approximation and reference solutions obtained from optimal trajectories generated at random initial points.

In \cite{MR4717768}, an intriguing connection is established between a specific optimization problem and the multitime Hopf formula to address multitime Hamilton–Jacobi (HJ) partial differential equations (PDEs). In \cite{MR4848555}, the authors formulate a least-squares problem to approximate the viscosity solution of the HJ equation, leveraging a monotone and consistent numerical scheme. An implicit solution formula based on the method of characteristics is proposed in \cite{PO2025} for computing the viscosity solution of the HJ equation. This formula is subsequently learned using a deep learning-based approach that scales effectively to high dimensions.

In this paper, we introduce an alternative supervised learning method to solve HJ PDEs in high dimensions. Our study stems from some recent advancements in Wasserstein Hamiltonian flow (WHF) \cite{chow2020wasserstein}, which describes a family of PDEs defined on the  Wasserstein manifold, the probability density set equipped with the optimal transport (OT) metric. Examples of WHFs include the Wasserstein geodesic \cite{MR4479897}, Schr\"odinger equation \cite{MR4405488}, and mean field control \cite{MR4322102}. A typical WHF consists of a transport (or Fokker-Planck) equation and a HJ equation. Coupling two equations together, they form a geometric flow with symplectic and Hamiltonian structures on the Wasserstein manifold (we refer to section \ref{sec-2} for more details). This inspires us to design a numerical scheme that can preserve the geometric properties of the original HJ equation and mitigate the curse of dimensionality at the same time. 

To achieve this goal, we must confront several difficulties. First, the  classical structure-preserving methods are often implicit in time and they become intractable when the spatial dimension grows high. Second, it is well-known that the characteristics of HJ equation may intersect and its classical solution may only exist up to a finite time. Third, the state-of-the-art numerical methods mainly focus on solving the viscosity solution, and may not capture {the geometric structure on the Wasserstein density manifold}. Last but not least, in some applications like the geometric optics, seismic waves and semi-classical limits of quantum dynamics, one may be more interested in other physical solutions, like the multi-valued solution and its statistical information \cite{MR2069944}. 

To overcome these challenges, we leverage the geometric structure of WHF and the approximation power of deep neural networks (DNNs) to design a supervised learning procedure. More precisely, we propose an approach consisting of the following steps. 

\begin{enumerate}
    \item Coupling the given HJ equation with a continuity equation that transports a probability density function to form a WHF on Wasserstein manifold. The transport velocity field is provided by the solution of the HJ equation. According to the theory of WHF, a particle version corresponding to the coupled system can be constructed leading to a system of Hamiltonian ordinary differential equations (ODEs). 
    \item Formulating a regression problem based on the Bregmann divergence following the OT theory. Its critical point satisfies the coupled system of WHF. This regression or its equivalent least squares expression are then used as the loss function in the learning process. 
    \item Generating the training data $(\{\bX_t \},\{\bP_t\})$ by applying a symplectic integrator to the particle version of WHF, which is the Hamiltonian {ODE system}  constructed in the first step. 
    \item Learning the solution {of} HJ equation by reducing the loss function evaluated on the training data $(\{\bX_t \},\{\bP_t\})$ via  minimization algorithms such as Adam \cite{kingma2014Adam}.
\end{enumerate}
Details on the first and second steps will be given in section \ref{sec-2}, and about the third and fourth steps in section \ref{sec-3}. 

The proposed method eases the computation burden of high dimensional HJ equation from three different aspects. (i) The loss function is expressed in term of expectation, which can be evaluated by employing the Monte Carlo integral methods and auto differentiation in DNNs. This allows us to carry out the calculation in higher dimensions without limiting the number of unknowns as the classical finite difference and finite element methods do. (ii) The training data $(\{\bX_t \},\{\bP_t\})$ is generated by solving ODEs, which can be scaled up to higher dimensions. {The numerical solvers for Hamiltonian ODEs that we choose are the symplectic integrators which could  preserve the symplectic structure of the phase flow exactly (see section 2 for the definitions) and maintain the Hamiltonian up to a relatively long time.}
(iii) The density function can be selected (supervised) so that its support covers the region of interest. This provides a mechanism to only generate training data concentrated at the place where the solution of HJ equation is needed, and it is different from many existing DNN based methods for high dimensional PDEs, like  physics-informed neural network (PINN) \cite{MR3881695}, deep-Ritz \cite{MR3767958}, or weak adversarial network \cite{MR4079373}, in which samples are usually taken everywhere in the domain. {It should be noted that there are a few recent works that propose samples in a sophisticated way, like the deep adaptive sampling method (see \cite{MR4531552} and references therein).} An added benefit is that the training data is computed by symplectic structure preserving schemes so that better geometric properties of the HJ equation can be retained in the learning procedure. 

More importantly, we would like to advocate two new features of the proposed method for theoretical analysis. The coupling strategy enables us to develop a novel error bound using the residual estimate with respect to the density controlling where and how the training data is sampled. In other words, the error estimate may vary depending on the chosen density. This is different from the traditional error estimates, and it is more suitable for machine learning-based methods in which random samples are used for the training. We establish the rigorous error estimate for the proposed method in section \ref{sec-3}.  In a special case when the initial density is selected as the uniform distribution, the proposed method generates training data using ODEs that resemble the bi-characteristic formulation. According to the uniqueness theorem of ODEs, the training data can be generated beyond the blow-up time that the classic solution of HJ equation doesn't exist anymore, for example, the characteristics intersect. In this sense, the supervised learning method may compute the solution of HJ equation after the blow-up time. We show several such examples along with other numerical experiments in section \ref{sec-4}.  

Although our proposed approach shares some similarities with the supervised learning formulation presented in \cite{nakamura2020causality, nakamura2021adaptive}, they have major differences too. The algorithm in \cite{nakamura2020causality, nakamura2021adaptive} is designed for the ``backward'' HJ equations originated from control with desirable terminal conditions, and the training data is generated by solving boundary value problems following the Pontryagin {maximal principle}. While our scheme is proposed for the ``forward'' HJ equation with given initial condition, and the training data is created by solving initial value Hamiltonian ODEs following particle formulation of WHF. More importantly, our derivation is conducted on the Wasserstein manifold, and it reveals the connection between the supervised learning scheme and a sup-inf problem originated from the mean-field control, which further provides a formulation for error analysis. It is also worth mentioning that the coupling idea is also used in \cite{meng2023primal}, in which the solution of HJ equation is reformulated as a saddle point problem and further solved by the primal-dual hybrid gradient algorithm.  In our scheme, we introduce a swarm of particles governed by the Hamiltonian ODEs corresponding to the WHF, and their trajectories are used as the data in the supervised learning. This leads to a minimization problem whose loss function can be computed by the Monte--Carlo method, and it is scalable to high-dimensional problems. Recently, \cite{Neklyudov2022ActionML} uses the similar connection between the optimal transport and action matching to propose a learning method for stochastic dynamics from samples.

\section{Density coupling strategy}\label{sec-2}

In this section, we introduce the preliminary and key ingredients for designing the supervised learning scheme of HJ equations. We first present
a regression problem via the Bregman divergence with Hamiltonian dynamics as the constrain, which plays the role of the loss function in section 3. Then we show that the critical point of this regression problem is indeed the HJ equation coupled with a transport equation for the probability density. The connection between the regression and a sup-inf optimization from the viewpoint of optimal transport is also discussed.

{
\subsection{Regression problem via the Bregman divergence}

 \begin{df}[Bregman divergence \cite{BREGMAN1967200}] \label{def: bregman}
  Suppose $f\in\mathcal{C}^1(\mathbb{R}^d)$ is a strict convex function. We define the Bregman divergence $D_f(\cdot:\cdot):\mathbb{R}^d\times\mathbb{R}^d\rightarrow\mathbb{R}_{\geq 0}$ induced by $f$ as
  $$D_f(q_1:q_2) = f(q_1) - f(q_2) - \nabla f(q_2)\cdot(q_1 - q_2).$$
\end{df}
It is known that the Bregman divergence is positive and  $D_f(q_1:q_2) = 0$ if and only if $q_1 = q_2$. 

\begin{lm}\label{thm Legendre and Bregman divergence}
  Suppose $f\in\mathcal{C}^2(\mathbb{R}^d)$ is $\alpha$-strongly convex and $L$-strongly smooth ($\alpha, L >0$), i.e., $\alpha I_d \preceq \nabla^2 f(q) \preceq L I_d$ for any $q\in\mathbb R^d$. Then, the Legendre transform $f^*$ of $f$ 
 belongs to $\mathcal{C}^2  (\mathbb{R}^d),$ and is $\frac{1}{L}-$strongly convex and $\frac{1}{\alpha}$-strongly smooth on $\mathbb{R}^d$. Furthermore, it holds that 
  $$f(q)+f^*(p)-q\cdot p = D_f(q : \nabla f^*(p)) = D_{f^*}(p:\nabla f(q)).$$
\end{lm}
Lemma  \ref{thm Legendre and Bregman divergence} uses standard arguments which are common in convex optimization. We defer its proof to the appendix. Denote $H^*$ as the \textit{Legendre Transform} of the given Hamiltonian $H(x,p)$ with respect to $p$, i.e., $H^*(x,v)\triangleq \sup_{p\in \mathbb{R}^d}\{v\cdot p - H(x,p)\}$ for any fixed $x\in \mathbb R^d, v\in \mathbb R^d$. Since $H\in\mathcal{C}^1(\mathbb{R}^{2d})$ is strictly convex with respect to $p$ for arbitrary $x$, $H^*(x,v)$ is also strictly convex with respect to $v$. Both $\frac {\partial }{\partial p}H(x,\cdot)$ and $\frac {\partial}{\partial v} H^*(x,\cdot)$ are  invertible for arbitrary $x\in\mathbb{R}^d$.

 We denote $\Psi$ the set of the equivalence class, consisting of $[\psi]$ with $\psi\in  C^1([0,T]\times \mathbb R^d).$, i.e.,  
 {\small
 \begin{align*}
 [\psi]=&\Big\{\phi \in C^1([0,T]\times \mathbb R^d)\Big| \phi=\psi+c \; \text{for a spatial constant function}\;  c \in C^1([0,T]\times \mathbb R^d) \Big\}.
 \end{align*}}

We consider a regression problem given by 
\begin{align}\label{DH}
 &(D_{H}\textrm{-Regression}) \,   \min_{[\psi]\in\Psi} \left\{\mathscr{L}^{D_{H,x}}_{\rho_0, g, T}(\psi)\right\}, \\\nonumber
 &  \mathscr{L}^{D_{H,x}}_{\rho_0, g, T}(\psi)  \triangleq \mathbb{E}_{\omega}\left[\int_0^T D_{H,x}(\nabla\psi(\bX_t(\omega), t):\bP_t(\omega))dt\right],
 \label{regression 2 expectation form }
\end{align}
where we denote $D_{H,x}(q_1,q_2)=D_{H(x,\cdot)}(q_1 : q_2)$, i.e., $D_{H,x}$ is the $x$-dependent Bregman divergence regarding $H(x, \cdot)$ and $\rho_0$ is a given probability density function. Here the random process $\{\bX_t\}_{t\in [0,T]}$ is 
defined on a complete probability space $(\Omega, \mathcal{F}, P),$ satisfying the following Hamiltonian  ODE
\begin{equation}
  \begin{cases}
    \dot \bX_t = \frac {\partial}{\partial p}H(\bX_t, \bP_t),\quad \bX_0 \sim \rho_0, \\
    \dot \bP_t = -\frac {\partial} {\partial x} H(\bX_t, \bP_t), \quad \bP_0 = \nabla g(\bX_0).
  \end{cases} \label{correspd Hamiltonian System}
\end{equation}
The initial value $\bX_0$ obeys the probability distribution with the density function $\rho_0$ (denote $\bX_0\sim \rho_0$ for simplicity).
It is known (see, e.g., \cite{MR2221614}) that the phase flow of the Hamiltonian system (2.2) preserve two intrinsic properties, i.e., the Hamiltonian structure $H(\bX_t,\bP_t) = H(\bX_0,\bP_0)$, and the canonical symplectic structure on the Euclidean space, i.e.,
$$\omega (t):=\text{d}\bX_t \wedge \text{d}\bP_t =\omega(0),$$
where “d” denotes the differential with respect to the phase variables $(\bX_0,\bP_0)$. The
 preservation of $\omega(\cdot)$ means that the oriented areas of projections onto the coordinate
planes is a geometric invariant.

Notice that  $\mathscr{L}^{D_{H,x}}_{\rho_0,g,T}(\phi)=\mathscr{L}^{D_{H,x}}_{\rho_0,g,T}(\psi)$  for $\phi\in [\psi]$.  This $D_{H}\textrm{-Regression}$ functional \eqref{DH} matches the gradient $\nabla\psi(\bX_t, t)$ to the momentum $\bP_t$ with respect to the Bregman divergence induced by the Hamiltonian $H$. The computation of $D_{H}\textrm{-Regression}$ can be approximated by the Monte--Carlo method once the samples are available. 
In addition, if we denote $\mu_t$ as the joint probability distribution of $(\bX_t, \bP_t)$ solved from the Hamiltonian system \eqref{correspd Hamiltonian System} for $0\leq t \leq T$, $\rho(\cdot, t)$ is the density of the $\bX$-marginal distribution of $\mu_t$. Since $\mu_t(x,p)$ can be conveniently sampled according to the Hamiltonian ODEs \eqref{correspd Hamiltonian System},  by the Fubini's theorem,  the regression problem \eqref{DH} can be rewritten as  
\begin{align}
  \quad \min_{[\psi]\in\Psi} \left\{\int_0^T \int_{\mathbb{R}^{2d}} D_{H, x }(\nabla\psi(x, t):p)~d\mu_t(x,p)dt \right\}.
  \label{regression 1}
\end{align}

By the classical results, one may propose a loss function to enforce $\nabla \psi(\bX_t) = \bP_t$
for all $(\bX_0 , \bP_0 )$  in the phase space, like in PINNs.
 In contrast, \eqref{DH} penalises deviation from this equality by coupling with a initial
 density function $\rho_0$. Besides, one can replace the cost function $D_{H,x}$ in \eqref{DH} by any bivariate function $d: \mathbb R^d \times \mathbb R^d\to [0,+\infty)$ such that $d(x,y) = 0 \Leftrightarrow x = y$. This will not affect the analysis in section 3. Here we use the regression problem \eqref{DH}  is due to its natural connections with the Hamiltonian dynamics and a sup-inf problem (see subsections \ref{sub-hamilton}-\ref{sec-sup-inf}).
 
In section \ref{sec-3}, we will take $D_{H,x}$ as the quadratic distance $|\cdot|^2$, and  propose the following least squares problem as the loss function in the supervised learning procedure, which may make the training easier. 
{\small\begin{align}
  (\textrm{Least Squares}) \,  \min_{[\psi]\in\Psi} ~\left\{\mathscr{L}^{|\cdot|^2}_{\rho_0,g,T}(\psi)\right\},\,
   \mathscr{L}^{|\cdot|^2}_{\rho_0, g, T}(\psi)  \triangleq  \mathbb{E}_{\omega} \left[\int_0^T |\nabla\psi(\bX_t(\omega), t)-\bP_t(\omega)|^2~dt \right]. \label{regression quad}
\end{align}}This does not weaken the performance of the original problem \eqref{regression 1} since 
$D_{H, x }(q_1:q_2)\approx \frac{1}{2}(q_1-q_2)^\top \frac {\partial^2}{\partial p^2}H (x,q_2) (q_1 - q_2)$ for sufficiently close $q_1,q_2$. Furthermore, one can verify the following result.
\begin{prop}
Suppose $H(x,p) = \frac{1}{2}|p|^2+V(x)$. Then $D_{H, x }(q_1 : q_2) = \frac{1}{2}|q_1 - q_2|^2$, and the corresponding regression \eqref{regression 2 expectation form } is equivalent to the least squares formulation \eqref{regression quad}.
\end{prop} 
Next, we give a consistency result on the regression problem \eqref{regression 2 expectation form } whose proof can be found  in the supplementary material. 

\begin{tm}[Consistency]\label{thm consist}
Suppose the Hamiltonian $H\in \mathcal C^1(\mathbb R^d \times \mathbb R^d)$ satisfies the conditions that {$\frac {\partial}{\partial x}  H, \frac {\partial}{\partial p}  H $} are Lipschitz {uniformly w.r.t $x,p$}, and that $H$ is strictly convex with respect to $p$ for any fixed $x \in \mathbb{R}^d $.
Assume that $\widehat{\psi}\in\mathcal{C}^2(\mathbb{R}^d\times[0, T])$ satisfies $\mathscr{L}_{\rho_0, g, T}^{D_{H, x }}(\widehat{\psi})=0,$ then $\widehat{\psi}$ solves the following gradient-version of the Hamilton-Jacobi equation
\begin{align}
  \nabla\left(\frac{\partial}{\partial t}\widehat{\psi}(x,t) + H(x, \nabla\widehat{\psi}(x,t))\right) = 0, \quad & \textrm{at } (x,t)\in\mathbb R^d \times (0, t] \textrm{ with } x\in \mathrm{Spt}(\rho_t);\label{supported HJE}\\
  &  \textrm{and}  \; \nabla\widehat{\psi}(x, 0) = \nabla g(x) ~~\textrm{with any } x \in \textrm{Spt}(\rho_0).\nonumber
\end{align}
Similarly, $\widehat \psi$ also solves \eqref{supported HJE} if $\mathscr{L}_{\rho_0, g, T}^{|\cdot|^2}(\widehat{\psi}) = 0$.
\end{tm}

\subsection{ Extending solution beyond singularity formation of HJ equation}\label{rk weighted momentum as grad HJ solu }
We would like to point out that the solution of dynamical ODEs \eqref{correspd Hamiltonian System}, and both definitions of the regression \eqref{regression 2 expectation form } and least square problems \eqref{regression quad} can exist even after the singularity formation in the solution of HJ equation \eqref{HJ}. This means that we can use the proposed method to compute the minimizers beyond the singularity time. An interesting question is what solution the proposed method computes. To answer it, Theorem \ref{thm consist} may give us some hints as it can be used to define a weak solution of HJ equation in the following sense. By swapping the integrals in  $\mathscr{L}_{\rho_0,g,T}^{D_{H, x }}$, it holds that 
\begin{align*}
\mathscr{L}_{\rho_0,g,T}^{D_{H, x }}(\psi)& = \int_0^T\int_{{\mathbb{R}^{2d}}} D_{H, x }(\nabla\psi(x, t):p)~d\mu_t(x,p)dt \\
   & = \int_0^T\int_{\mathbb{R}^d} \left(\int_{\mathbb{R}^d} D_{H, x }(\nabla\psi(x, t):p)~d\mu_t(p|x)\right)~\rho_{t}(x)~dxdt.
\end{align*}
The minimizer $\widehat \psi$ of $\mathscr{L}_{\rho_0,g,T}^{D_{H, x }}$ can be viewed as a weak solution of the HJ equation since taking the first variation on $\psi$ leads to   
\begin{equation*}
  -\nabla\cdot\left(\rho_{t}(x) \left(\int_{\mathbb{R}^d} \nabla_{q_1}D_{H, x }(\nabla\widehat{\psi}(x,t):p)~d\mu_t(p|x)  \right)  \right)=0.
\end{equation*}
Here $\nabla_{q_1}D_{H, x }(\cdot:\cdot)$ is the partial derivative with respect to the first variable $q_1$ of $D_{H, x }(q_1 : q_2).$ In particular, if $H(x,p)=\frac{1}{2}|p|^2+V(x)$,  the minimizer of  $\mathscr{L}_{\rho_0, g, T}^{|\cdot|^2}$ solves the following elliptic equation
\begin{equation}
  -\nabla\cdot(\rho_{t}(x)(\nabla\widehat{\psi}(x,t) - \bar{p}(x,t) )) = 0. \quad \textrm{where } \bar{p}(x,t)  =  \int_{\mathbb{R}^d}  p 
 ~  d\mu_t(p|x). \quad \textrm{for } t\in[0, T].  \label{optimal psi when H is quad} 
\end{equation}
{By the integration by parts formula, \eqref{optimal psi when H is quad} yields that for any $f\in \mathcal C^1(\mathbb R^d),$
$$ \int_{\mathbb R^d}\rho_{t}(x)(\nabla\widehat{\psi}(x,t) - \bar{p}(x,t))\nabla f(x)d x=0.$$
Thus, the weighted divergence free vector field $\nabla\widehat{\psi}(x,t) - \bar{p}(x,t)$ is orthogonal to the gradient of any test function $f$ with respect to the $L^2(\rho_{t})$ inner product.
}
 To sum up, in the proposed regression problem, $\nabla\widehat{\psi}$ can be viewed as the orthogonal (with respect to the $L^2(\rho_{t})$ inner product) projection of the $\mu_t(\cdot | x)$-weighted momentum $\bar{p}(x,t)$ to the space of gradient fields.

This definition comes with several benefits. On the one hand,  Theorem \ref{thm consist} verifies that the minimizer $\widehat \psi$ solves the HJ equation \eqref{supported HJE} in the strong sense (in the gradient form) before the time $T_*$ that the classical solution develops caustics. On the other hand, the lifetime of the minimizer $\widehat \psi$ of  $\mathscr{L}_{\rho_0, g, T}^{|\cdot|^2}$ goes beyond $T_*$ since the conditional distribution $\mu_t(\cdot|x)$ on momentum is not based on the Dirac type function centered at certain positions $x$. {Although the gradient $\nabla \phi$ may be multi-valued and has information about which mono-momentum to match with in the particle level,} we treat $\widehat{\psi}$ as the $\mu_t(\cdot|x)$-weighted ``solution'' associated with the Hamilton-Jacobi equation \eqref{HJ} in this paper. However, how to theoretically understand the numerical solution after the singularity remains as an open question (see also the supplementary material), which is beyond the scope of this paper. Furthermore, by modifying the cost functional in the regression problem, one may construct different types of weak solutions of HJ equations. This is another topic that deserves further investigation and careful discussion. 

It is worth mentioning that when taking $H(x,p)=\frac{1}{2}|p|^2+V(x)$, the corresponding joint distribution $\mu(\cdot,\cdot,t)$, the marginal distribution $\rho_{t}$, and $\nabla\widehat{\psi}(\cdot,t)$, are related to the semiclassical limit of the following Schr\"{o}dinger equation
\begin{equation}
\epsilon \mathrm{i} \frac{\partial}{\partial t}\Psi^{\epsilon}(x,t) = -\frac{\epsilon^2}{2} \Delta\Psi^\epsilon(x,t) + V(x)\Psi^\epsilon(x,t). \quad \Psi^\epsilon(x,0)=\sqrt{\rho_0(x)} ~ e^{\mathrm{i}\frac{g(x)}{\epsilon}}. \label{schrodinger}
\end{equation}
Denote $\overline{\Psi}$ the conjugate of $\Psi.$
It is known that the Wigner transform $w^\epsilon(\cdot, \cdot, t)$ associated with $\Psi^\epsilon(\cdot, t)$, i.e.,
\[w^\epsilon(x,p,t) = \frac{1}{(2\pi)^d} \int \Psi^\epsilon(x+\frac{\epsilon}{2}\eta, t)\overline{\Psi^\epsilon(x - \frac{\epsilon}{2}\eta, t)}e^{\mathrm{i}p\cdot\eta}~d\eta  \]
 weak-$*$ converges to the distribution $w(\cdot, \cdot, t)$ which satisfies the Liouville equation on phase space with the initial Wigner measure $w_{in}(\cdot,\cdot)=w(\cdot, \cdot, 0)$ (see, e.g., \cite[Theorem 3.2]{jin2011mathematical}).
If $\mu_0(\cdot,\cdot)$ is  the initial Wigner measure $w_{in}(\cdot,\cdot)$, then $\mu_t(\cdot,\cdot)$ is indeed $w(\cdot,\cdot,t)$ according to the uniqueness of the solution of the Liouville equation.
Furthermore, $|\Psi^\epsilon(\cdot,t)|^2=\int w^\epsilon(\cdot,dp,t)$ weakly converges to $\rho_t$; and $\int p w^\epsilon(\cdot, dp, t)$ weakly converges to 
 $\int p w(\cdot,dp,t)= \rho_{t}(\cdot)\int p d\mu_t(p|\cdot) = \rho_{t}(x)\bar{p}(x,t).$
In particular, when dimension $d=1$, recall \eqref{optimal psi when H is quad} admits the exact solution $\nabla\widehat{\psi}(x,t) = \bar{p}(x,t)$, the flux $\rho_{t}(x)\nabla\widehat{\psi}(x,t)$ is formally the semiclassical limit of weighted momentum $\int p~ w^\epsilon(x,p,t)~dp.$
}

\subsection{Connections with Hamiltonian dynamics}
\label{sub-hamilton}
In this part, we give the connections between  the regression problem \eqref{DH} (or \eqref{regression 1}) and the HJ equation \eqref{HJ}  coupled with a continuity equation from the viewpoint of classical Hamiltonian dynamics. To explain it clearly, let us assume that 
the Hamiltonian $H:(x,p)\mapsto H(x,p)$ belongs to $\mathcal C^2(\mathbb R^d\times \mathbb R^d)$ and 
being strictly convex with respect to the second variable $p$ for arbitrary fixed first variable $x$. 
Suppose that the  solution $u$ of \eqref{HJ} exists, is unique and smooth in both time and space.

Consider the random process $\bX_t$ satisfying 
$$\dot \bX_t= {\frac {\partial}{\partial p}} H(\bX_t, {\nabla} u(\bX_t,t)).$$
Then the probability density function $\rho(\cdot,t)$ of $\bX_t$ satisfies
\begin{align}
     & \partial_t \rho(x,t) + \nabla\cdot(\rho(x,t) {\frac {\partial}{\partial p}} H(x,{\nabla}   u(x,t))) = 0, \quad \rho(\cdot, 0)=\rho_0, \label{FPE PDE system 2}
\end{align}
which is a transport (continuity) equation.  
Let us consider the dynamics of the momentum defined by $\bP_t(\omega):= {\nabla}  u(\bX_t(\omega), t)$. By taking the time derivative of $\bP_t$, we get 
\begin{equation}
     \dot\bP_t \! = \! \frac{\partial}{\partial t}{\nabla}  u(\bX_t, t) + {\nabla^2 }u(\bX_t, t)\dot\bX_t \! = \! \frac{\partial}{\partial t}{\nabla} u(\bX_t, t) + {\nabla^2}u(\bX_t, t)\frac {\partial}{\partial p}H(\bX_t,{\nabla} u(\bX_t, t)),\label{dynamic P}
\end{equation}
where ${\nabla^2u(x, t)}$ is the Hessian matrix of $u(x,t).$
If we differentiate \eqref{HJ} on both sides with respect to $x$, we have
\begin{equation}
\frac{\partial}{\partial t} {\nabla}  u(x,t) + \frac {\partial }{\partial x}H(x,{\nabla} u(x,t)) + {\nabla^2_x} u(x,t)\frac {\partial}{\partial p}H(x, {\nabla}  u(x,t)) = 0,~~ {\nabla} u(\cdot, 0) = \nabla g(x).  \label{nabla x of HJE}
\end{equation}
By setting $x=\bX_t$ in \eqref{nabla x of HJE} and substituting back into \eqref{dynamic P}, we obtain that $$\dot\bP_t = -{\frac {\partial}{\partial x} }H(\bX_t,{\nabla} u(\bX_t, t)) = -{\frac {\partial}{\partial x}} H(\bX_t, \bP_t).$$
{To sum up, we have recovered the constrained Hamiltonian ODE system \eqref{correspd Hamiltonian System} in the regression problem \eqref{DH} and find that the minimizer of \eqref{DH} is the solution $u$ of \eqref{HJ}.}

It should be pointed out that \eqref{HJ} coupled with the continuity equation \eqref{FPE PDE system 2}  is related to the WHF  introduced in \cite{chow2020wasserstein}. Indeed, by rewriting \eqref{nabla x of HJE} as 
$${\nabla}  \Big(\frac{\partial}{\partial t}  u(x,t) + H(x,{\nabla} u(x,t)) \Big) = 0$$
(see also \cite[Page 1020]{MR4405488}), one can obtain a coupled system of PDEs corresponding to the particle system \eqref{correspd Hamiltonian System},
\begin{align}
    & \partial_t {\rho}(x,t) + \nabla\cdot({\rho}(x,t){\frac {\partial}{\partial p} }H(x,\nabla \widehat u(x,t))) = 0, \quad \rho(\cdot, 0)=\rho_0; \label{FPE PDE system 3}\\
  & \partial_t \widehat u(x,t) + H(x,\nabla \widehat u(x,t)) = 0,~ \widehat{u}(\cdot,0) = g(\cdot),
  \label{HJE PDE system 3}
\end{align}
where 
$\widehat u(x,t)=u(x,t)+c(t)$
for any arbitrary $c(\cdot)\in \mathcal{C}^1([0, T])$. The phase flow
of  \eqref{FPE PDE system 3}-\eqref{HJE PDE system 3} preserves the Hamiltonian structure $\int_{\mathbb R^d}H(x,\nabla \widehat u(x)) \rho(x) dx$, and the infinite-dimensional symplectic structure (see, e.g., \cite{MR997295}), i.e.,
$$\widehat \omega(t):=\int_{\mathbb R^d} \text{d} \rho(x,t) \wedge \text{d} \widehat u(x,t) dx=\widehat \omega(0).$$ 
In particular, when {$H(x,p)=\frac 12|p|^2$}, the coupled system \eqref{FPE PDE system 3}-\eqref{HJE PDE system 3} is the Wasserstein geodesic equation \cite{villani2009optimal}, which is the critical point of the Benamou-Brenier formula defining the OT distance on Wasserstein manifold \cite{benamou2000computational}. 
 
This approach of coupling offers additional freedom in choosing the initial density $\rho_0$ which ultimately controls the support of the coupled density $\mathrm{Spt}(\rho(\cdot,t))$, hence where and how the samples $(\{\bX_t \},\{\bP_t\})$ are drawn. As a by-product, solving \eqref{FPE PDE system 3}-\eqref{HJE PDE system 3} on $\mathrm{Spt}(\rho(\cdot,t))$, can recover the solution of original Hamiltonian--Jacobi equation \eqref{HJ} up to a spatial constant function. It should be noticed that the solution solved by  \eqref{FPE PDE system 3}-\eqref{HJE PDE system 3} is consistent with the classical solution of \eqref{HJE PDE system 3} when $T<T_*$ with $T_*$ being the first time that \eqref{HJE PDE system 3} develops a singularity. On the other hand, the Hamiltonian system \eqref{correspd Hamiltonian System} is always well-posed even if the PDE \eqref{HJE PDE system 3} does not admit classical solutions. This inspires us to design a new way to learn the solution of \eqref{HJ} even beyond the singularity.
 
{
\subsection{Connections with a variational problem}
\label{sec-sup-inf}

In this part, we show the connection between \eqref{DH} and the following sup-inf problem
\begin{equation} 
  \sup_{\psi\in\mathcal C^1}\;\inf_{\widetilde\rho\in AC([0, T];\mathcal{P}_2)}\;\{\mathscr{J}_{\rho_0, \rho_T, T}(\widetilde \rho, \psi)\},  \label{original supinf problem}  
\end{equation}
here $AC([0, T];\mathcal{P}_2)$ denotes the space of absolute continuous trajectories in the probability space equipped with 2-Wasserstein distance. We define
\begin{equation}
\label{original saddle functional J}
\begin{split}
  \mathscr{J}_{\rho_0, \rho_b, T} (\widetilde\rho, \psi) = & \int_0^T\int_{\mathbb R^d} -(\partial_t\psi(x,t) +H(x,\nabla\psi(x,t)))\widetilde\rho(x,t)~dx dt \\
  & + \int_{\mathbb{R}^d} \psi(x, T) \rho_T(x)~dx- \int_{\mathbb{R}^d} \psi(x, 0)\rho_0(x)~dx.
\end{split}
\end{equation}
This formulation originates from the optimal transport associated with the initial density $\rho_0 = \rho(\cdot, 0)$ and target $\rho_T = \rho(\cdot, T)$ (see, e.g., \cite{benamou2000computational, villani2009optimal, ambrosio2005gradient, chow2020wasserstein} and references
therein).
Here we use $\widetilde \rho$ as variable of the functional so as to distinguish it from the solution $\rho$ to the continuity equation \eqref{FPE PDE system 3}.
By taking the first variation of \eqref{original saddle functional J}
 w.r.t $\psi,\widetilde \rho $, by standard arguments, one can show that  \eqref{FPE PDE system 3}-\eqref{HJE PDE system 3} (g is unknown
since $\rho_T$ is given) forms the critical point of \eqref{original supinf problem}. In fact,  \eqref{FPE PDE system 3} can be derived from  \eqref{original saddle functional J} by taking the first variation of  \eqref{original saddle functional J} w.r.t $\psi$ and using the integration by parts formula.

We want to point out that in the standard OT formulation, the terminal density $\rho_T$ is given independently. This is different from  \eqref{FPE PDE system 3}-\eqref{HJE PDE system 3} considered in subsection \ref{sub-hamilton} where $g$ is given. Since $\rho(x,t)$ is the probability density of $\boldsymbol{X}_t$ given by the Hamiltonian system \eqref{correspd Hamiltonian System} on $[0, T]$, which is uniquely determined by the initial conditions $\rho_0$ and $g$. It implies that $\rho_T=\rho(x,T)$ is also determined by $\rho_0$ and $g$. For this reason, we use notation $\mathscr{L}_{\rho_0, g, T}(\psi)$ in \eqref{def L reduced} to emphasize the dependence on $g$. If we directly replace $\widetilde \rho$ in \eqref{original saddle functional J} by the optimal density $\rho$, then \eqref{original supinf problem} can be rewritten as the following optimization problem
\begin{equation}
  \sup_{[\psi]\in\Psi
  } \; \{\mathscr{L}_{\rho_0, g, T}(\psi) \}, \label{variational problem}
\end{equation}
where
\begin{align}\mathscr{L}_{\rho_0,g,T}(\psi) = & \int_0^T\int_{\mathbb{R}^d} -\left(\partial_t \psi(x,t) + H(x, \nabla\psi(x,t) )\right) \rho_t(x)~dx dt \nonumber \\
  & + \int_{\mathbb{R}^d} \psi(x, T) \rho_T(x)~dx- \int_{\mathbb{R}^d} \psi(x, 0)\rho_0(x)~dx. \label{def L reduced}
\end{align}}

\begin{lm}\label{lemma: calculation L equals regression funcitonal}
  Suppose that $T > 0$ is the given terminal time, and that the Hamiltonian $H\in \mathcal{C}^1(\mathbb{R}^d\times \mathbb{R}^d)$  is 
  {strongly convex with respect to the momentum $p$ for any $x\in\mathbb{R}^d$}. Assume $\rho_0\in\mathcal{C}^1(\mathbb{R}^d)$ and $g\in\mathcal{C}^1(\mathbb{R}^d)$. Then
  \begin{equation}
  \footnotesize
  \mathscr{L}_{\rho_0, g, T}(\psi) = - \int_0^T\int_{\mathbb{R}^{2d}} D_{H,x}(\nabla\psi(x,t):p)~d\mu_t(x,p) dt + \int_0^T\int_{\mathbb{R}^{2d}}    H^*(x, \nabla_pH(x, p))~d\mu_t(x,p) dt, \label{mathscrL = regression}
  \end{equation}
  {where we denote $D_{H,x}(q_1 : q_2) = D_{H(x, \cdot)}(q_1 : q_2)$, i.e., $D_{H,x}$ is the $x$-dependent Bregman divergence regarding $H(x,\cdot)$.}
\end{lm}

{The proof of Lemma \ref{lemma: calculation L equals regression funcitonal} is done by direct calculation of
$ {\mathscr{L}_{\rho_0, g, T}(\psi)}$ and  Lemma  \ref{thm Legendre and Bregman divergence}. For completeness, we provide the proofs in the supplementary material.} The second term on the right-hand side of \eqref{mathscrL = regression} does not involve $\psi$, and thus can be treated as a constant, {which implies that the variational problem \eqref{def L reduced} is equivalent to the regression problem \eqref{DH}.}

\section{Supervised learning scheme via density coupling}
\label{sec-3}

In this section, we present the supervised learning scheme based on the density coupling strategy and the regression formulation \eqref{regression 2 expectation form }. 

\subsection{Algorithm}\label{sub : alg}
Our method for computing the Hamilton-Jacobi equation \eqref{HJ} associated with the probability density distribution $\rho_0$ consists of the following two main steps.
\begin{itemize}
    \item (Generating sample trajectories on phase space) Sample $N$ particles $\{x_0^{(k)}\}_{k=1}^N$ from $\rho_0$ with momentum $p_0^{(k)} = \nabla g(x_0^{(k)})$, and apply a suitable geometric integrator to solve the Hamiltonian system 
    \begin{equation}
        \label{alg: Hamilton system}
        \begin{split}
          &  \dot{x}_t^{(k)}  = \partial_p H(x_t^{(k)}, p_t^{(k)})  \\
          &\dot{p}^{(k)}_t  = -\partial_x H(x_t^{(k)}, p_t^{(k)})  
        \end{split} \quad \textrm{with initial condition } (x_0^{(k)}, \nabla 
          g(x_0^{(k)})).
    \end{equation}
    at time steps $t_i = ih$, with $h=\frac{T}{M}$, $1\leq i \leq M$ for each $k\in\{1,2,...,N\}$. We denote the numerical solutions at  $t_i$ as $\{(\tilde{x}_{t_i}^{(k)},\tilde{p}_{t_i}^{(k)})\}$, $1\leq k\leq N$.
    \item (Compute $\psi$ via supervised learning) Set up the neural network $\psi_\theta:\mathbb{R}^d\times [0, T]\rightarrow \mathbb{R}$, and minimize the sum of average discrepancies between each $\nabla_x\psi_\theta(\tilde{x}_{t_i}^{(k)}, t_i)$ and $\tilde{p}_{t_i}^{(k)}$ at each time step $t_i$ evaluated on a random batch $\{\widetilde{x}^{(k_j)}\}_{j=1}^{N_0} \subset \{\widetilde{x}^{(k)}\}$ with batchsize $N_0$. More precisely, we denote   
{\begin{equation} 
  \textrm{Loss}(\theta) = \frac{1}{M}\sum_{i=1}^M \left(\frac{1}{N_0}\sum_{k=1}^{N_0} |\nabla_x\psi_\theta(\tilde{x}_{t_i}^{(k)}, t_i)-\tilde{p}_{t_i}^{(k)}|^{  2  }\right). \label{ Loss}
\end{equation}}
 We apply stochastic gradient descent algorithms such as Adam's method \cite{kingma2014Adam} to minimize $\textrm{Loss}(\theta)$ with respect to the parameter $\theta$ in $\psi_\theta$.
We summarize our method in Algorithm \ref{alg1}.

\end{itemize}
    \begin{algorithm}[htb!]
    \caption{Computing the gradient field of Hamilton-Jacobi equation \eqref{HJ} associated with initial density function $\rho_0$. }\label{alg1}
    \begin{algorithmic}
    \State Set up neural network $\psi_\theta:\mathbb{R}^d\times [0,T]\rightarrow \mathbb{R}$; {He initialization \cite{he2015delving} is used to initialize the parameter $\theta$.}
    \State Sample $\{x_0^{(k)}\}_{k=1}^N$ from $\rho_0$;
    \State Apply a suitable geometric integrator to solve the Hamiltonian system \eqref{alg: Hamilton system} with initial condition $x_0=x_0^{(k)}, p_0 = \nabla g(x_0^{(k)})$ to obtain the trajectory $(\tilde{x}_{t_i}^{(k)}, \tilde{p}_{t_i}^{(k)})$ at time steps $0\leq t_1\leq \dots \leq t_M=T$ for each $k$, $1\leq k \leq N$.
    \For{    $\textrm{Iter} = 0$ to $ N_{\textrm{Iter}}$}
       \State Pick random batch with size $N_0 \leq N$ from $\{\widetilde{x}^{(k)}\}$;
       \State Evaluate $\textrm{Loss}(\theta)$ defined as in \eqref{ Loss};
       \State Apply Adam's method with learning rate $lr$ to perform gradient descent $$\theta \leftarrow \theta -lr  \; \nabla_\theta \textrm{Loss}(\theta);$$
    \EndFor
    \State $\nabla_x\psi_\theta(\cdot, t)$ ($0\leq t \leq T$) is the computed gradient field of the Hamilton-Jacobi equation \eqref{HJ}.
    \end{algorithmic}
    \end{algorithm}
{ In our algorithm, we have the freedom to choose the geometric integrator to discretize the Hamiltonian system \eqref{correspd Hamiltonian System}. There are various choices such as symplectic Runge--Kutta schemes, symplectic partitioned Runge--Kutta Methods, Str\"{o}mer--Verlet scheme, etc. We refer interested readers to \cite{MR2221614} and references therein for further details. Such structure-preserving methods could preserve the properties, such as symplectic structure and quadratic conservative quantities, of the original system as much as possible \cite{MR4405488}. We also would like to mention that in the generating procedures, one may replace the symplectic integrator as other high-order nonsymplectic integrators if the first time that the classical solution develops the singularity is not long. To further explore the difference caused by the symplectic scheme and high-order non-symplectic scheme, we have included an extended discussion in section \ref{subsec-4.4}.  However, it remains challenging to ascertain which scheme is superior and warrants further investigation. 
    
A few observations have been made during our implementation of the proposed algorithm. 

First, {Theorem \ref{thm consist} suggests that both the regression problem \eqref{regression 2 expectation form } and \eqref{regression quad} are consistent with respect to equation \eqref{supported HJE}. However, in practice, to perform the supervised learning in an efficient and stable way, one needs to avoid the case in which the Hessian (with respect to $p$) of the Hamiltonian $H$ possesses a large conditional number. {We adopt the least squares regression \eqref{regression quad} and use the quadratic loss \eqref{ Loss} in our implementation.}} 

Second, it may be difficult for a single neural network to learn the solution on the entire time interval $[0, T]$, especially when $T$ is large or when the solution experiences large-scale oscillations.  In such cases, in order to improve the performance of our method, we split the time interval $[0, T]$ into smaller sub-intervals, train different $\psi_\theta$ on each sub-interval respectively, and then concatenate the solution together. We refer the reader to section \ref{example: harmonic} for further details.

Third, we may re-sample the points $\{x_0^{(k)}\}_{1\leq k\leq N}$ from $\rho_0$ and repeat the procedure in each training iteration to update $\theta$. According to our experience, such a strategy produces numerical solutions with similar quality compared to that computed by the method with fixed samples throughout the simulations. }
    
\subsection{Bound on the residual}\label{sec: bound res }
In this part, we estimate the density weighted residual of the numerical solution $\psi_\theta$ produced from the proposed algorithm. Let us denote $\tilde{\Phi}_h:\mathbb{R}^{2d}\rightarrow \mathbb{R}^{2d}$ as the solution map of the chosen geometric integrator for \eqref{correspd Hamiltonian System}, and 
\begin{equation*}
    (\tilde{x}_{t_i}, \tilde{p}_{t_i}) = \tilde{\Phi}_h^{ (i) }(x_0,\nabla g (x_0)) \triangleq \underbrace{\tilde{\Phi}_h\circ \dots \circ\tilde{\Phi}_h}_{i ~\tilde{\Phi}_h  \textrm{s composing together} }(x_0, \nabla g (x_0)), 
\end{equation*}
where  the stepsize $ h=\frac{T}{M}$, $(\tilde{x}_{t_i}, \tilde{p}_{t_i})$ is the \textit{numerical solution} solved at time $t_i=ih$ with initial condition $x_0$ and $p_0=\nabla g(x_0)$.  We denote $\tilde{\rho}_{t_i}$ the probability density of random variable $\tilde x_{t_i}$. Let $r\geq 2$ be the order of the local truncation error of numerical solver $\tilde{\Phi}_h${\footnote{Suppose $(x_h, p_h)$ is the exact solution of \eqref{residual est: Hamilton system} with initial condition $(x_0, p_0)$ after one time step $h$, then 
\begin{equation}
  |\tilde{\Phi}_h(x_0, p_0) - (x_h, p_h)|=C_{\tilde{\Phi}_h}(x_0, p_0)h^r,  \label{num schm solver Phi order r}
\end{equation}
where $C_{\tilde{\Phi}_h}((x_0, p_0))$ is a constant only depending on the Hamiltonian $H$, the initial condition $(x_0,p_0)$, and the numerical scheme.  }.} Correspondingly, we denote $\Phi_t:\mathbb{R}^{2d}\rightarrow\mathbb{R}^{2d}$ as the flow map of the Hamiltonian system 
\begin{equation}
    \label{residual est: Hamilton system}
      \dot{x}_t  = \partial_p H(x_t, p_t),  \quad
      \dot{p}_t  = -\partial_x H(x_t, p_t),
\end{equation}
i.e., $\Phi_t((x_0, p_0)) = (x_t, p_t)$ for $t\in [0,T]$.

{
Let the neural network approximation $\psi_{\theta}$ be trained by minimizing the loss \eqref{ Loss}
with data generated by a numerical integrator of order $r$ for \eqref{alg: Hamilton system} with initial sample  $\{x_{t_0}^{(k)}\}_{k\le N}$ drawn from $\rho_0$. In this subsection, assume that the approximation $\psi_\theta$ is
regular enough, for example $\psi_{\theta}\in \mathcal C^{1,2}([0,T]\times \mathbb R^d)\cap \mathcal C^{0,3}([0,T]\times \mathbb R^d)$.
}
We consider the loss vector of the supervised learning at each sample point as
\begin{equation}
  e_{t_i}^{(k)} =  \nabla\psi_\theta(\tilde{x}_{t_i}^{(k)}, t_i) - \tilde{p}_{t_i}^{(k)}  .  \label{def every particle error}
\end{equation} Let us set 
\begin{equation}
{\varepsilon_i^N}=\frac{1}{N}\sum_{k=1}^N {|e_{t_{i}}^{(k)}|} \quad \text{and} \quad {\delta_i^{N,h}}= \frac{1}{N}\sum_{k=1}^N \frac {|e_{t_{i+1}}^{(k)}-e_{t_{i}}^{(k)}|}{h}\label{def: epsilon N and delta N }
\end{equation}
as the empirical average of the training loss and its  difference quotient  at time node $t_i$, respectively. We note that when $\nabla\psi_{\theta}$ is Lipschitz on the support of the probability density function,  {$e_{t_{i}}^{(k)}$ is continuous with respect to {$t_i$}. In particular, if there is no training error (i.e., $e_{t_i}^{(k)}=0$), we have ${\varepsilon_i^N}={\delta_i^{N,h}}=0.$} Our estimate on the $L^1$-residual of $\nabla \psi_\theta$ is presented in the next theorem. 

\begin{tm}[Posterior estimation on $L^1$ residual of Hamilton-Jacobi equation]\label{thm est}Suppose that $\frac{\partial H}{\partial p}$ and $\frac{\partial H}{\partial x}$ are Lipschitz  {(uniformly w.r.t $x,p$)} with constants $L_1$ and $L_2$ respectively, the initial distribution $\rho_0$ has a compact support, $\epsilon \in (0,1)$ is a given constant, $M$ is large enough such that $M \geq  \max\{T, \frac{T}{2}(L_1+L_2)e^{L_1+L_2}\}$, and the time stepsize is taken as $h=\frac{T}{M}$. Assume that the neural network $\psi_\theta$ is trained by minimizing the loss \eqref{ Loss} with data generated by a numerical integrator of order $r$ for \eqref{alg: Hamilton system} with initial samples $\{x_{t_0}^{(k)}\}_{k=1}^N$ drawn from $\rho_0$.
Then with probability $1-\epsilon$,  $\psi_\theta$ satisfies
\begin{align}
    \int_{\mathbb{R}^d}  &  \left|\nabla\left(\frac{\partial}{\partial t}\psi_\theta(x, t_i )  +  H(x, \nabla\psi_\theta(x,t_i))\right)\right| ~\tilde{\rho}_{t_i}(x) dx \nonumber\\
     & \leq  ~ \frac{1}{2}\lambda(\theta, i )h + \eta(\theta, {i}  )h^{r-1} +\delta_i^{N,h}+\nu(\theta,i)\varepsilon_{i}^N +R(\theta, i)\sqrt{\frac{\ln M + \ln\frac{2}{\epsilon}}{2N}} ,   \label{thm final ineq}
\end{align}
at $t_i=ih$, $i=1,\dots, M$. Here, $\lambda(\theta, i),\eta(\theta, i),\nu(\theta, i), R(\theta, i)$ are non-negative constants depending on the parameter $\theta$, time node $t_i$, Hamiltonian $H$, initial distribution $\rho_0$, and numerical scheme  $\tilde{\Phi}_h$.
\end{tm}

\begin{proof}
Let us  focus on the $k$-th trajectory $\{(\tilde{x}_{t_i}^{(k)}, \tilde{p}_{t_i}^{(k)} )\}_{i=0}^M$. At time node $t_i$, $i\le M-1,$
we denote 
\begin{equation*}
  (\widehat{x}_{\tau}^{(k)}, \widehat{p}_{\tau}^{(k)}) = \Phi_\tau(\tilde{x}_{t_i}^{(k)}, \tilde{p}_{t_i}^{(k)}), \quad \tau\geq 0.
\end{equation*}
For simplicity, we omit the superscript $(k)$ of each $(\tilde{x}_{t_i}^{(k)}, \tilde{p}_{t_i}^{(k)} )$, $(x_{t}^{(k)}, p_{t}^{(k)})$, $(\widehat{x}_\tau^{(k)}, \widehat{p}_\tau^{(k)})$ and $e_i^{(k)}$. We start by considering
\begin{align}
  \nabla\psi_\theta(\tilde{x}_{t_{i+1}}, t_{i+1}) - \nabla\psi_\theta(\tilde{x}_{t_i}, t_i) = \tilde{p}_{t_{i+1}} - \tilde{p}_{t_i} + (e_{i+1} - e_i)\label{pf_eq1}
\end{align}
The left-hand side of \eqref{pf_eq1} can be recast as
\begin{align*}
(\nabla\psi_\theta(\widehat{x}_h, t_{i+1}) - \nabla\psi_\theta(\tilde{x}_{t_i}, t_i)) + (\nabla\psi_\theta(\tilde{x}_{t_{i+1}}, t_{i+1}) - \nabla\psi_\theta(\widehat{x}_h, t_{i+1})),
\end{align*}
where the first term can be formulated as
{\small
\begin{align*}
  \nabla\psi_\theta(\widehat{x}_h, t_{i+1}) - \nabla\psi_\theta(\tilde{x}_{t_i}, t_i)&= \int_0^h \frac{d}{d\tau}\nabla\psi_\theta(\widehat{x}_\tau, t_i + \tau)~d\tau\\ 
  & = \int_0^h \nabla^2\psi_\theta(\widehat{x}_\tau,t_i+\tau)\frac{\partial}{\partial p}H(\widehat{x}_\tau, \widehat{p}_\tau) +\frac{\partial}{\partial t}\nabla\psi_\theta(\widehat{x}_\tau, t_i+\tau)~d \tau.
\end{align*}}
For the second equality, we recall that $\dot{\widehat{x}}_\tau = \frac{\partial}{\partial p}H(\widehat{x}_\tau, \widehat{p}_\tau)$.

On the other hand, the right-hand side of \eqref{pf_eq1} can be formulated as
\begin{align*}
  (\widehat{p}_h - \tilde{p}_{t_i}) + (\tilde{p}_{t_{i+1}} - \widehat{p}_h) + (e_{i+1} - e_i),
\end{align*}
where the first term can be rewritten as
\begin{equation*}
  \widehat{p}_h - \tilde{p}_{t_i} = \int_0^h \dot{\widehat{p}}_\tau ~d\tau = \int_0^h - \frac{\partial}{\partial x}H(\widehat{x}_\tau,  \widehat{p}_\tau)~d\tau.
\end{equation*}
Combining the previous calculations, we obtain 
\begin{align}
  & \int_0^h \frac{\partial}{\partial t}\nabla\psi_\theta(\widehat{x}_\tau, t_i+\tau) + \nabla^2\psi_\theta(\widehat{x}_\tau, t_i+\tau)\frac{\partial}{\partial p}H(\widehat{x}_\tau, \widehat{p}_\tau) + \frac{\partial}{\partial x}H(\widehat{x}_\tau,  \widehat{p}_\tau)~d \tau \nonumber \\
  & =  (\nabla\psi_\theta(\widehat{x}_h, t_{i+1}) - \nabla\psi_\theta(\tilde{x}_{t_{i+1}}, t_{i+1})) + (\tilde{p}_{t_{i+1}} - \widehat{p}_h) + (e_{i+1} - e_i).\label{main eq }
\end{align}
Let us denote $R_{t_i}=\underset{1\leq k \leq N}{\max}\{|\tilde{x}_{t_i}^{(k)}|+|\tilde{p}_{t_i}^{(k)}|\}$, $L=L_1+L_2$ and $C = |\partial_pH(0,0)|+|\partial_xH(0,0)|$. We first claim that 
\begin{equation}\label{bound-control-particle}
  |\widehat{x}_\tau - \tilde{x}_{t_i}|+|\widehat{p}_\tau - \tilde{p}_{t_i}|\leq 2(LR_{t_i}+C+1)\tau.
\end{equation}
We refer to section \ref{appendix: proof of thm 3.1} for its proof. Denote the time-space region $E_i\subset\mathbb{R}^d\times\mathbb{R}_+$ as
\begin{equation}\label{def-ei}
  E_i = \{(y,s)~|~|y|\leq R_{t_i}+(LR_{t_i}+C+1)h, ~ t_i\leq s \leq t_{i+1}\}.
\end{equation}
Notice that $(\widehat{x}_\tau, t_i+\tau)\in E_i$ for any $0\leq\tau\leq h$. We define
\begin{equation}
  L_{\theta, i}^A = \textrm{Lip}_{E_i}(\partial_t\nabla\psi_\theta) \triangleq \sup_{(y,s),(y',s')\in E_i} \frac{|\partial_t\nabla\psi_\theta(y,s)-\partial_t\nabla\psi_\theta(y',s')|}{|y-y'|+|s-s'|}  \label{lip partial_t nabla psi},
\end{equation}
i.e., $L_{\theta, i}^A$ as the Lipschitz constant of vector function $\partial_t\nabla\psi_\theta(x,t)$ on $E_i$. Then we have 
\begin{equation}
   |\partial_t\nabla\psi_\theta(\widehat{x}_\tau, t_i+\tau) - \partial_t\nabla\psi_\theta(\tilde{x}_{t_i}, t_i)|\leq L_{\theta, i}^A(|\widehat{x}_\tau - \tilde{x}_{t_i}|+\tau)\leq L_{\theta, i}^A\ 3(LR_{t_i}+C+1)h.  \label{lip dist 1}
\end{equation}
Let us denote
\begin{equation}
M_{\theta, i}= \sup_{ x \in \textrm{supp}(\tilde{\rho}_{t_i})}\|\nabla^2\psi_\theta(x, t_i)\|,\label{notation M_theta, i}
\end{equation}
and
\begin{equation}
  L_{\theta, i}^B = \textrm{Lip}_{E_i}(\nabla^2\psi_\theta) \triangleq \sup_{(y,s),(y',s')\in E_i} \frac{\| \nabla^2\psi_\theta(y,s)- \nabla^2\psi_\theta(y',s')\|}{|y-y'|+|s-s'|},   \label{notation lip nabla2 psi}
\end{equation}
where $\|\cdot\|$ is the $2$-norm of the square matrix.

For convenience, we introduce  
\[
    \mathscr{D}\psi_\theta(x,p,t) = \frac{\partial}{\partial t}\nabla\psi_\theta(x, t) + \nabla^2\psi_\theta(x, t)\frac{\partial}{\partial p}H(x, p) + \frac{\partial}{\partial x}H(x, p).
\]
Then we can bound 
\begin{equation}
    |\mathscr{D}\psi_\theta(\widehat{x}_\tau, \widehat{p}_\tau, t_i+\tau) - \mathscr{D}\psi_\theta(\tilde{x}_{t_i}, \tilde{p}_{t_i}, t_i)|\leq \lambda(\theta, i )\tau,  \label{lip dist combine}
\end{equation}
where the positive constant $\lambda(\theta, i)$ is shown in section \ref{appendix: proof of thm 3.1}.

We reformulate \eqref{main eq } as
\begin{align*}
h\mathscr{D}\psi_\theta(\tilde{x}_{t_i}, \tilde{p}_{t_i}, t_i) = & ~ \int_0^h \mathscr{D}\psi_\theta(\tilde{x}_{t_i}, \tilde{p}_{t_i}, t_i)-\mathscr{D}\psi_\theta(\widehat{x}_{\tau}, \widehat{p}_{\tau}, t_i+\tau) ~ d\tau   \\
  & + (\nabla\psi_\theta(\widehat{x}_h, t_{i+1}) - \nabla\psi_\theta(\tilde{x}_{t_{i+1}}, t_{i+1})) + (\tilde{p}_{t_{i+1}} - \widehat{p}_h) + (e_{i+1} - e_i).
\end{align*}
We have the following estimate 
{\small\begin{align}
  |\mathscr{D}\psi_\theta(\tilde{x}_{t_i}, \tilde{p}_{t_i}, t_i)| & \leq \frac{1}{h}\int_0^h |\mathscr{D}\psi_\theta(\widehat{x}_\tau, \widehat{p}_\tau, t_i+\tau) - \mathscr{D}\psi_\theta(\tilde{x}_{t_i}, \tilde{p}_{t_i}, t_i)|~d\tau \label{main estimate}  \\
  & \quad + \frac{1}{h}|\nabla\psi_\theta(\widehat{x}_h, t_{i+1}) - \nabla\psi_\theta(\tilde{x}_{t_{i+1}}, t_{i+1})| + \frac{|\tilde{p}_{t_{i+1}}-\widehat{p}_h|}{h}+\frac{|e_{i+1}-e_i|}{h}.    \nonumber
\end{align}}
Let us define
\begin{equation}\label{def-di}
  D_i = \{ x ~ | ~ |x|\leq R_{t_i} + 3(LR_{t_i} + C + 1)h ~ \}.
\end{equation}
and
\begin{equation}
  L_{\theta, i}^C = \textrm{Lip}(\nabla\psi_\theta(\cdot, t_i)) \triangleq \sup_{y, y' \in D_i } \frac{| \nabla\psi_\theta(y,t_i)- \nabla\psi_\theta(y',t_i) |}{|y-y'|}.   \label{notation lip x nabla psi} 
\end{equation} 
This, together with \eqref{main estimate}, leads to (see the proof in section \ref{appendix: proof of thm 3.1})
\begin{align}\label{over-est-sample}
  & \left|\nabla\left(\frac{\partial}{\partial t}\psi_\theta(\tilde{x}_{t_i}, t_i ) + H(\tilde{x}_{t_i}, \nabla\psi_\theta(\tilde{x}_{t_i}))\right)\right|\\\nonumber
  \leq & ~  \frac{1}{2}\lambda(\theta, i )h + (L_{\theta, i+1}^C+1)C_{\tilde{\Phi}_h}(\tilde{x}_{t_i}, \tilde{p}_{t_i})h^{r-1} + \frac{|e_{i+1}-e_i|}{h}+(M_{\theta, i}L_1+L_2)e_i.
\end{align}
We finally take average over the sample points $\{\tilde{x}_{t_i}^{(k)}\}_{1\leq k\leq N}$. This leads to
{
\begin{align}
  & \frac{1}{N}\sum_{k=1}^N \left|\nabla\left(\frac{\partial}{\partial t}\psi_\theta(\tilde{x}_{t_i}^{(k)}, t_i ) + H(\tilde{x}_{t_i}^{(k)}, \nabla\psi_\theta(\tilde{x}_{t_i}^{(k)}))\right)\right| \nonumber\\
  \leq & ~ \frac{1}{2}\lambda(\theta, i )h + \underbrace{(L_{\theta, i+1}^C+1)\frac{1}{N}\sum_{k=1}^N C_{\tilde{\Phi}_h}(\tilde{x}_{t_i}^{(k)}, \tilde{p}_{t_i}^{(k)})}_{\textrm{denote as} ~ \eta(\theta, i)}h^{r-1} 
 \label{main ineq}  \\
  & + \frac{1}{N}\sum_{k=1}^N\frac{|e_{i+1}^{(k)}-e_i^{(k)}|}{h}+\underbrace{(M_{\theta, i}L_1+L_2)}_{\textrm{denote as}~\nu(\theta, i)}|e_i^{(k)}|.  \nonumber 
\end{align}
}
This provides an upper bound on the empirical average of the $L^1$-residual of $\psi_\theta$ using the computed samples $\{\tilde{x}_{t_i}^{(k)}\}_{1\leq k\leq N}$ at time node $t_i$.

To further estimate the expectation of the $L^1$-residual at all the time nodes $\{t_1,\dots, t_T\}$, let us denote $\tilde{\rho}_{t_i}=(\tilde{\Phi}_h\circ\dots\circ \tilde{\Phi}_h)_\sharp \rho_0$ as the probability density function of the numerical solution $\tilde{x}_{t_i}$ computed by the chosen scheme starting from $x_0\sim\rho_0$. For simplicity, let us denote the residual term of the Hamilton-Jacobi equation as
\begin{equation*}
  \mathcal{R}[\psi_\theta](x,t)=\nabla\left(\frac{\partial}{\partial t}\psi_\theta(x, t) + H(x, \nabla\psi_\theta(x,t))\right).
\end{equation*}

For a fixed time $t_i$ and samples $\{\tilde{x}_{t_i}^{(k)}\}_{1\leq k \leq N }\sim \tilde{\rho}_{t_i}$, by Hoeffding's inequality (see e.g. \cite{SSBD2014}), for any $0<\delta<1$, with probability $1-\delta$, we can bound the gap between the expectation and the empirical average of the $L^1$ residual as
\begin{equation}  \left|\int_{\mathbb{R}^d} |\mathcal{R}[\psi_\theta](x,t_i)| \tilde{\rho}_{t_i}~dx - \frac{1}{N}\sum_{k=1}^N |\mathcal{R}[\psi_\theta](\tilde{x}_{t_i}^{(k)},t_i)|\right|\leq \underbrace{\sup_{x\in\textrm{supp}({\tilde{\rho}_{t_i}})}|\mathcal{R}[\psi_\theta](x,t_i)|}_{\textrm{denote as}~ R(\theta,i)}\sqrt{\frac{\ln\frac{2}{\delta}}{2N}}.\label{hoeffding ineq}
\end{equation}

Since we assume that $\textrm{supp}(\rho_0)$ is a bounded set, and the solution map $\tilde{\Phi}_h$ of the numerical scheme is continuous, then $\textrm{supp}(\tilde{\rho}_{t_i})$ is also bounded. Thus $R(\theta, i)$ is guaranteed to be finite.

By combining \eqref{main ineq} and \eqref{hoeffding ineq}, for any time node $t_i$, with probability $1-\delta$, we can estimate the average $L^1$ residual of Hamilton-Jacobi equation at time $t_i$ as
\begin{align}
    & \int_{\mathbb{R}^d}   \left|\nabla\left(\frac{\partial}{\partial t}\psi_\theta(x, t_i ) + H(x, \nabla\psi_\theta(x,t_i))\right)\right| \tilde{\rho}_{t_i}~dx\nonumber\\
    \leq & ~ \frac{1}{2}\lambda(\theta, i )h + \eta(\theta,i)h^{r-1} + \left( \frac 1 N \sum_{k=1}^N\frac{|e_{i+1}^{(k)} - e_i^{(k)}|}{h} + \nu(\theta, i)|e_i^{(k)}| \right) +R(\theta, i)\sqrt{\frac{\ln\frac{2}{\delta}}{2N}}.\label{expectation residual ineq}
\end{align}
If we denote the subset $\Omega_{t_i}$ of the sample space on which \eqref{expectation residual ineq} holds. It follows that $\mathbb{P}(\Omega_{t_i}^c)\leq \delta$. Then we have   
\begin{equation*}
  \mathbb{P}\;\left(\bigcap_{i=1}^M   \Omega_{t_i} \right) = 1- \mathbb{P}\; \left( \bigcup_{i=1}^M \Omega_{t_i}^c \right) \geq 1-\sum_{i=1}^M \mathbb{P} \left(\Omega_{t_i}^c\right)\geq 1 - M\delta.
\end{equation*}
By letting $M\delta=\epsilon $, we have shown that for the fixed neural network $\psi_\theta$, initial distribution with density $\rho_0$ and initial samples $\{x_{t_0}^{(k)}\}_{k=1}^N\sim\rho_0$, with probability $1-\epsilon$, 
\begin{align}
    & \int_{\mathbb{R}^d}   \left|\nabla\left(\frac{\partial}{\partial t}\psi_\theta(x, t_i ) + H(x, \nabla\psi_\theta(x,t_i))\right)\right| \tilde{\rho}_{t_i}~dx \nonumber\\
    \leq & ~ \frac{1}{2}\lambda(\theta, i )h + \eta(\theta,i)h^{r-1} +\delta_i^{N,h}+\nu(\theta,i)\varepsilon_{i}^N + R(\theta,i)\sqrt{\frac{\ln M + \ln\frac{2}{\epsilon}}{2N}} \label{expectation residual ineq any time}
\end{align}
holds at any time {node} $t_i$, $i=1,2,\dots,M$.
\end{proof}

In Theorem \ref{thm est}, it can be seen that the constants $\lambda(\theta,i),\eta(\theta,i)$, $\nu(\theta,i),R(\theta,i)$ depend on the neural network approximation $\psi_\theta$ and the sample complexity $N$.  In the proof of Theorem \ref{thm est}, the set $E_i$ also depends on $N$. We refer to Remark \ref{rk-thm 3.1} in section \ref{appendix: proof of thm 3.1} for their dependence on $\psi_\theta$ and $N$. 

{We want to highlight that the \textit{posterior} estimation on the $L^1$-residual of $\nabla \psi_\theta$ consists of three parts: the numerical error depending on the geometric integrator $\frac{1}{2}\lambda(\theta, i )h + \eta(\theta, {i}  )h^{r-1}$ in \eqref{thm final ineq}, the training error $\delta_i^{N,h}+\nu(\theta,i)\varepsilon_{i}^N$ caused by the neural network approximation, and the sampling error $R(\theta, i)((\ln M + \ln\frac{2}{\epsilon})/(2N))^{1/2}$ due to the Monte--Carlo method.   
For the results about explicit bound of $\varepsilon_i^N$, one may use the McDiarmid's inequality \cite{SSBD2014} and Rademacher complexity $\textrm{Rad}(F)$ of the function set $F=\{\mathcal{R}[\psi_\theta]\circ \tilde{\Phi}_h^{i}\}_{i=0,1,\dots, M}$, as well as Masaart's Lemma \cite{SSBD2014} on estimating the upper bound of $\textrm{Rad}({F})$. Since $\varepsilon_i^N$ mainly relies on the approximation power of $\psi_\theta$, which is another topic beyond the scope of this work, we omit its detailed discussion here.}

We note that the error estimate \eqref{thm final ineq} is established for \textit{density-weighted} residual of $\nabla\psi_\theta$. Here the probability density $\tilde \rho_{t_i}$ of numerical solution $\tilde{x}_{t_i}$ is solved via the geometric integrator $\tilde{\Phi}_h$. We anticipate smaller residual values of $\nabla\psi_\theta$ at the region on which $\tilde\rho_{t_i}$ possesses a higher probability. On the contrary, no estimate is provided outside of the support of $\tilde \rho_{t_i}$. Such an observation is verified in the later section \ref{sec: verify rel btw residual and tile rho ti }. {It can be seen that the approximate solution $\psi_{\theta}$ depends on the choice of $\rho_0.$ The question on the sensitivity and dependence on the choice of $\rho_0$ is very interesting and important. We give a brief discussion on this issue with a specific example in the supplementary material and hope to this issue in the future.
}

{We would like to remark that,  if assuming the existence of the classical solution,  one can show that the temporal convergence order of numerical integrator in  proposed algorithm can be improved to  $r-2$ ($r>2$) via similar arguments as in the proof of Theorem \ref{thm est}.} 
Besides, the error analysis in Theorem \ref{thm est} works for any $T>0$ even when $T$ goes beyond the threshold time $T_*$ of classical solution. However, when $t_i$ is approaching (or even surpassing) $T_*$, the superposition of momentum vectors in the configuration space often leads to {a} larger training loss $\mathcal E_i$, which increases the error upper bound in \eqref{thm final ineq}. Such increment in the loss values $\mathcal E_i$ is reflected in several numerical examples demonstrated in section \ref{numerical example separable hamilton }. This is justifiable because the classical solution itself even cannot be extended beyond $T_*$, and we are not able to control the residual value of $\nabla\psi_\theta$ when time $t_i$ approaches (or surpasses) $T_*$. On the other hand, in our proposed algorithm, the numerical solution $\psi_\theta$ extends naturally beyond $T_*$, which can be treated as the approximation to the $\mu_t(\cdot|x)-$weighted ``solution'' $\widehat{\psi}$ to the HJ equation \eqref{HJ} discussed in {section} \ref{rk weighted momentum as grad HJ solu }. Several numerical examples of such $\mu_t(\cdot |x)-$weighted ``solution'' are also demonstrated in next section.

\section{Numerical tests}\label{sec-4}

{In our implementation, we set $\psi_\theta(\cdot, \cdot):\mathbb{R}^{d+1}\rightarrow \mathbb{R}$ as neural network with ResNet \cite{he2016deep} structure in our implementation. To be more precise, we consider the following neural network $\mathcal{NN}_{\theta}^{L, \widetilde{d}}(\cdot, \cdot):\mathbb{R}^{d+1}\rightarrow\mathbb{R}$ with depth $L$ and width (hidden dimension) $\widetilde{d}$ as
\begin{equation*}
  \mathcal{NN}^{L, \widetilde{d}}_\theta(x,t) = f_L\circ f_{L-1}\circ \dots f_2\circ f_1(x,t), 
\end{equation*}
with 
\begin{align*}
  f_k(y) = 
  \begin{cases}
      \sigma(A_ky+b_k) \quad \textrm{for } k=1,\\
      \sigma(y+ \kappa(A_ky+b_k)) \quad \textrm{for } 2\leq k \leq L-1,\\
      A_ky \quad \textrm{for } k=L.
  \end{cases} 
\end{align*}If not specified, we choose the activation function $\sigma(\cdot)$ as the hyperbolic tangent function $\mathrm{tanh}(\cdot)$. And $\kappa \in\mathbb{R^+}$ is the stepsize of each layer, we choose $\kappa=0.5$ in our experiments. Furthermore, $A_1 \in \mathbb{R}^{\widetilde{d}\times (d+1)}, b_1\in\mathbb{R}^{\widetilde{d}}$, $A_k\in\mathbb{R}^{\widetilde{d}\times \widetilde{d}}, b_k\in\mathbb{R}^{\widetilde{d}}$ for all $2\leq k \leq L-1$, and $A_L \in\mathbb{R}^{1\times\widetilde{d}}$ compose the parameter $\theta\in\mathbb{R}^{(L-2)\widetilde{d}^2 + \widetilde{d}(d+2) + (L-1)\widetilde d}$ of this neural network. {We omit the bias term in the final layer, as our goal is to compute the gradient field $\nabla_x \psi_\theta$, where the bias has no effect and is therefore discarded.} We apply the Adam method \cite{kingma2014Adam} to train $\psi_\theta$ in Algorithm \ref{alg1}. We {pick the random batch size $N_0=1200$} for all the numerical experiments discussed in this section. All experiments are conducted on a Tesla T4 GPU with 16GB high-bandwidth memory, using PyTorch 2.6.0 and CUDA 12.5.} The training time for $\psi_\theta$ on each time interval is around 3-10 minutes for problems with dimensions varying from 2 to 30. The Python code used to reproduce the numerical examples presented in this paper is available at {\color{blue} \href{https://github.com/LSLSliushu/HJ_via_density_coupling}{\texttt{github.com/LSLSliushu/HJ\_via\_density\_coupling}}}.

\subsection{{Residual and error bounds}}\label{sec: verify rel btw residual and tile rho ti }

Theorem \ref{thm est} states that the expectation of the residual can be bounded, where the expectation is taken with respect to the distribution $\tilde{\rho}_{t_i}$ of samples used for training $\psi_\theta$. Thus we anticipate a smaller residual value %
on the support of $\tilde\rho_{t_i}$; On the other hand, the residual outside of the support of $\tilde \rho_{t_i}$ can not be controlled due to lack of learning samples. This is observed in the following examples. 

Consider the Hamilton-Jacobi equation on $\mathbb{R}^2\times [0,T]$ with $T=3$, $H(x,p)=\frac{|p|^2}{2}+\frac{|x|^2}{2}$ and initial data $u(x)=\frac{|x|^2}{2}$.  {We choose $\rho_0 = \mathcal N((3, 3), I)$, i.e., the normal distribution shifted by $(3, 3).$} {We set $\psi_\theta = \mathcal{NN}^{L, \widetilde{d}}_\theta$ with $L=7, \widetilde d = 40$}. We choose the number of time subintervals $M=40$, and the number of samples $N=7500.$ We set the learning rate $ lr=0.5 \cdot 10^{-4}$ and perform Adam's method for $N_{\textrm{Iter}} = 8000$ iterations. We plot the heat map of the residual term 
\begin{equation}
  \textrm{Res}(x,t) = \left|\nabla\left(\frac{\partial}{\partial t}\psi_\theta(x, t) + H(x, \nabla\psi_\theta(x, t))\right)\right|  \label{def: res}
\end{equation}
together with the samples $\{x_{t_i}^{(k)}\}_{k=1}^N$ at different time nodes $t_i$ in the first row of Figure \ref{residue and error samples}. The support of the samples mostly overlaps with the region on which the residual value $\textrm{Res}(x,t)$ is small. 
A similar observation is also found about the error between $\nabla\psi_\theta(x,t)$ and the real solution $\nabla u(x,t)$, where $u(x,t) =\frac{1}{2}\cot(t+\frac{\pi}{4})|x|^2$, i.e.
\begin{equation}
 \textrm{Err}(x,t)=|\nabla \psi_\theta(x,t) - \nabla u (x,t)|.\label{def:  error}
\end{equation} 
The results are demonstrated in the second row of Figure \ref{residue and error samples}.
\begin{figure}[htb!]
\begin{subfigure}{.24\textwidth}
  \centering
  \includegraphics[width=\linewidth]{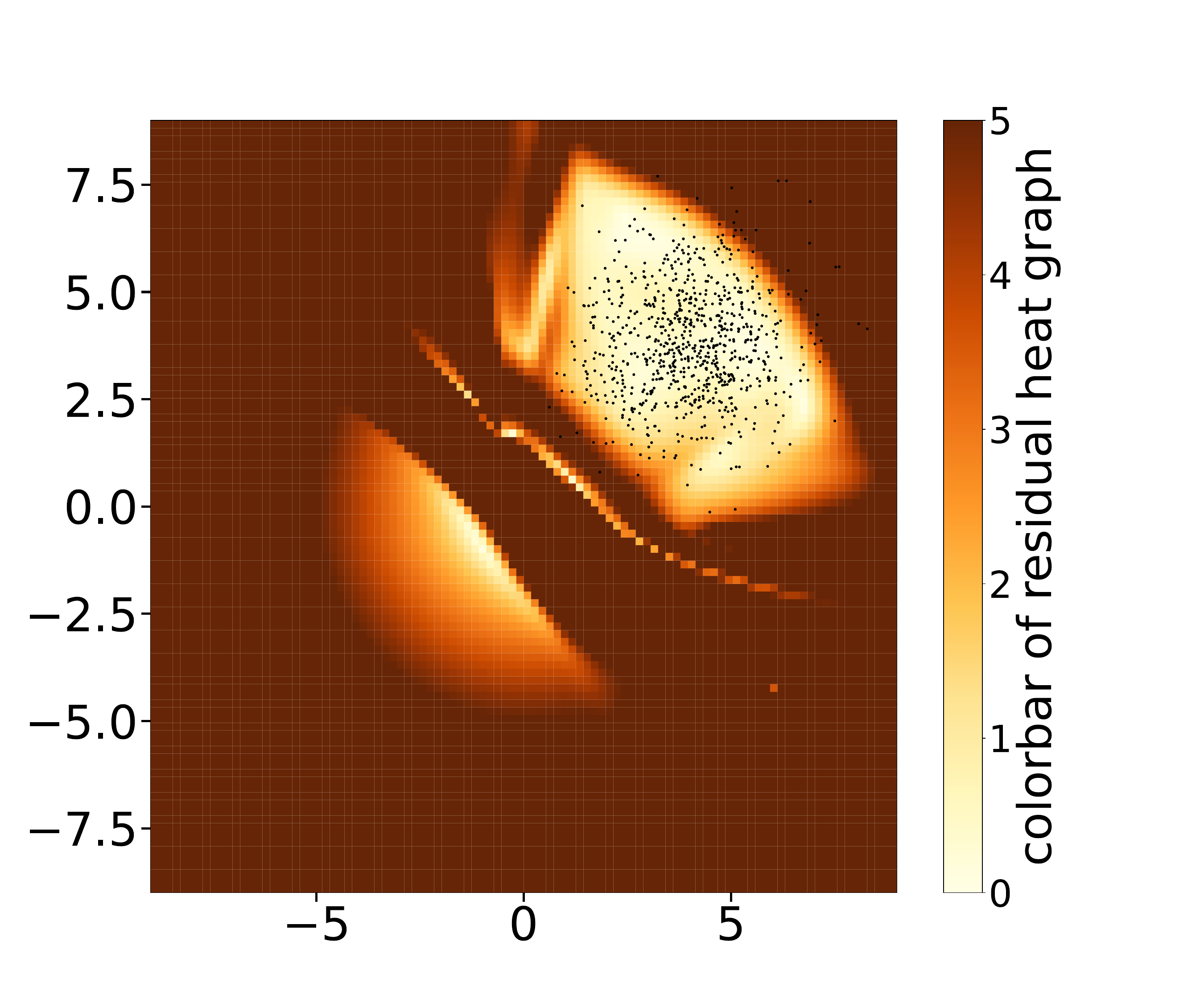}
\end{subfigure}
\begin{subfigure}{.24\textwidth}
  \centering
  \includegraphics[width=\linewidth]{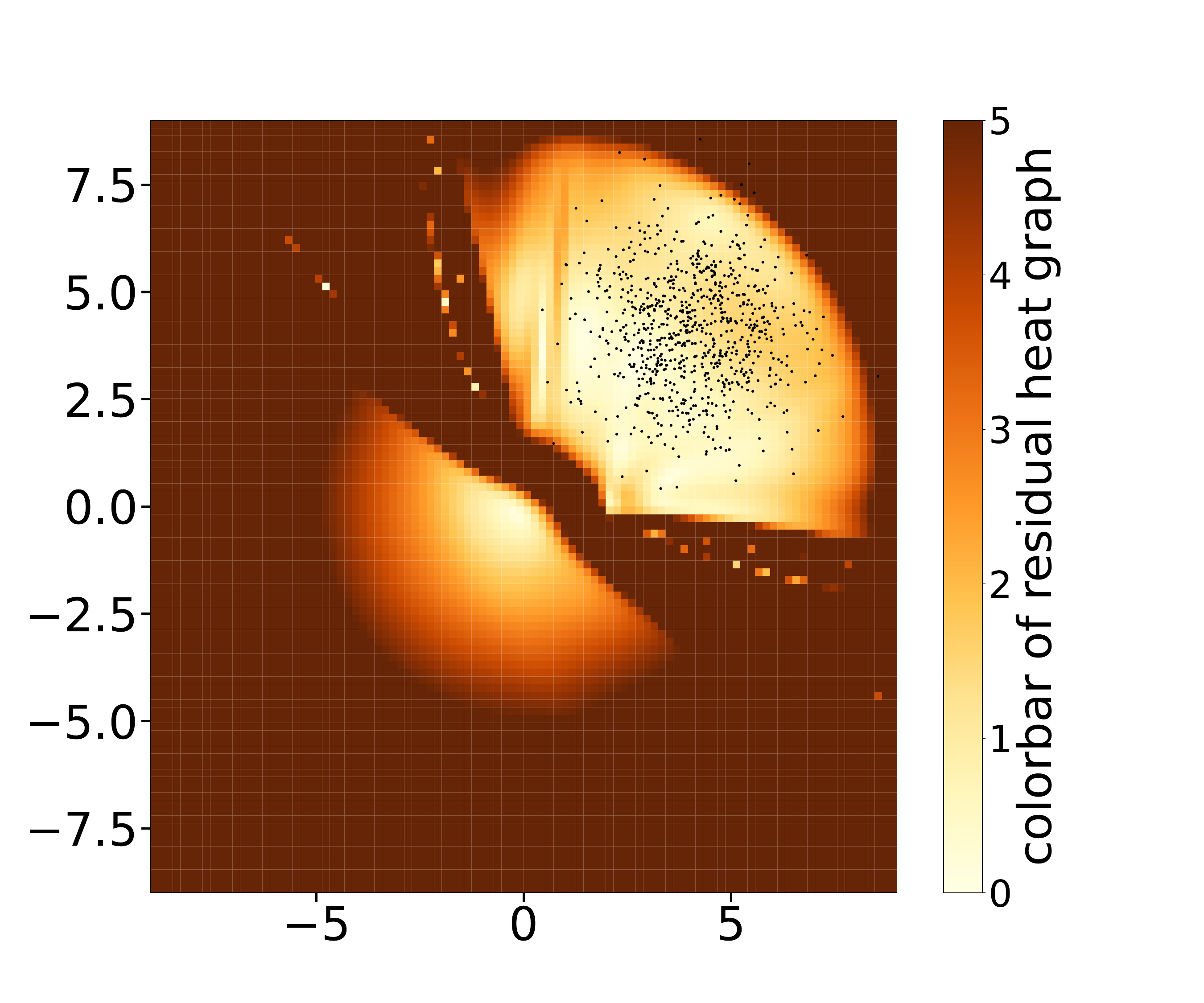}
\end{subfigure}
\begin{subfigure}{.24\textwidth}
  \centering
  \includegraphics[width=\linewidth]{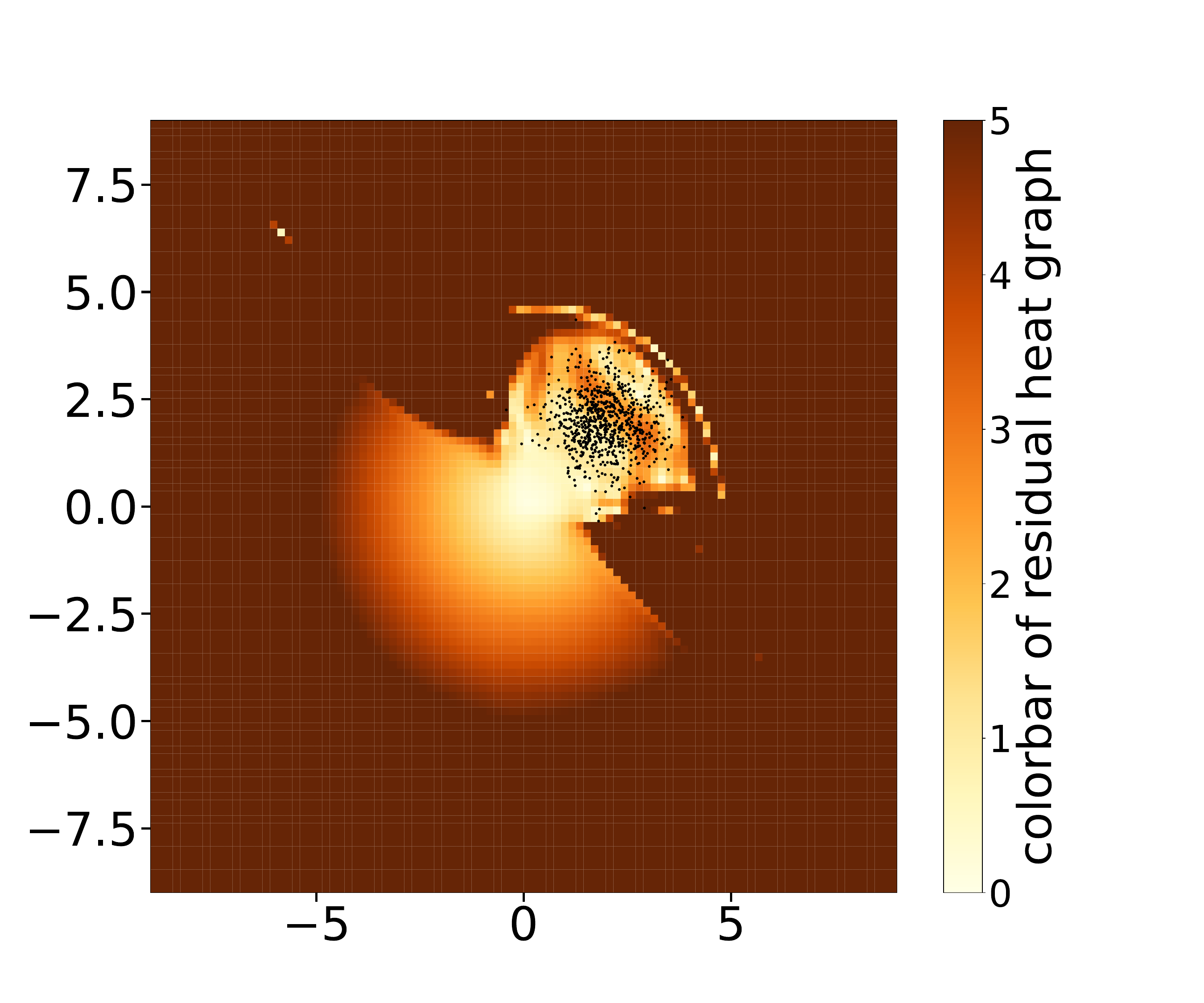}
\end{subfigure}
\begin{subfigure}{.24\textwidth}
  \centering
  \includegraphics[width=\linewidth]{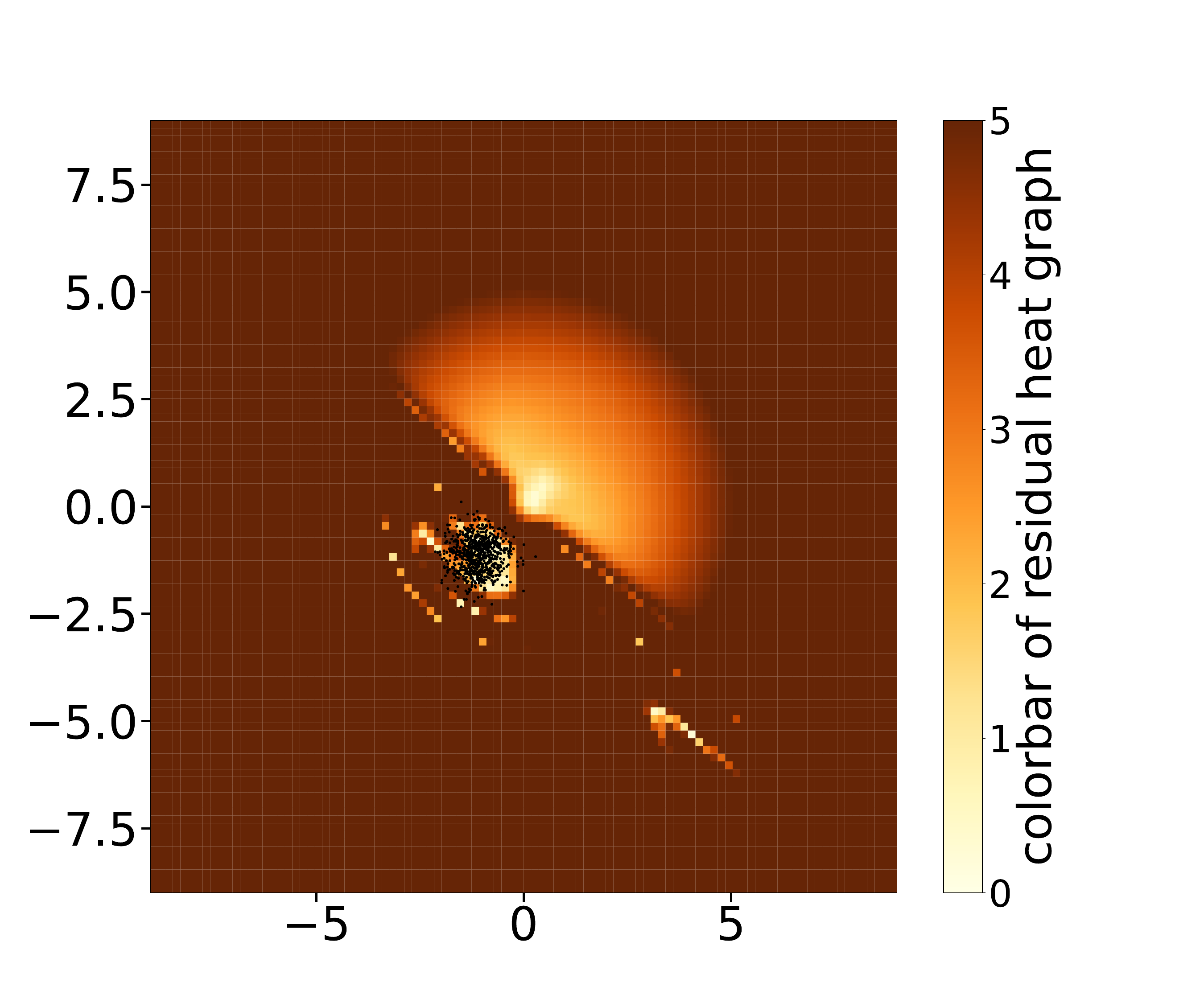}
\end{subfigure}
\begin{subfigure}{.24\textwidth}
  \centering
  \includegraphics[width=\linewidth]{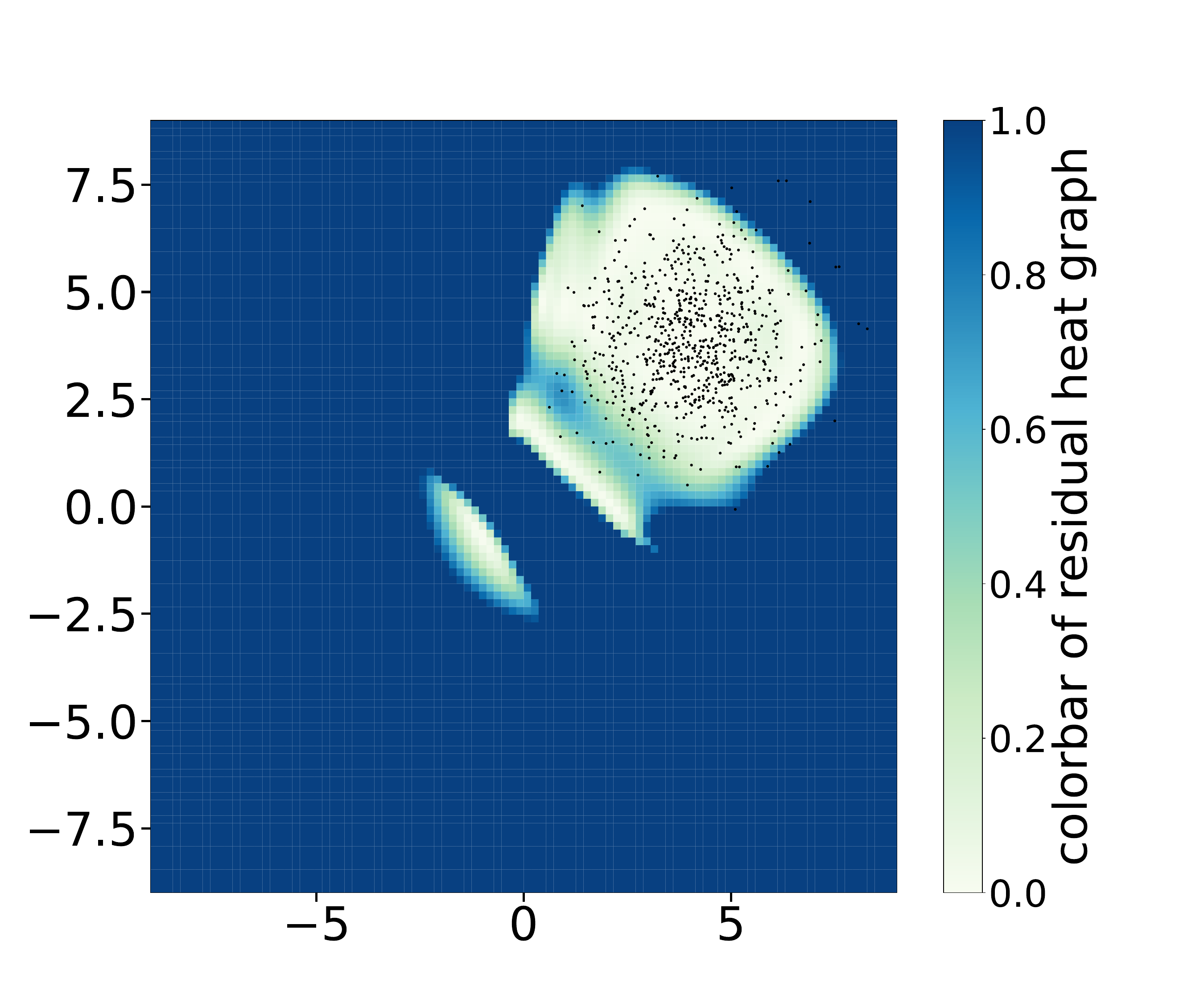}
  \caption{$t=3/8$}
\end{subfigure}
\begin{subfigure}{.24\textwidth}
  \centering
  \includegraphics[width=\linewidth]{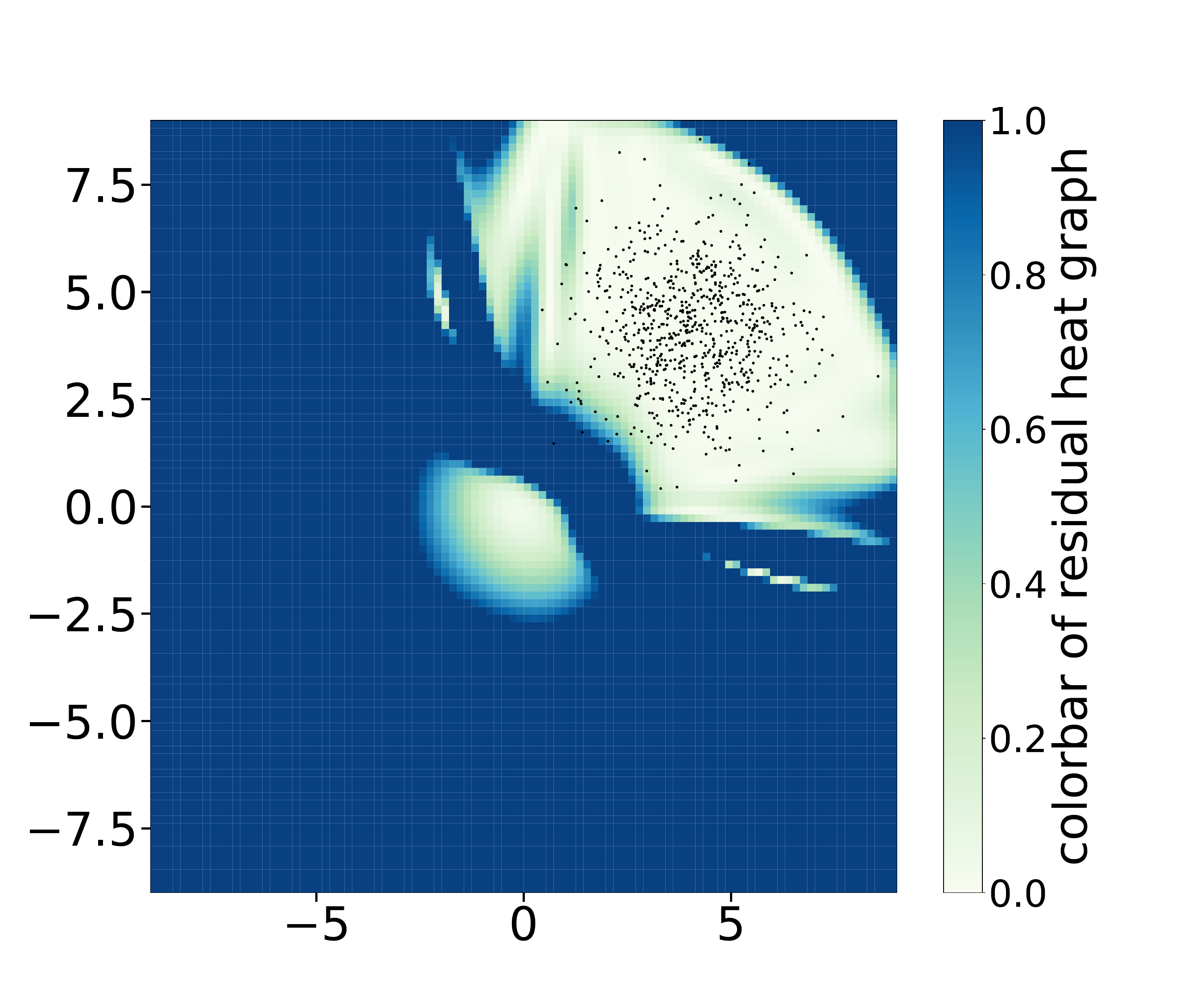}
  \caption{$t=9/8$}
\end{subfigure}
\begin{subfigure}{.24\textwidth}
  \centering
  \includegraphics[width=\linewidth]{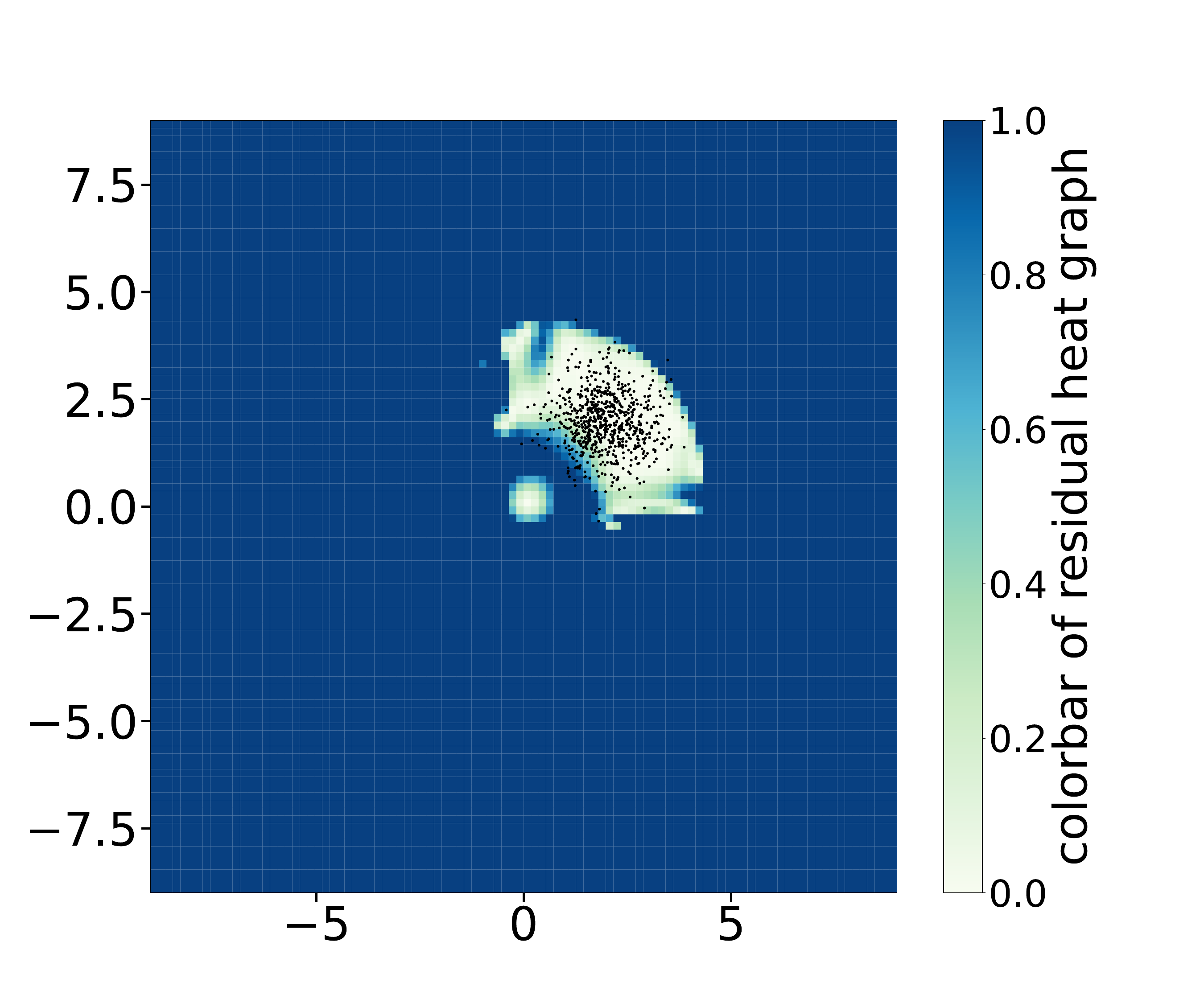}
  \caption{$t=15/8$}
\end{subfigure}
\begin{subfigure}{.24\textwidth}
  \centering
  \includegraphics[width=\linewidth]{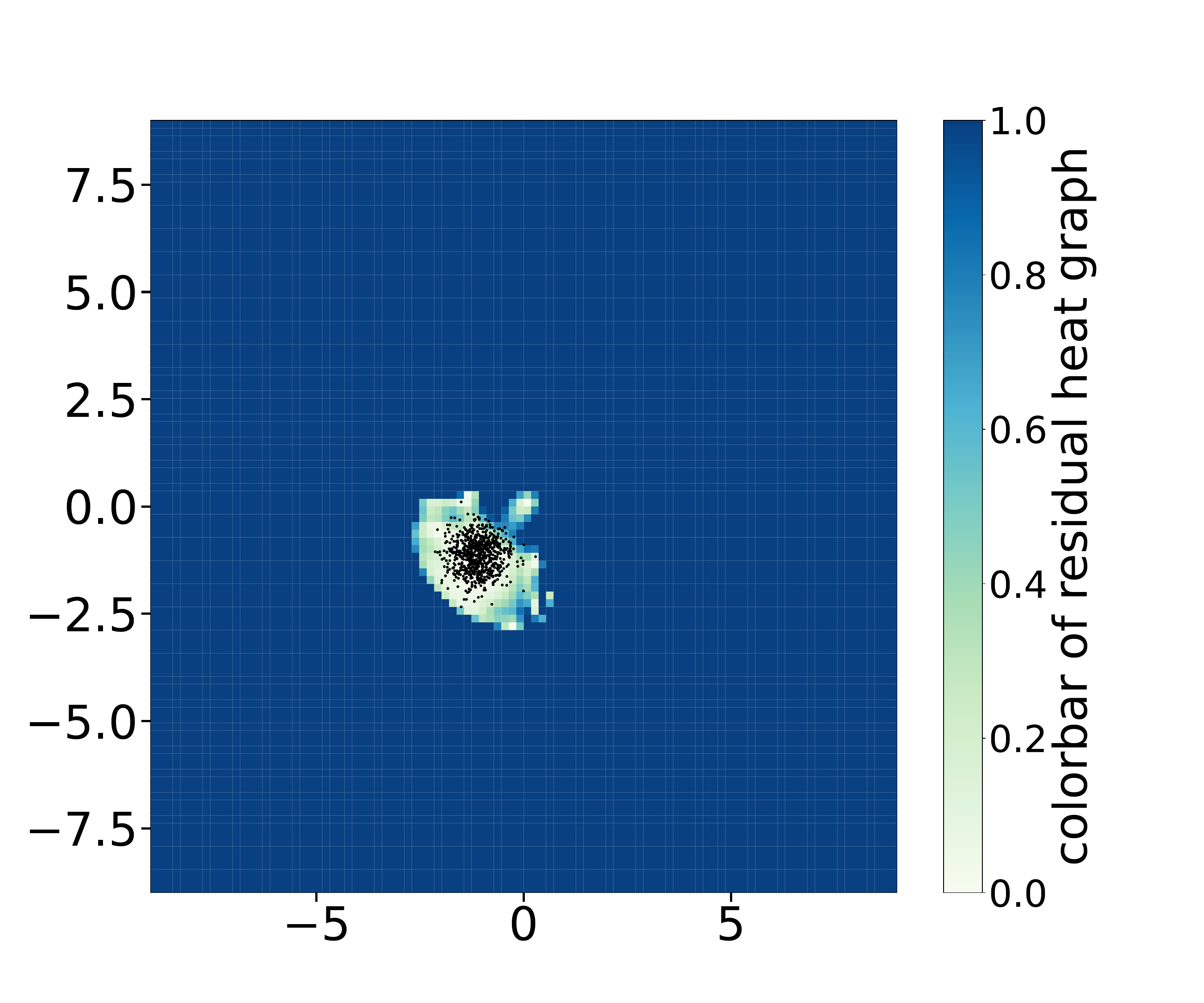}
  \caption{$t=21/8$}
\end{subfigure}
\caption{{(Top row) Heat graphs of the residual $\textrm{Res}(x, t)$ of the numerical solution $\psi_\theta$ and the sample points (black) at different time stages $t$. (Bottom row) Heat graphs of the error $\textrm{Err}(x, t)$ of the numerical solution $\psi_\theta$ and the sample points (black) at different time stages $t$.}
}
\label{residue and error samples}
\end{figure}
Another interesting question is how the sample size $N$ affects the accuracy of the numerical solution $\nabla\psi_\theta$. To test it, we train $\psi_\theta$ by using different sample size $N$ while keeping other hyperparameters unchanged. We examine the relationship between the error $\|\nabla\psi_\theta(\cdot,t)-\nabla u(\cdot,t)\|^2_{L^2(\rho_t)}$ and the sample size $N$ on time interval $[0, 0.25]$, where we discretize the time interval into $M=100$ subintervals. 

We repeat Algorithm \ref{alg1} for different sample sizes $N=16\cdot 2^k$ with $k=0,1, \dots 9$. We approximate the $L^2(\rho_t)$ discrepancy between numerical solution $\nabla\psi_\theta$ and real solution $\nabla u$ by using the Monte--Carlo method with a large sample size 45000. We conduct the numerical experiments on the same Hamilton-Jacobi equation with dimensions being $2$ and $10$ respectively. The results are plotted in Figure \ref{loglog plot error - N }, showing that the accuracy of the proposed method improves as the number of sample sizes $N$ increases. 
\begin{figure}[htb!]
\centering
\begin{subfigure}{.24\textwidth}
  \centering
  \includegraphics[width=\linewidth]{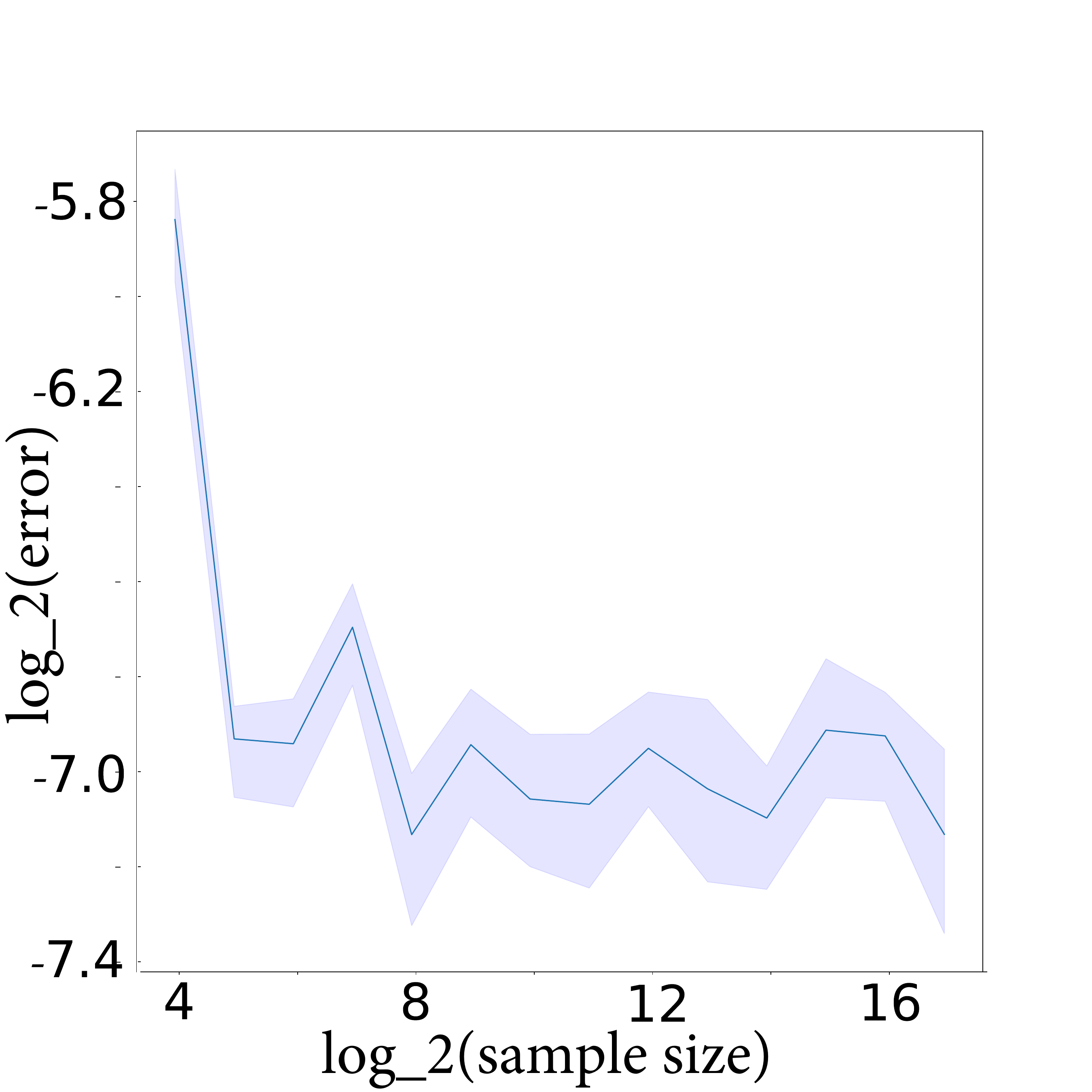}
  \caption{\footnotesize $d=2$, $t=0.1$}
\end{subfigure}
\begin{subfigure}{.24\textwidth}
  \centering
  \includegraphics[width=\linewidth]{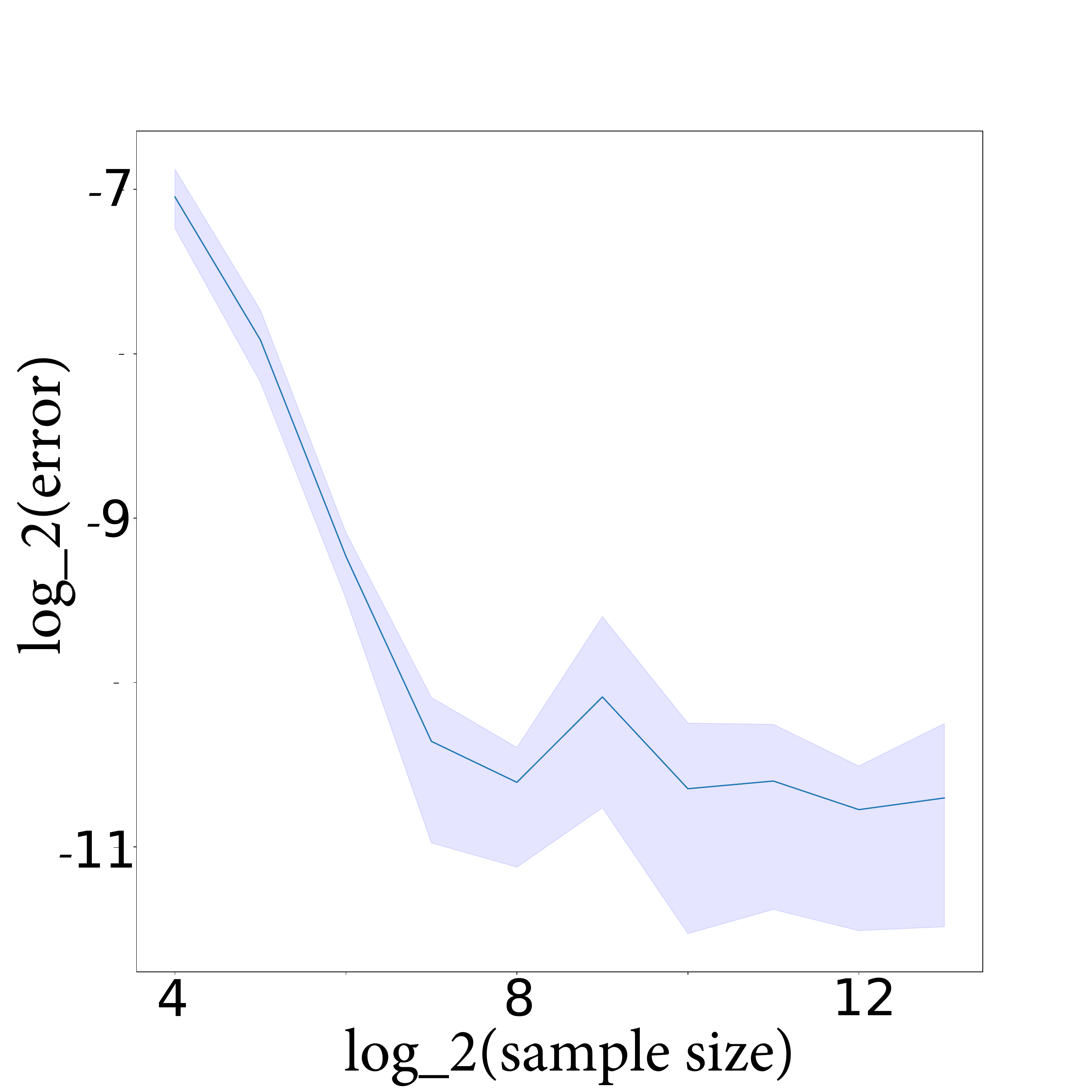}
  \caption{\footnotesize $d=2$, $t=0.2$}
\end{subfigure}
\begin{subfigure}{.24\textwidth}
  \centering
  \includegraphics[width=\linewidth]{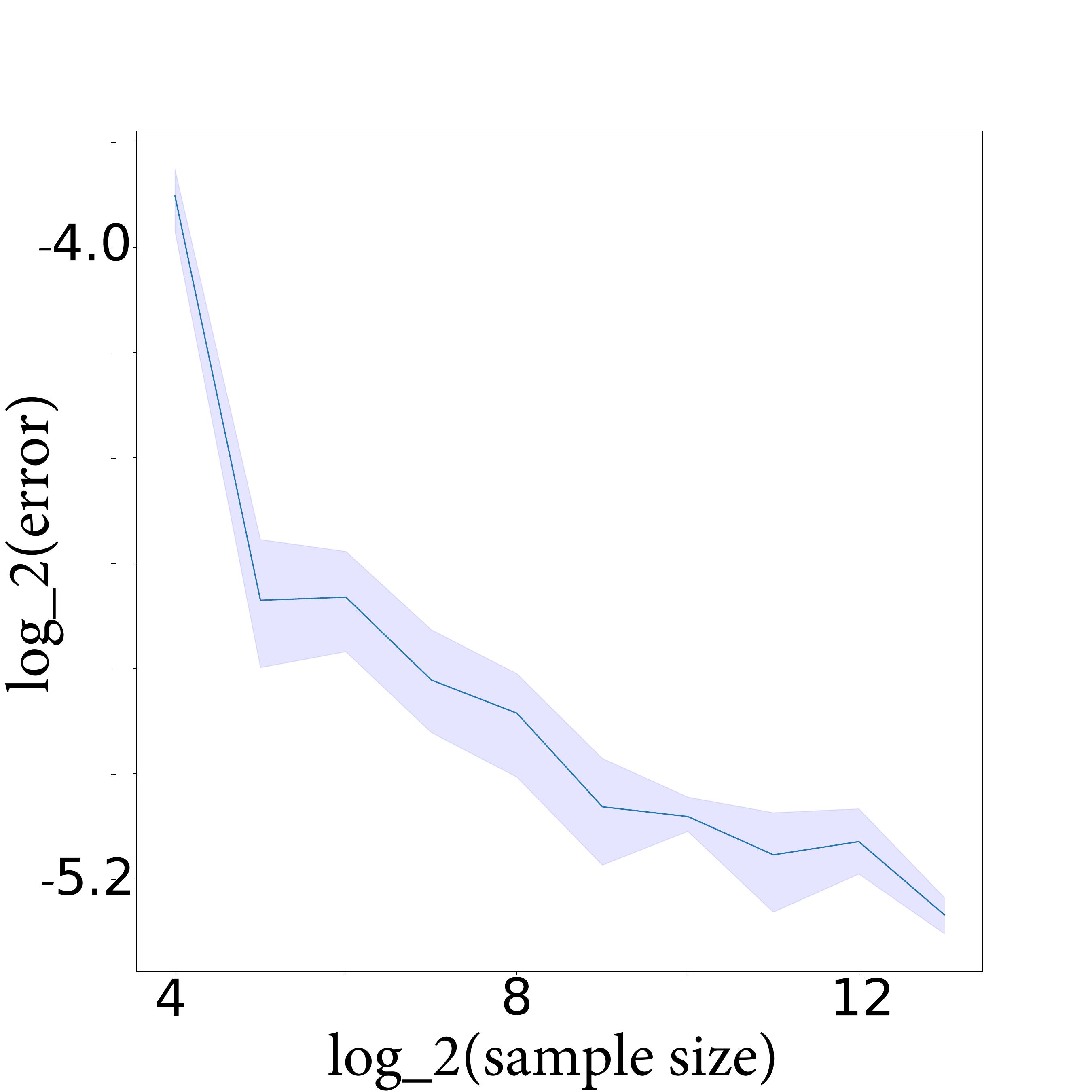}
  \caption{\footnotesize  $d=10$, $t=0.1$}
\end{subfigure}
\begin{subfigure}{.24\textwidth}
  \centering
  \includegraphics[width=\linewidth]{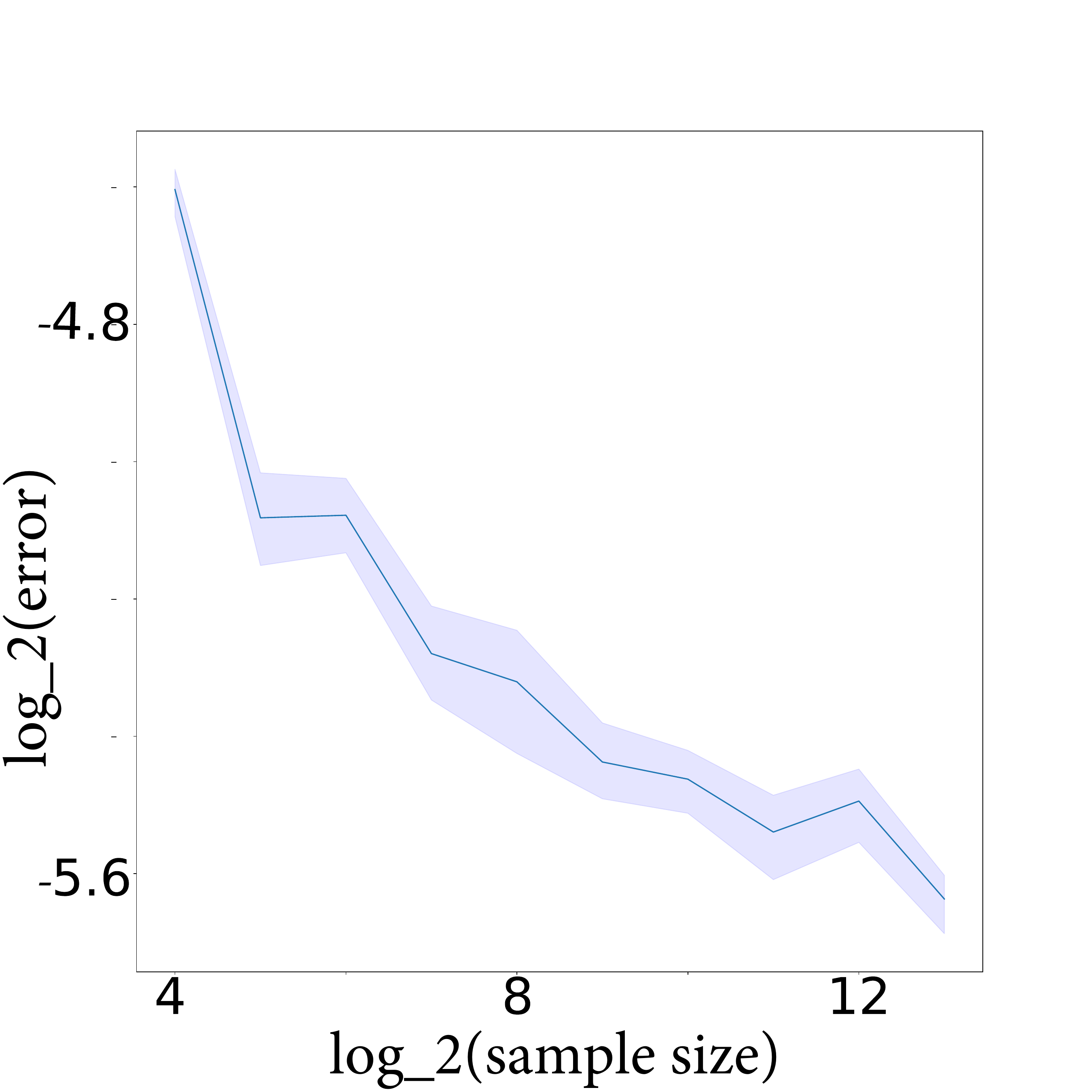}
  \caption{\footnotesize  $d=10$, $t=0.2$}
\end{subfigure}
\caption{{Average error versus sample size plots ($\log_2-\log_2$) for 2D and 10D HJ equation 
(plots with confidence interval ($25\%-75\%$) based on $40$ sets of data)}} \label{loglog plot error - N }
\end{figure}

\subsection{Behavior of neural network solution near caustics}\label{subsec: study_near_caustics}

As discussed in section \ref{rk weighted momentum as grad HJ solu }, one of the remarkable features of our treatment is its ability to compute the solution beyond the time $T_*$ at which caustics form--that is, multiple characteristics possess same state position $x$ but exhibit different momentum directions $p$. The classical solution no longer exists beyond $T_*$. However, our method remains applicable by matching the $\mu_t(\cdot|x)$-weighted momentum via a least-squares approach. In this section, we present a 2D Hamilton–Jacobi equation with  caustic formation.

We consider the Hamiltonian $H(x, p) = \frac{1}{2}p^\top\boldsymbol{\Sigma} p$ as the degenerate quadratic kinetic energy, where $\boldsymbol{\Sigma}=\boldsymbol{\eta}\boldsymbol{\eta}^\top = \frac{1}{2}\boldsymbol{1}\boldsymbol{1}^\top$, $\boldsymbol{\eta}=\frac{1}{\sqrt{2}}\boldsymbol{1}$. Here we define $\boldsymbol{1}=(1,1)^\top$ as a $2$-D vector. In the first test, we set $g(x)=\cos(\boldsymbol{\eta}^\top  x)$, and $\rho_0=\mathcal U(E)$ as the uniform distribution on the square region $E=[-\frac{\pi}{\sqrt{2}}, \frac{\pi}{\sqrt{2}}]^2$ and solve the HJ equation on the time interval $[0, T]$ with $T=3$. By restricting the equation on straight lines along the diagonal ($\boldsymbol{\eta}$) direction, one can verify that the classical solution $u(\cdot, t)$ of \eqref{HJ} takes the form $u(x,t) = f(\boldsymbol{\eta}^\top x, t)$, where $f(\cdot, t):\mathbb{R}\rightarrow \mathbb{R}$ solves 
\[ \frac{\partial f(z,t)}{\partial t}+\frac12|\partial_z f(z,t)|^2=0, \quad f(z,0) = \cos(z). \]
Here, we denote $z$ as the coordinate along the diagonal direction. Consider $\{z_t, p_t\}_{t\geq 0}$ as the bi-characteristics of the above equation. For arbitrary $\xi \in\mathbb{R},$ we assume $z_0=\xi, p_0 = -\sin(\xi).$ One can solve for $z_t = \xi - t\sin(\xi), p_t = -\sin(\xi).$ Let us denote the map $\varphi_t:\mathbb{R} \rightarrow \mathbb{R},\xi \mapsto \xi-t\sin(\xi)$. It can be checked that the map $\varphi_t(\cdot)$ remains injective as $t<T_*:=1$. Since we have $\partial_z f(z_t, t)=p_t$ for $t<T_*$, this leads to 

\begin{equation}
  \partial_z f(z,t)=-\sin(\varphi_t^{-1}(z)),  \label{def classic solution examp sinusoidal}
\end{equation}
for $t<T_*=1$. As $t$ increases beyond $T_*$, caustics will develop. Solving the least square problem \eqref{regression quad} leads to our proposed weak solution. We present an explicit formula for $\widehat{f}(z, t)$, which denotes the restriction of $u(\cdot, t)$ on the diagonal line $\ell$ passing through the origin. It can be verified that the conditional distribution $\mu(\cdot|z)$ for the momentum is
\[ \mu_t(\cdot|z) = \frac{1}{C_z} \sum_{\xi \in \varphi_t^{-1}(z)} \frac{\widehat{\rho}_0(\xi)}{|\varphi_t'(\xi)|} \delta_{-\sin(\xi)}. \]
Here $\widehat{\rho}_0$ denotes the distribution of $\rho_0$ conditioned on $\ell$. Therefore, $\widehat{\rho}_0 = \mathcal U([-\pi, \pi])$; we denote $C_z := \sum_{\xi \in \varphi_t^{-1}(z)} \frac{\widehat{\rho}_0(\xi)}{|\varphi_t'(\xi)|}$; and $\delta_y$ denotes the Dirac measure concentrated on $y\in\mathbb{R}$. As the coupled density $\widehat{\rho}_t$ is always supported on $[-\pi, \pi]$ as $t$ increases from $0$ to $T=3$, we shall only focus on the solution $\widehat{f}(z, t)$ defined on $[-\pi, \pi] \times [0, T]$. We can calculate\footnote{It is worth mentioning that, when $t \in [T_*, 3],$ there exist $\pm z_t^*:= \pm (\sqrt{t^2 - 1} - \mathrm{arccos}\frac{1}{t}) \in [-\pi, \pi]$, such that $\xi = \pm \arccos(\frac{1}{t}) \in \varphi_t^{-1}(\pm z_t^*)$, which will lead to zero Jacobian $|\varphi_t'(\xi)| = 0$. Under such scenarios, we should define $\widehat{f}(\pm z_t^*, t) := \pm \frac{\sqrt{t^2-1}}{t},$ which turns out to be the limit of the right-hand side of \eqref{def weighted momentum examp sinusoida} as $z\rightarrow \pm z_t^*$. For the sake of brevity, we keep the notation concise in \eqref{def weighted momentum examp sinusoida}.}
\begin{align}\label{def weighted momentum examp sinusoida}
 \partial_z \widehat{f}(z, t) = \int_{\mathbb{R}} p \ d\mu_t(p|z) =
  - \frac{1}{C_z} \sum_{\xi\in\varphi_t^{-1}(z)} \frac{\sin (\xi)}{2\pi|1 - t\cos \xi|}, \quad \textrm{for } z\in [-\pi, \pi], t \in [0, 3].
\end{align}
\noindent
\textbf{Capturing jump discontinuities:} As demonstrated in Figure \ref{fig: 2D_caustic_compare_NN_exact}, the weak solution $\widehat{f}(\cdot, t)$ develops two jump discontinuities as $t$ evolves beyond $T_*$. For all the presented examples in this section, we pick $N_{iter}=20000$ and set $lr=0.5\cdot 10^{-4}$. We first set $\psi_\theta=\mathcal N\mathcal N_\theta^{L, \widetilde{d}}$ with $L=6, \widetilde{d}=50$ using $\mathrm{tanh}(\cdot)$ activation functions. The numerical results are provided in Figure \ref{fig: 2D_caustic_compare_NN_exact} and \ref{fig: 2D_caustic_compare_NN_exact_separate_time}.
\begin{figure}[htb!]
\begin{subfigure}{.49\textwidth}
  \centering
  \includegraphics[width=\linewidth]{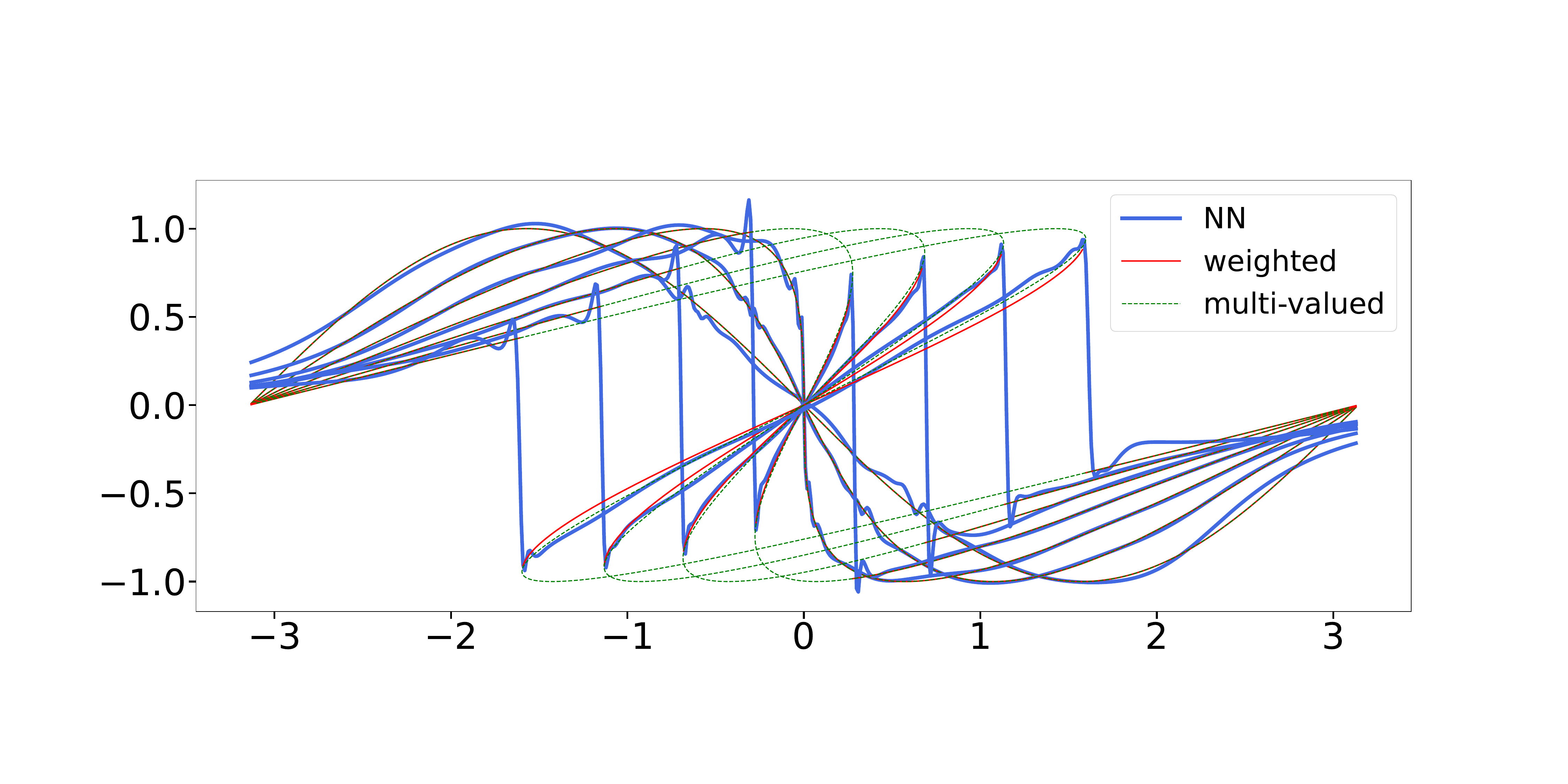}
\end{subfigure}
\begin{subfigure}{.49\textwidth}
  \centering
  \includegraphics[width=\linewidth]{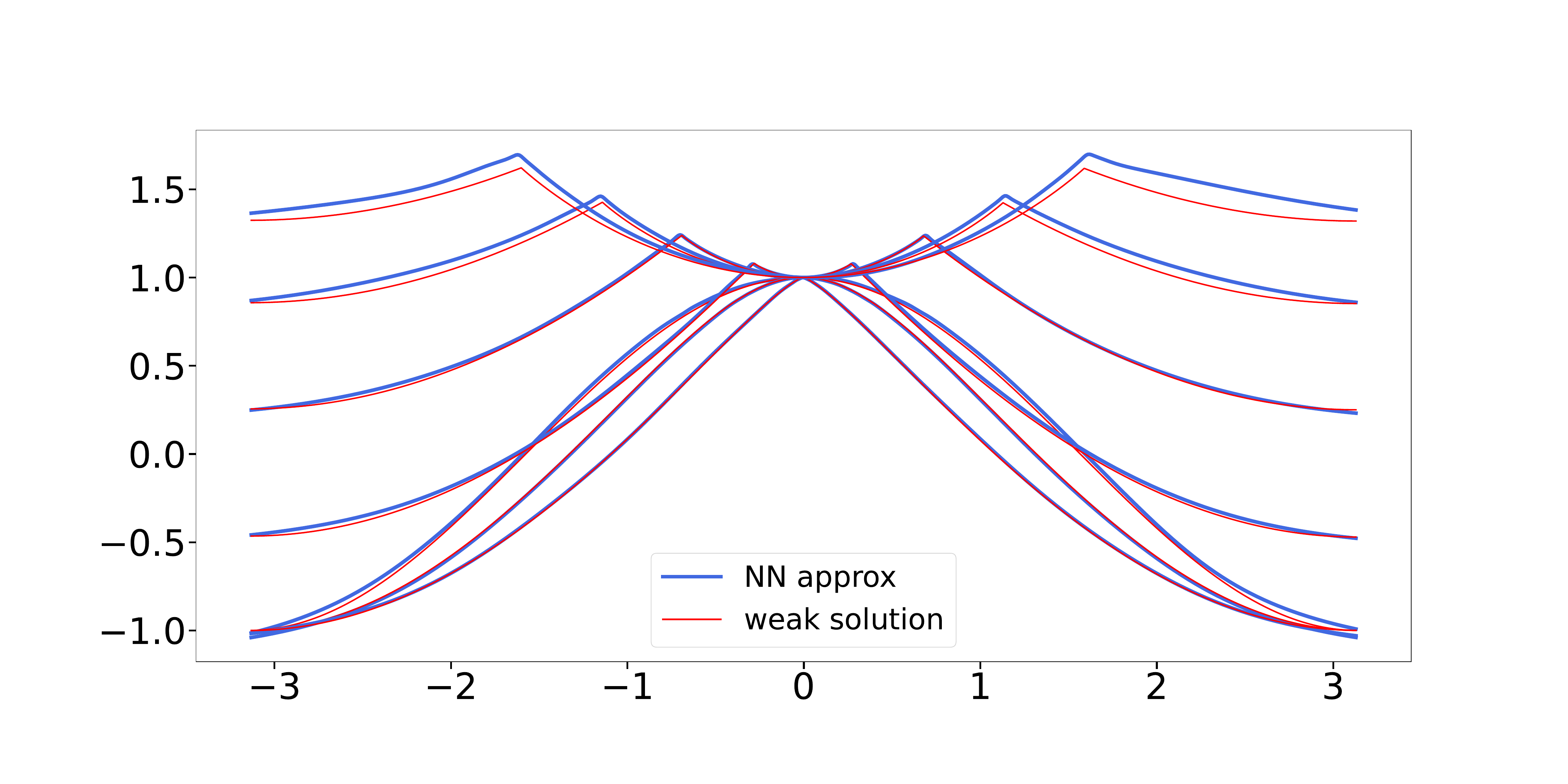}
\end{subfigure}
\vspace{-0.3cm}
\caption{{\textbf{Left}: Plots of NN approximation $\boldsymbol{\eta}^\top\nabla\psi_\theta(\cdot, t)$ (blue), weighted momentum (gradient of weak solution) $\partial_z \widehat{f}(\cdot, t)$ (red), and multi-valued momentum (green) along diagonal line $\ell$. \textbf{Right}: Plots of NN approximation $\psi_\theta(\cdot)$ (blue), weak solution $\widehat{f}(\cdot, t)$ (red) along $\ell$, we fix the function value equal to $0$ at the origin.}}
\label{fig: 2D_caustic_compare_NN_exact}
\end{figure}

\begin{figure}[h!]
\begin{subfigure}{.24\textwidth}
  \centering
  \includegraphics[width=\linewidth]{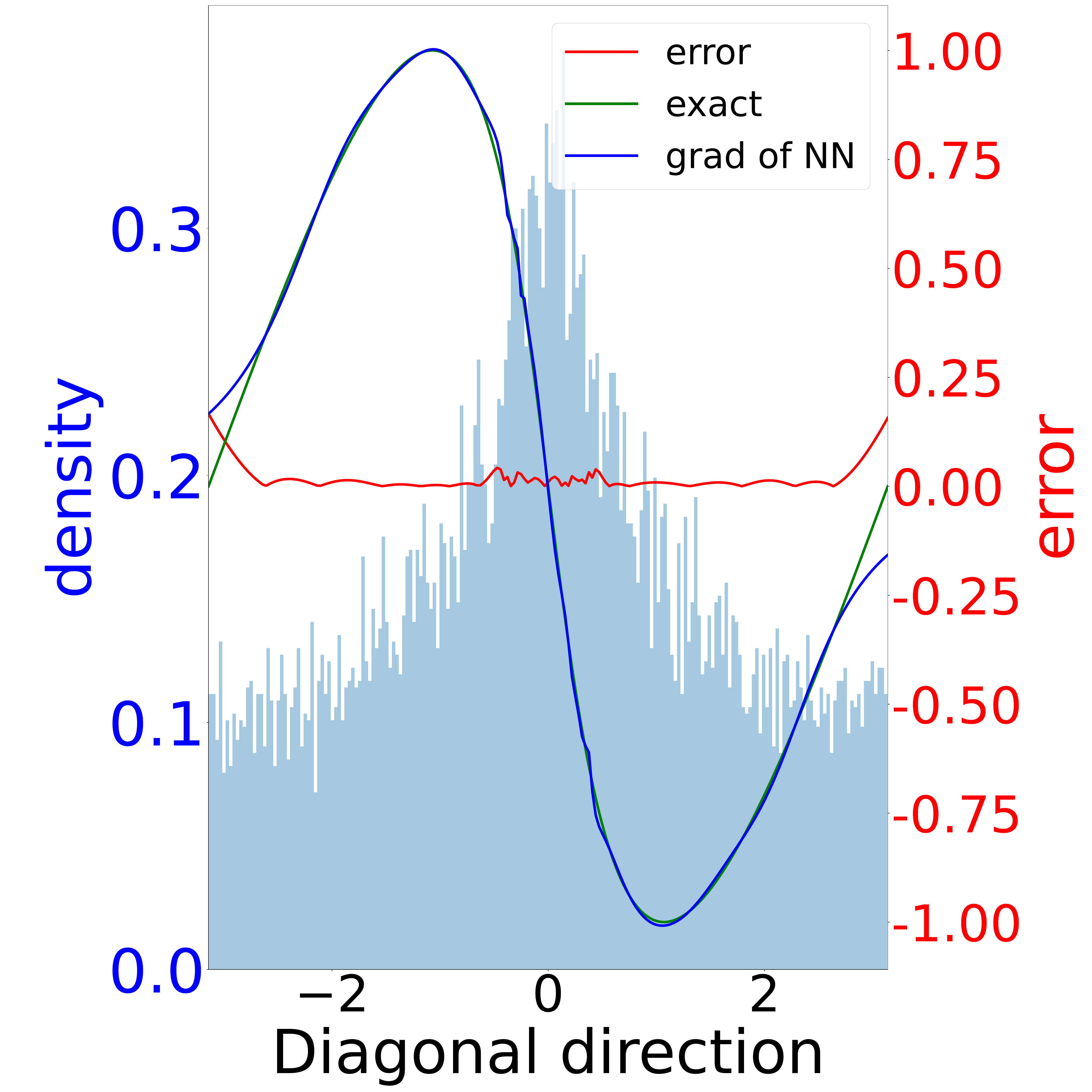}
  \subcaption{$t=0.5$}
\end{subfigure}
\begin{subfigure}{.24\textwidth}
  \centering
  \includegraphics[width=\linewidth]{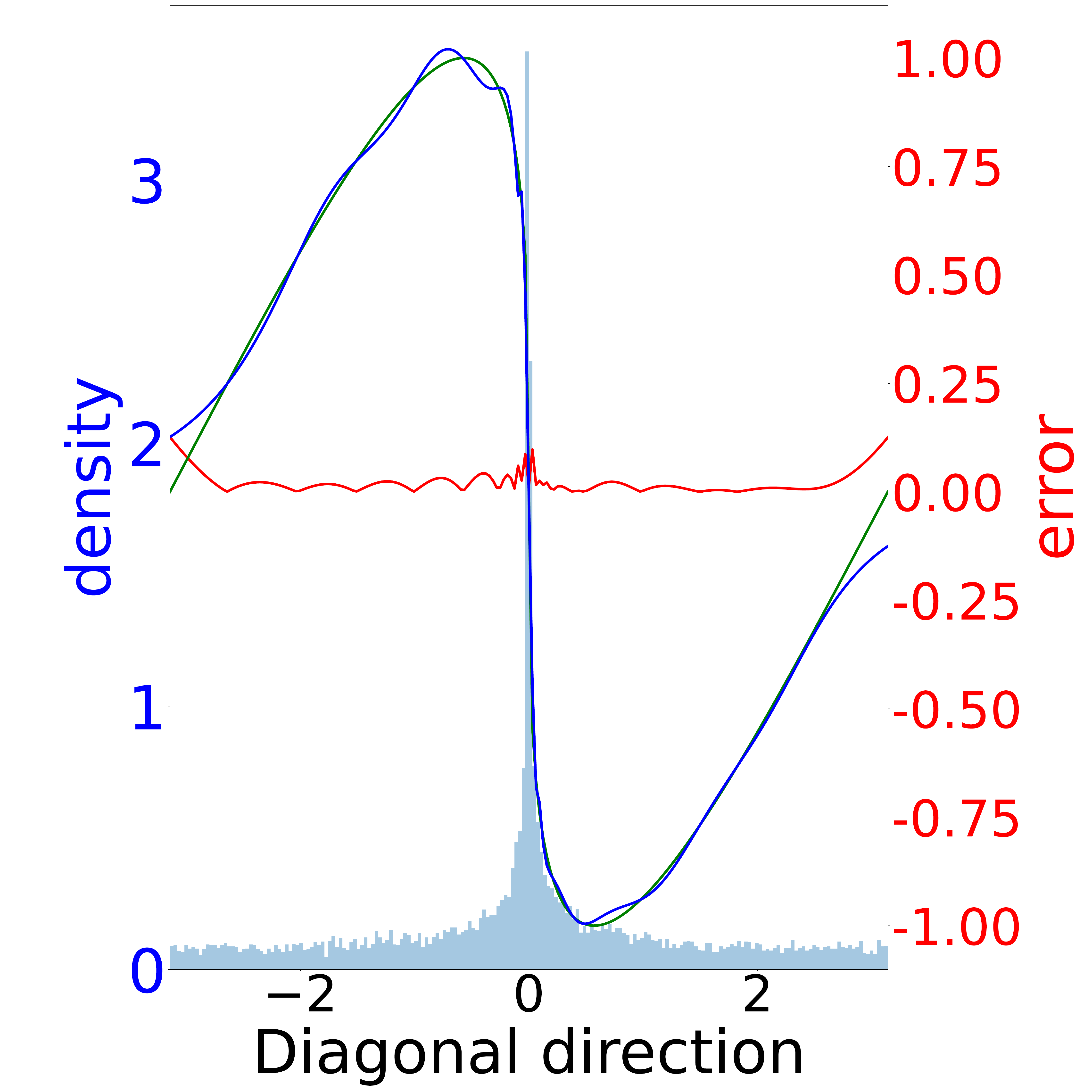}
  \subcaption{$t=1.0 (T_*)$}
\end{subfigure}
\begin{subfigure}{.24\textwidth}
  \centering
  \includegraphics[width=\linewidth]{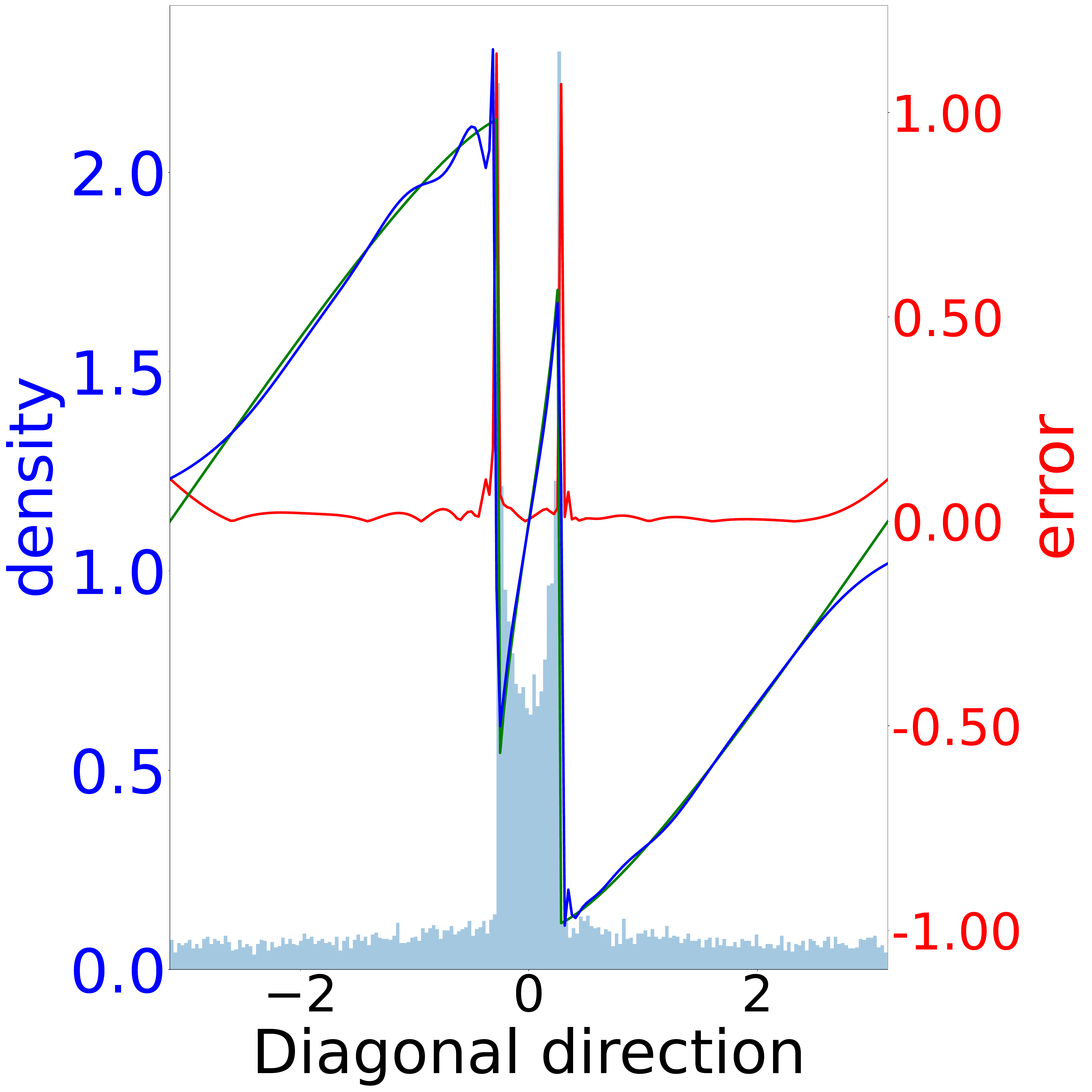}
  \subcaption{$t=1.5$}
\end{subfigure}
\begin{subfigure}{.24\textwidth}
  \centering
  \includegraphics[width=\linewidth]{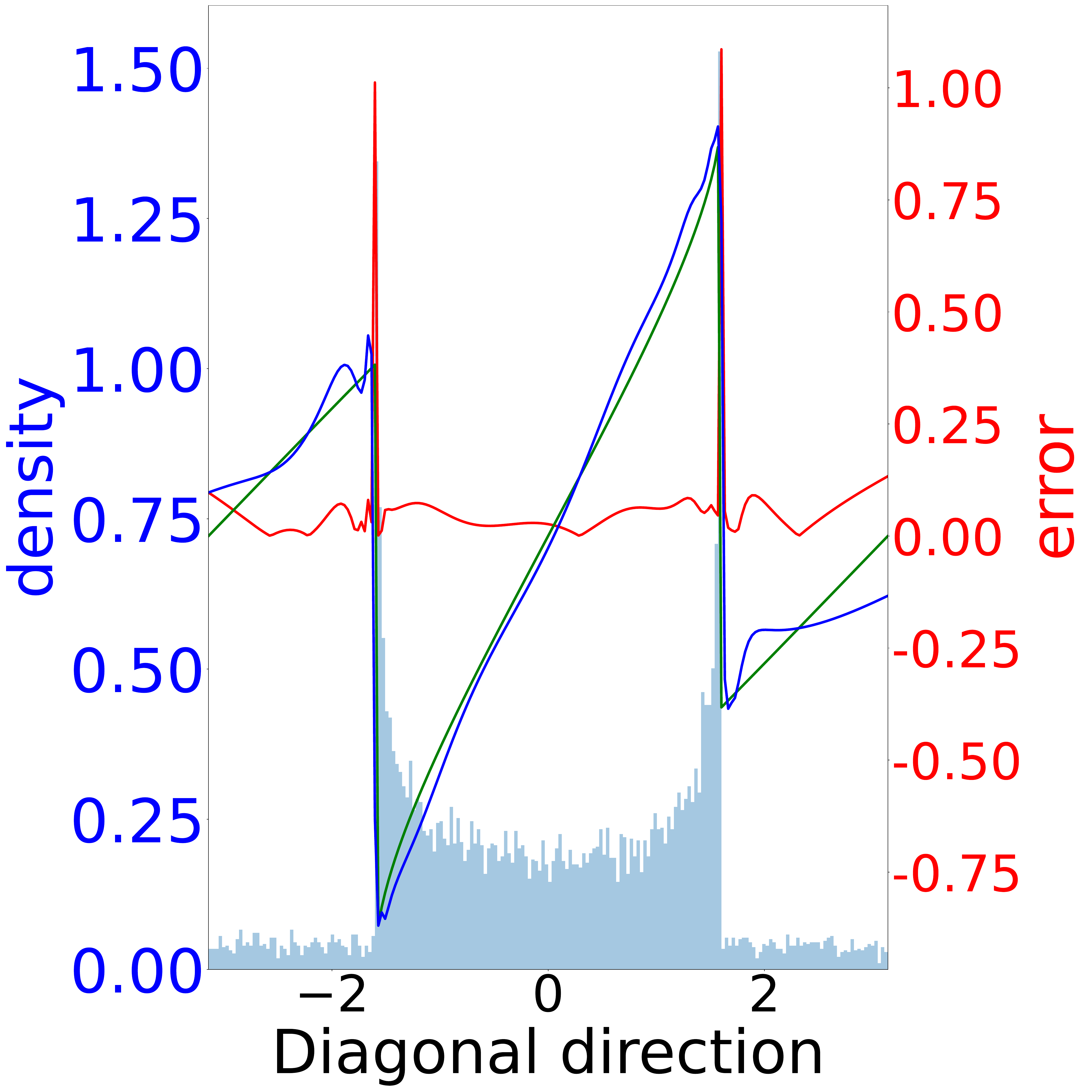}
  \subcaption{$t=3.0$}
\end{subfigure}
\vspace{-0.4cm}
\caption{{ \textbf{Left to Right}: Plots of NN approximation $\boldsymbol{\eta}^\top\nabla\psi_\theta(z\boldsymbol{\eta})$ (blue), weighted momentum (weak solution) $\partial_z \widehat{f}(z, t)$ (green), and the error term $|\boldsymbol{\eta}^\top\nabla\psi_\theta(z\boldsymbol{\eta}) - \partial_z \widehat{f}(z, t)|$ versus $z$ ($-\pi \leq z \leq \pi$) (red) at different time points. The background histogram indicates the probability distribution $\widehat{\rho}_t$ at time $t$.}}
\label{fig: 2D_caustic_compare_NN_exact_separate_time}
\end{figure}
Although the neural network approximation exhibits oscillatory behavior as $z$ approaches the discontinuity points $\pm z_t^*$, it accurately captures the value of $\partial_z \widehat{f}(\cdot, t)$ outside small neighborhoods around $\pm z_t^*$, given a moderately higher sample density.

Due to the page limitation, we  present a series of the same 2D Hamilton–Jacobi equation, with or without caustic formation, in section \ref{SM-NN}, generated by varying the initial values $g(\cdot)$ and initial densities $\rho_0$. We also analyze the numerical behavior of the neural network approximation with different activation functions in section \ref{SM-NN}.

\subsection{{Solving HJ equations}}\label{numerical example separable hamilton } 
In this part, we first test our algorithm on the {separable Hamiltonian $H(x,p) = K(p) + V(x)$ with the quadratic kinetic energy $K(p) = \frac{1}{2}|p|^2$}. For the considered example, we apply our method to solve equation \eqref{HJ} with the one-step St\"ormer–Verlet scheme \cite{MR2221614} for the corresponding Hamiltonian system \eqref{alg: Hamilton system}. Extra examples including an HJ equation with non-separable Hamiltonian $H(x,p)$ are presented in subsections \ref{example: sinusoidal_initial}-\ref{example: nonseparable hamiltonian }. Finally, in example \ref{example: LQC inverted pendulum }, we apply our algorithm to the linear quadratic control (LQC) problem of inverted pendulums with terminal density constraint. The hyperparameters used in our algorithm for each numerical example are listed in Table \ref{tab : hyperparameter } of the supplementary material.

\subsubsection{Example with quadratic potential}\label{example: harmonic}
We set {the potential and the initial condition}  as {$V(x) = \frac 12 |x|^2$} and $g(x)=\frac{1}{2}|x|^2$. {We choose $\rho_0 = \mathcal N(\underbrace{(3,...,3)}_{30}, I)$ }  {and} solve this equation on $[0, 5].$ 

It can be verified directly that $u(x,t)=\frac{1}{2}\cot(t+\frac{\pi}{4})|x|^2$ is the classical solution to the equation on $[0, \frac{3\pi}{4})$. When $t$ approaches  $T_*=\frac{3\pi}{4}$, {this classical} solution blows up.  Our method is able to compute both the classical solution as well as the extended solution beyond $T_*$.

{The solution to this HJ equation possesses a rather strong oscillatory profile along time $t$. Due to the rigidity of the neural network, given $T=5$, it is generally difficult for a single neural network to capture the overall shape of $\{u(x, t)\}_{t\geq 0}$ \cite{ziyin2020neural}.}

{As a remedy, in order to make our computation more efficient, we apply the multi-interval training strategy in this example.} We separate $[0, T]$ into multiple shorter subintervals and train different neural networks on each subinterval. Our experiments indicate that such treatment of training the networks independently on each subinterval and concatenating together improves the flexibility of the numerical solution $\psi_\theta(x,t)$ and thus enhances the performance. To be more specific, we divide $[0, T]$ into $M_T=25$ equal intervals, i.e., $[0, T]=\bigcup_{k=1}^{M_T}I_k$ with each $I_k=[\frac{k-1}{M_T}T, \frac{k}{M_T}T)$ for $1\leq k\leq M_T-1$ and $I_{M_T} = [\frac{M_T-1}{M_T}T, T]$. We train $\psi_{\theta_k}$ on each $I_k$ and set $\psi_\theta(x,t) = \sum_{k=1}^{M_T}\chi_{I_k}(t)\psi_{\theta_k}(x,t)$ as our numerical solution. Here $\chi_{I_k}$ is the indicator function of time interval $I_k$.

We demonstrate the numerical solutions in Figure \ref{harmonic 30d plot vector field}. Since the solution is a high dimensional function, we plot its graph {on  the $5$-th and $15$-th coordinates}. For convenience, we call it {$5\text{th}$-$15\text{th}$} plane. It is observed that both the solution and vector field have good agreements with their exact counterparts at the regions where samples are drawn.  
\begin{figure}[htb!]
\begin{subfigure}{.18\textwidth}
  \centering
  \includegraphics[width=\linewidth]{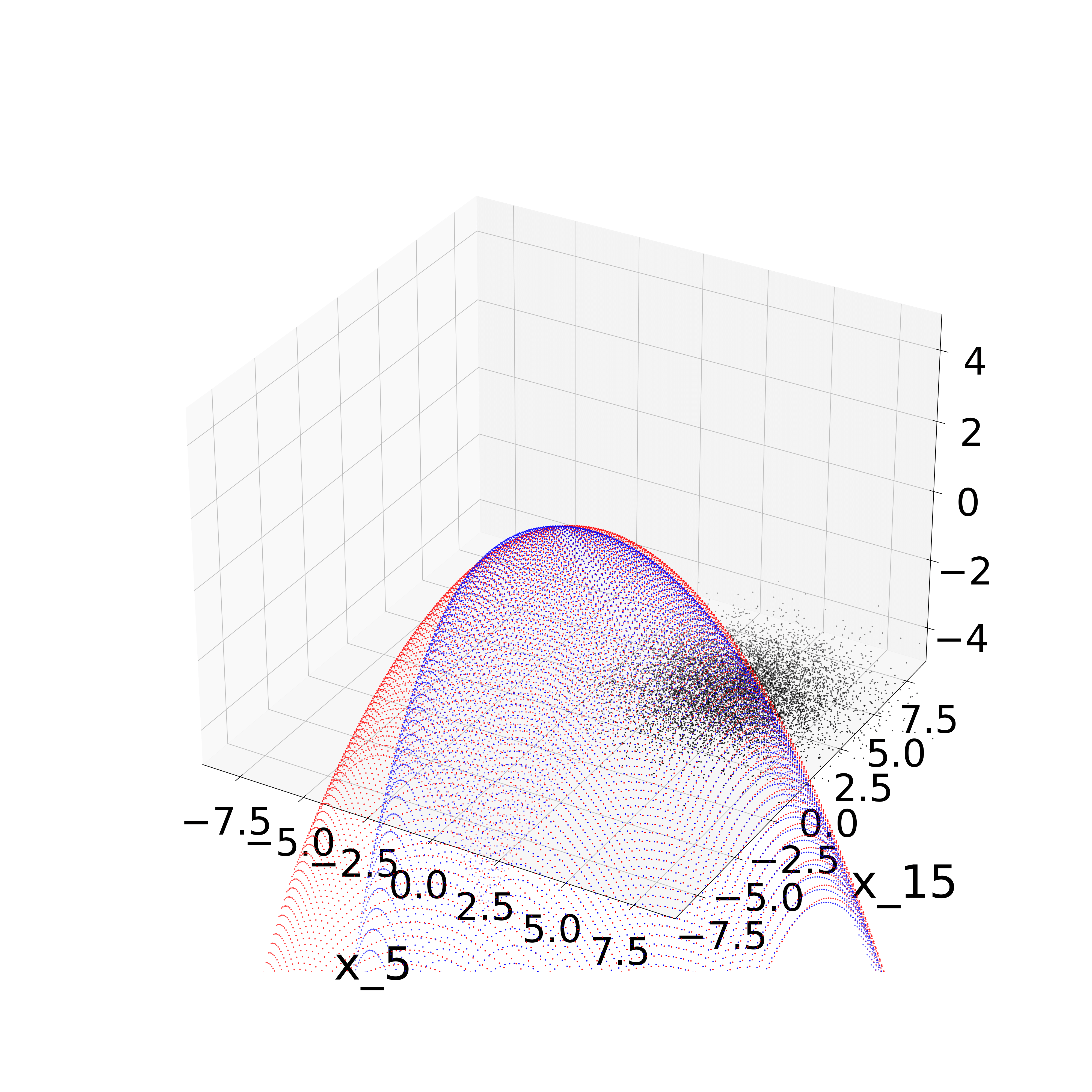}
\end{subfigure}
\begin{subfigure}{.18\textwidth}
  \centering
  \includegraphics[width=\linewidth]{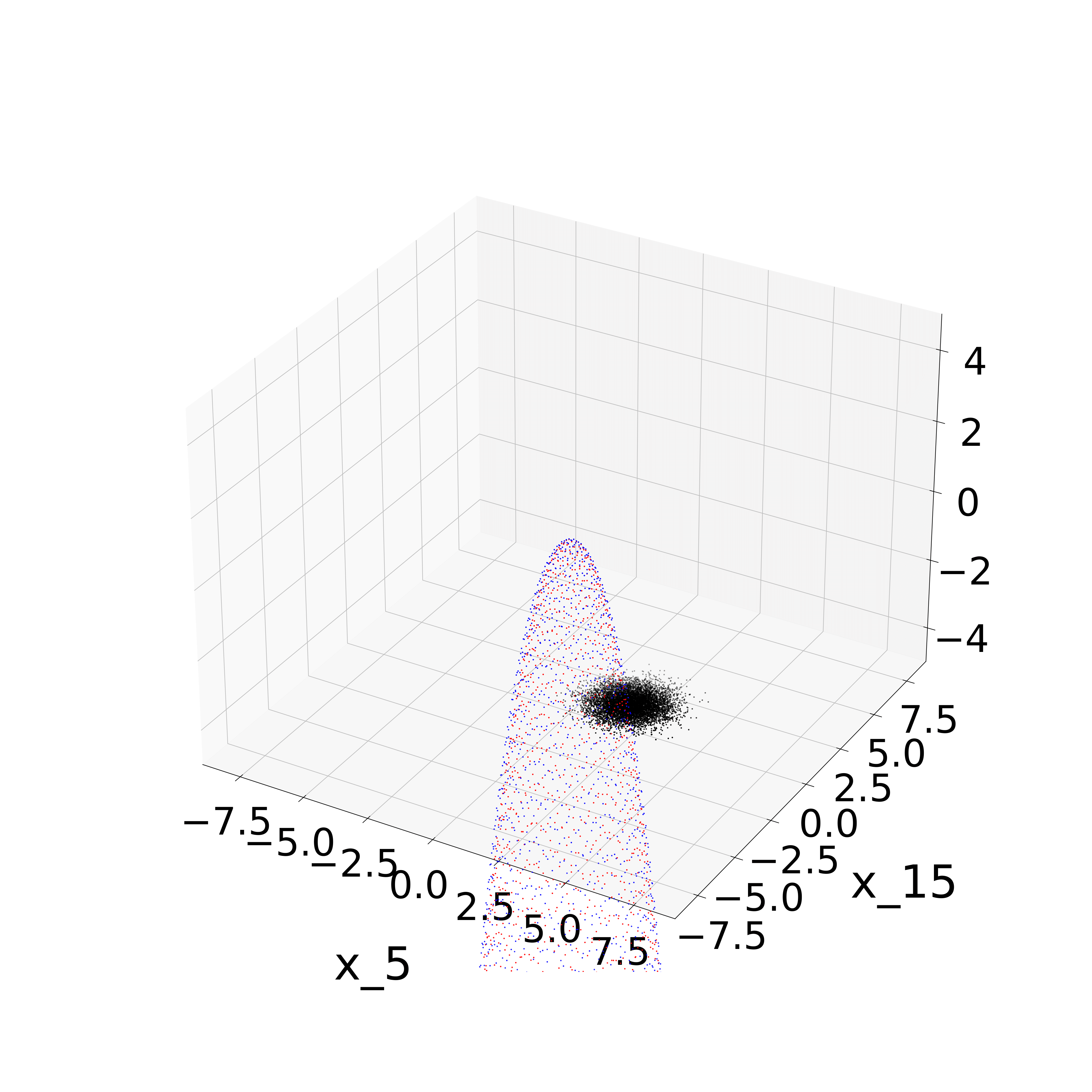}
\end{subfigure}
\begin{subfigure}{.18\textwidth}
  \centering
  \includegraphics[width=\linewidth]{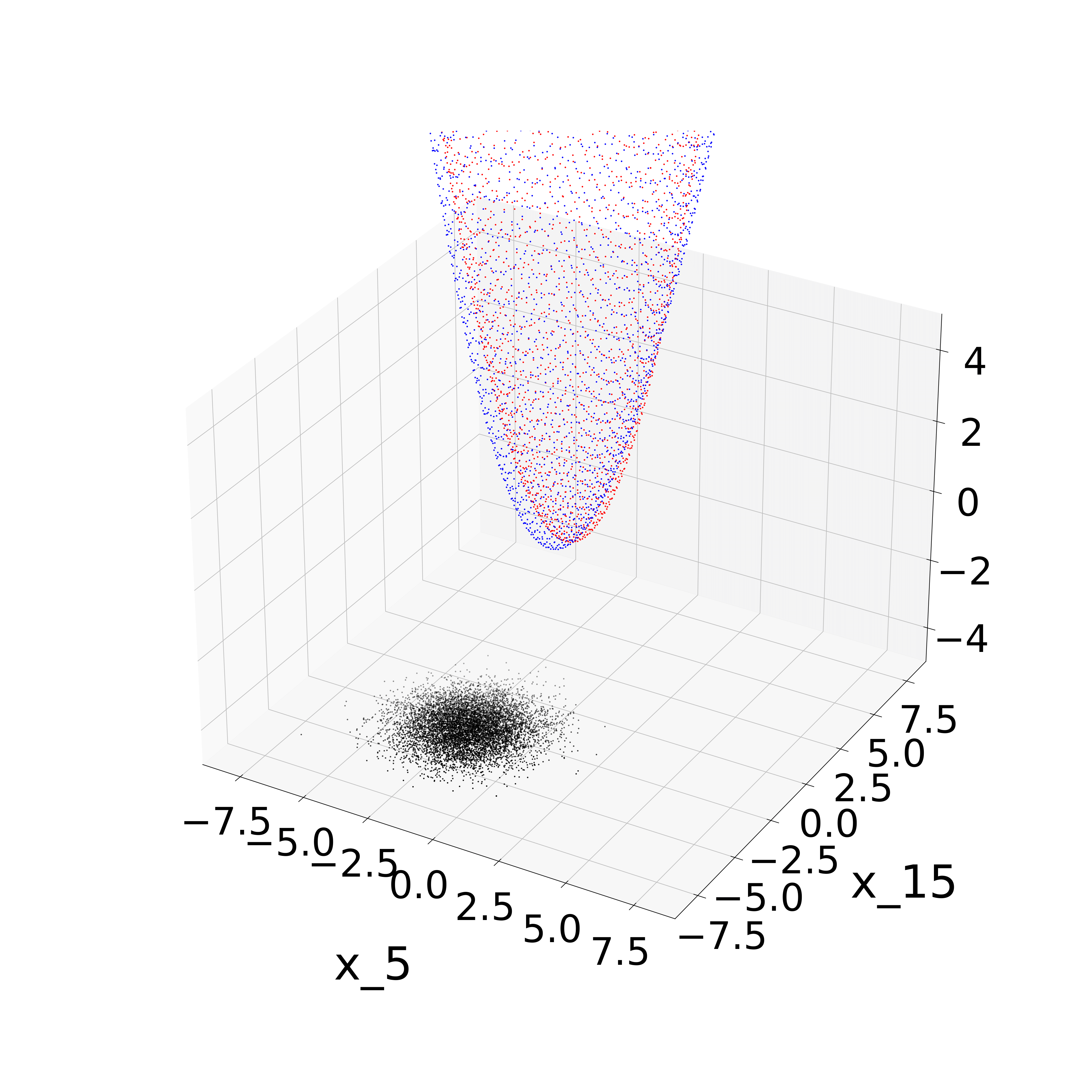}
\end{subfigure}%
\begin{subfigure}{.18\textwidth}
  \centering
  \includegraphics[width=\linewidth]{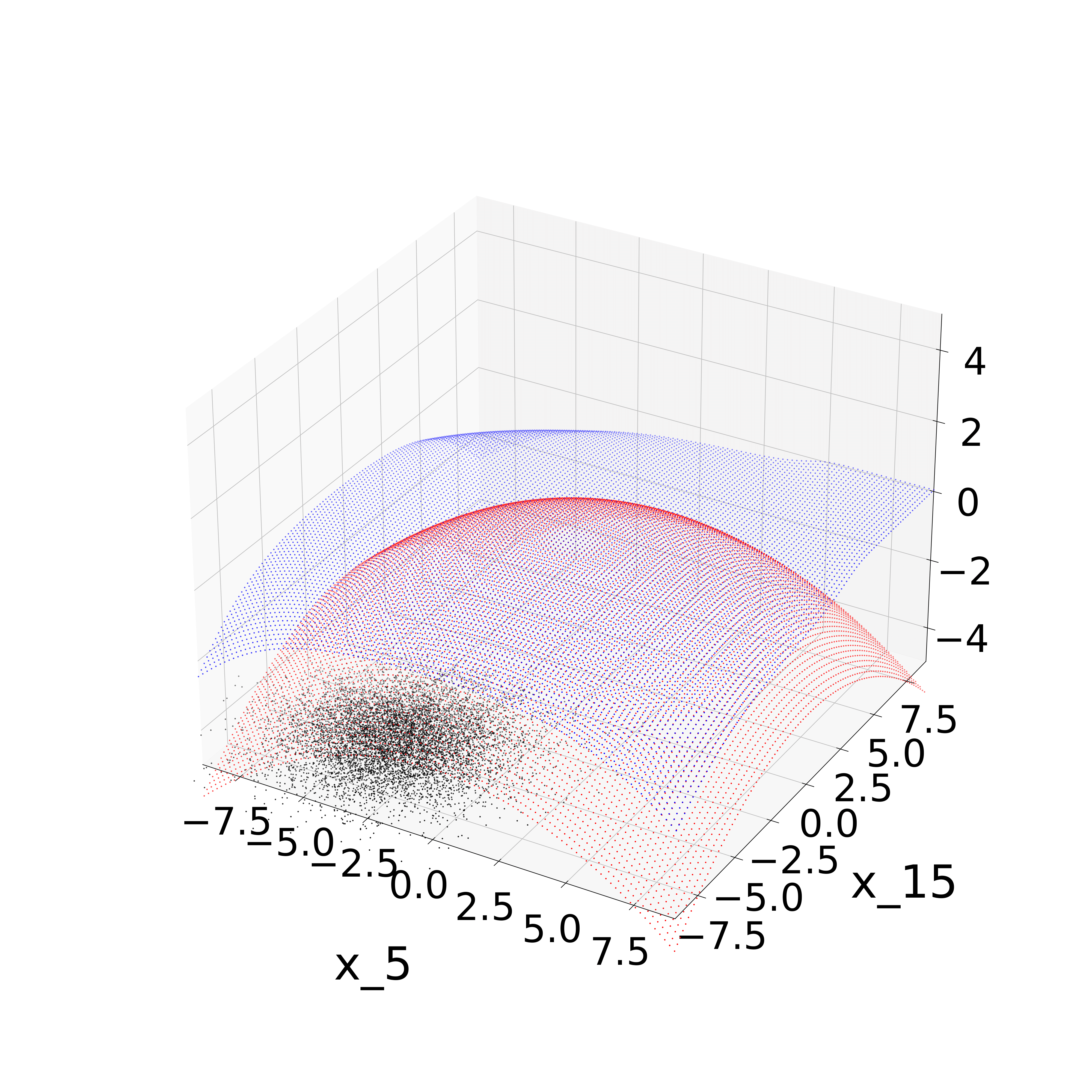}
\end{subfigure}
\begin{subfigure}{.18\textwidth}
  \centering
  \includegraphics[width=\linewidth]{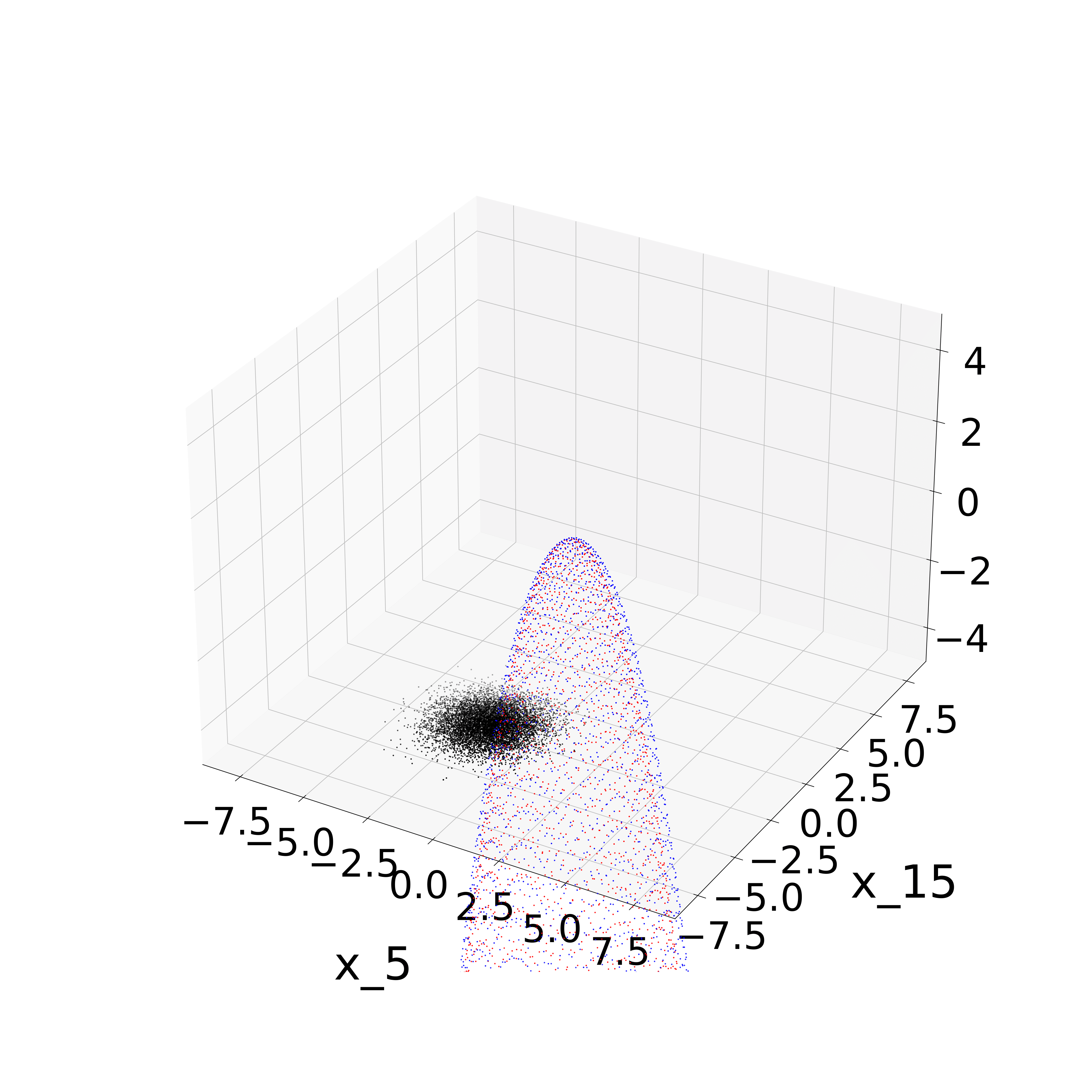}
\end{subfigure}
\begin{subfigure}{.18\textwidth}
  \centering
  \includegraphics[width=\linewidth]{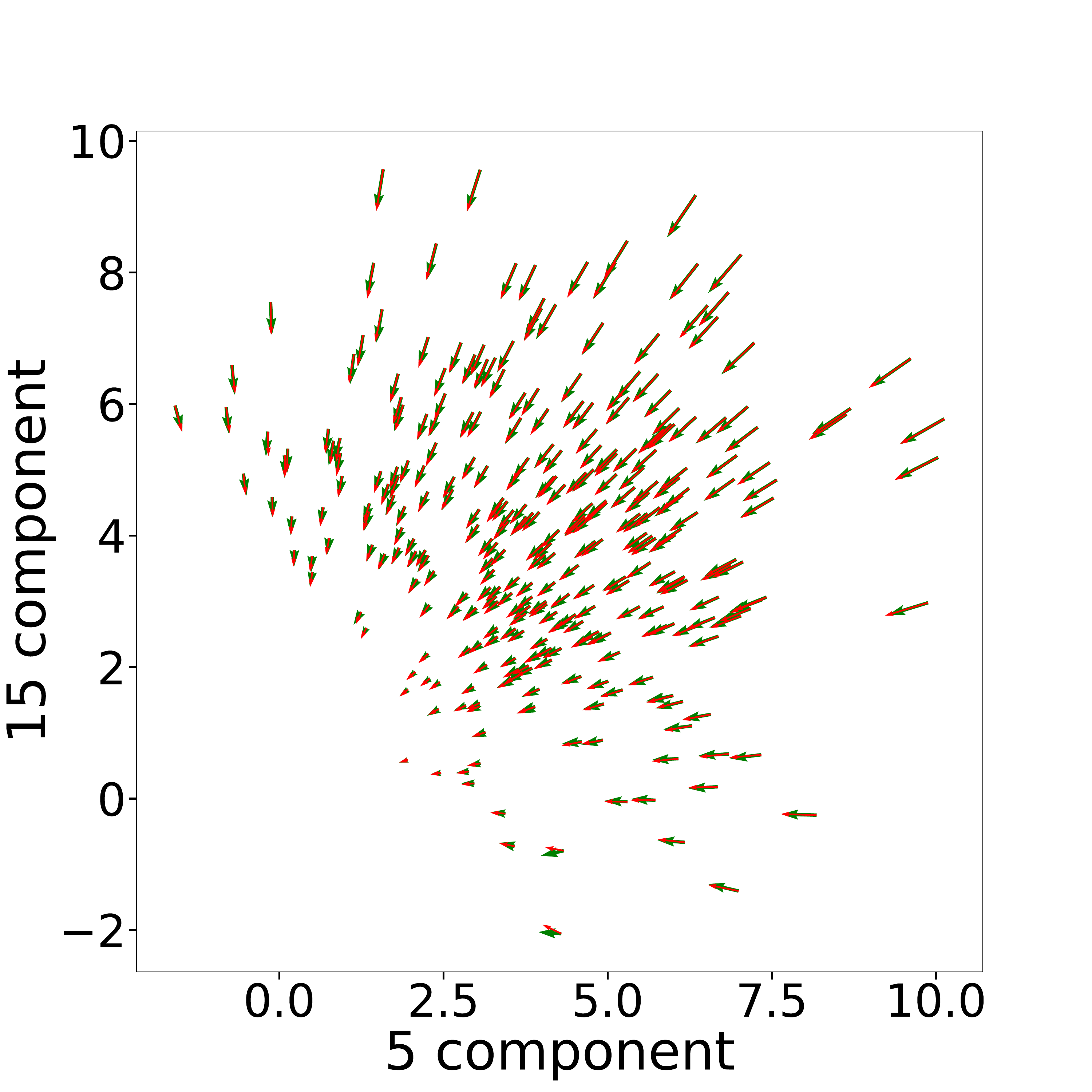}
  \caption{$t=1.0$}
\end{subfigure}\hfill
\begin{subfigure}{.18\textwidth}
  \centering
  \includegraphics[width=\linewidth]{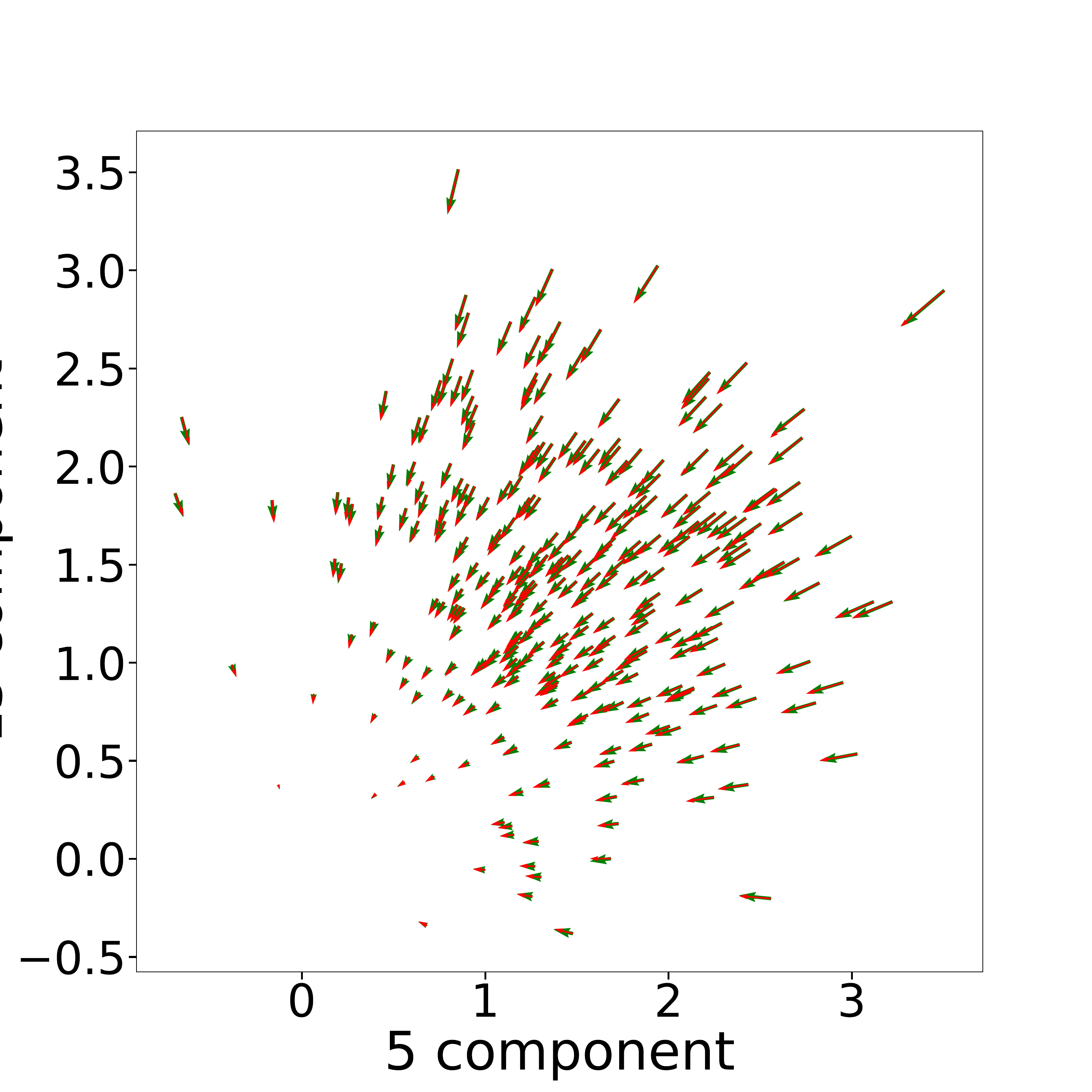}
  \caption{$t=2.0$}
\end{subfigure}\hfill
\begin{subfigure}{.18\textwidth}
  \centering
  \includegraphics[width=\linewidth]{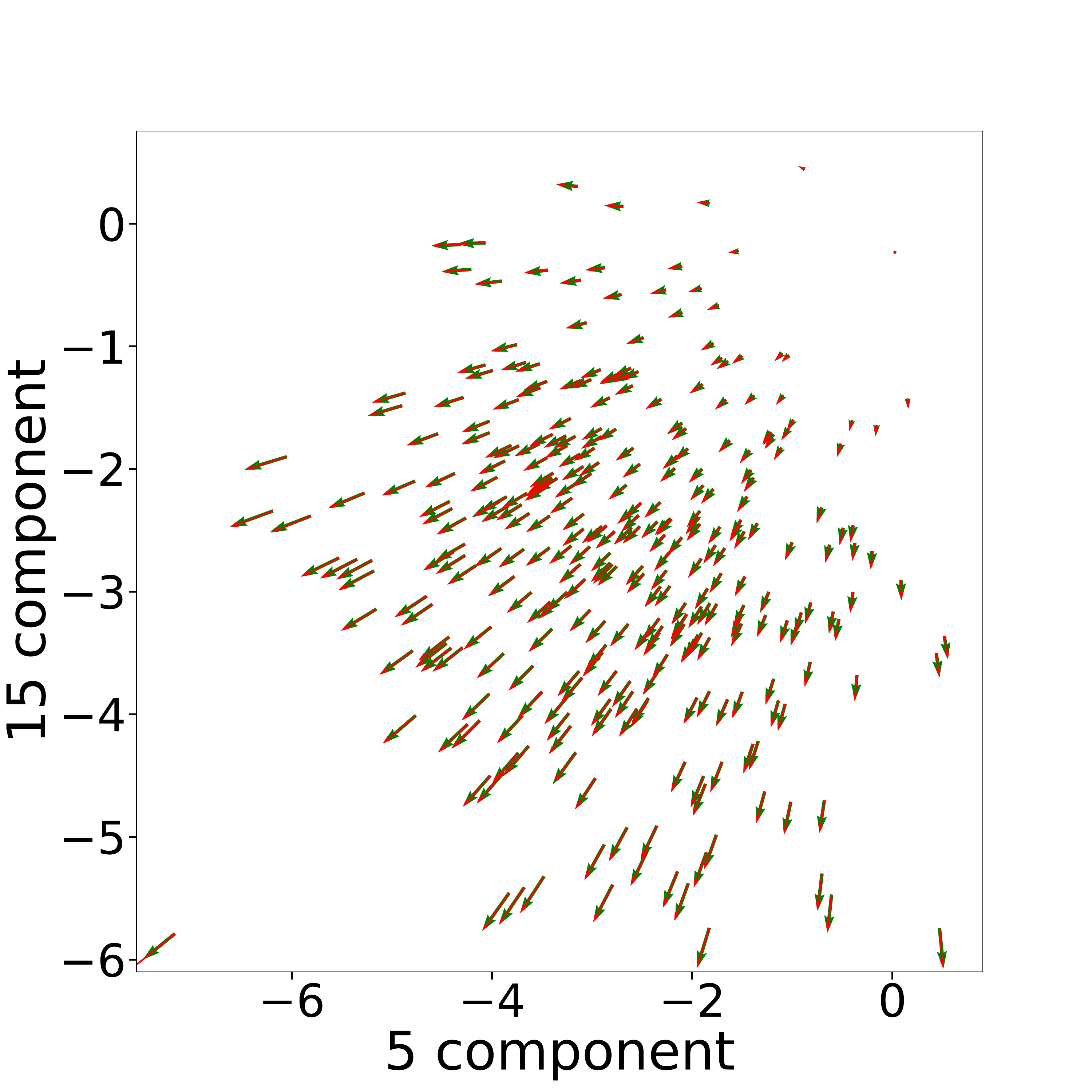}
  \caption{$t=3.0$}
\end{subfigure}\hfill
\begin{subfigure}{.18\textwidth}
  \centering
  \includegraphics[width=\linewidth]{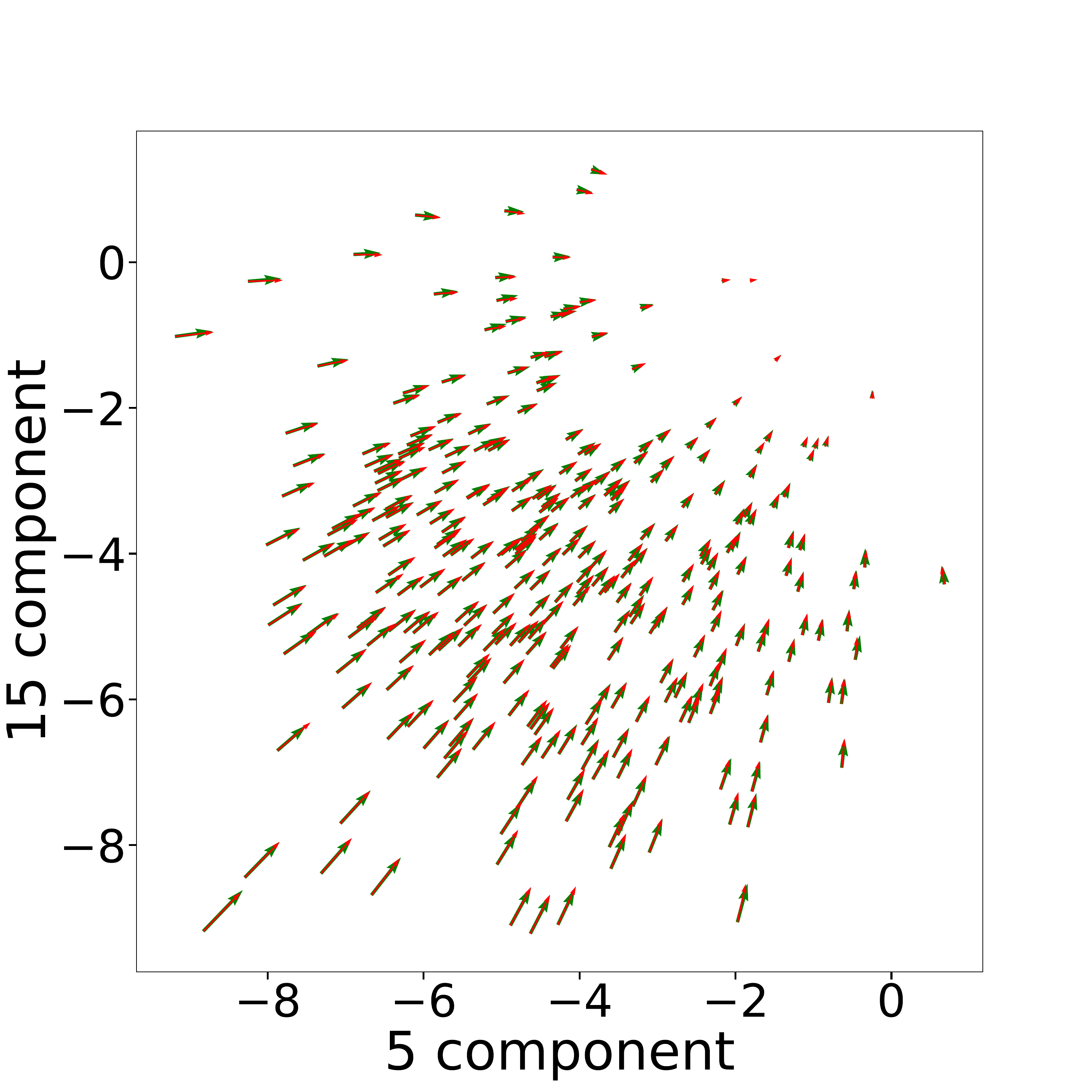}
  \caption{$t=4.0$}
\end{subfigure}\hfill
\begin{subfigure}{.18\textwidth}
  \centering
  \includegraphics[width=\linewidth]{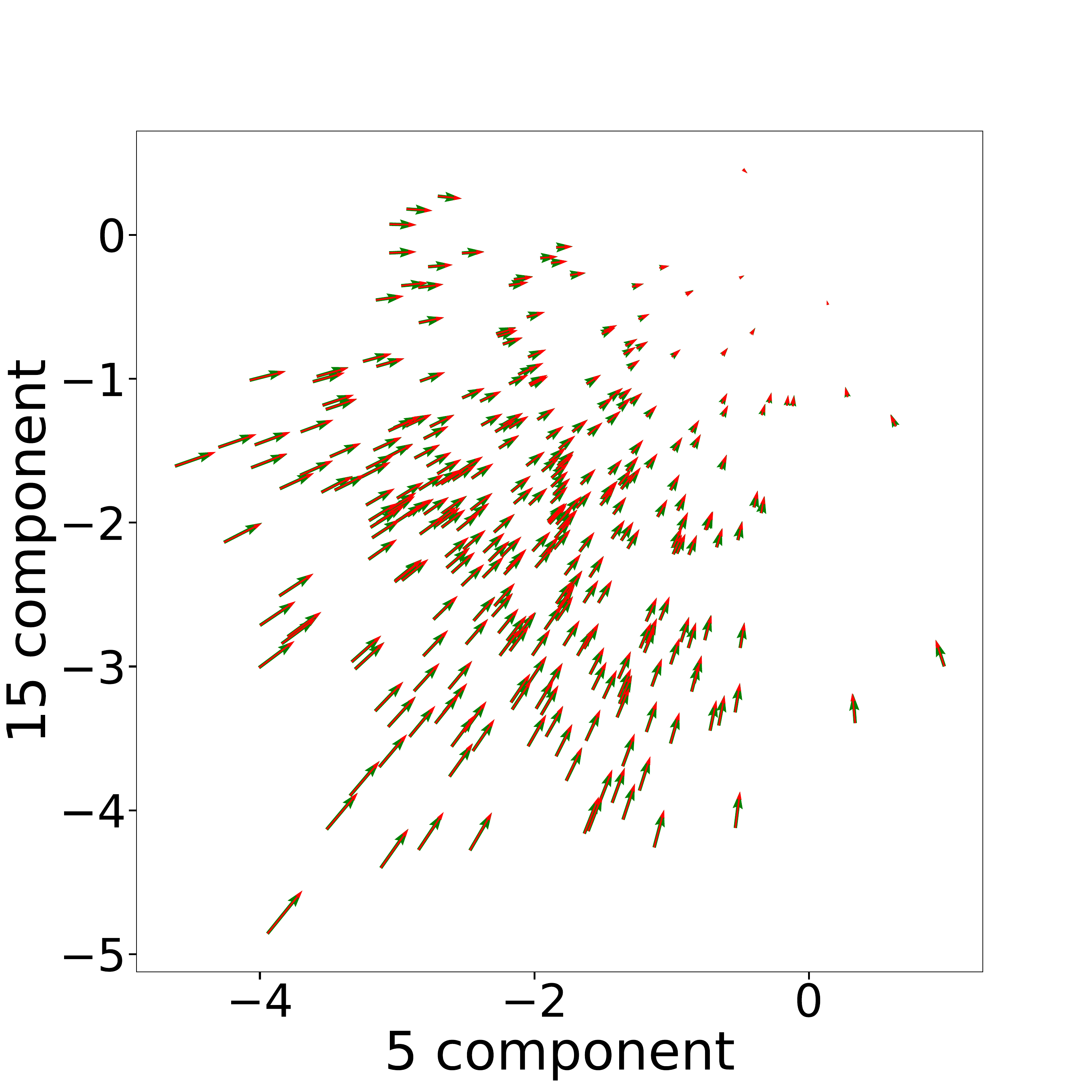}
  \caption{$t=5.0$}
\end{subfigure}  \hfill
\vspace{-0.3cm}
\caption{{\textbf{First row:} Graphs of the numerical solution $\psi_\theta$ (blue) and the exact solution (red) at different time stages on the {$5\text{th}-15\text{th}$} plane; \textbf{Second row:} Plots of vector fields $\nabla \psi_\theta(\cdot, t)$ (green) with momentums of samples (red) at different time stages on the {$5\text{th}-15\text{th}$} plane.}}
\label{harmonic 30d plot vector field}
\end{figure}
{Recall the $\{\epsilon_i^N\}$ defined in \eqref{def: epsilon N and delta N },} we calculate the total loss {$\sum_{i=(j-1)l}^{jl-1} \varepsilon_i^N$} among the time nodes located in the subinterval $I_j$, where  $l=\frac{M}{M_T}$, and plot $\sum_{i=(j-1)l}^{jl-1} \varepsilon_i^N$ ($1\leq j\leq M_T$) versus time in Figure \ref{error vs time plots} (left-most) in subsection \ref{append: hyperparams_alg}.

It is clear that the error increases significantly around $T_*=\frac{3\pi}{4}\approx2.36$. According to our experience, it is intrinsically difficult to compute the solution near singular point $T_*$.

\subsubsection{Swarming particles with LQC of inverted pendulum with terminal distribution constrain}\label{example: LQC inverted pendulum }

We apply the proposed method to the inverted pendulum model \cite{franklin2002feedback}\cite{sultan2003inverted}. We refer the reader to \cite{franklin2002feedback} for the basic settings of inverted pendulum control. Let us denote the position of the cart as $x_t$, and the angle between the stick and the vertical direction as $\zeta_t$ at time $t$ (we take the counter-clockwise as the positive direction for $\zeta_t$). Suppose we exert a force $u_t$ on the cart at time $t$, the mechanics of the cart and the stick are governed by the following differential equation (The equation has been linearized at $\zeta \approx 0, \dot \zeta \approx 0$.)
\begin{align*}
  u_t = & (M+m)\ddot x_t - ml\ddot \zeta_t  \\
  l\ddot \zeta_t  = & g\zeta_t  +  \ddot x_t.
\end{align*}
This yields
\begin{align*}
  \ddot x_t = & \frac{m}{M}g\zeta_t + \frac{u_t}{M} \\
  \ddot \zeta_t = & \frac{M+m}{Ml}g\zeta_t + \frac{u_t}{Ml}.
\end{align*}
By introducing $y_t = \dot x_t, \phi_t = \dot \zeta_t,$ we consider the following dynamics of 
\[  \textbf{X}_t = (x_t, y_t, \zeta_t, \phi_t)  \] 
with the external force $u_t$ as the control,
\begin{equation*}
  \dot {\left[\begin{array}{c}
       x_t  \\
       y_t  \\
       \zeta_t    \\
       \phi_t  
  \end{array}\right]} = \left[\begin{array}{cccc}
      0  & 1 & 0 &  0  \\
      0  & 0 & \frac{m}{M} & 0 \\
      0 & 0 & 0 & 1 \\
      0 & 0 & \frac{M+m}{Ml}g & 0 
  \end{array}\right]  \left[ \begin{array}{c}
       x_t  \\
       y_t  \\
       \zeta_t  \\
       \phi_t    
  \end{array} \right]  +  \left[\begin{array}{c}
       0  \\
       \frac{1}{M}  \\
       0  \\
       \frac{1}{Ml}
  \end{array}\right] \left[ u_t \right] \overset{\textrm{denote as}}{ = } A \textbf{X}_t + B u_t. 
\end{equation*}
We wish to exert the control $\{u_t\}$ to this dynamics so that both the cart and the stick stay stably, and at the same time, minimize the effort $u_t$ paid to the control. Thus, we consider the following cost functional
\begin{equation*}
  \int_0^T \frac12 \textbf{X}_t^\top Q\textbf{X}_t + \frac12 R u_t^2 + \frac12 \textbf{X}_T^\top P_1 \textbf{X}_T.
\end{equation*}
Here we pick $Q = P_1 = \textrm{diag}(1, 0, 1, 0)$, $R=1$. This is a optimal control problem in 4-dimensional phase space of $x, \zeta$. We assume the terminal distribution $\rho_T$ as $\mathcal N(0, \sigma_x^2 I_2)\otimes \mathcal U([-\zeta_0, \zeta_0]) \otimes \mathcal N (0, \sigma_{\dot \zeta}^2)$. That is, if $(x, \dot x, \zeta, \dot \zeta)\sim \rho_T$, then $(x, \dot x)\sim \mathcal N(0, \sigma_x^2 I_2)$, $\theta \sim \mathcal U([-\zeta_0, \zeta_0])$, $\dot\zeta \sim \mathcal N(0, \sigma_{\dot \zeta }^2)$. Here $\mathcal U([a,b])$ denotes the uniform distribution on the interval $[a, b]$. In this example, we set $\sigma_x = \sigma_{\dot \zeta} = 0.2, \zeta_0 = \frac{\pi}{20}$. We pick terminal $T=2$.  We then apply our algorithm to compute for $\{\psi_\theta(\cdot, t)\}_{0 \leq t \leq T}$ as the solution to the HJ equation,
\begin{align}
  \frac{\partial \widetilde{u}(x,t)}{\partial t} + \underbrace{\frac12 (B^\top \nabla \widetilde{u}(x,t))^\top R^{-1} (B^\top \nabla \widetilde{u}(x,t)) - \nabla \widetilde{u}(x,t)^\top Ax - \frac12 x^\top Qx}_{H(x, \nabla u(x))} = 0, \nonumber & \\
  \widetilde{u}(x,0) = \frac{1}{2} x^\top P_1x. &  \nonumber
\end{align}
To carry out our computation, we evolve the associated   Hamiltonian system 
\begin{equation*}
          \dot q_t = \partial_p H(q_t, p_t) , \quad \dot p_t = - \partial_x H(q_t, p_t). \quad \textrm{with } q_0 \sim \widetilde{\rho}_0, ~ p_0 = P_1 q_0.
\end{equation*}
This yields the linear ODE system
\begin{equation}
    {\left[\begin{array}{c}
        \dot  q_t \\
        \dot  p_t  
    \end{array}\right]} = \left[ \begin{array}{cc}
       -A    &   BR^{-1}B^\top  \\
        Q    &    A^\top
    \end{array} \right] \left[ \begin{array}{c}
         q_t    \\
         p_t    
    \end{array} \right], \quad  \begin{array}{c}
         q_0 \sim \widetilde{\rho}_0,  \\
         p_0 = P_1 q_0. 
    \end{array}    \label{Hamiltonian system associated with inverted pendulum}
\end{equation}
We denote $\widetilde{\rho}_T$ as the density of $\mathrm{Law}(q_T)$. The
initial samples are drawn from $\widetilde{\rho}_0 = \rho_T$.
A more detailed discussion is included in subsection \ref{example: LQC inverted pendulum 1}.

Moreover, upon evolving \eqref{Hamiltonian system associated with inverted pendulum}, we denote $\widetilde{\rho}_T $ as the distribution of terminal particles. We set the initial distribution of the swarm  $\rho_0$ as $\widetilde{\rho}_T$. 
For any samples of $\rho_0$, we calculate the trajectory under our learned control $\{\psi_\theta(\cdot, t)\}_{0 \leq t \leq T}$ and compare it with the trajectory under the optimal control (i.e., the trajectories solved from the Pontryagin's minimum principle, see  \eqref{Pontryagin Principle of LCQ problem} in subsection \ref{example: LQC inverted pendulum 1}). The results are demonstrated in the Figure \ref{fig: LQC plot traj compare} and \ref{fig: LQC plot x vs t and theta vs t compare}. 
\begin{figure}[htb!]
    \centering
    \begin{subfigure}{0.45\textwidth}
        \includegraphics[trim={0cm 0cm 0cm 10cm},clip, width=\textwidth]{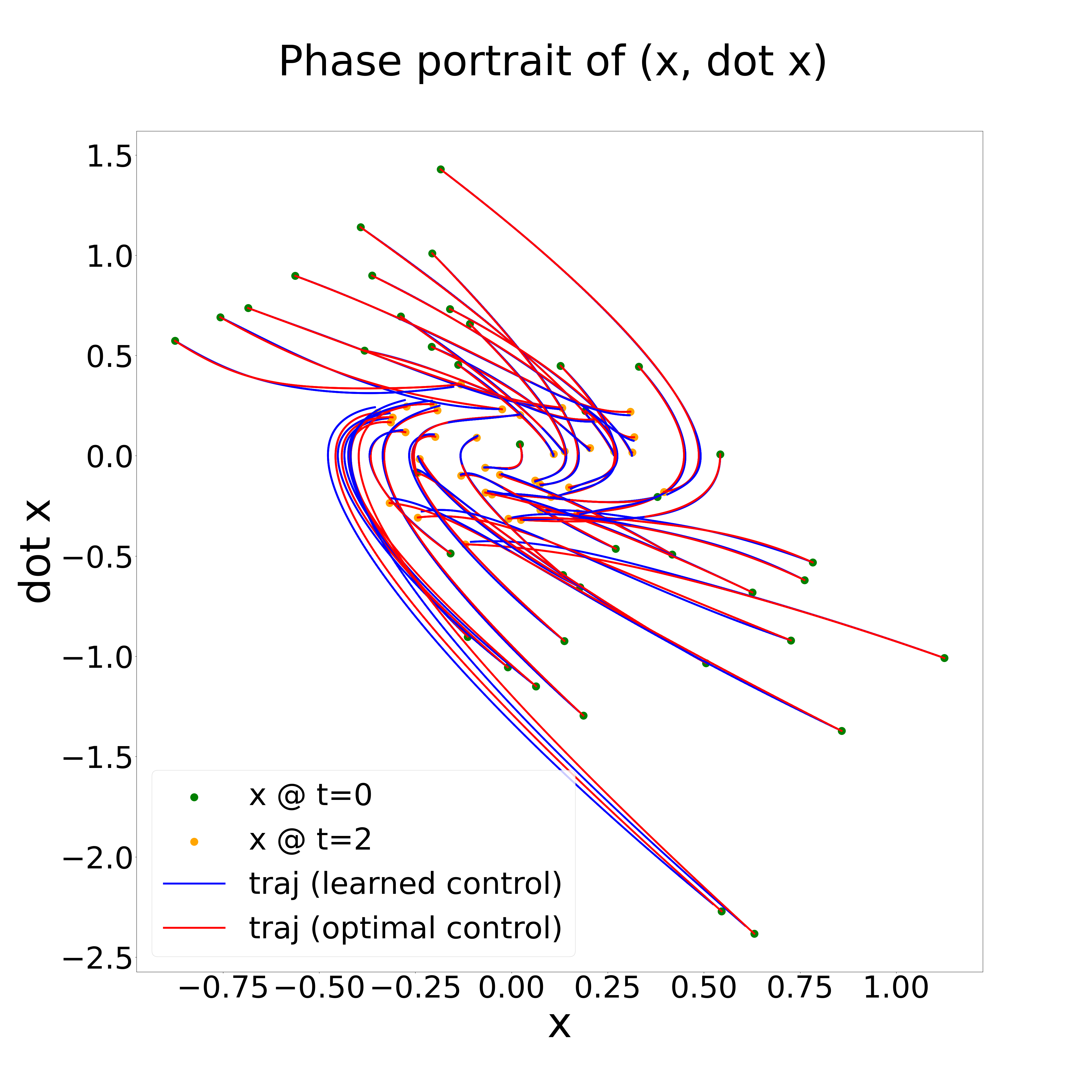}
    \end{subfigure}
    \begin{subfigure}{0.45\textwidth}
        \includegraphics[trim={0cm 0cm 0cm 10cm},clip, width=\textwidth]{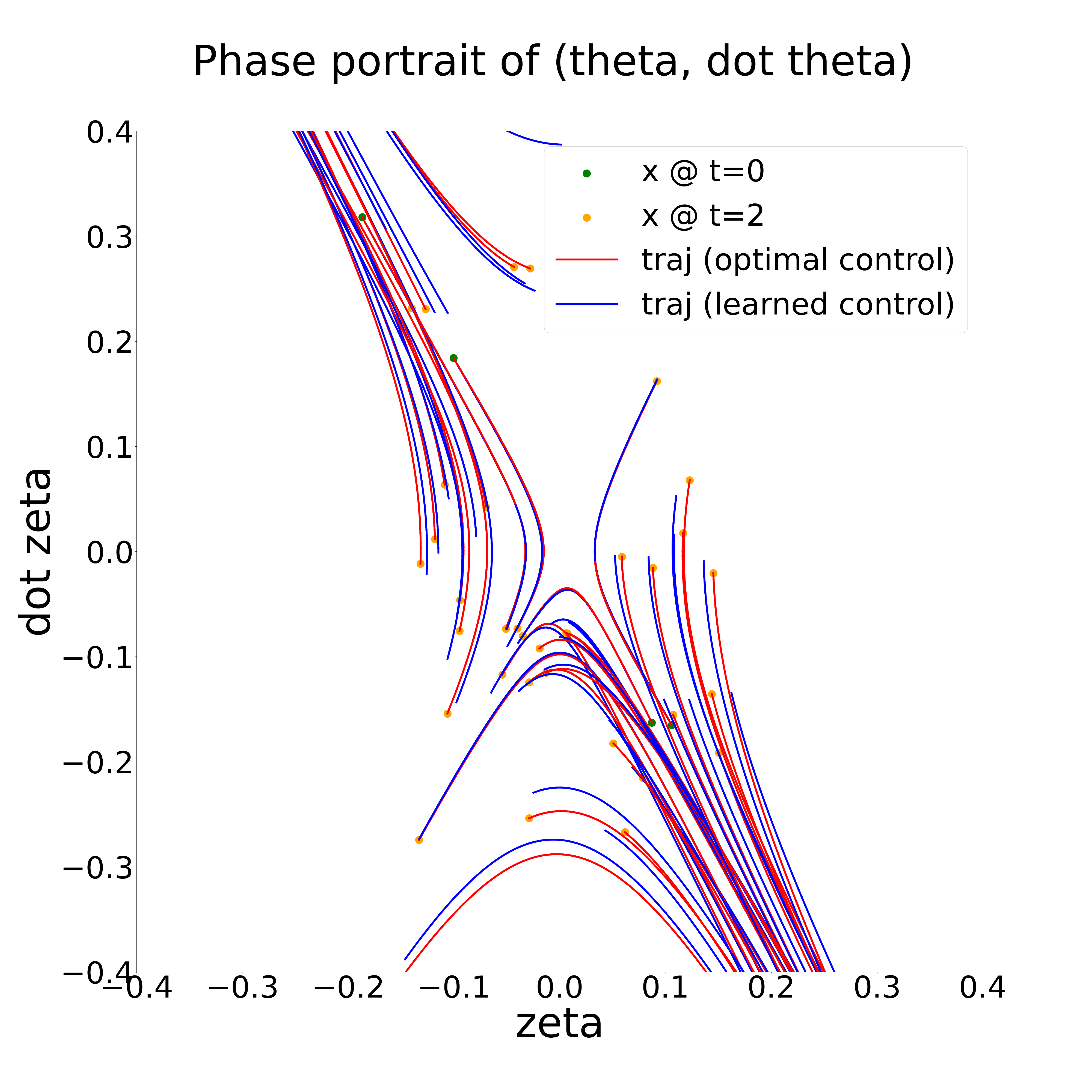}
    \end{subfigure}
    \vspace{-0.3cm}
    \caption{{Plot of different trajectories under the learned control $-R^{-1}B^\top\nabla\psi_\theta(\cdot, t)$ (blue) and the corresponding trajectories under the optimal control (red). \textbf{Left:} plot on $(x, \dot x)$ plane; \textbf{Right:} plot on $(\zeta, \dot\zeta)$. We showcase the control over 40 agents.}}\label{fig: LQC plot traj compare}
\end{figure}
{The $L^2$ loss decay curve shown in the rightmost subfigure \ref{fig: LQC plot x vs t and theta vs t compare}} converges exponentially to $0$, suggesting that our algorithm works properly on this example.
\begin{figure}[htb!]
    \centering
    \begin{subfigure}{0.32\textwidth}
        \includegraphics[width=\textwidth]{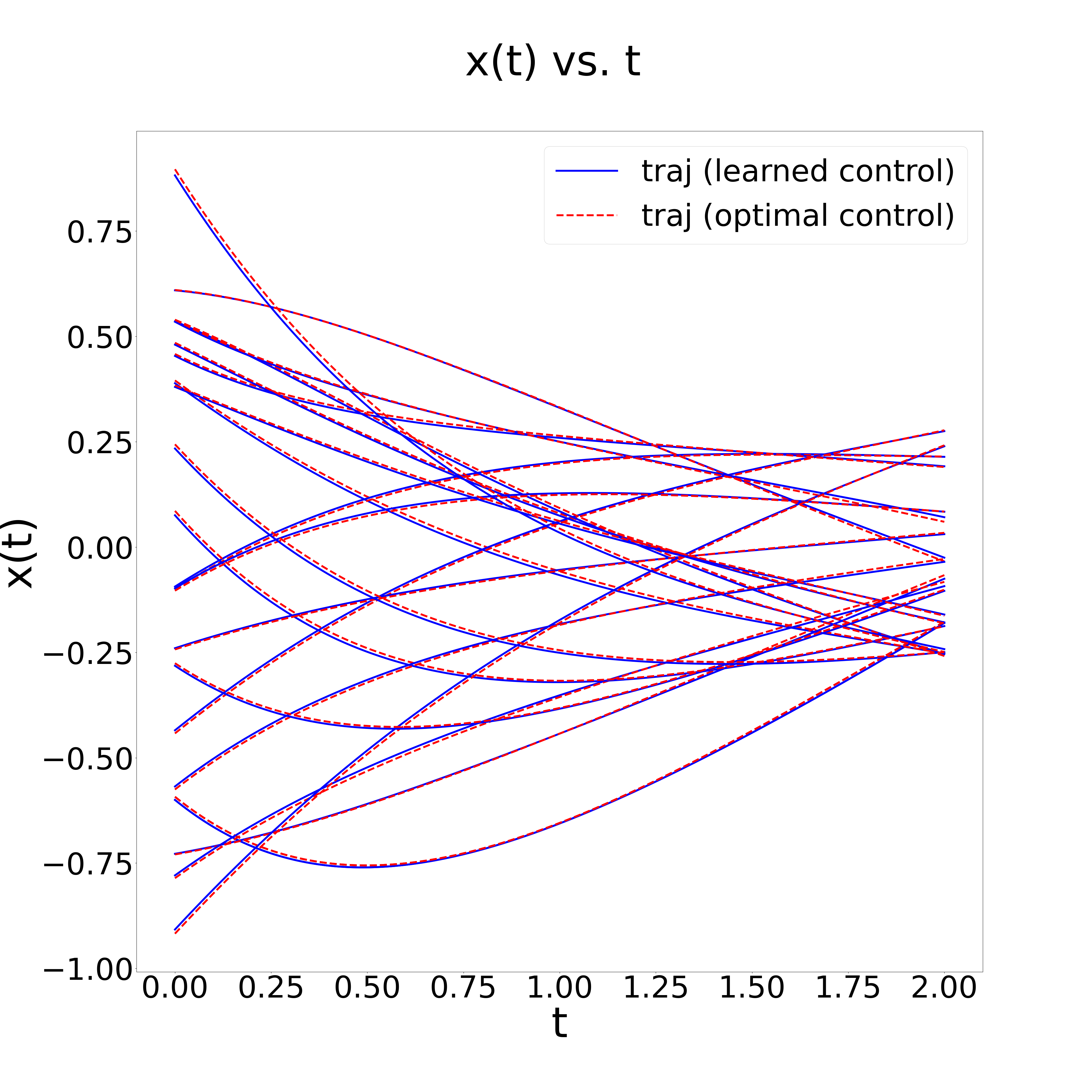}
    \end{subfigure}
    \begin{subfigure}{0.32\textwidth}
        \includegraphics[width=\textwidth]{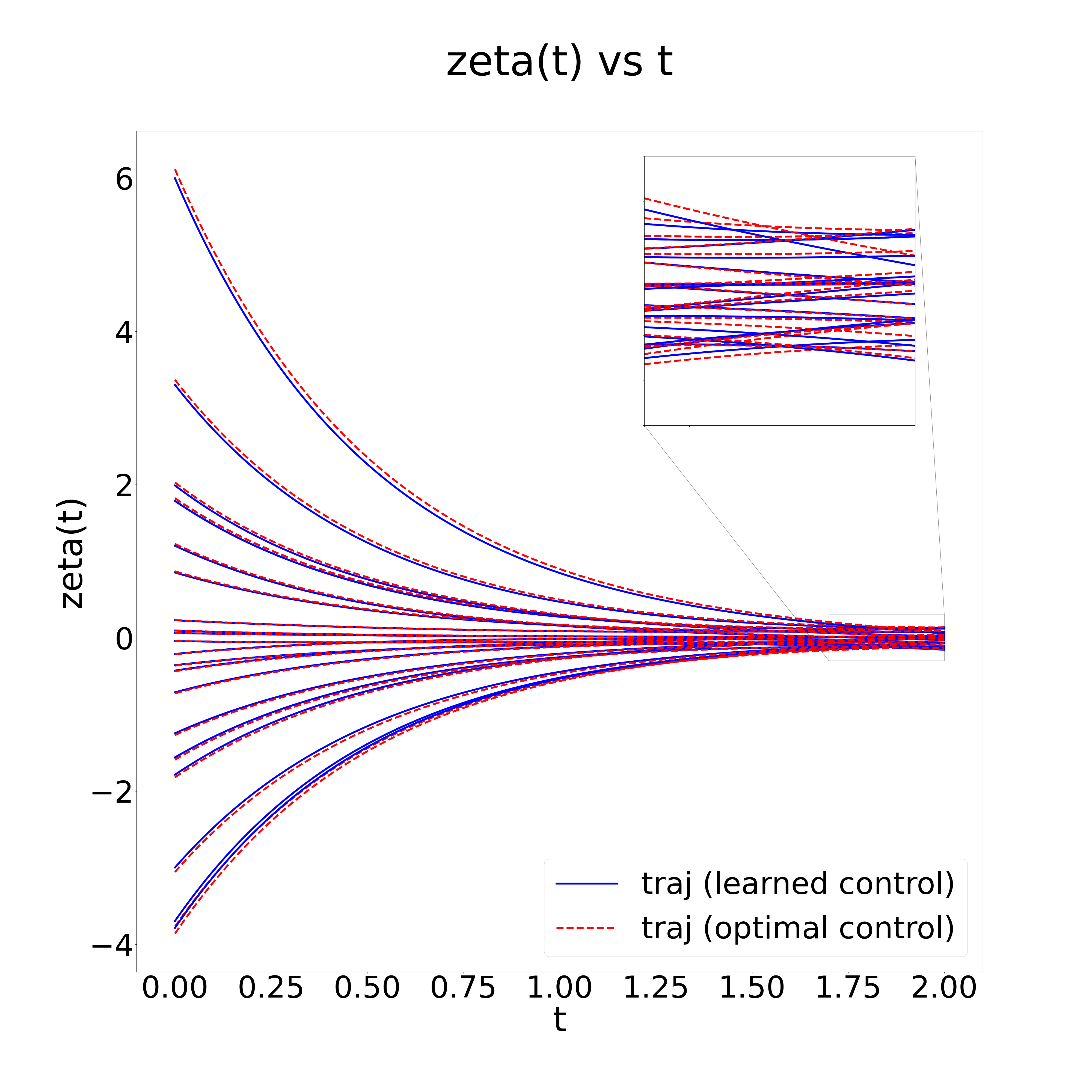}
    \end{subfigure}
    \begin{subfigure}{0.32\textwidth}
        \includegraphics[width=\linewidth]{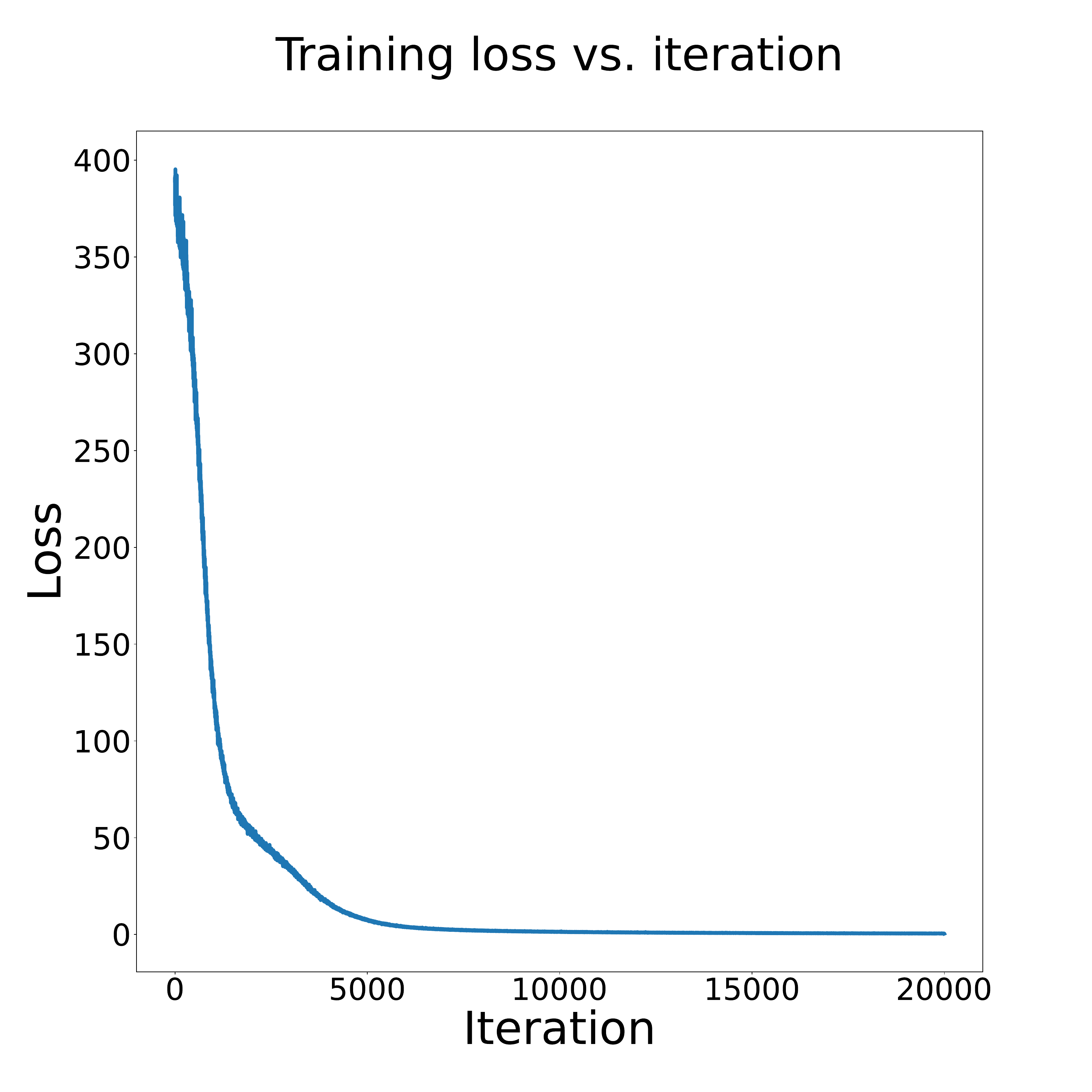}
    \end{subfigure}
    \vspace{-0.3cm}
    \caption{{\textbf{Left two figures:} Plot of different trajectories under the learned control $-R^{-1}B^\top\nabla\psi_\theta(\cdot, t)$ (blue) and the corresponding trajectories under the optimal control (red). First figure: plot of $x_t$ vs. $t$; Second figure: plot of $\zeta_t$ vs. $t$. \textbf{Right figure:} Plot of $L^2$ loss vs iteration in training. For clarity, we illustrate the control of 20 agents here.}}\label{fig: LQC plot x vs t and theta vs t compare}
\end{figure}

\section{Conclusion}\label{ sec-5}
{In this paper, we propose a supervised learning algorithm to compute the first-order {HJ} equation by the density-coupling strategy. Such treatment is inspired by the Wasserstein Hamiltonian flow, which bridges the {HJ} equation and its associated Hamiltonian ODE system. We then reformulate our method as a regression algorithm using the Bregman divergence. Furthermore, we provide error estimation on the $L^1$ residual term for the {proposed} method. The efficiency of our algorithm is verified by a series of numerical examples.

Multiple research directions may serve as the proceeding of this work. To name some of them, 
\begin{itemize}
    \item Our method can compute the solution to the {HJ} equation beyond the caustics, which is different from the commonly considered viscosity solution \cite{crandall1983viscosity}. Is it possible to modify our algorithm at points at which caustics develop to compute the viscosity solution of the {HJ}  equation?
    \item As mentioned in {section} \ref{rk weighted momentum as grad HJ solu }, our treatment leads to a new way to extend the classical solution of HJ equation beyond the caustics. What are the mathematical properties of such a solution? What is the relationship between this solution and the viscosity solution to the HJ equation?
    \item As discussed in section \ref{sec: bound res }, we are not able to control the residual outside of the support of the swarm of particles. How can we propose the initial distribution $\rho_0$ such that the support of $\rho_t$ covers the desired region on which we wish to obtain the accurate solution to the HJ equation?
\end{itemize}
We leave these topics to be investigated in the future.}

\bibliographystyle{plain}
\bibliography{references}
\newpage

\appendix
\section{Proof of Lemma 2.1}
In order to prove Lemma 2.1, we first prove the following result.
\begin{lemma}\label{lemm: regularity of nabla f inv }
  Suppose $f\in \mathcal C^2(\mathbb{R}^d)$, and $\alpha I \preceq \nabla^2 f \preceq L I$ with $L \geq \alpha>0$. Then $\nabla f:\mathbb{R}^d\rightarrow \mathbb{R}^d$ is invertible, if we denote $(\nabla f)^{-1}$ as the inverse function of $\nabla f$, we have $(\nabla f)^{-1} \in \mathcal C^1(\mathbb{R}^d;\mathbb{R}^d)$, and $\nabla((\nabla f)^{-1}) = (\nabla^2f\circ\nabla f^{-1})^{-1}$. 
\end{lemma}
\begin{proof}
  We first prove that $\nabla f$ is invertible. For arbitrary $p\in\mathbb{R}^d$, consider $g(x) = -p\cdot x + f(x)$, then $g$ is $\alpha-$strongly convex. There exists unique $x'\in\mathbb{R}^d$ s.t. $\nabla g(x')=0$, i.e., $\nabla f(x') = p$; furthermore, for any $x''$ such that $\nabla f(x'')=p$ we have $\nabla g(x'')=0$, the uniqueness yields $x''=x'$. This proves that $\nabla f$ is a bijective map on $\mathbb{R}^d$. We denote $(\nabla f)^{-1}$ as the inverse map of $\nabla f$. To show the continuity of $(\nabla f)^{-1}$, for any $\epsilon > 0$, choose $\delta < \alpha \epsilon$. For fixed $p\in\mathbb{R}^d$, consider any $q$ with $\|q-p\|<\delta,$ denote $x=\nabla f^{-1}(p)$, $y = \nabla f^{-1}(q)$, from $\alpha-$strongly convexity, we have $\|\nabla f(y) - \nabla f(x)\|\geq \alpha \|y-x\|$, this yields $\|\nabla f^{-1}(q) - \nabla f^{-1}(p)\|\leq \frac{\|q-p\|}{\alpha}<\epsilon$. This verifies the continuity of $\nabla f^{-1}$.
  
  We then show $(\nabla f)^{-1}$ is differentiable. Since $f\in\mathcal C^2$, $\nabla f\in\mathcal C^1$. So $\nabla f$ is differentiable, which indicates that for any $x,y\in\mathbb{R}^d$, 
  \[ \nabla f(y) - \nabla f(x) = \nabla^2 f(x)(y-x) + r(x,y), \]
  where $r:\mathbb{R}^d\times \mathbb{R}^d\rightarrow \mathbb{R}^d$ is certain vector function satisfying $\lim_{y\rightarrow x} \frac{\|r(x,y)\|}{\|y-x\|} = 0.$ Denote $p = \nabla f(x), q = \nabla f(y)$, the above equation yields,
  \[ q - p = \nabla^2 f(x)(\nabla f^{-1}(q) - \nabla f^{-1}(p)) + r(x,y). \]
  This is 
  \begin{equation} 
    \nabla f^{-1}(q) - \nabla f^{-1}(p) = (\nabla^2 f(x))^{-1} (q - p) - (\nabla^2 f(x))^{-1} r(x,y).  \label{diff of nabla f inv }
  \end{equation}
  Denote $\hat{r}(q, p) = -(\nabla^2 f(x))^{-1}r(x,y)$, we have
  \[ \|\hat{r}(q, p)\| \leq \|\nabla^2 f(x))^{-1}\|\cdot \frac{\|r(x,y)\|}{\|y-x\|}\cdot\frac{\|y-x\|}{\|q - p\|}\cdot \|q-p\|. \]
  Since $\|\nabla^2 f(x)^{-1}\|\leq \frac{1}{\alpha}$, and $\frac{\|y-x\|}{\|q-p\|} = \frac{\|y-x\|}{\|\nabla f(y) - \nabla f(x)\|}\leq \frac{1}{L}.$ This yields
  \[ \|\hat{r}(q, p)\| \leq \frac{1}{\alpha L}\frac{\|r(x,y)\|}{\|y-x\|}\cdot \|q-p\|. \]
  Now send $q\rightarrow p$, due to the continuity of $\nabla f^{-1}$, we know $y\rightarrow x$. The above inequality yields $r(q,p)=o(\|q-p\|),$ which verifies the differentiability of $\nabla f^{-1}.$ Furthermore, by \eqref{diff of nabla f inv },  we know the Jacobian of $\nabla f^{-1}$ is $\nabla(\nabla f^{-1})(p) = (\nabla^2 f(\nabla f^{-1}(p)))^{-1}$, which is continuous. This verifies $\nabla f^{-1}\in\mathcal C^1$.
\end{proof}

\begin{proof}[Proof of Lemma 2.1]
  By Lemma \ref{lemm: regularity of nabla f inv }, we know $\nabla f$ is bijective, and we denote $\nabla f^{-1}\in\mathcal{C}^1$ as its inverse.
  According to the definition of Legendre transformation, 
  \[ f^*(p) = \sup_{\xi\in\mathbb{R}^d} \{\xi \cdot p - f(\xi)\}, \]
  since $\xi\cdot p - f(\xi)$ is $\alpha-$strongly concave as a function  of $\xi,$ for any $p\in \mathbb R^d,$ there is a unique maximizer $\xi_*$, which solves $\nabla f(\xi_*) = p,$ i.e., $\xi_* = (\nabla f)^{-1}(p).$ Thus $f^*(p) = (\nabla f)^{-1}(p)\cdot p - f((\nabla f)^{-1}(p))$, since $\nabla f^{-1}\in\mathcal C^1$, $f^*$ is at least $\mathcal C^1$, use $\nabla(\nabla f^{-1}(p)) = (\nabla^2 f(\nabla f^{-1}(p)))^{-1},$ we have
  \[ \nabla f^*(p) = \nabla((\nabla f)^{-1}(p)) p + \nabla f^{-1}(p) - f(\nabla f^{-1}(p)) = \nabla f^{-1}(p). \]
  Since $\nabla f^{-1}\in\mathcal C^1$, we know $\nabla f^*\in\mathcal C^1$, this leads to $f^*\in\mathcal C^2$.

  Furthermore, we have $\nabla^2 f^*(p) = \nabla(\nabla f^{-1}(p)) = [\nabla^2 f(\nabla f^{-1}(p))]^{-1},$ this yields $\frac{1}{L}I \preceq\nabla^2 f^*\preceq \frac{1}{\alpha}I.$
  
  On the other hand, recall that $\xi_* = \nabla f^{-1}(p) = \nabla f^*(p),$ we have $f^*(p) = \nabla f^*(p)\cdot p - f(\nabla f^*(p)).$ Thus,
  \begin{align*}
    f(q)+f^*(p)-q\cdot p  &= f(q) + \nabla f^*(p)\cdot p - f(\nabla f^*(p)) - q\cdot p\\
    & = f(q) - f(\nabla f^*(p)) - p\cdot(q - \nabla f^*(p))   \\
    & = f(q) - f(\nabla f^*(p)) - \nabla f(\nabla f^*(p))\cdot(q - \nabla f^*(p)) \\
    & = D_f(q : \nabla f^*(p)). 
  \end{align*}
  For the third equality, we use the fact that $\nabla f^*(p) = (\nabla f)^{-1}(p)$ for any $p\in\mathbb{R}^d.$

  To prove the fact that $f(q)+f^*(p)-q\cdot p = D_{f^*}(p : \nabla f(q)),$ one only needs to treat $g=f^*\in\mathcal{C}^2(\mathbb{R}^d)$ with $\frac{1}{L} I \preceq \nabla^2 g \preceq \frac{1}{\alpha} I $, and $g^* = f^{**} = f$,\footnote{This is true for any $f\in\mathcal{C}(\mathbb{R}^d)$ that is convex, c.f. Chapter 11 of \cite{rockafellar2009variational}. } 
  and then apply the above argument to $g$.
\end{proof}

{
\section{Proof of Theorem 2.1}
\begin{proof}[Proof of Theorem 2.1]
Given the Lipschitz condition on the vector field $(\frac {\partial }{\partial x}H^\top, $ $ \frac {\partial}{\partial p}H^\top)^\top$, it is known that the underlying Hamiltonian system considered admits a unique solution with continuous trajectories  for arbitrary initial condition $(\bX_0, $ $\nabla g(\bX_0))$. 

Let us recall the probability space $(\Omega, \mathcal{F}, \mathbb P)$ used to describe the randomness of the Hamiltonian system. Since 
\[ \mathbb{E}_\omega \left[\int_0^T D_{H,x}(\nabla\widehat{\psi}(\bX_t(\omega), t):\bP_t(\omega))~dt\right]    = 0,  \] 
then by the fact that Bregman divergence $D_{H,x}$ is always non-negative, we obtain
\begin{equation*}
  \int_0^T D_{H,x}(\nabla\widehat{\psi}(\bX_t(\omega), t):\bP_t(\omega))~dt = 0, \quad  \mathbb P\textrm{-almost surely.}
\end{equation*}
Thus, there exists a measurable subset $\Omega'\subset\Omega$ with $\mathbb P(\Omega')=1$ such that 
\[ \int_0^T D_{H,x}(\nabla\widehat{\psi}(\bX_t(\omega'), t):\bP_t(\omega'))~dt = 0, \quad \forall ~\omega'\in\Omega'.  \]
By using the continuity and non-negativity of $D_{H,x}(\nabla\widehat{\psi}(\bX_t(\omega'), t):\bP_t(\omega'))$ with respect to $t$, we have
 \begin{equation}
   \nabla\widehat{\psi}(\bX_t(\omega'), t) = \bP_t(\omega') \quad \textrm{for} ~ ~ 0\leq t\leq T.  \label{nabla psi = p }
 \end{equation}
When $t=0$, we have $\nabla\widehat{\psi}(\bX_0(\omega'), 0) = \bP_0(\omega')$. Recall the initial condition of the Hamiltonian System, we have $\bP_0(\omega') = \nabla g(\bX_0(\omega'))$. This yields $\nabla\widehat{\psi}(\bX_0(\omega'), 0) = \nabla g(\bX_0(\omega'))$ for any $\omega'\in\Omega'$, which yields 
 \begin{equation}
   \nabla\widehat{\psi}(x,0) = \nabla g(x)\quad \textrm{for all } x\in\textrm{Spt}(\rho_0).  \label{consistent pf equ 1} 
 \end{equation}  
 On the other hand, for $t\in(0, T],$ by differentiating on both sides of \eqref{nabla psi = p } w.r.t. $t$, we obtain
\begin{equation}
  \frac{\partial}{\partial t} \nabla\widehat{\psi}(\bX_t(\omega'), t) + \nabla^2\widehat{\psi}(\bX_t(\omega'), t)\dot\bX_t(\omega') = \dot\bP_t(\omega').  \label{proof thm 1}
\end{equation}
Recall that we have 
\begin{align*}
\dot\bX_t = \frac {\partial}{\partial p}H(\bX_t, \bP_t)=\frac {\partial}{\partial p} H(\bX_t, \nabla\widehat{\psi}(\bX_t, t)),\\
\dot\bP_t = -\frac {\partial}{\partial x}H( \bX_t,  \bP_t ) = -\frac {\partial}{\partial x}H(\bX_t, \nabla\widehat{\psi}(\bX_t, t)).
\end{align*} Plugging these into \eqref{proof thm 1} yields
\begin{align*}
  \frac{\partial}{\partial t} \nabla\widehat{\psi}(\bX_t(\omega'), t) + \nabla^2\widehat{\psi}(\bX_t(\omega'), t)\frac {\partial}{\partial p} H(\bX_t(\omega'), \nabla\widehat{\psi}(\bX_t(\omega'), t)) \\
  = -\frac {\partial}{\partial x}H(\bX_t(\omega'), \nabla\widehat{\psi}(\bX_t(\omega'), t)),
\end{align*}
which leads to
\begin{equation*}
  \nabla\left( \frac{\partial}{\partial t}\widehat{\psi}(\bX_t(\omega'),t) + H(x, \nabla\widehat{\psi}(\bX_t(\omega'),t)) \right) = 0, \quad \forall~\omega'\in\Omega'.
\end{equation*}
Since the probability density distribution of $\bX_t$ is $\rho_t$, we have proved that 
\begin{equation}
\nabla\left( \frac{\partial}{\partial t}\widehat{\psi}(x,t) + H(x, \nabla\widehat{\psi}(x,t)) \right) = 0, \quad \forall~~x\in\textrm{Spt}(\rho_t).  \label{consistent pf equ 2}
\end{equation}
Combining \eqref{consistent pf equ 1} and \eqref{consistent pf equ 2} proves  this theorem.

On the other hand, if $\mathscr{L}_{\rho_0, g, T}^{|\cdot|^2}(\widehat\psi)=0$. By using the fact that $|\nabla \widehat{\psi}(\boldsymbol{X}_t(\omega), t) - \boldsymbol{P}_t(\omega)|^2$ is continuous and non-negative for a.s. $\omega\in\Omega,$ we can repeat the previous proof to show the same assertion still holds.
\end{proof}
}

\section{Proof of Lemma 2.2}\label{appendix: proof of lemma 2.2 }
\begin{proof}[Proof of Lemma 2.2]

Let us first consider the term
\begin{equation}
 \int_{\mathbb{R}^d} \psi(x,t)\rho_t(x)dx = \mathbb{E}_{\bX_t} \psi(\bX_t, t).  \label{proof lemma 3 a}
\end{equation}
By differentiating \eqref{proof lemma 3 a} w.r.t. time $t$, we obtain
\begin{equation*}
\frac{d}{dt}\left( \int_{\mathbb{R}^d} \psi(x,t)\rho_t(x)    dx\right) = \mathbb{E} \left[ \nabla\psi(\bX_t, t)\cdot\dot\bX_t + \frac{\partial\psi(\bX_t, t)}{\partial t} \right]. 
\end{equation*}
The right-hand side of the above equation equals
\begin{align*}
  & \mathbb{E}_{\bX_t, \bP_t} \nabla\psi(\bX_t, t)\cdot {\frac {\partial}{\partial p}} H(\bX_t, \bP_t) + \mathbb{E}_{\bX_t} \left[\frac{\partial \psi(\bX_t, t)}{\partial t}\right] \\
  = & \int_{\mathbb{R}^{2d}} \nabla\psi(x, t)\cdot {\frac {\partial}{\partial p}} H(x, p)~d\mu_t(x,p) + \int_{\mathbb{R}^d} \frac{\partial\psi(x, t)}{\partial t}\rho_t(x)dx.
\end{align*}
Combining the above equations, we have
{\small
\begin{equation}
  \int_{\mathbb{R}^d} - \partial_t\psi(x, t) d \rho_t(x) = \int_{\mathbb{R}^{2d}} \nabla\psi(x, t)\cdot {\frac {\partial}{\partial p}} H(x, p)~d\mu_t(x,p) - \frac{d}{dt}\left( \int_{\mathbb{R}^{2d}} \psi(x,t)\rho_t(x)dx   \right). \label{proof lemma 3 b}
\end{equation}}Plugging \eqref{proof lemma 3 b} into the formula of $\mathscr{L}_{\rho_0,g,T}(\psi )$ yields that
{\small
\begin{align}\nonumber
&\mathscr{L}_{\rho_0,g,T}(\psi )\\\nonumber
 = & \int_0^T \left( \int_{\mathbb{R}^{2d}} \nabla\psi(x, t)\cdot {\frac {\partial}{\partial p}} H(x, p)~d\mu_t(x,p) - \frac{d}{dt}\left( \int_{\mathbb{R}^{2d}} \psi(x,t)\rho_t(x)  dx\right) \right)~dt \nonumber\\
  & + \int_0^T\int_{\mathbb{R}^d} -H(x, \nabla\psi(x,t))\rho_t(x)~dx~dt \nonumber\\
  & + \int_{\mathbb{R}^d} \psi(x, T)\rho_t(x)~dx - \int_{\mathbb{R}^d} \psi(x, 0)\rho_0(x)~dx. \nonumber\\
  = & \int_0^T \int_{\mathbb{R}^{2d}}  (\nabla\psi(x, t)\cdot {\frac {\partial}{\partial p}} H(x, p) - H(x, \nabla\psi(x,t)))~d\mu_t(x,p)dt \nonumber\\
  = & \int_0^T \int_{\mathbb{R}^{2d}} (\nabla\psi(x, t)\cdot {\frac {\partial}{\partial p}} H(x, p) - H(x, \nabla\psi(x,t)) - H^*(x, {\frac {\partial}{\partial p}} H(x, p)))~d\mu_t(x,p)dt\nonumber\\
  & + \int_0^T\int_{\mathbb{R}^{2d}} H^*(x,{\frac {\partial}{\partial p}} H(x,p))~d\mu_t(x,p)dt.
  \label{proof lemma 3 c}
\end{align}}The second equality is obtained by integrating the time-derivative of \eqref{proof lemma 3 a} on $[0, T]$ as well as by using the fact that $\rho_t(\cdot)$ is the density of $\bX-$marginal of $\mu_t$. 

Based on  Lemma 2.1, choosing $f$ as $H^*$ and  $f^*$ as the Hamiltonian $H$, and letting $q = {\frac {\partial}{\partial p}}H(x,p)$ and $p = \nabla\psi(x,t)$, we obtain 
{
\begin{align*} &H^*(x, {\frac {\partial}{\partial p}} H(x, p)) + H(x, \nabla\psi(x,t)) - \nabla\psi(x, t)\cdot {\frac {\partial}{\partial p}} H(x, p)
\\
&={D_{H,x}}(\nabla\psi(x,t):\nabla_v H^*(x, {\frac {\partial}{\partial p}} H(x,p))).
 \end{align*}}
Since ${\frac {\partial}{\partial v}} H^*(x, \cdot) = ({\frac {\partial}{\partial p}}H(x, \cdot))^{-1}$, the right-hand side of the above equality  leads to $D_{H, x}(\nabla\psi(x,t):p)$. Plugging this back to \eqref{proof lemma 3 c} proves Lemma 2.2.
\end{proof}

{\section{Discussion on the dependence on $\rho_0$}\label{SM-density}

We give a brief discussion about the dependence on $\rho_0$ for the computed solution $\psi_{\theta}$ via an example with the classical Hamiltonian $H(x,p)=\frac 12|p|^2+V(x).$
Assume that the solution $\psi_{\theta}$ exists for any initial density and is regular enough in time and space. For simplicity, we ignore the numerical error caused by symplectic integrator and consider the difference between $\psi_{\theta,\rho_{0,1}}$ and $\psi_{\theta,\rho_{0,2}}$ with different initial densities $\rho_{0,1}$ and $\rho_{0,2}.$. 

Then by remark 2.1, one can verify for any test function $f\in \mathcal C^1([0,T]\times \mathbb R^d)$ 
\begin{align}\label{equ-weak}
\int_{0}^T\int_{\mathbb R^d} \Big(\nabla \psi_{\theta,\rho_{0,i}}(x,t)- \bar {p}_i(x,t)\Big)\nabla f(x,t) \rho_{i}(x,t)dx dt=0.
\end{align}
Here $i\le 2$, $\rho_{i}(x,t)$ is the marginal density of the position of the particle $X_t^i$ with different initial data $x_0^{(i)}$, and $\bar {p}_i(x,t)=\int_{\mathbb R^d} p d \mu^{(i)}_{t}(p|x)$ with $\mu_{t}^{(i)}(p|x)$ being the conditional distribution of the joint distribution $\mu_{t}^{(i)}(x,p).$

In particular, when the characteristic lines do not intersect, 
by \eqref{equ-weak} one can infer that   
$\nabla \psi_{\theta,\rho_{0,1}}(x,t)=\nabla \psi_{\theta,\rho_{0,2}}(x,t)$
in the intersection of the supports of $\rho_{0,1}(t,\cdot)$ and $\rho_{0,2}(t,\cdot)$. Moreover, in this case $\bar {p_i}(x,t)=P_t^{i}$ with the initial value $(X_t^{(i)})^{-1}(x),\nabla g\big((X_t^{(i)})^{-1}(x)\big)$ as the initial value of the underlying Hamiltonian ODE. Since characteristic lines do not intersect, it is not hard to see that 
\begin{align}\label{prop-pi}
&|\bar {p_i}(x,t)|\le C\sup_{x\in supp(\rho_{0,i})}|\nabla g(x)|,\; i\le 2,\\\nonumber 
&\bar {p_1}(x,t)=\bar {p_2}(x,t), \;\text{for any fixed} \; (x,t).
\end{align}

By subtracting \eqref{equ-weak} for $i=1,2,$ one further has that 
\begin{align}\label{equ-weak1}
    &\int_{0}^T\int_{\mathbb R^d} \Big(\nabla \psi_{\theta,\rho_{0,1}}(x,t)- \nabla \psi_{\theta,\rho_{0,2}}(x,t)\Big)\nabla f(x,t) \rho_{1}(x,t)dx dt\\\nonumber
    &+\int_{0}^T\int_{\mathbb R^d} \nabla \psi_{\theta,\rho_{0,2}}(x,t) \nabla f(x,t) (\rho_{2}(x,t)-\rho_1(x,t))dx dt\\\nonumber
    &-\int_{0}^T\int_{\mathbb R^d} \Big(\bar{p_1}(x,t)- \bar{p_2}(x,t)\Big)\nabla f(x,t) \rho_{1}(x,t)dx dt\\\nonumber
    &-\int_{0}^T\int_{\mathbb R^d} \bar{p_2}(x,t) \nabla f(x,t) (\rho_1(x,t)-\rho_2(x,t)) dx dt
    =0.
\end{align}

Taking $f=\psi_{\theta,\rho_{0,1}}(x,t)-  \psi_{\theta,\rho_{0,2}}(x,t)$ and using Young's inequality, by the symmetry of $\rho_{0,i}$ and \eqref{prop-pi}, one can obtain 
\begin{align*}
    &\sup_{i\le 2}\int_{0}^T\int_{\mathbb R^d}|\nabla \psi_{\theta,\rho_{0,1}}(x,t)- \nabla \psi_{\theta,\rho_{0,2}}(x,t)|^2\rho_{i}(x,t)dx dt
    \\&
    \le C \int_0^T \int_{\mathbb R^d}\Big(1+|\nabla \psi_{\theta,\rho_{0,1}}(x,t)|^2+|\nabla \psi_{\theta,\rho_{0,1}}(x,t)|^2\Big)|\rho_1(x,t)-\rho_2(x,t)|dxdt\\
    &+C  \int_0^T \int_{\mathbb R^d}\Big(1+|\bar{p_1}(x,t)|^2+ \bar{p_2}(x,t)|^2\Big)|\rho_1(x,t)-\rho_2(x,t)|dxdt.
\end{align*}
This, together with the fact that $\rho_i(t,\cdot)$ is continuous w.r.t. the initial density, implies that the approximate solution $\psi_{\theta}$ is continuous w.r.t. the initial  density.

After the characteristic  lines intersect, the analysis is more complicate and relies on the properties of conditional distribution  $\mu_t^{(i)}(p|x)$ and the averaged momentum $\bar{p_i}(t,x).$ It is beyond the scope of this current work. We hope to address and study this issue in the future.
}

\section{More details in the proof of Theorem 3.1}\label{appendix: proof of thm 3.1}

\subsection{Proof of \eqref{bound-control-particle}}
We estimate the distance between $\widehat{x}_\tau$ and $\widehat{x}_0=\tilde{x}_{t_i}$ by considering
{\small
\begin{align}
  |\widehat{x}_\tau-\widehat{x}_0|   \leq \int_0^\tau |\frac{\partial}{\partial p}H(\widehat{x}_s, \widehat{p}_s)|~ds & \leq \int_0^\tau |\frac{\partial}{\partial p}H(\widehat{x}_0, \widehat{p}_0)|+|\frac{\partial}{\partial p}H(\widehat{x}_0, \widehat{p}_0)-\frac{\partial}{\partial p}H(\widehat{x}_s, \widehat{p}_s)|~ds  \nonumber \\
  & \leq \tau |\frac{\partial}{\partial p}H(\widehat{x}_0, \widehat{p}_0)|+L_1\int_0^\tau |\widehat{x}_s - \widehat{x}_0|+|\widehat{p}_s - \widehat{p}_0|~ds,\label{estimate x}
\end{align}
}where the second inequality is due to the Lipschitz property of $\frac{\partial H}{\partial p}$. Similarly, for $\widehat{p}_\tau$ and $\widehat{p}_0=\tilde{p}_{t_i}$, we have
\begin{equation}
  |\widehat{p}_\tau - \widehat{p}_0|\leq \int_0^\tau |-\frac{\partial}{\partial x}H(\widehat{x}_s, \widehat{p}_s)|~ds\leq \tau |\frac{\partial}{\partial x}H(\widehat{x}_0, \widehat{p}_0)|+L_2\int_0^\tau |\widehat{x}_s - \widehat{x}_0|+|\widehat{p}_s - \widehat{p}_0|~ds    \label{estimate p}
\end{equation}
By adding \eqref{estimate x} and \eqref{estimate p} and applying the Gr\"{o}nwall's inequality, we obtain
\begin{equation}\label{estimate x p}
\begin{split}
  |\widehat{x}_\tau - \tilde{x}_{t_i}|+|\widehat{p}_\tau - \tilde{p}_{t_i}|&\leq (|\frac{\partial}{\partial p}H(\tilde{x}_{t_i}, \tilde{p}_{t_i})|+|\frac{\partial}{\partial x}H(\tilde{x}_{t_i}, \tilde{p}_{t_i})|)\\
  &\times \left(\tau + \frac{e^{(L_1+L_2)\tau}-(L_1+L_2)\tau - 1}{L_1+L_2}  \right),
  \end{split}
\end{equation}

From the Lipschitz property and the inequality $e^x\leq 1+x+\frac{1}{2}e^xx^2$ for $x\geq 0$, the right hand side of \eqref{estimate x p}
 can be further bounded by
\begin{equation*}
   \Big(  (L_1+L_2)(|\tilde{x}_{t_i}|+|\tilde{p}_{t_i}|)+(|\partial_pH(0,0)|+|\partial_xH(0,0)|) \Big) \left(\tau + \frac{1}{2}e^{(L_1+L_2)\tau}(L_1+L_2)\tau^2\right).
\end{equation*}
Let us denote $R_{t_i}=\underset{1\leq k \leq N}{\max}\{|\tilde{x}_{t_i}^{(k)}|+|\tilde{p}_{t_i}^{(k)}|\}$, $L=L_1+L_2$ and $C = |\partial_pH(0,0)|+|\partial_xH(0,0)|$. Since we assume that $$M \geq  \max\{T, \frac{T}{2}(L_1+L_2)e^{L_1+L_2}\},$$
the time stepsize 
$$h\leq\frac{T}{M}\leq \min\{1, \frac{2}{L_1+L_2}e^{-(L_1+L_2)}\}.$$ Then for $0\leq \tau\leq h$, we have $\frac{1}{2}e^{(L_1+L_2)\tau}(L_1+L_2)\tau^2\leq \frac{1}{2}e^{Lh}Lh\cdot \tau\leq \tau$. Thus, \eqref{estimate x p} can be bounded by
\begin{equation*}
  |\widehat{x}_\tau - \tilde{x}_{t_i}|+|\widehat{p}_\tau - \tilde{p}_{t_i}|\leq 2(LR_{t_i}+C+1)\tau.
\end{equation*}

\subsection{Proof of \eqref{lip dist combine}}
{
 Direct calculation yields that
{\footnotesize  
\begin{align}\label{lip dist 2}
    & \left|\nabla^2\psi_\theta(\widehat{x}_\tau, t_i+\tau)\frac{\partial}{\partial p}H(\widehat{x}_\tau, \widehat{p}_\tau) - \nabla^2\psi_\theta(\tilde{x}_{t_i}, t_i)\frac{\partial}{\partial p}H(\tilde{x}_{t_i}, \tilde{p}_{t_i})\right| \\
    = & \left| (\nabla^2\psi_\theta(\widehat{x}_\tau, t_i+\tau) - \nabla^2\psi_\theta(\tilde{x}_{t_i}, t_i))\frac{\partial}{\partial p}H(\widehat{x}_\tau, \widehat{p}_\tau) + \nabla^2\psi_\theta(\tilde{x}_{t_i}, t_i)(\frac{\partial}{\partial p}H(\widehat{x}_{\tau}, \widehat{p}_\tau) - \frac{\partial}{\partial p}H(\tilde{x}_{t_i}, \tilde{p}_{t_i})) \right|\nonumber\\
    \leq & L_{\theta, i}^B(|\widehat{x}_\tau - \tilde{x}_{t_i}|+\tau)\left| \frac{\partial}{\partial p}H(\widehat{x}_\tau, \widehat{p}_\tau)\right|+\|\nabla^2\psi_\theta(\tilde{x}_{t_i}, t_i)\|L_1(|\widehat{x}_\tau - \tilde{x}_{t_i}|+|\widehat{p}_\tau - \tilde{p}_{t_i}|)\nonumber\\\nonumber
    \leq & L_{\theta, i}^B(2(LR_{t_i}+C+1)\tau+\tau)(|\partial_pH(0,0)|+L_1(R_{t_i}+2(LR_{t_i}+C+1)\tau)) \\\nonumber
    &+ M_{\theta, i}2L_1(LR_{t_i}+C+1)\tau,\nonumber
\end{align}
}
and that 
\begin{equation}
  \left|\frac{\partial}{\partial x}H(\widehat{x}_\tau, \widehat{p}_\tau) -\frac{\partial}{\partial x}H(\tilde{x}_{t_i}, \tilde{p}_{t_i}) \right|\leq L_2(|\widehat{x}_\tau - \tilde{x}_{t_i}|+|\widehat{p}_\tau - \tilde{p}_{t_i}|)\leq 2L_2(LR_{t_i}+C+1)\tau\label{lip dist 3}
\end{equation}
Combining \eqref{lip dist 1},\eqref{lip dist 2} and \eqref{lip dist 3} together, we obtain \eqref{lip dist combine} with 
{\small \begin{align}
  \lambda(\theta, i) = & 3L_{\theta, i}^A(LR_{t_i}+C+1) + L_{\theta, i}^B 3(LR_{t_i}+C+1)(|\partial_pH(0,0)|  \nonumber\\
  & +L_1(R_{t_i}+2(LR_{t_i}+C+1)h))+ 2 L_1M_{\theta, i}(LR_{t_i}+C+1) + 2L_2(LR_{t_i}+C+1).  \label{notation lambda}
\end{align}}

\subsection{Proof of \eqref{over-est-sample}}
Using \eqref{lip dist combine}, the first term on the right hand side of \eqref{main estimate} is upper bounded by $\frac{1}{2}\lambda(\theta, i )h.$ 

Recall the notation used in \eqref{num schm solver Phi order r}. Since we assume that the numerical scheme for integrating the Hamiltonian system has local truncation error of order $r$, the second term can be bounded by
\begin{equation}
  \frac{1}{h}|\nabla\psi_\theta(\widehat{x}_h, t_{i+1}) - \nabla\psi_\theta(\tilde{x}_{t_{i+1}}, t_{i+1})| \leq L_{\theta, i+1}^C\frac{|\widehat{x}_h-\tilde{x}_{t_{i+1}}|}{h}\leq L_{\theta, i+1}^C C_{\tilde{\Phi}_h}(\tilde{x}_{t_i}, \tilde{p}_{t_i})h^{r-1}.
\end{equation}
Similarly, the last two terms in \eqref{main estimate} can be bounded by $C_{\tilde{\Phi}_h}(\tilde{x}_{t_i}, \tilde{p}_{t_i})h^{r-1}$. 

The left hand side of \eqref{main estimate} can be recast as
\begin{equation*}
  |\mathscr{D}\psi_\theta(\tilde{x}_{t_i}, \nabla \psi_\theta(\tilde{x}_{t_i}), t_i) + (\mathscr{D}\psi_\theta(\tilde{x}_{t_i}, \tilde{p}_{t_i}, t_i) - \mathscr{D}\psi_\theta(\tilde{x}_{t_i}, \nabla \psi_\theta(\tilde{x}_{t_i}), t_i))|. 
\end{equation*}
Since $\nabla\psi_\theta(\tilde{x}_{t_i}, t_i) = \tilde{p}_{t_i}+e_{i}$, we have
\begin{align*}
  &|\mathscr{D}\psi_\theta(\tilde{x}_{t_i}, \tilde{p}_{t_i}, t_i) - \mathscr{D}\psi_\theta(\tilde{x}_{t_i}, \nabla \psi_\theta(\tilde{x}_{t_i}), t_i)|\\
   & \leq \|\nabla^2\psi_\theta(\tilde{x}_{t_i}, t_i)\|  L_1|\tilde{p}_{t_i}-\nabla\psi_\theta(\tilde{x}_{t_i})|+L_2|\tilde{p}_{t_i}-\nabla\psi_\theta(\tilde{x}_{t_i})|\\
  & \leq (M_{\theta, i}L_1+L_2)e_i.
\end{align*}
Let us recall
\begin{equation*}
  \mathscr{D}\psi_\theta(\tilde{x}_{t_i}, \nabla \psi_\theta(\tilde{x}_{t_i}), t_i)=\nabla\left(\frac{\partial}{\partial t}\psi_\theta(\tilde{x}_{t_i}, t_i ) + H(\tilde{x}_{t_i}, \nabla\psi_\theta(\tilde{x}_{t_i}))\right),
\end{equation*}
thus, \eqref{main estimate} leads to \eqref{over-est-sample}.}

\begin{rk}\label{rk-thm 3.1}
In Theorem \ref{thm est}, it can be seen that the constants $\lambda(\theta,i),\eta(\theta,i)$, $\nu(\theta,i),R(\theta,i)$ depend on the neural network approximation $\psi_\theta$ and $N$. In the proof of Theorem \ref{thm est}, the set $E_i$ also depends on $N$. Here we would like to explain their dependence on $\psi_\theta$ and $N$.
 
Since $(\tilde{x}_{t_i}^{(k)},\tilde{p}_{t_i}^{(k)})_{i\le M}$  is  the numerical solution of some symplectic integrator of \eqref{alg: Hamilton system}, the standard Gronwall's argument leads that there exists a constant $C(x_0^{(k)}, $ $p_0^{(k)},T)<+\infty$ such that
$$\underset{1\leq i \leq M}{\max}\{|\tilde{x}_{t_i}^{(k)}|+|\tilde{p}_{t_i}^{(k)}|\}\le C(x_0^{(k)},p_0^{(k)},T).$$
Thus, $\sup_{i\le M}R_{t_i}$ is uniformly bounded w.r.t. $(x_0^{(k)},p_0^{(k)}), k\le N$.   By \eqref{def-ei} and \eqref{def-di}, the sets $E_i$ and $  D_i $ are also uniformly bounded w.r.t. $k\le N$.

With regard to $\lambda(\theta,i),$ by \eqref{notation lambda}, \eqref{lip partial_t nabla psi}, \eqref{notation M_theta, i} and \eqref{notation lip nabla2 psi}, it holds that $$\lambda(\theta,i)\le C_1\Big(\textrm{Lip}_{E_i}(\partial_t\nabla\psi_\theta) ,\sup_{ x \in \textrm{supp}(\tilde{\rho}_{t_i})}\|\nabla^2\psi_\theta(x, t_i)\| ,\textrm{Lip}_{E_i}(\nabla^2\psi_\theta)\Big)<+\infty$$
for some constant $C_1\Big(\textrm{Lip}_{E_i}(\partial_t\nabla\psi_\theta) ,\sup\limits_{ x \in \textrm{supp}(\tilde{\rho}_{t_i})}\|\nabla^2\psi_\theta(x, t_i)\| ,\textrm{Lip}_{E_i}(\nabla^2\psi_\theta)\Big)>0.$ 

From \eqref{main ineq}, 
one has that for some $C_2>0,$
$$\eta(\theta,i)\le C_2\Big(\textrm{Lip}_{D_i}(\nabla\psi_\theta(\cdot, t_i))\Big)<\infty,$$
and that for some $C_3>0$,
$$\nu(\theta,i)\le C_3\Big(\sup_{ x \in \textrm{supp}(\tilde{\rho}_{t_i})}\|\nabla^2\psi_\theta(x, t_i)\|\Big)<\infty.$$
Since we assume that $\textrm{supp}(\rho_0)$ is a bounded set, and the solution map  of the numerical scheme is continuous, then one can verify that for some $C_4>0,$
$$R(\theta, i)\le C_4\Big(\sup_{ x \in \textrm{supp}(\tilde{\rho}_{t_i})}   \|\partial_t\nabla\psi_\theta\|,\sup_{ x \in \textrm{supp}(\tilde{\rho}_{t_i})}\|\nabla^2\psi_\theta(x, t_i)\|\Big)<\infty.$$ 
\end{rk}

\section{A stronger version of Theorem 3.1}
\begin{theorem}
Under the condition of Theorem 3.1, in addition assume that the classical solution of HJ PDE exists. Then with the probability $1-\epsilon$, the neural network $\psi_{\theta}$ satisfies 
{\small
\begin{align*}
 \int_{\mathbb{R}^d}  &  \left|\nabla\left(\frac{\partial}{\partial t}\psi_\theta(x, t_i )  +  H(x, \nabla\psi_\theta(x,t_i))\right)\right| ~\tilde{\rho}_{t_i}(x) dx \nonumber\\
     &\le  ~ C_{\theta,i}h^{r-2}+\frac 1{N}\sum_{k=1}^N|\sum_{j\in N(i)}a_{ij}e_j^k\frac 1h| +\nu(\theta,i)(|\nabla e_i^{(k)}|+|e_i^{(k)}|)+R(\theta, i)\sqrt{\frac{\ln M + \ln\frac{2}{\epsilon}}{2N}},
\end{align*}}at $t_i=ih$, $i=1,\dots, M$. Here, $a_{ij}$ is the coefficient and $j\in N(i)$ denotes the node to be used in the numerical differentiation formula $I^h(f)(t_i)=\sum_{j\in N(i)}a_{ij}f(t_i)\frac 1 h$ of order $r_1\ge r-2$. The constants $C(\theta, i),\nu(\theta, i), R(\theta, i)$ are non-negative depending on the parameter $\theta$, time node $t_i$, Hamiltonian $H$, initial distribution $\rho_0$, the exact solution of HJ PDE and the numerical solution of  temporal numerical scheme.
\end{theorem}

\begin{proof}
 We use the same notations as in the proof of Theorem 3.1. Let us denote the residual term of optimal neural network  as
\begin{equation*}
  \mathcal{R}[\psi_\theta](x,t)=\nabla\left(\frac{\partial}{\partial t}\psi_\theta(x, t) + H(x, \nabla\psi_\theta(x))\right).
\end{equation*}
and the residual term of the weak solution as 
$$\mathcal{R}_{exa}[\psi](x,t):=\nabla\left(\frac{\partial}{\partial t}\psi(x, t) + H(x, \nabla\psi(x,t))\right).$$
Note that if $\psi$ is the strong solution of HJ equation, then $\mathcal{R}_{exa}[\psi](x,t)=0.$

For the sample particle $\widetilde x_{t_i}^{(k)},k\le N, i\le M$, it holds that 
{\small
\begin{align*}
    \frac {1}{N} \sum_{k=1}^N \mathcal R \psi_{\theta}(\widetilde x_{t_i}^{(k)},t_i)&=  \frac {1}{N} \sum_{k=1}^N \Big(\mathcal R \psi_{\theta}(\widetilde x_{t_i}^{(k)},t_i)- \mathcal R_{exa} \psi( x_{t_i}^{(k)},t_i)\Big)
    \\
    &= \frac {1}{N} \sum_{k=1}^N \Big(\mathcal D \psi_{\theta}(\widetilde x_{t_i}^{(k)}, \widetilde p_{t_i}^{(k)},t_i)- \mathcal D \psi(x_{t_i}^{(k)},  p_{t_i}^{(k)},t_i)\Big)\\
    &= \frac {1}{N} \sum_{k=1}^N \Big(\mathcal D \psi_{\theta}(\widetilde x_{t_i}^{(k)}, \widetilde p_{t_i}^{(k)},t_i)- \mathcal D \psi_{\theta}(x_{t_i}^{(k)},  p_{t_i}^{(k)},t_i)\Big)
    \\
    &+\frac {1}{N} \sum_{k=1}^N \Big(\mathcal D \psi_{\theta}( x_{t_i}^{(k)},   p_{t_i}^{(k)},t_i)- \mathcal D \psi(x_{t_i}^{(k)},  p_{t_i}^{(k)},t_i)\Big).
\end{align*}}Next we estimate the two terms on the right hand side. First, we split the first term as 
\begin{align*}
&\mathcal D \psi_{\theta}(\widetilde x_{t_i}^{(k)}, \widetilde p_{t_i}^{(k)},t_i)- \mathcal D \psi_{\theta}(x_{t_i}^{(k)},  p_{t_i}^{(k)},t_i)\\
&= \nabla \frac {\partial}{\partial t}\psi_{\theta}(\widetilde x_{t_i}^{(k)},t_i)-\nabla \frac {\partial}{\partial t}\psi_{\theta}( x_{t_i}^{(k)},t_i)
\\
&+\nabla^2 \psi_{\theta}(\widetilde x_{t_i}^{(k)},t_i)\frac {\partial}{\partial p} H(\widetilde x_{t_i}^{(k)},\widetilde p_{t_i}^{(k)})-\nabla^2 \psi_{\theta}(x_{t_i}^{(k)},t_i)\frac {\partial}{\partial p} H( x_{t_i}^{(k)}, p_{t_i}^{(k)})\\
&+\frac {\partial}{\partial x}H(\widetilde p_{t_i}^{(k)},\nabla \psi_{\theta}(x_{t_i}^{(k)},t_i)))-\frac {\partial}{\partial x}H(p_{t_i}^{(k)},\nabla \psi_{\theta}(x_{t_i}^{(k)},t_i))).
\end{align*}
By using the finite support property of $\rho_{t_i}$ and $\tilde \rho_{t_i}$ and Lipschitz property of $\psi_{\theta}$ on bounded domain, 
\begin{align*}
    |\nabla \frac {\partial}{\partial t}\psi_{\theta}(\widetilde x_{t_i}^{(k)},t_i)-\nabla \frac {\partial}{\partial t}\psi_{\theta}( x_{t_i}^{(k)},t_i)|\le L_{\theta,i}^A|\widetilde x_{t_i}^{(k)}-x_{t_i}^{(k)}|.
\end{align*}
Similarly, one can obtain that 
\begin{align*}
&\Big|\nabla^2 \psi_{\theta}(\widetilde x_{t_i}^{(k)},t_i)\frac {\partial}{\partial p} H(\widetilde x_{t_i}^{(k)},\widetilde p_{t_i}^{(k)})-\nabla^2 \psi_{\theta}(x_{t_i}^{(k)},t_i)\frac {\partial}{\partial p} H( x_{t_i}^{(k)}, p_{t_i}^{(k)})\Big|\\
&\le L_{\theta,i}^B(|\widetilde x_{t_i}^{(k)}-x_{t_i}^{(k)}|+|\widetilde p_{t_i}^{(k)}-p_{t_i}^{(k)}|)
\end{align*}
and that 
\begin{align}\label{error1}
\Big|\frac {\partial}{\partial x}H(\widetilde p_{t_i}^{(k)},\nabla \psi_{\theta}(x_{t_i}^{(k)},t_i)))-\frac {\partial}{\partial x}H(p_{t_i}^{(k)},\nabla \psi_{\theta}(x_{t_i}^{(k)},t_i)))\Big|\\\nonumber
\le L_{\theta,i}^C(|\widetilde x_{t_i}^{(k)}-x_{t_i}^{(k)}|+|\widetilde p_{t_i}^{(k)}-p_{t_i}^{(k)}|).
\end{align}
Here $L_{\theta,i}^{A},L_{\theta,i}^{B},L_{\theta,i}^{C}$ are finite depending on the support of $\rho_0.$
Note that the global error of the numerical scheme $|\widetilde x_{t_i}^{(k)}-x_{t_i}^{(k)}|+|\widetilde p_{t_i}^{(k)}-p_{t_i}^{(k)}|$ is of order $r-1$. Thus, $$\frac {1}{N} \sum_{k=1}^N \Big(\mathcal D \psi_{\theta}(\widetilde x_{t_i}^{(k)}, \widetilde p_{t_i}^{(k)},t_i)- \mathcal D \psi_{\theta}(x_{t_i}^{(k)},  p_{t_i}^{(k)},t_i)\Big)\sim O(h^{r-1}).$$ 

Notice that  
\begin{align*}
&\mathcal D \psi_{\theta}( x_{t_i}^{(k)}, p_{t_i}^{(k)},t_i)- \mathcal D \psi(x_{t_i}^{(k)},  p_{t_i}^{(k)},t_i)\\
&=\nabla \frac {\partial}{\partial t}\psi_{\theta}(x_{t_i}^{(k)},t_i)
-\nabla \frac {\partial}{\partial t}\psi(x_{t_i}^{(k)},t_i)\\
&
+\nabla^2 \psi_{\theta}(x_{t_i}^{(k)},t_i)\frac {\partial}{\partial p} H(x_{t_i}^{(k)},p_{t_i}^{(k)})
-\nabla^2 \psi(x_{t_i}^{(k)},t_i)\frac {\partial}{\partial p} H(x_{t_i}^{(k)},p_{t_i}^{(k)})\\
&+\frac {\partial}{\partial x}H(x_{t_i}^{(k)},\nabla \psi_{\theta}(x_{t_i}^{(k)},t_i))
-\frac {\partial}{\partial x}H(x_{t_i}^{(k)},\nabla \psi(x_{t_i}^{(k)},t)).
\end{align*}
 Using the fact that $e_{t_i}^{k}=\nabla \psi_{\theta}(\widetilde x_{t_i}^{(k)},t_i)-\widetilde p_{t_i}^{(k)}$ and the mean value theorem, 
we get 
\begin{align}
    \nabla \psi_{\theta}(x_{t_i}^{(k)},t_i)
    &=   \nabla \psi_{\theta}(\widetilde x_{t_i}^{(k)},t_i)+ \nabla \psi_{\theta}(x_{t_i}^{(k)},t_i)-  \nabla \psi_{\theta}(\widetilde x_{t_i}^{(k)},t_i)\nonumber\\
&=\nabla \psi_{\theta}(\widetilde x_{t_i}^{(k)},t_i)+ \int_0^1  \nabla^2 \psi_{\theta}((1-\alpha_1) \widetilde x_{t_i}^{(k)}+\alpha_1 x_{t_i}^{(k)},t_i)(x_{t_i}^{(k)}-\widetilde x_{t_i}^{(k)}) d\alpha_1\\\nonumber
&=\widetilde p_{t_i}^{(k)}+e_i^k+O(|\widetilde x_{t_i}^{(k)}-x_{t_i}^{(k)}|). \label{decomposition}
\end{align}

Notice that in the error estimate, directly using the fact that $\nabla \psi(\widetilde x_{t_i},t_i)=\widetilde p_{t_i}$ and forward difference method may lead to a lower order of convergence in time for the numerical discretization since less information is known for the time derivative of $\widetilde p_{t_i}.$ Instead, our strategy is using a high order numerical differentiation formula to approximate the time derivative first and then applying the fact that $\nabla \psi(\widetilde x_{t_i},t_i)=\widetilde p_{t_i}$.
To this end, we approximate $\frac {\partial}{\partial t} \nabla \psi_{\theta}$ using a high order linear numerical differential formula $I_{h}(\nabla \psi_{\theta})$, i.e., for any sufficient smooth function $f.$
\begin{align*}
I_{h}(f)(t_i)=\sum_{j\in N(i)} a_{ij} f(t_j)\frac 1{h}=f'(t_i)+O(h^{r_1}),
\end{align*}
where $a_{ij}\in \mathbb R$ and $t_j$ are the nodes close to $t_i.$ 

Using the numerical differentiation formula and the mean value theorem, as well as the fact that $p_{t}^{(k)}=\nabla \psi(x_{t}^{(k)},t)$, it follows that 
\begin{align*}
\nabla \frac {\partial}{\partial t}\psi_{\theta}(x_{t_i}^{(k)},t_i)
-\nabla \frac {\partial}{\partial t}\psi(x_{t_i}^{(k)},t_i)
&= \frac {\partial}{\partial t} \nabla \psi_{\theta}( x_{t_i}^{(k)},t_i)-\frac {\partial}{\partial t} p_t^{(k)}|_{t=t_i}\\
&=I_h(\nabla \psi_{\theta}( x_{t}^{(k)},t))|_{t=t_i}
-I_h(p_t^{(k)})|_{t=t_i}+O(h^{r_1})\\
&=I_h(\nabla \psi_{\theta}( x_{t}^{(k)},t)-p_t^{(k)})|_{t=t_i}+O(h^{r_1}).
\end{align*}
According to \eqref{decomposition}, it follows that 
\begin{align}
\nabla \frac {\partial}{\partial t}\psi_{\theta}(x_{t_i}^{(k)},t_i)
-\nabla \frac {\partial}{\partial t}\psi(x_{t_i}^{(k)},t_i)
&=\sum_{j\in N(i)}a_{ij}(\nabla \psi_{\theta}(x_{t_j}^{(k)},t_j)-p_{t_j}^{k})\frac 1 h+O(h^{r_1})\nonumber \\
&=\sum_{j\in N(i)}a_{ij}(\widetilde p_{t_j}^{(k)}+e_{j}^k-p_{t_j}^{k})\frac 1{h}+O(h^{r-2})+O(h^{r_1})\nonumber\\
&=\sum_{j\in N(i)}a_{ij}  e_{j}^k\frac 1h +O(h^{r-2})+O(h^{r_1}). \label{upperbound1}
\end{align}

Next we give the estimate  for the term $\nabla^2 \psi_{\theta}(x_{t_i}^{(k)},t_i)\frac {\partial}{\partial p} H(x_{t_i}^{(k)},p_{t_i}^{(k)})
-\nabla^2 \psi(x_{t_i}^{(k)}$ $,t_i) \frac {\partial}{\partial p} H(x_{t_i}^{(k)},p_{t_i}^{(k)})$.
By using the mean value theorem and \eqref{decomposition} again, we obtain that 
\begin{align*}
  & \nabla^2 \psi_{\theta}(x_{t_i}^{(k)},t_i)\frac {\partial}{\partial p} H(x_{t_i}^{(k)},p_{t_i}^{(k)})
-\nabla^2 \psi(x_{t_i}^{(k)},t_i)\frac {\partial}{\partial p} H(x_{t_i}^{(k)},p_{t_i}^{(k)})\\
&=
(\nabla^2 \psi_{
\theta
}(x_{t_i}^{(k)},t_i)-\nabla p_{t_i}^{(k)})\frac {\partial}{\partial p} H(x_{t_i}^{(k)},p_{t_i}^{(k)})\\
&=(\nabla \widetilde p_{t_i}^{(k)} -\nabla p_{t_i}^{(k)})\frac {\partial}{\partial p} H(x_{t_i}^{(k)},p_{t_i}^{(k)})
+\nabla e_i^k \frac {\partial}{\partial p} H(x_{t_i}^{(k)},p_{t_i}^{(k)})
+O(h^{r-1}).
\end{align*}
Since the order of time integrator will not depends on the formulation of the coefficient of ODEs, one has $\nabla \widetilde p_{t_i}^{(k)} -\nabla p_{t_i}^{(k)}\sim O(h^{r-1}).$ 
As a consequence, it holds that 
\begin{align}
    &\nabla^2 \psi_{\theta}(x_{t_i}^{(k)},t_i)\frac {\partial}{\partial p} H(x_{t_i}^{(k)},p_{t_i}^{(k)})
-\nabla^2 \psi(x_{t_i}^{(k)},t_i)\frac {\partial}{\partial p} H(x_{t_i}^{(k)},p_{t_i}^{(k)})\\\nonumber
&=\nabla e_i^k \frac {\partial}{\partial p} H(x_{t_i}^{(k)},p_{t_i}^{(k)})
+O(h^{r-1}).\label{upperbound2}
\end{align}

It suffices to estimate the term $\frac {\partial}{\partial x}H(x_{t_i}^{(k)},\nabla \psi_{\theta}(x_{t_i}^{(k)},t_i))
-\frac {\partial}{\partial x}H(x_{t_i}^{(k)},\nabla \psi(x_{t_i}^{(k)},t)).$ For this term, using the mean value theorem, \eqref{decomposition} and the order of the numerical scheme, we get 
{\small\begin{align}
&\frac {\partial}{\partial x}H(x_{t_i}^{(k)},\nabla \psi_{\theta}(x_{t_i}^{(k)},t_i))
-\frac {\partial}{\partial x}H(x_{t_i}^{(k)},\nabla \psi(x_{t_i}^{(k)},t)) \nonumber \\
&=\int_0^1 \frac {\partial^2}{\partial x \partial 
 p }H(x_{t_i}^{(k)}, \alpha_2\nabla \psi_{\theta}(x_{t_i}^{(k)},t_i)
+(1-\alpha_2)\nabla \psi(x_{t_i}^{(k)},t))(\nabla \psi_{\theta}(x_{t_i}^{(k)},t_i)-\nabla \psi(x_{t_i}^{(k)},t_i)) d\alpha_2 \nonumber \\
&=\int_0^1 \frac {\partial^2}{\partial x \partial 
 p }H(x_{t_i}^{(k)}, \alpha_2\nabla \psi_{\theta}(x_{t_i}^{(k)},t_i)
+(1-\alpha_2)\nabla \psi(x_{t_i}^{(k)},t))(\widetilde p_{t_i}^{(k)}-p_{t_i}^{(k)}) d\alpha_2
+O(h^{r-1}) \label{upperbound3} \\
&+ \int_0^1 \frac {\partial^2}{\partial x \partial 
 p }H(x_{t_i}^{(k)}, \alpha_2\nabla \psi_{\theta}(x_{t_i}^{(k)},t_i)
+(1-\alpha_2)\nabla \psi(x_{t_i}^{(k)},t))e_i^kd\alpha_2. \nonumber
\end{align}}

Combining the estimates \eqref{upperbound1}-\eqref{upperbound3},  we obtain that 
{\small
\begin{align*}
&\frac 1 N \sum_{k=1}^N\mathcal D \psi_{\theta}( x_{t_i}^{(k)}, p_{t_i}^{(k)},t_i)- \mathcal D \psi(x_{t_i}^{(k)},  p_{t_i}^{(k)},t_i)\\
&= \frac 1 N \sum_{k=1}^N\sum_{j\in N(i)}a_{ij}  e_{j}^k\frac 1h+\nabla e_i^k \frac {\partial} {\partial p}H(x_{t_i}^{(k)},p_{t_i}^{(k)})\\
&+\int_0^1 \frac {\partial^2}{\partial x \partial 
 p }H(x_{t_i}^{(k)}, \alpha_2\nabla \psi_{\theta}(x_{t_i}^{(k)},t_i)
+(1-\alpha_2)\nabla \psi(x_{t_i}^{(k)},t))e_i^kd\alpha_2+O(h^{r-2})+O(h^{r_1}).
\end{align*}}Taking $r_1\ge r-2$, and using \eqref{error1} and the  Taylor expansion, we further obtain that 
\begin{align} \label{overall error}
  &\frac {1}{N} \sum_{k=1}^N \mathcal R \psi_{\theta}(\widetilde x_{t_i}^{(k)},t_i) \\\nonumber&=O(h^{r-2})+\frac {1}{N} \sum_{k=1}^N \Big(\sum_{j\in N(i)}a_{ij}  e_{j}^k\frac 1h+\nabla e_i^k \frac {\partial} {\partial p}H(x_{t_i}^{(k)},p_{t_i}^{(k)})  \\\nonumber 
  &+\int_0^1 \frac {\partial^2}{\partial x \partial 
 p }H(x_{t_i}^{(k)}, \alpha_2\nabla \psi_{\theta}(x_{t_i}^{(k)},t_i)
+(1-\alpha_2)\nabla \psi(x_{t_i}^{(k)},t))e_i^kd\alpha_2\Big) \nonumber\\  
&=O(h^{r-2})+\frac 1{N}\sum_{k=1}^N|\sum_{j\in N(i)}a_{ij}e_j^k\frac 1h| +\nu(\theta,i)(|\nabla e_i^{(k)}|+|e_i^{(k)}|), \nonumber
\end{align}
where 
\begin{align*}\nu(\theta,i)
&=\sup_{x_{t_i}\sim \rho_{t_i}}\Big(|\frac {\partial}{\partial p} H(x_{t_i},p_{t_i})|+|\int_0^1 \frac {\partial^2}{\partial x \partial 
 p }H(x_{t_i}^{(k)}, \alpha_2\nabla \psi_{\theta}(x_{t_i}^{(k)},t_i)
\\
&\quad +(1-\alpha_2)\nabla \psi(x_{t_i}^{(k)},t))d\alpha_2|\Big).
\end{align*}

To further estimate the expectation of the $L^1$-residual at all the time nodes $\{t_1,\dots, t_T\}$, let us denote $\tilde{\rho}_{t_i}=(\tilde{\Phi}_h\circ\dots\circ \tilde{\Phi}_h)_\sharp \rho_0$ as the probability density function of the numerical solution $\tilde{x}_{t_i}$ computed by the chosen scheme starting from $x_0\sim\rho_0$.
For a fixed time $t_i$ and samples $\{\tilde{x}_{t_i}^{(k)}\}_{1\leq k \leq N }\sim \tilde{\rho}_{t_i}$, by Hoeffding's inequality (see e.g. \cite{SSBD2014}), for any $0<\delta<1$, with probability $1-\delta$, we can bound the gap between the expectation and the empirical average of the $L^1$ residual as

\begin{equation}  
    \left|\int_{\mathbb{R}^d} \mathcal{R}[\psi_\theta](x,t_i) \tilde{\rho}_{t_i}~dx - \frac{1}{N}\sum_{k=1}^N \mathcal{R}[\psi_\theta](\tilde{x}_{t_i}^{(k)},t_i)\right|\leq \underbrace{\sup_{x\in\textrm{supp}({\tilde{\rho}_{t_i}})}|\mathcal{R}[\psi_\theta](x,t_i)|}_{\textrm{denote as}~ R(\theta,i)}\sqrt{\frac{\ln\frac{2}{\delta}}{2N}}.\label{hoeffding ineq_appendix}
\end{equation}
Similarly, for the samples $\{x_{t_i}^{(k)}\}_{1\le k\le N}\sim \rho_{t_i}$, for any $0<\delta<1,$ with probability $1-\delta,$
it holds that{\small 
\begin{align}
\left|\int_{\mathbb{R}^d} \mathcal{R}_{exa}[\psi](x,t_i) {\rho}_{t_i}~dx - \frac{1}{N}\sum_{k=1}^N \mathcal{R}_{exa}[\psi](\tilde{x}_{t_i}^{(k)},t_i)\right|\leq \underbrace{\sup_{x\in\textrm{supp}({{\rho}_{t_i}})}|\mathcal{R}_{exa}[\psi](x,t_i)|}_{\textrm{denote as}~ R_{exa}(i)}\sqrt{\frac{\ln\frac{2}{\delta}}{2N}}.\label{hoeffding ineq1}
\end{align}}Since we assume that $\textrm{supp}(\rho_0)$ is a bounded set, and the solution maps of the numerical scheme and the ODE system is continuous, then $\textrm{supp}(\tilde{\rho}_{t_i}),\textrm{supp}  ({\rho}_{t_i})$ are also bounded. Thus $R(\theta, i),R_{exa}(i)$ is guaranteed to be finite. Indeed, $R_{exa}(i)=0$ by our assumption.
Combining \eqref{overall error}, \eqref{hoeffding ineq_appendix}, and \eqref{hoeffding ineq1}, and using the similar arguments as in the proof of Theorem 3.1, we obtain the desired result where $C_{\theta,i}h^{r-2}$ is the upper bound of $\mathcal O(h^{r-2}).$
\end{proof}

\section{Behaviors of neural network solution near caustics}\label{SM-NN}

This section is a continuation of subsection \ref{subsec: study_near_caustics}. We  present a series  of 2D Hamilton–Jacobi equation, with or without caustic formation,  generated by varying the initial values $g(\cdot)$ and initial densities $\rho_0$. We also analyze the numerical behavior of the neural network approximation with different activation functions.
  
\textbf{Various activation functions:} In addition to $\mathrm{tanh}(\cdot),$ we test the same example using ResNets with different activation functions $\sigma(\cdot)$, including ReLU, softplus, and the $\sin(\cdot)$ function. The numerical performance of each activation function is summarized in Figure \ref{fig: 2D_caustic_compare_activations}. Based on the results, one may prefer $\tanh(\cdot)$ or $\sin(\cdot)$ over softplus or ReLU, as the latter appear less effective for approximating gradient fields. The sinusoidal activation is particularly well-suited for tasks involving periodic or oscillatory behavior. Although $\tanh(\cdot)$ yields smoother nonlinear approximations, neural networks with this activation still successfully capture the shocks that arise in the dynamical process. This observation highlights the broader applicability of $\tanh(\cdot)$, and therefore, we adopt it as the activation function for the subsequent examples presented in this article.
\begin{figure}[htb!]
\begin{subfigure}{.245\textwidth}
  \centering
  \includegraphics[width=\linewidth]{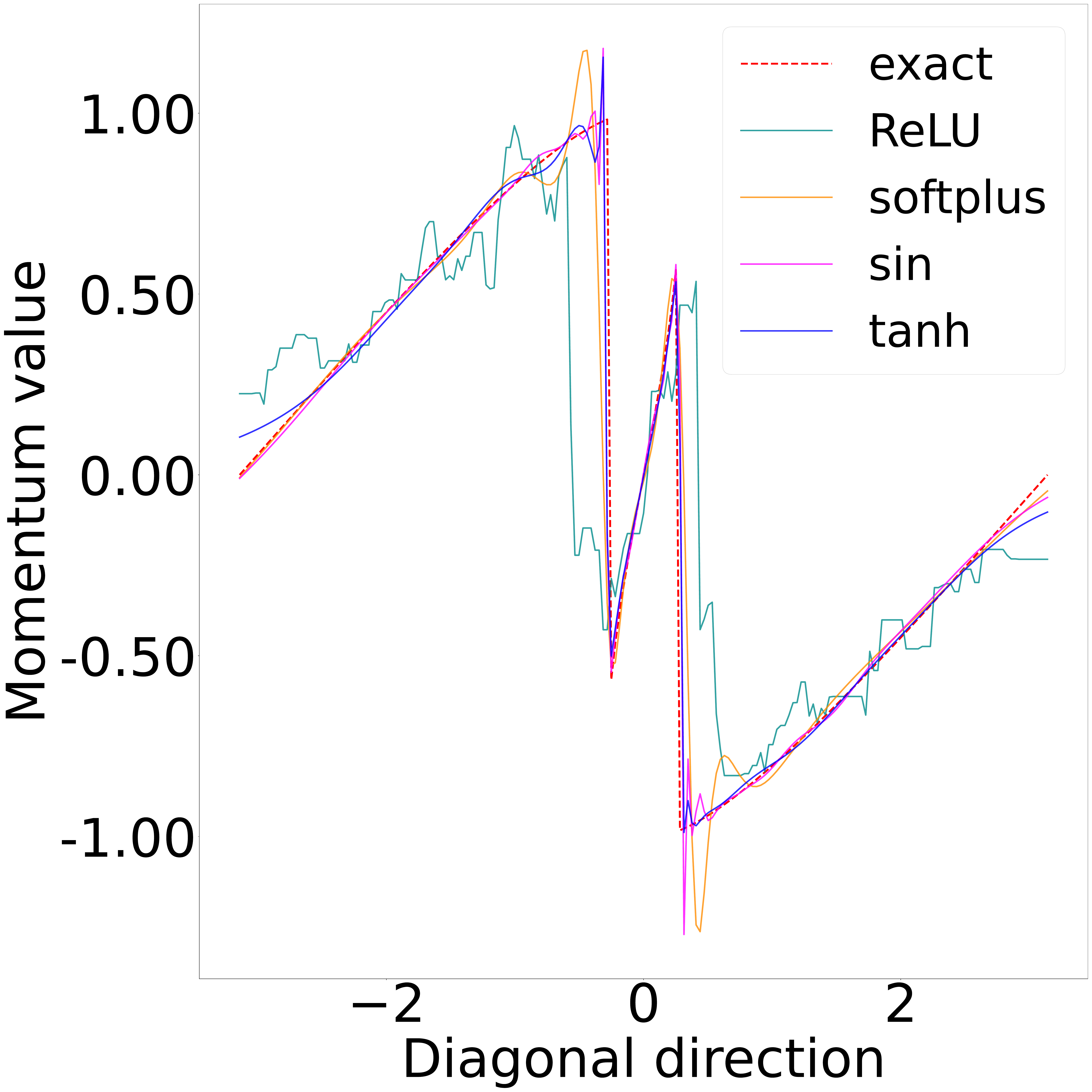}
  \subcaption{$t=1.5$}\label{subfig1}
\end{subfigure}
\begin{subfigure}{.245\textwidth}
  \centering
  \includegraphics[width=\linewidth]{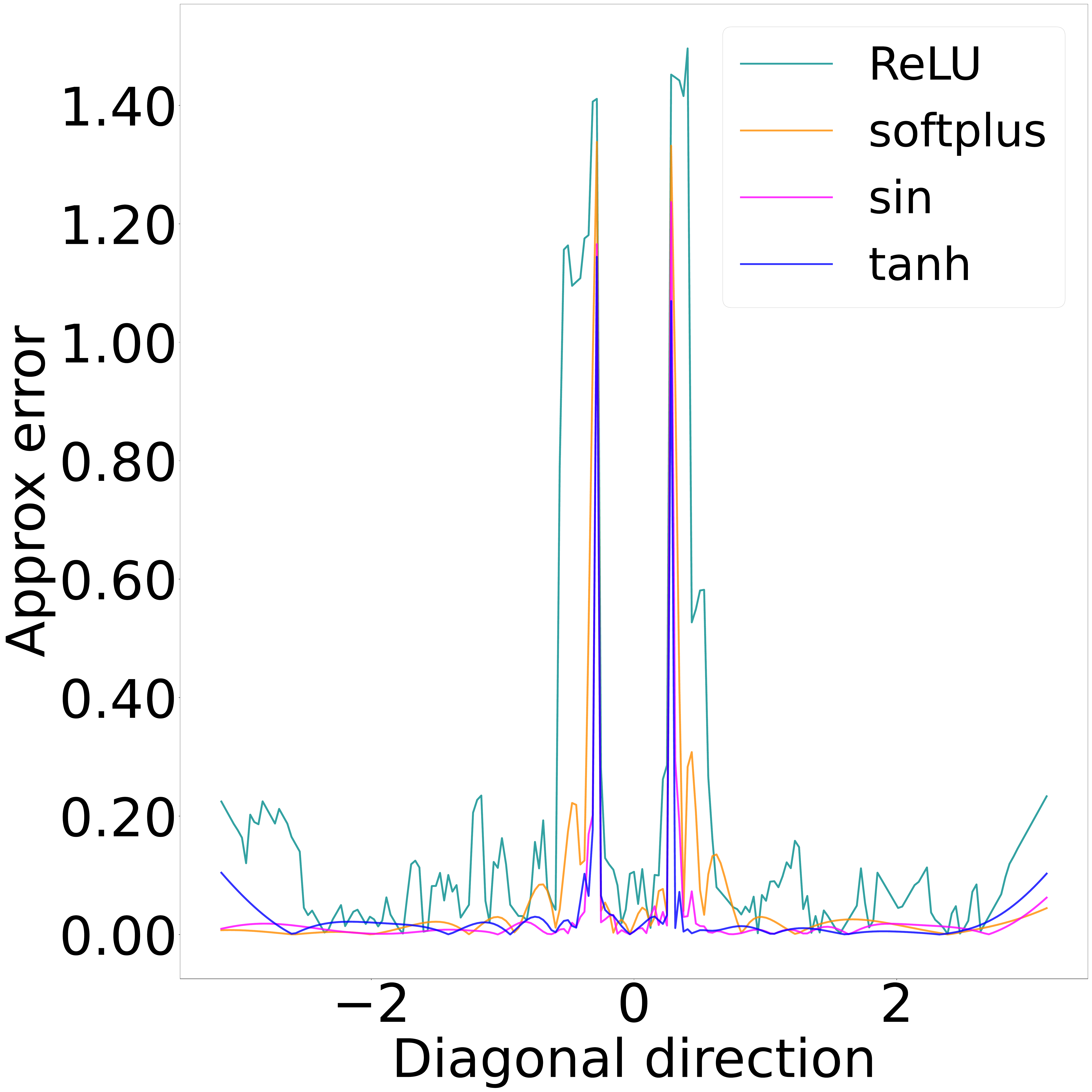}
  \subcaption{$t=1.5$}\label{subfig2}
\end{subfigure}
\begin{subfigure}{.245\textwidth}
  \centering
  \includegraphics[width=\linewidth]{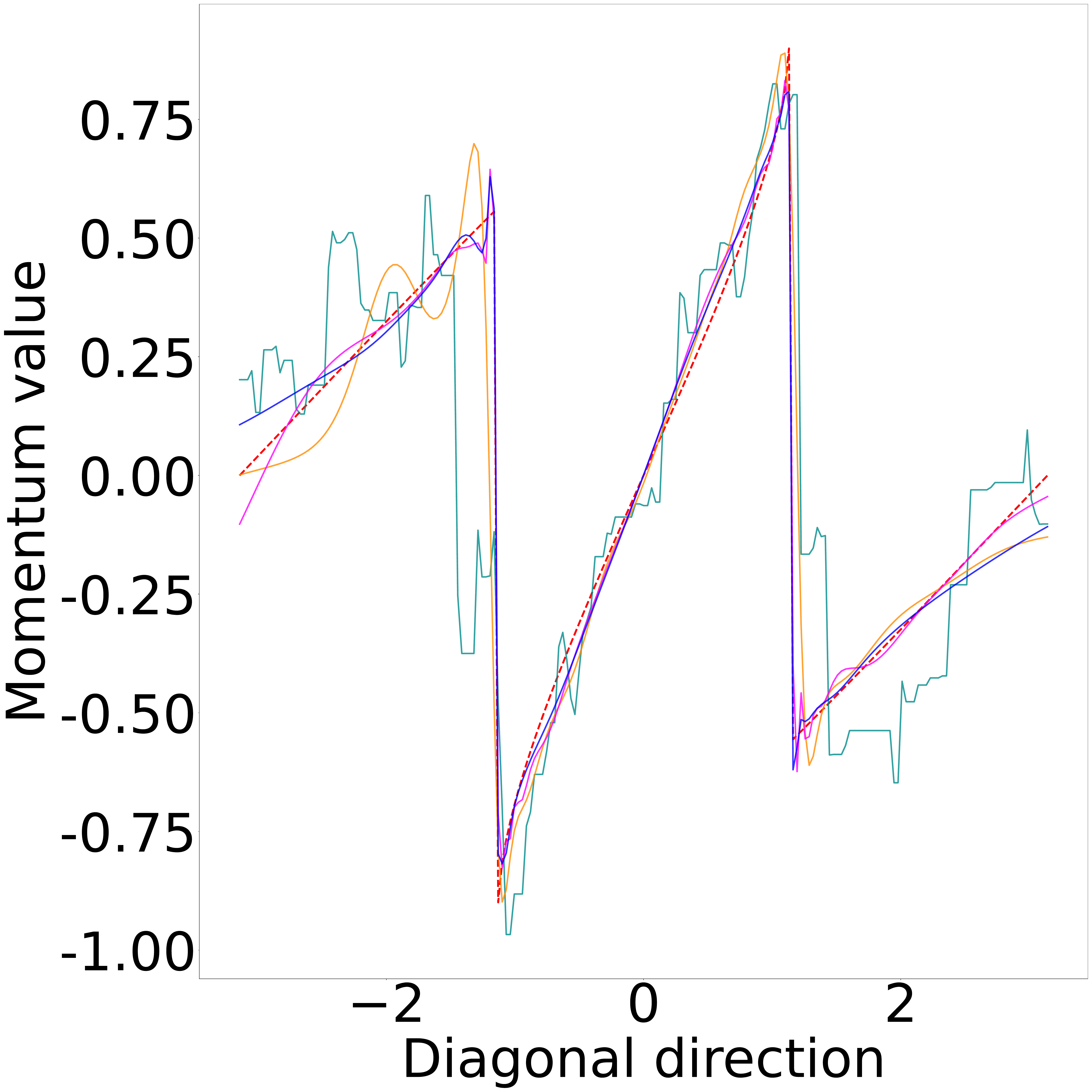}
  \subcaption{$t=2.5$}\label{subfig3}
\end{subfigure}
\begin{subfigure}{.245\textwidth}
  \centering
  \includegraphics[width=\linewidth]{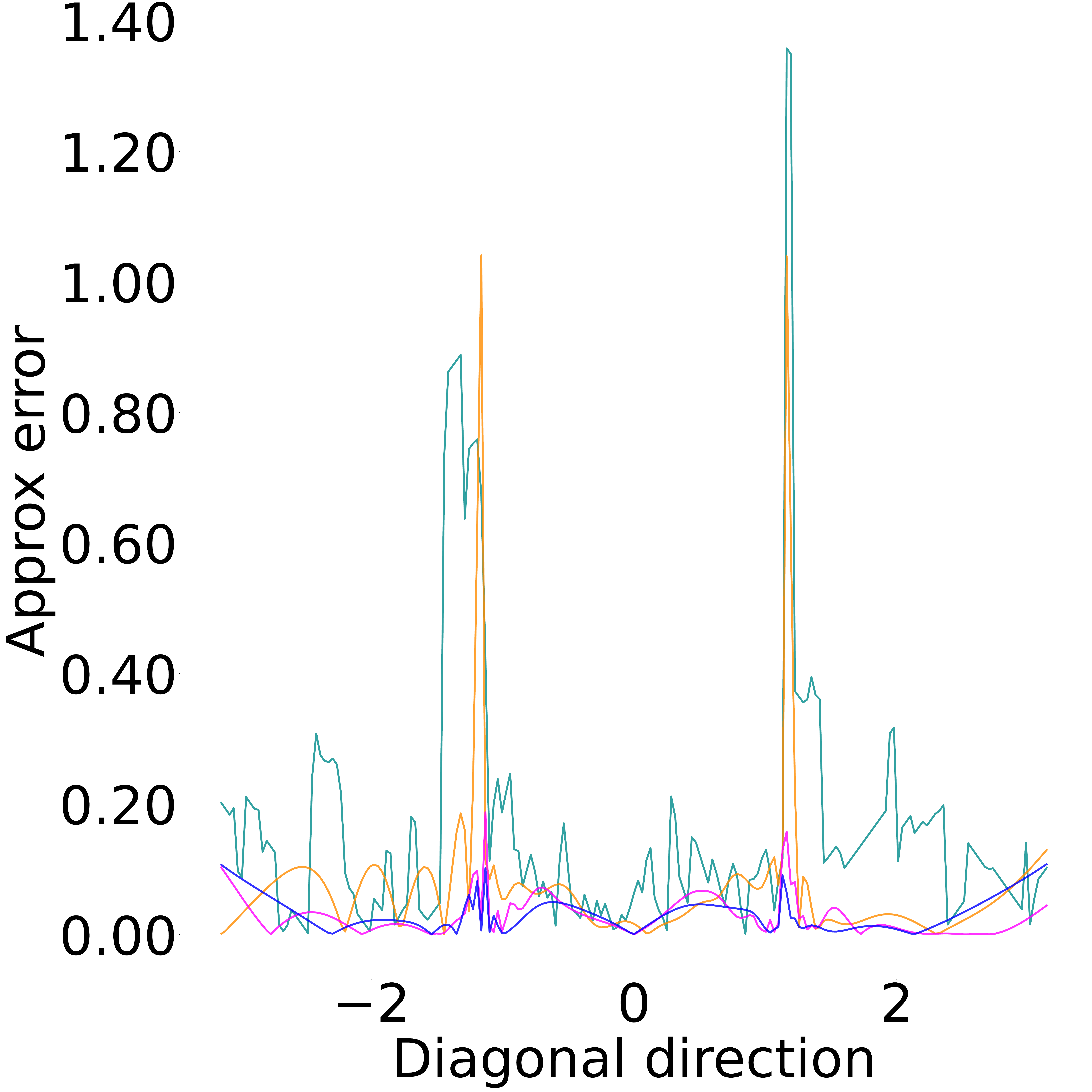}
  \subcaption{$t=2.5$}\label{subfig4}
\end{subfigure}
\vspace{-0.3cm}
\caption{{ \ref{subfig1}: Plots of $\boldsymbol{\eta}^\top\nabla\psi_\theta(z\boldsymbol{\eta})$ with the exact weighted momentum $\partial_z \widehat{f}(z, t)$ on $[-\pi,\pi]$ at $t=1.5$; \ref{subfig2}: Plots of the error $|\boldsymbol{\eta}^\top\nabla\psi_\theta(z\boldsymbol{\eta}) - \partial_z \widehat{f}(z, t)|$ on $[-\pi,\pi]$ at $t=1.5$. We test $\psi_\theta$ as ResNets using different activation functions; \ref{subfig3}-\ref{subfig4}: Same plots at $t=2.5$.}}
\label{fig: 2D_caustic_compare_activations}
\end{figure}

\noindent
\textbf{Non-continuous $\rho_0$:} We consider a more challenging example in which the initial density $\rho_0$ is discontinuous, defined as $\rho_0 = \frac{1}{|E|}(\lambda_1\chi_{E_1} + \lambda_2\chi_{E_2})$, where $\lambda_1 = \frac{1}{11}$, $\lambda_2 = \frac{10}{11}$, $|E|=2\pi^2$, and $E_1 := \{(x, y) \in E \mid x + y < 0 \}$, $E_2 := E \setminus E_1$. Here $\chi_{E_1}, \chi_{E_2}$ are indicator functions supported on $E_1, E_2$. The target function $\partial_z \widehat{f}(\cdot, t)$ develops three discontinuities on the interval $[-\pi, \pi]$ once $t$ exceeds the critical time $T_* = 1$. As illustrated on the left of Figure~\ref{fig: 2D_caustic_rho_1_10_and_non_caustics}, the discontinuity in $\rho_0$ further reduces the regularity of the weighted momentum, introducing additional challenges for neural network approximation.

\noindent
\textbf{Caustic-free initial condition:} In contrast to the caustic-forming case, we consider a scenario without caustic development by setting $g(x) \! = \! -\!\cos(\boldsymbol{\eta}^\top x)$, $\rho_0 = \mathcal{U}(E)$, and $T = 3$. The numerical results, shown on the right side of Figure~\ref{fig: 2D_caustic_rho_1_10_and_non_caustics}, demonstrate agreement between the neural network approximation and the exact classical solution. Moreover, the training loss readily drops to zero as the computing time increases.
\begin{figure}[htb!]
    \centering
    \begin{minipage}{0.45\textwidth}
      \centering
      \includegraphics[width=0.85\linewidth]{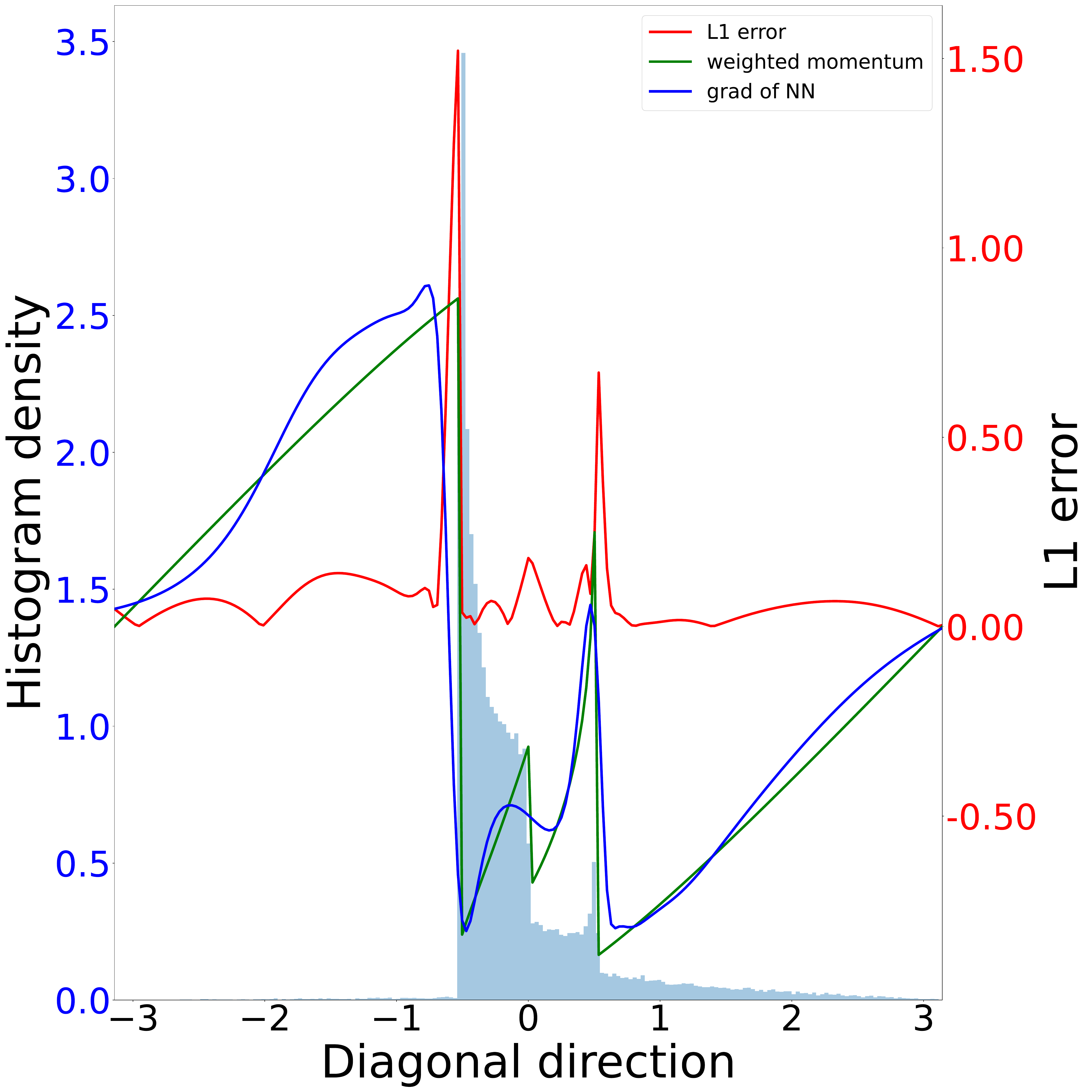}
      \subcaption*{{ Non-continuous $\rho_0$, $g(x)=\cos x$, $t=1.8$.}}
    \end{minipage}\hfill
    \begin{minipage}{0.53\textwidth}
        \centering
        \begin{minipage}{\textwidth}
        \centering
            \subcaptionbox*{}{\includegraphics[trim={0cm 15cm 0cm 17cm}, clip, width=0.78\linewidth]{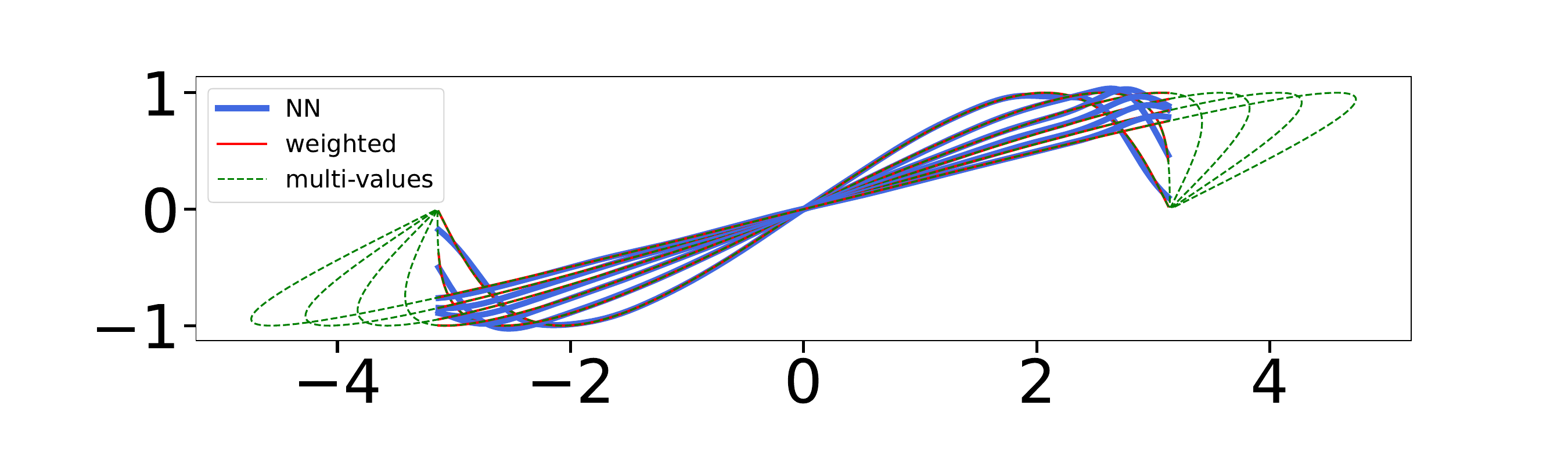}}\vspace{-7.9mm}
            \subcaptionbox*{{ \centering $\rho_0=\mathcal U(E)$,  $g(x)=-\cos x$.}}
            {\includegraphics[trim={0cm 1.4cm 0cm 5cm}, clip, width=0.62\linewidth]{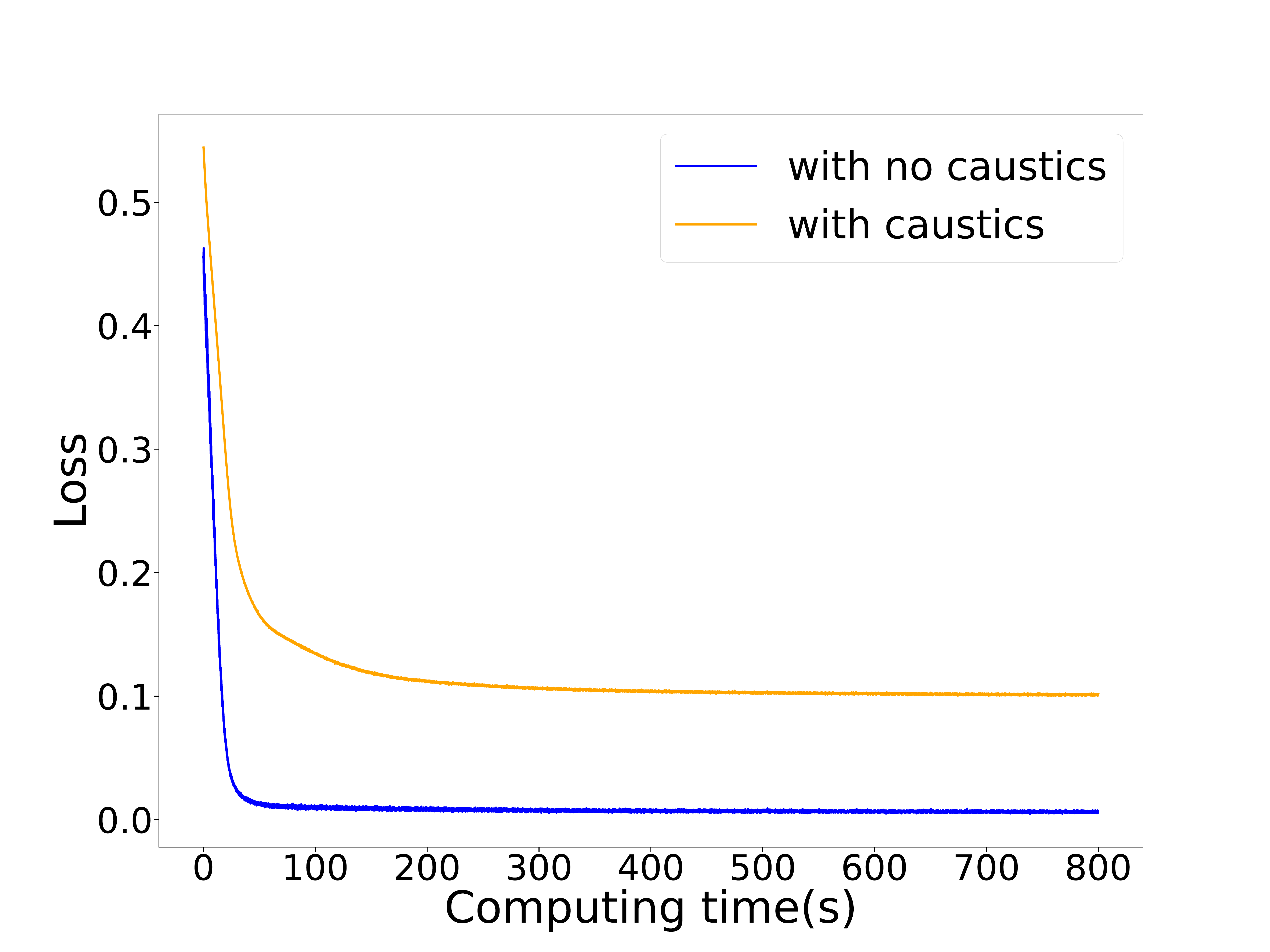}}
        \end{minipage}
    \end{minipage}
    \vspace{-0.3cm}
    \caption{{ \textbf{Left}: Plots of $\boldsymbol{\eta}^\top\nabla\psi_\theta(\cdot, t)$ at time $t=1.8$ for non-continuous $\rho_0 := (\lambda_1\chi_{E_1} + \lambda_2\chi_{E_2})/|E|$ and initial condition $g(x)=\cos(\boldsymbol{\eta}^\top x)$; \textbf{Right}: The case with uniform $\rho_0=\mathcal U(E)$ and initial condition $g(x)=-\cos(\boldsymbol{\eta}^\top x)$. \textbf{Top}: Plots of NN approximation $\boldsymbol{\eta}^\top\nabla\psi(z\boldsymbol{\eta}, t)$ with reference solution for $t\in[0, 3]$; \textbf{Bottom}: Training loss versus wall time(s) for the two cases: one that develops caustics ($g(x) = \cos(\boldsymbol{\eta}^\top x)$) and one that does not ($g(x) = -\cos(\boldsymbol{\eta}^\top x)$).}}\label{fig: 2D_caustic_rho_1_10_and_non_caustics}
\end{figure}

\section{More examples for Solving HJ equations}\label{more-example}

This section continues the discussion from subsection \ref{numerical example separable hamilton }. We apply our method to solve equation \eqref{HJ}  for both separable and non-separable Hamiltonians.

{ 
\subsection{Example with Degenerate Kinetic Energy and Sinusoidal Initial}\label{example: sinusoidal_initial}
{We consider a similar example to the one discussed in Section \ref{subsec: study_near_caustics}, with $d = 20$. Consider the Hamiltonian as the degenerate quadratic kinetic energy $K(p)=\frac{1}{2}p^\top\boldsymbol{\Sigma} p + \tau\boldsymbol{\eta}^\top p$ with $\boldsymbol{\Sigma}=\frac{1}{d}\boldsymbol{1}\boldsymbol{1}^\top$, $\boldsymbol{\eta}=\frac{1}{\sqrt{d}}\boldsymbol{1}$, $\tau=3$. We set the initial value $g(x)=\cos(\sqrt{3}\boldsymbol{\eta}^\top x)$ and $\rho_0 = \mathcal{U}([-4.5, 4.5]^d)$. We solve this equation on $[0, \frac{2}{3}]$. Similar to the discussion in section \ref{subsec: study_near_caustics}, the classical solution $ u (x,t)$ of \eqref{HJ} takes the form $ u (x,t) = f(\boldsymbol{\eta}^\top x, t)$, where $f(\cdot, t):\mathbb{R} \rightarrow \mathbb{R}$ satisfies
\begin{equation*}
  \partial_z f(z,t)=-\sqrt{3}\sin(\sqrt{3}\varphi_t^{-1}(z)), 
\end{equation*}
Here we denote $\varphi_t(\xi)=\xi+t(\tau-\sqrt{3}\sin(\sqrt{3}\xi))$. $\varphi_t$ remains injective as $0\leq t \leq T_*:=\frac13$, and the classical solution to this HJ equation \eqref{HJ} exists on $[0, T_*]$} We demonstrate the numerical solutions in Figure \ref{sinusodial example graph}.
\begin{figure}[htb!]
\begin{subfigure}{.24\textwidth}
  \centering
  \includegraphics[width=\linewidth]{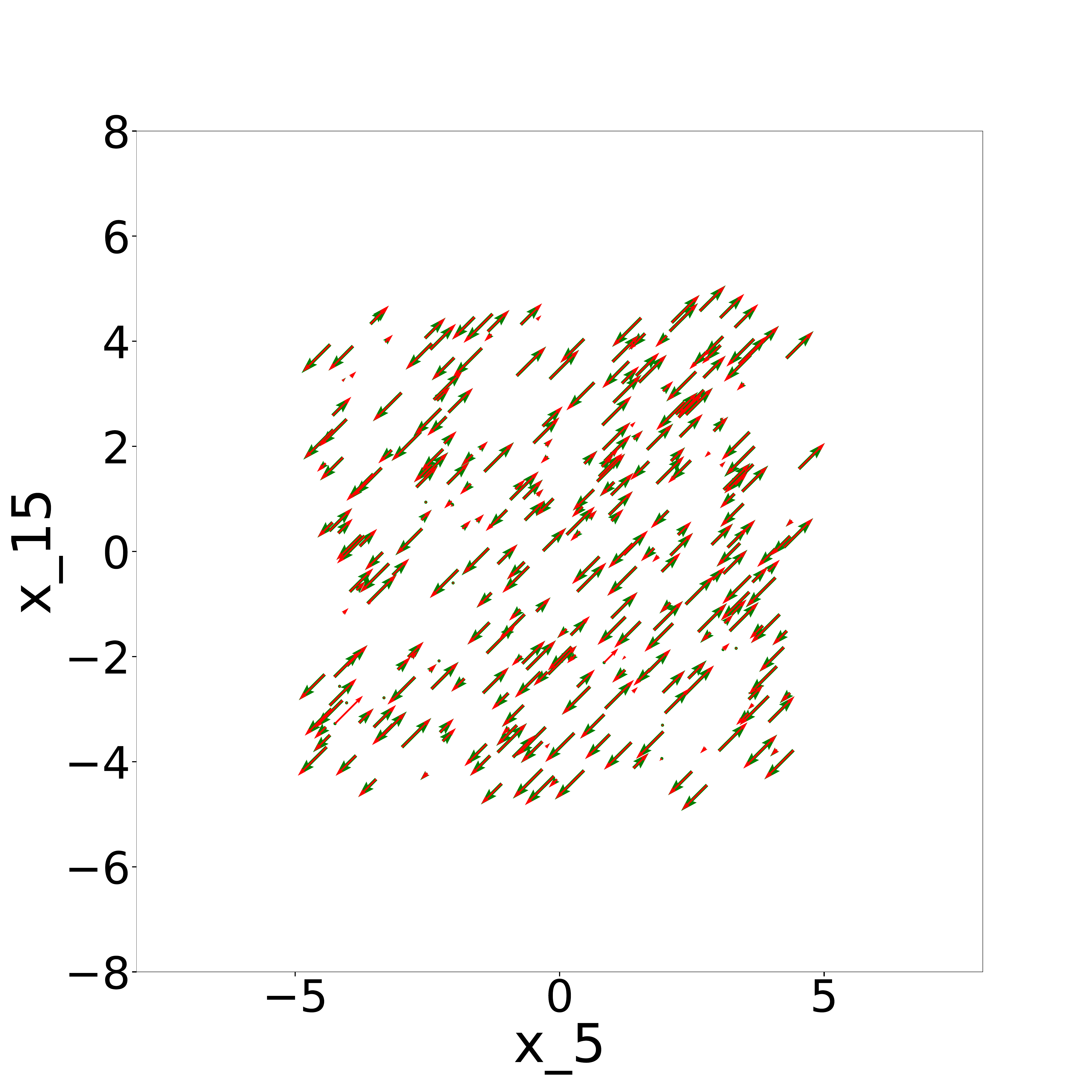}
  \caption{$t=\frac{1}{9}$}
\end{subfigure}
\begin{subfigure}{.24\textwidth}
  \centering
  \includegraphics[width=\linewidth]{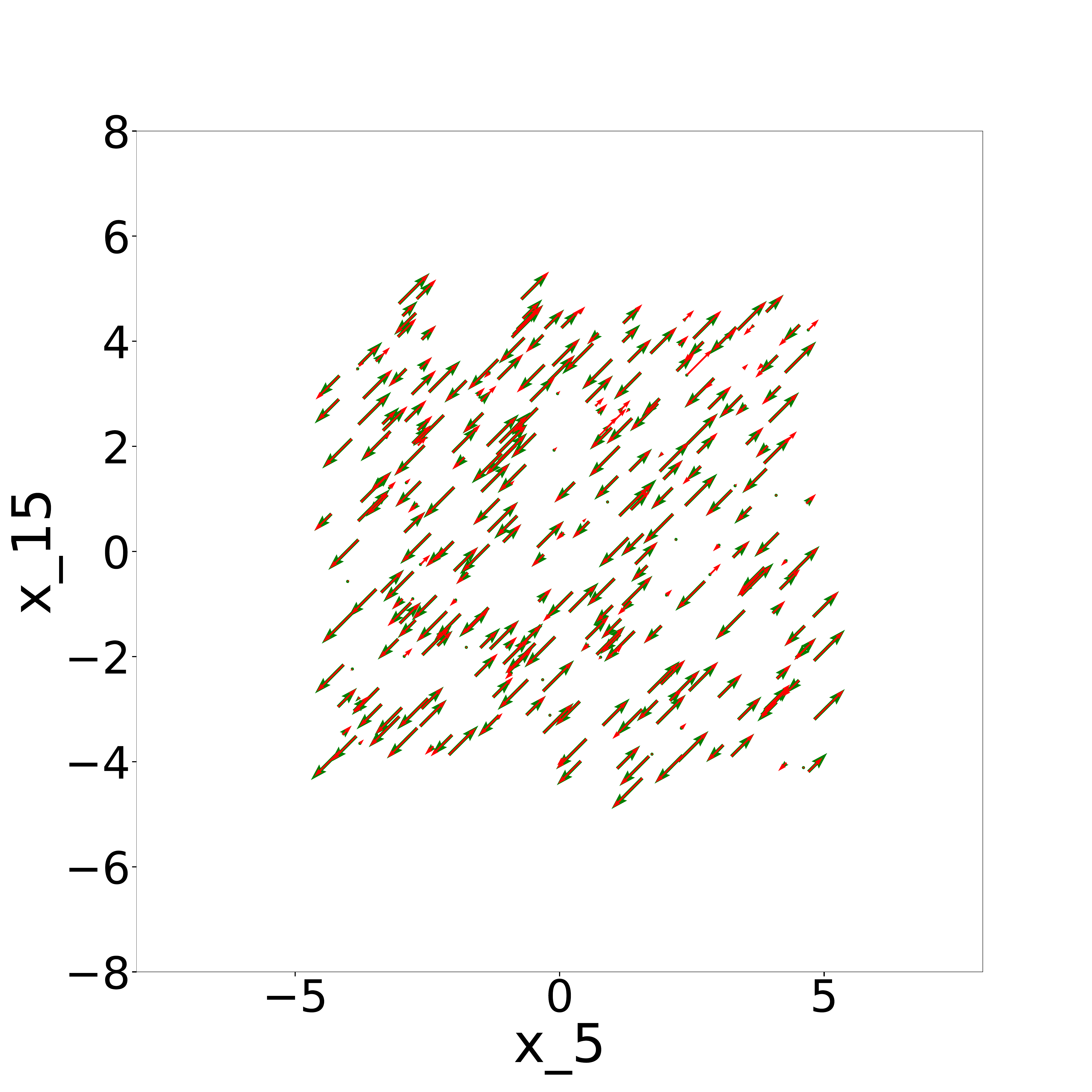}
  \caption{$t=\frac{1}{3}$ ($T_*$)}
\end{subfigure}
\begin{subfigure}{.24\textwidth}
  \centering
  \includegraphics[width=\linewidth]{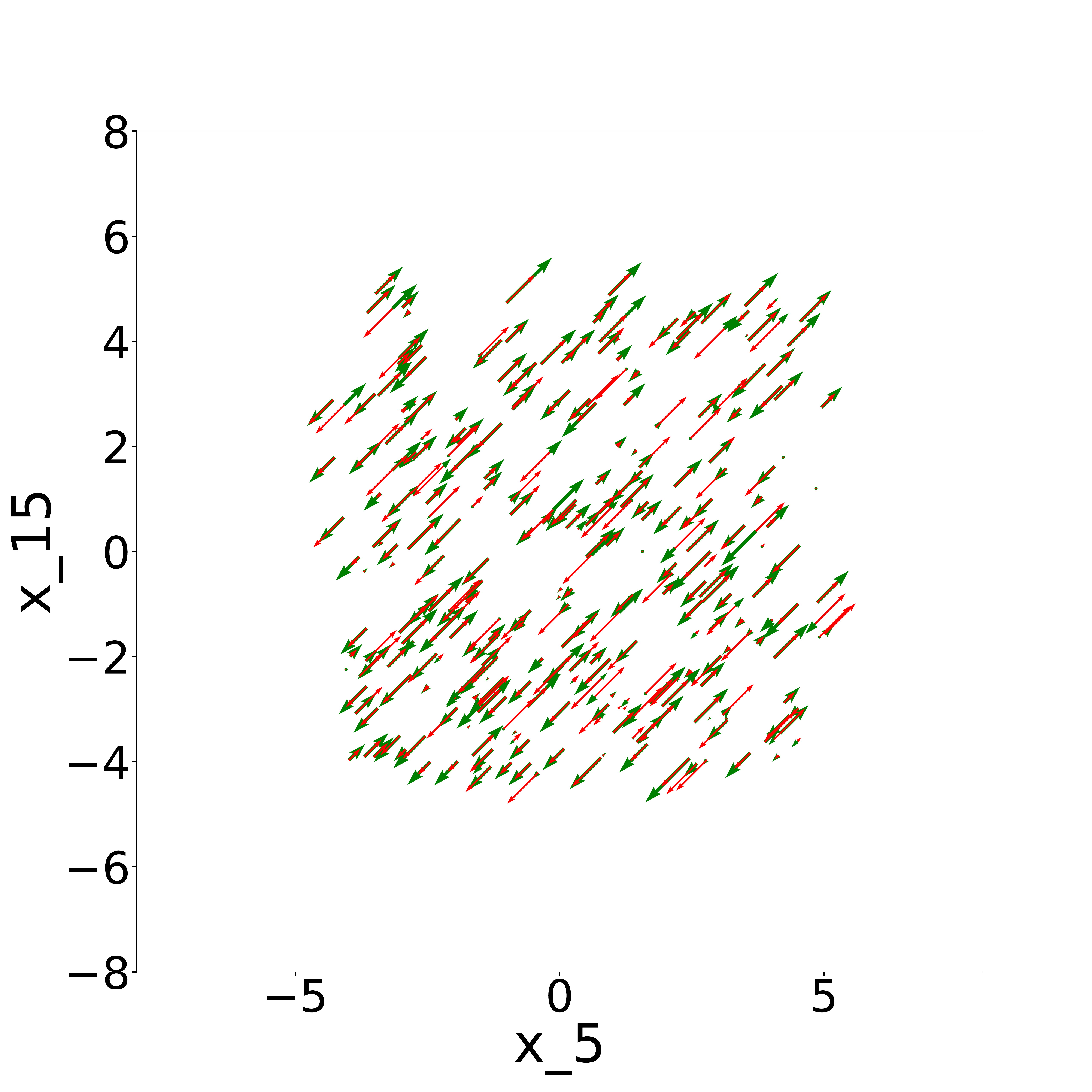}
  \caption{$t=\frac{5}{9}$}
\end{subfigure}
\begin{subfigure}{.24\textwidth}
  \centering
  \includegraphics[width=\linewidth]{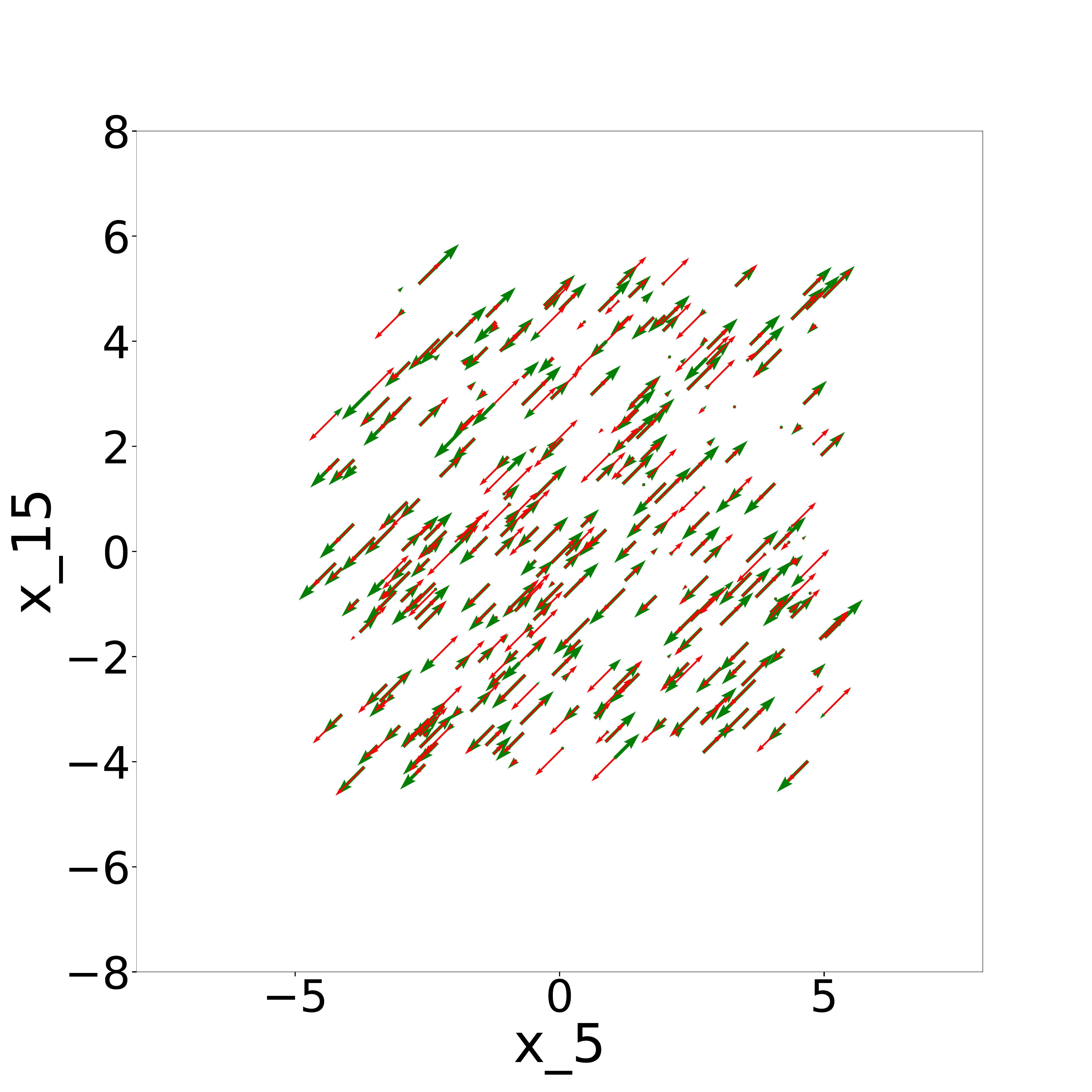}
  \caption{$t=\frac{2}{3}$}
\end{subfigure}
\vspace{-0.3cm}
\caption{Plots of vector fields $\nabla \psi_\theta(\cdot, t)$ (green) with momentums of samples (red) at different time stages on the {$5\text{th}-15\text{th}$} plane.}
\label{sinusodial example graph}
\end{figure}
The comparison between numerical solution and the exact solution along the diagonal line in $\mathbb{R}^{20}$ before $T_*$ is provided in Figure \ref{compare on diag line}. They show good agreement. 
\begin{figure}[htb!]
\begin{subfigure}{.32\textwidth}
  \centering
  \includegraphics[trim={0cm 2cm 0cm 0cm}, clip, width=\textwidth]{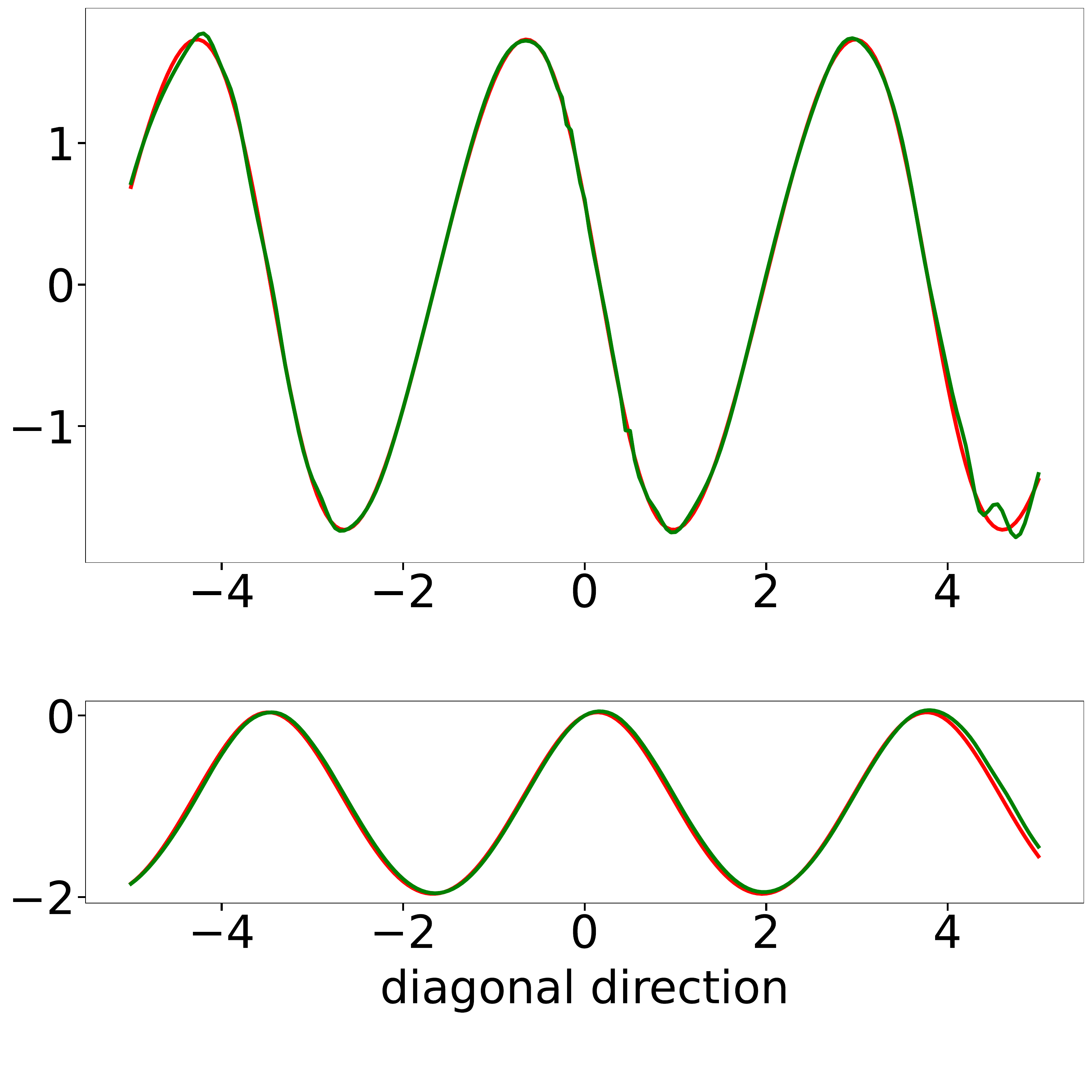}
  \caption{$t=\frac19$}
\end{subfigure}\hfill
\begin{subfigure}{.32\textwidth}
  \centering
  \includegraphics[trim={0cm 2cm 0cm 0cm}, clip, width=\textwidth]{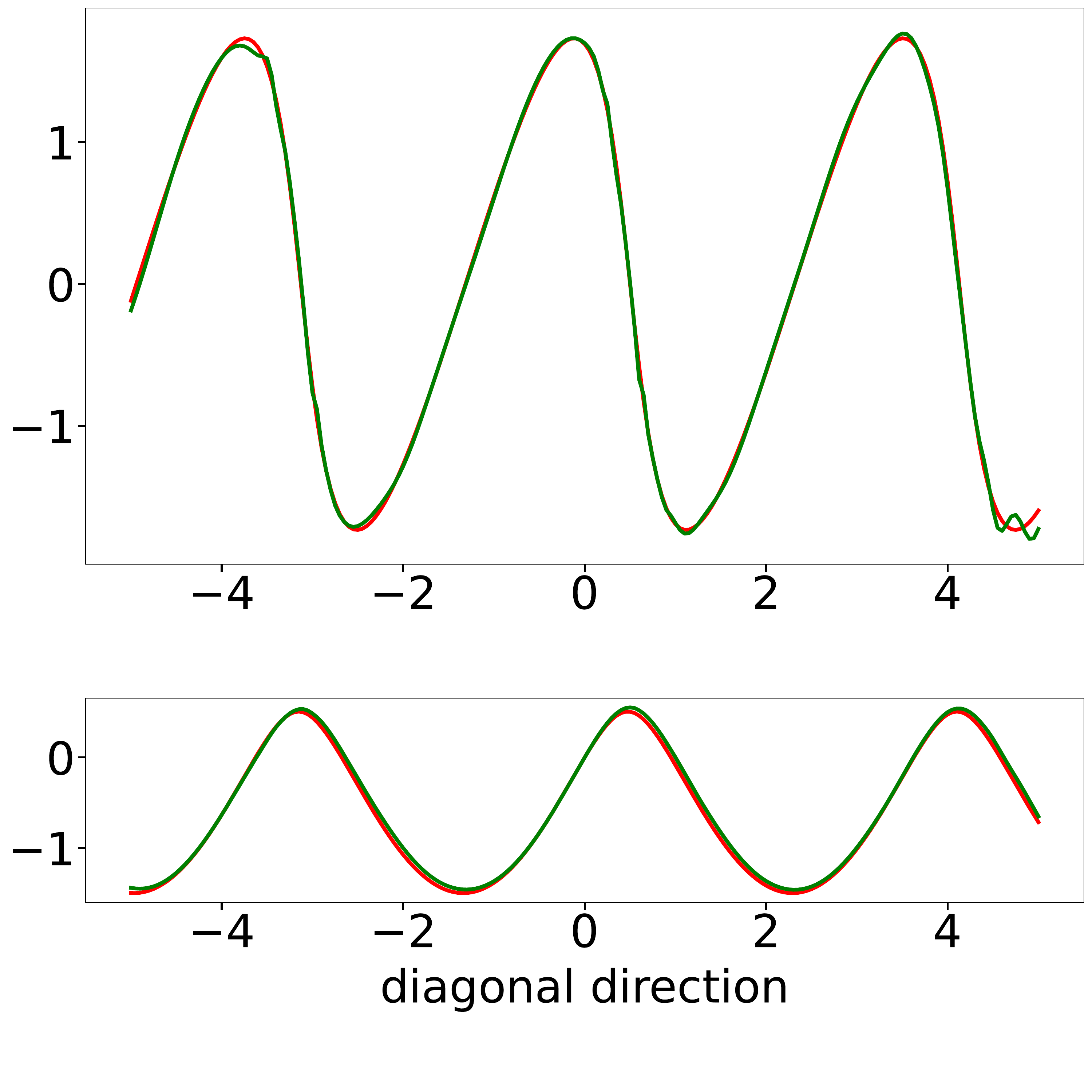}
  \caption{$t=\frac16$}
\end{subfigure}\hfill
\begin{subfigure}{.32\textwidth}
  \centering
  \includegraphics[trim={0cm 2cm 0cm 0cm}, clip, width=\textwidth]{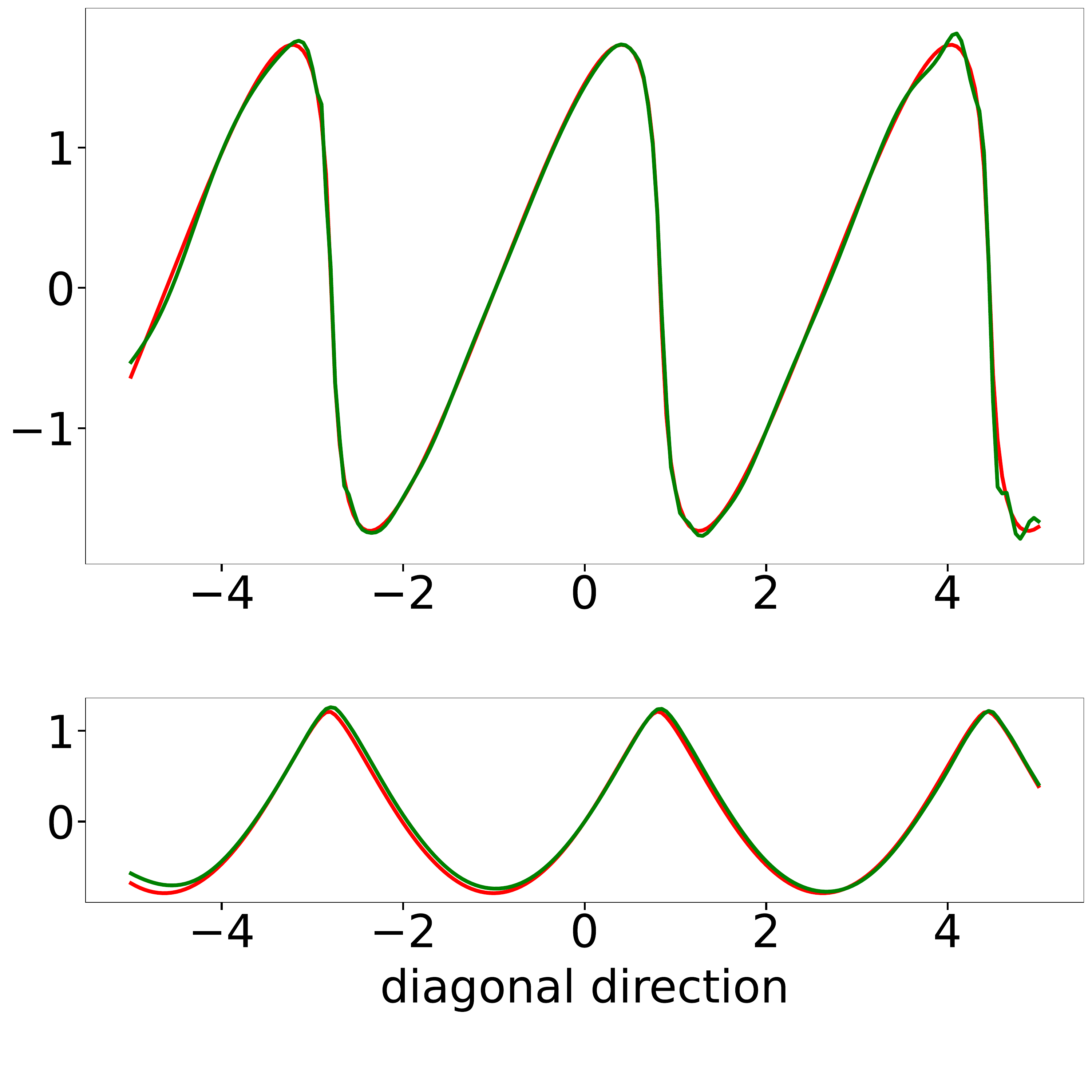}
  \caption{$t=\frac{5}{18}$}
\end{subfigure}\hfill
\vspace{-0.3cm}
\caption{{Comparison between numerical solution (green) and exact solution (red) along the diagonal line. \textbf{First row:} directional derivative $\boldsymbol{\eta}^\top\nabla\psi_\theta(x,t)$ (green) and $\boldsymbol{\eta}^\top\nabla  u(x,t)$ (red); \textbf{Second row:} $\psi_\theta(x,t)$ (green) and $\psi(x,t)$ (red).}}\label{compare on diag line}
\end{figure}
We further plot the loss $\frac{1}{N}\sum_{k=1}^N|e_{t_i}^{(k)}|^2$ (recall $e_{t_i}^{(k)}$ defined in \eqref{def every particle error}) versus time $t_i$ in Figure \ref{error vs time plots} (left subfigure). One can observe that the loss remains small before $T_*=\frac{1}{3}$ and increases significantly afterward.

\subsection{Example with Sinusoidal Potential and Gaussian Mixture as the Initial Distribution }\label{example: Shi Jin}
We now consider {the} Hamiltonian with a sinusoidal potential energy $H(x,p) = \frac{1}{2}|p|^2 + \cos(2x_{i_1}+0.4) + \cos(2x_{i_2}+0.4) $, {the} initial condition $u(x,0)=g(x)=\sin(x_{i_1}+0.15)+\sin(x_{i_2}+0.15)$, and the initial distribution $\rho_0=\frac{1}{2}(\mathcal{N}(\mu_1, I) + \mathcal{N}(\mu_2, I))$, 
where $\mu_1=-\frac{\pi}{2}(\boldsymbol{e}_{i_1} + \boldsymbol{e}_{i_2})$ and $\mu_2=\frac{\pi}{2}(\boldsymbol{e}_{i_1} + \boldsymbol{e}_{i_2})$. Here $\boldsymbol{e}_i$ denotes the vector with $i$-th entry being $1$ and remaining entries all $0$; and $i_1,i_2$ are two different integers between $1$ and $d$. In this example, we set $d=30$, $i_1 = 10, i_2 = 20$. We solve the equation on $[0, 1]$. A similar equation in one dimension was first considered in \cite{jin2005computing} and \cite{jin2011mathematical} in which the multivalued physical observables for the semiclassical limit of the Schr\"{o}dinger equation was computed.

We demonstrate the numerical solutions in Figure \ref{more general  30d plot vector field}. 
\begin{figure}[htb!]
\begin{subfigure}{.19\textwidth}
  \centering
  \includegraphics[width=\linewidth]{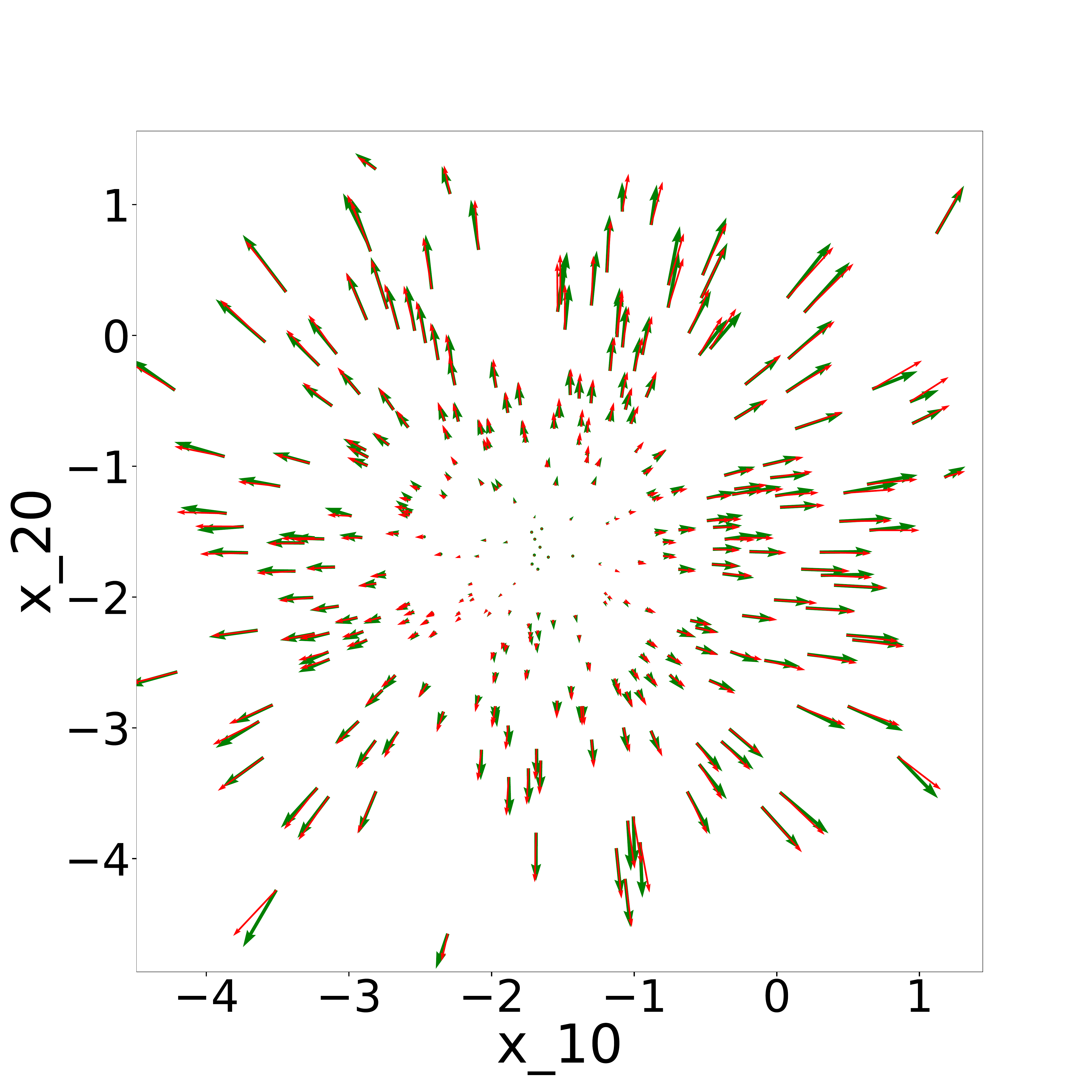}
  \caption{$t=0.2$}
\end{subfigure}
\begin{subfigure}{.19\textwidth}
  \centering
  \includegraphics[width=\linewidth]{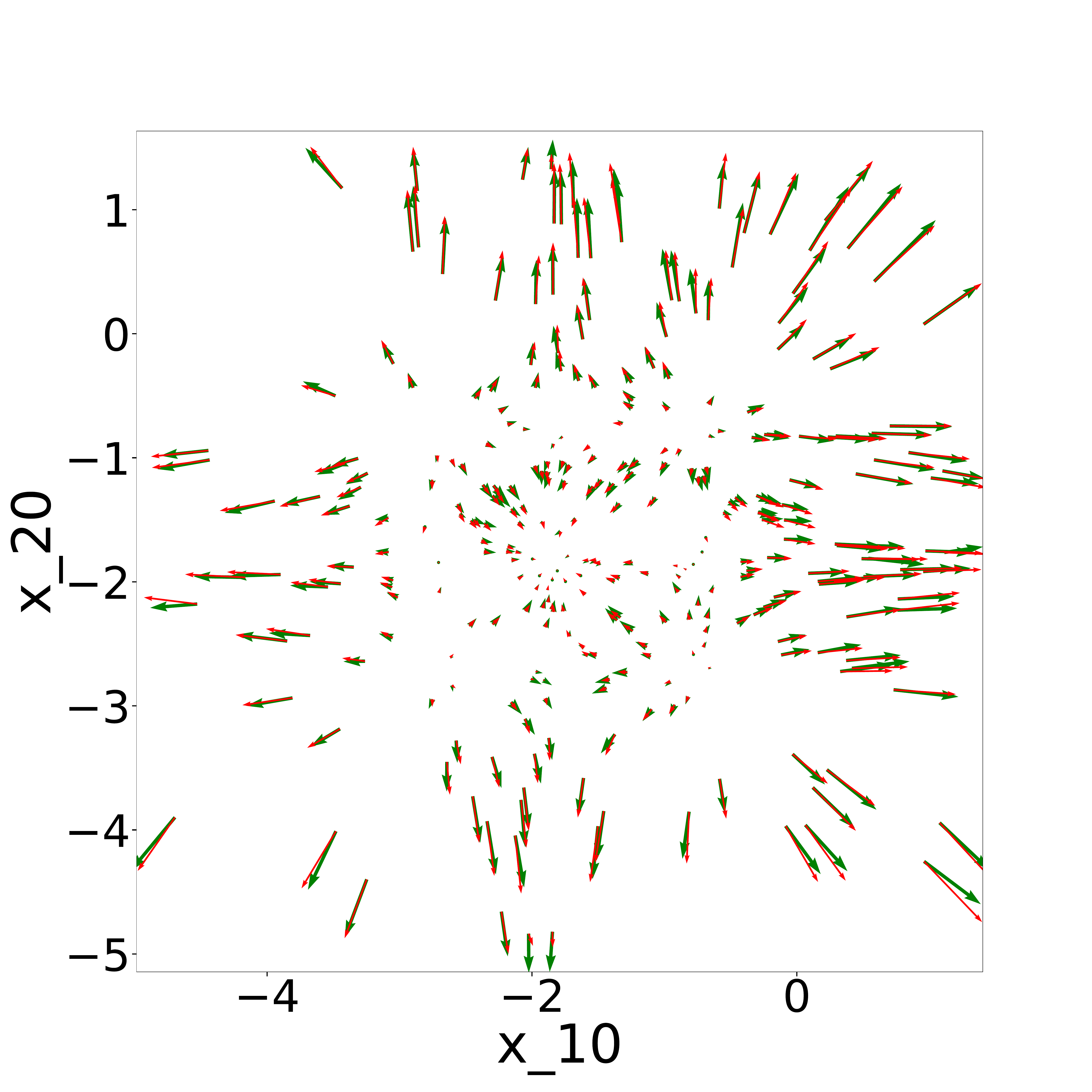}
  \caption{$t=0.4$}
\end{subfigure}
\begin{subfigure}{.19\textwidth}
  \centering
  \includegraphics[width=\linewidth]{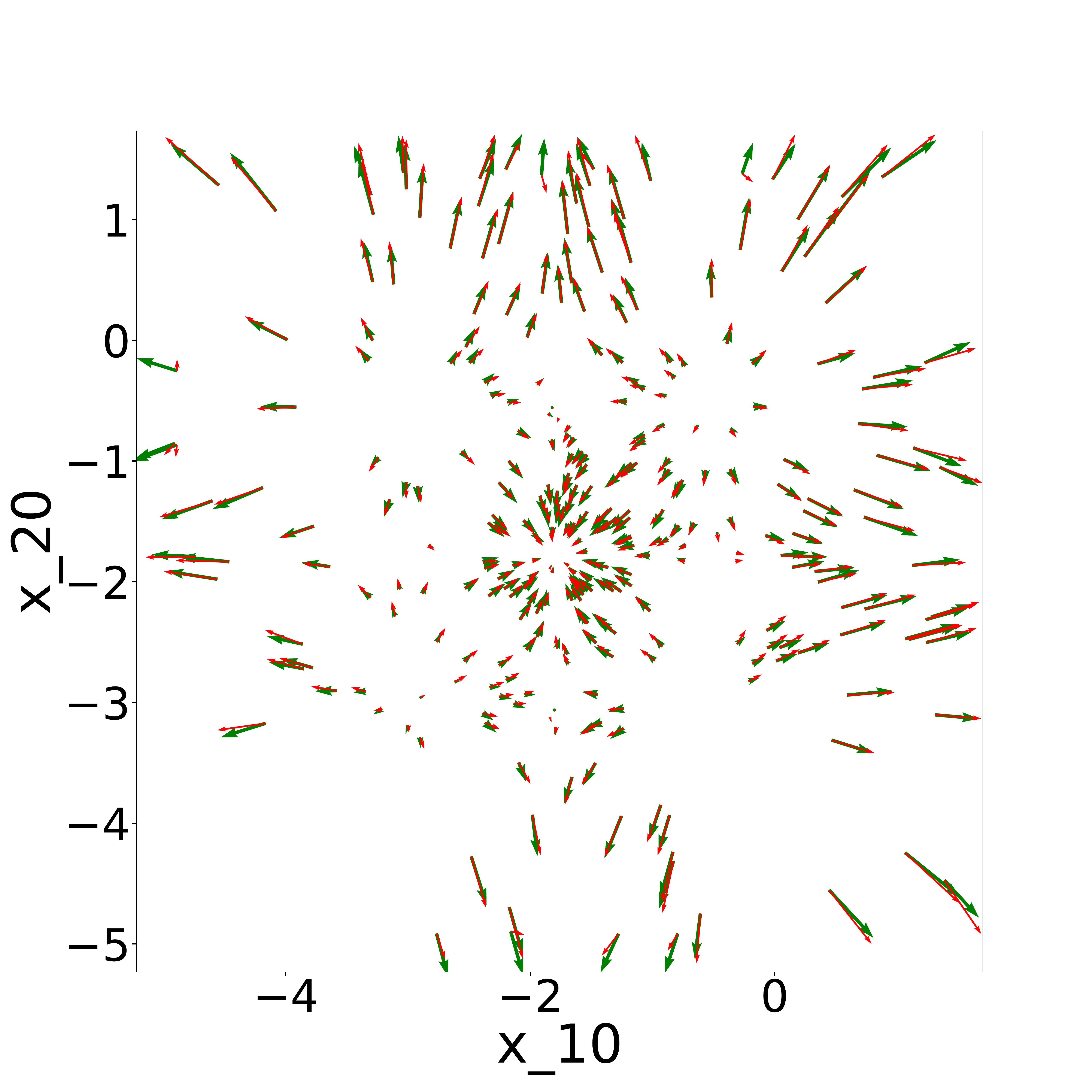}
  \caption{$t=0.6$}
\end{subfigure}
\begin{subfigure}{.19\textwidth}
  \centering
  \includegraphics[width=\linewidth]{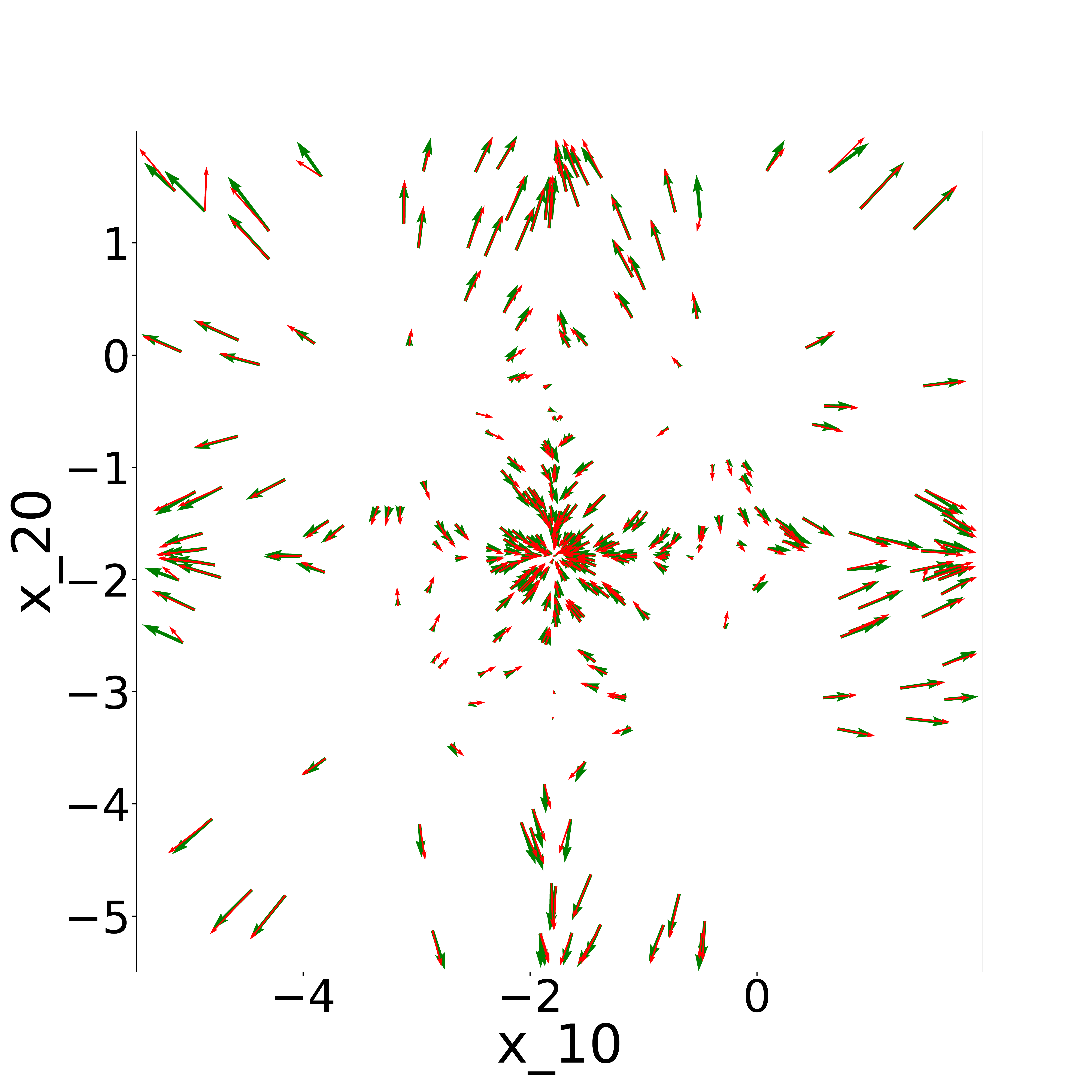}
  \caption{$t=0.8$}
\end{subfigure}
\begin{subfigure}{.19\textwidth}
  \centering
  \includegraphics[width=\linewidth]{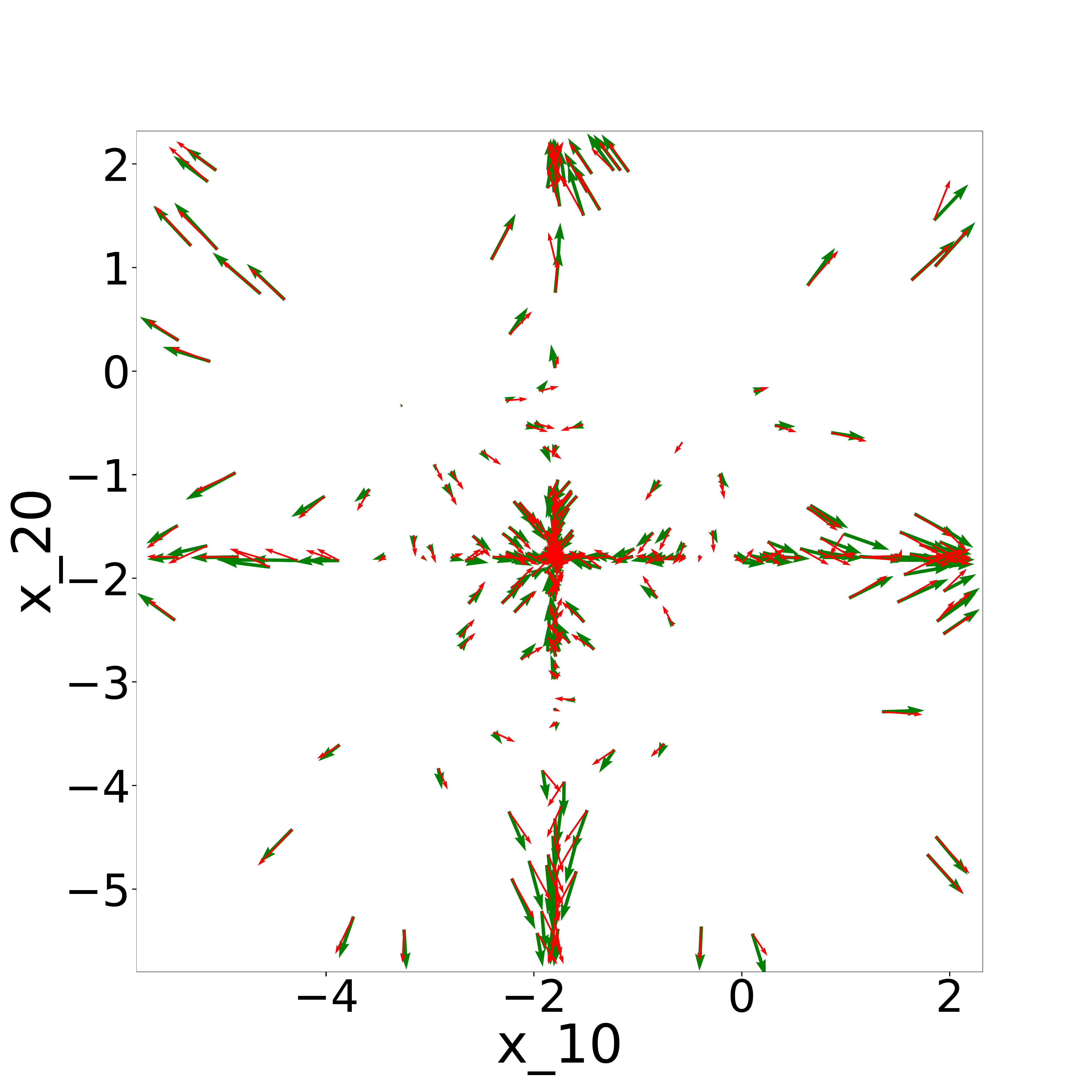}
  \caption{$t=1.0$}
\end{subfigure}
\vspace{-0.3cm}
\caption{{Plots of vector fields $\nabla \psi_\theta(\cdot, t)$ (green) with momentums of samples (red) at different time stages on the {$10\text{th}-20\text{th}$} plane.}}
\label{more general  30d plot vector field}
\end{figure}
Similarly, we plot the loss $\frac{1}{N}\sum_{k=1}^N|e_{t_i}^{(k)}|^2$ versus time nodes $t_i$ in Figure \ref{error vs time plots}(middle subfigure), which shows a significant increase in loss after $t=0.4$. We don't know the exact solution for this example. The numerical result suggests that the kinks of the solution may develop at $T_* \approx 0.4$.

\subsection{Example of non-separable Hamiltonian}\label{example: nonseparable hamiltonian }
In this example, we consider the following non-separable Hamiltonian 
\begin{equation} 
  H(x,p) = \frac{1}{2}(|x|^2+1)(|p|^2+1).   \label{def: non separable Hamiltonian }
\end{equation}
We take the initial value $u(x, 0) = g(x) = 0$ {and solve this equation} on $[0, 1]$.  We set the initial distribution $\rho_a = \mathcal N(0, 2I)$ {and the dimension $d=10$.}
We adopt the explicit symplectic scheme (with $\omega=10$) proposed in \cite{TaoPhysRevE} to integrate the Hamiltonian system \eqref{alg: Hamilton system} associated with the Hamiltonian \eqref{def: non separable Hamiltonian }. {The phase portraits are provided in the left of Figure \ref{non separable Hamiltonian : graph and 2d plot vector field }. Corresponding numerical results are presented in the right figures of Figure \ref{non separable Hamiltonian : graph and 2d plot vector field }.} The gradient field and the momentum match well before $t=0.4$ and after $t=0.9$. This is also verified in the $\frac{1}{N}\sum_{k=1}^N|e_{t_i}^{(k)}|^2$-versus-$t_i$ plot presented in Figure \ref{error vs time plots} (right subfigure).
\begin{figure}[htbp]
    \centering
    \begin{minipage}{0.39\textwidth}
        \centering
        \includegraphics[trim={7cm 5.5cm 20cm 24cm}, clip, width=\linewidth]{numerical_examples/non-separable_Hamiltonian_modif/trajectory_in_phase_space.pdf}
    \end{minipage}
    \begin{minipage}{0.6\textwidth}
        \centering
        \begin{minipage}{0.32\textwidth}
            \subcaptionbox*{}{\includegraphics[trim={5cm 1.8cm 3cm 2cm},clip, width=\linewidth]{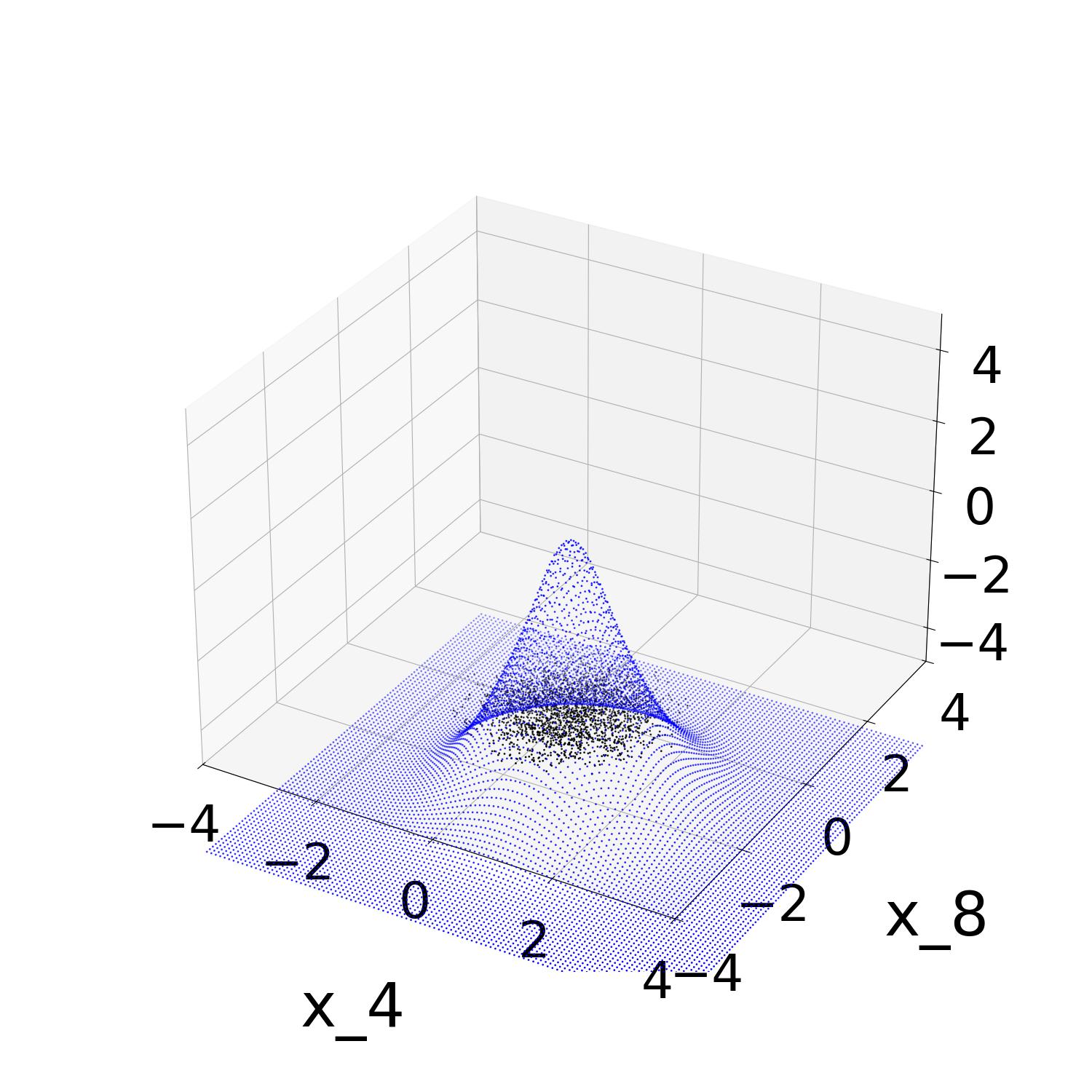}}\vspace{-6.5mm}
            \subcaptionbox*{$t=0.3$}{\includegraphics[trim={2cm 6cm 11cm 15cm},clip, width=\linewidth]{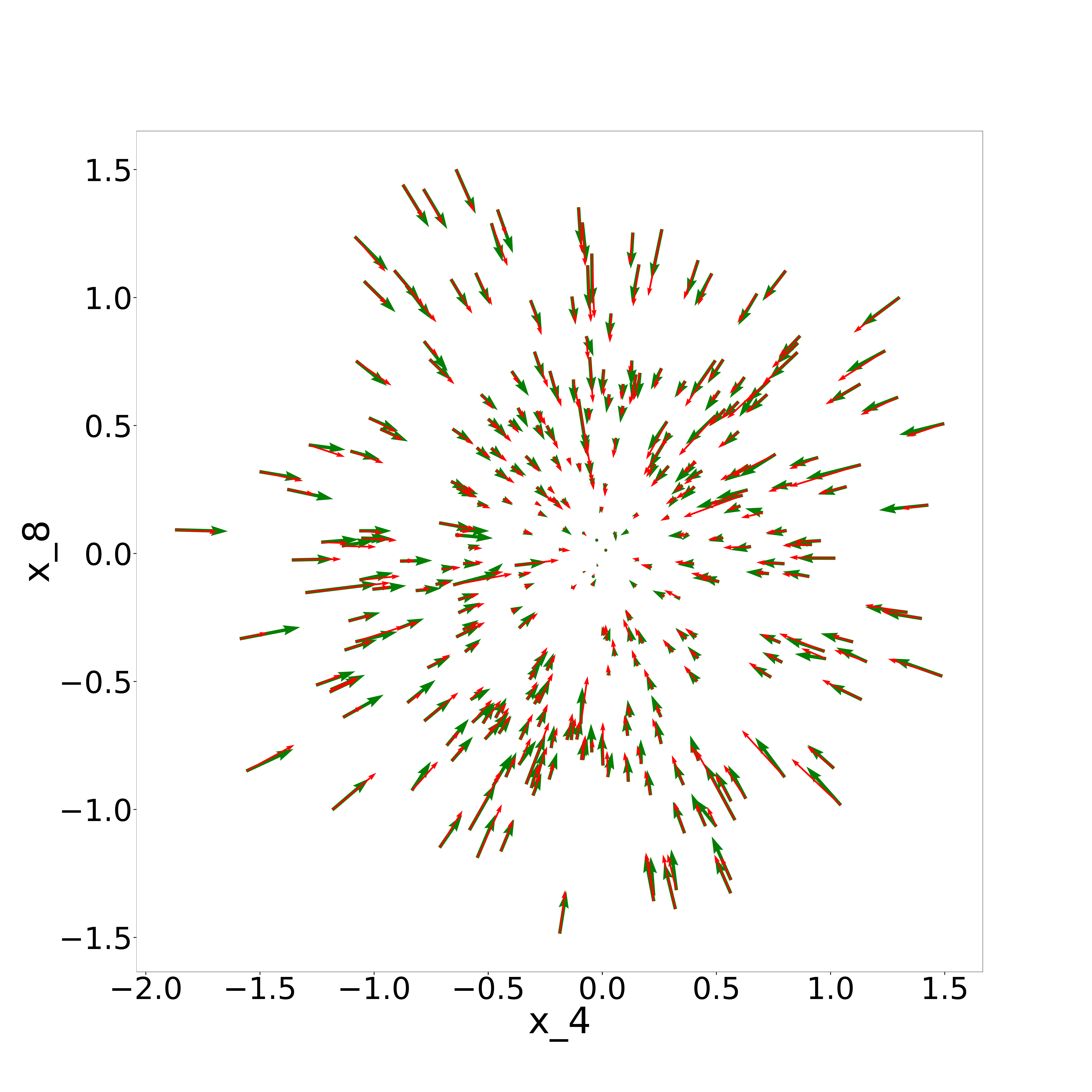}}
        \end{minipage}
        \begin{minipage}{0.32\textwidth}
            \subcaptionbox*{}{\includegraphics[trim={5cm 1.8cm 3cm 2cm},clip, width=\linewidth]{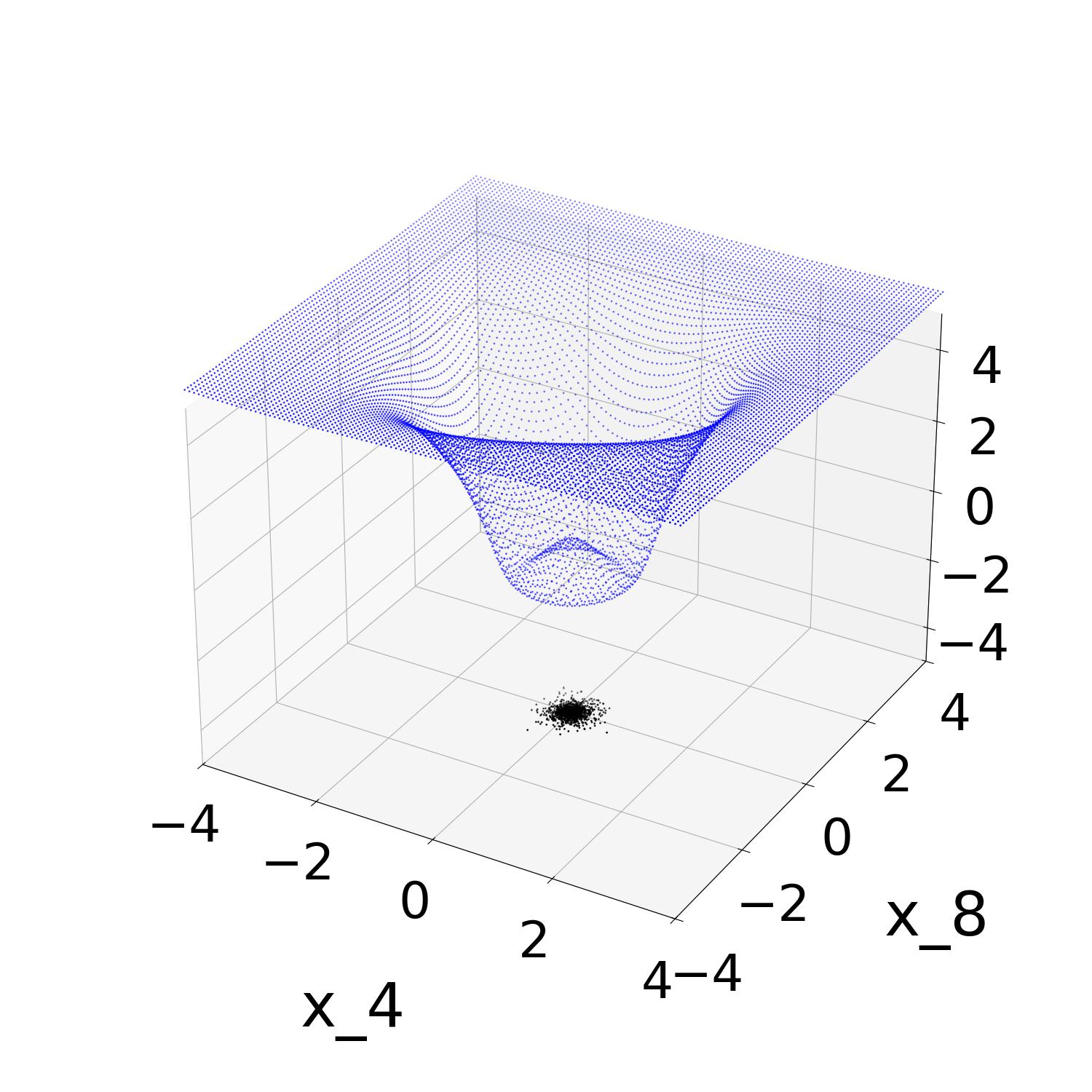}}\vspace{-6.5mm}
            \subcaptionbox*{$t=0.6$}{\includegraphics[trim={2cm 6cm 11cm 15cm},clip, width=\linewidth]{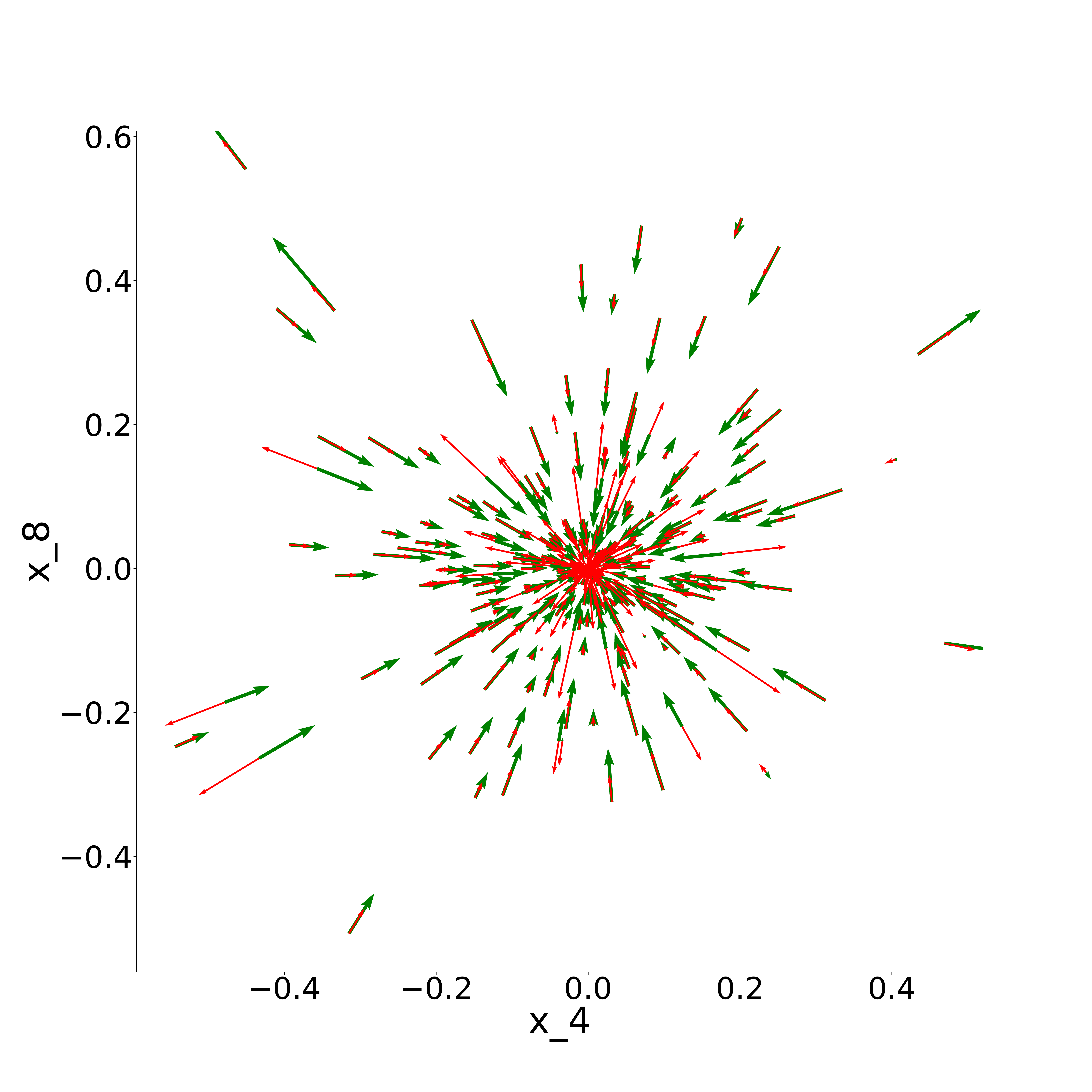}}
        \end{minipage}
        \begin{minipage}{0.32\textwidth}
            \subcaptionbox*{}{\includegraphics[trim={5cm 1.8cm 3cm 2cm},clip, width=\linewidth]{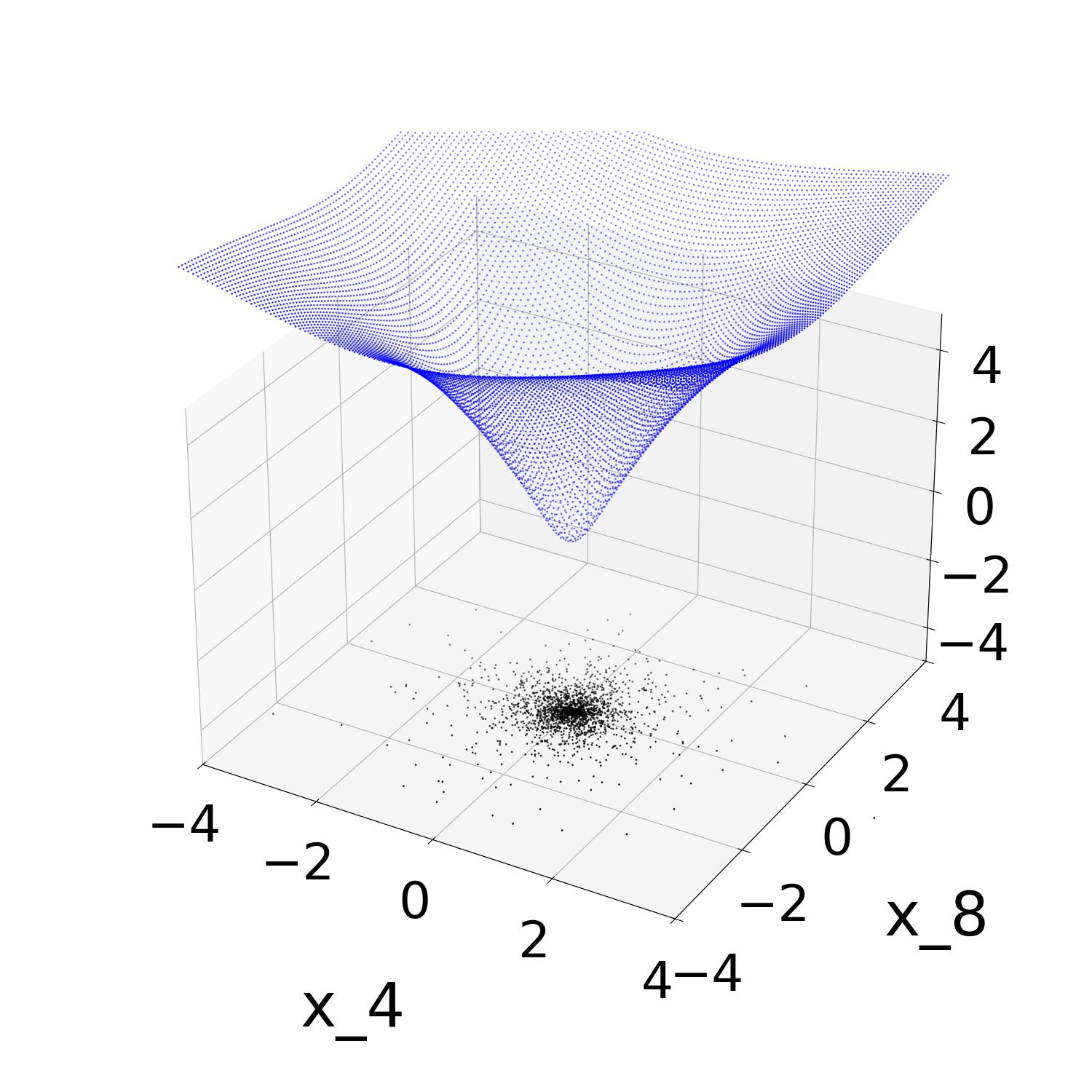}}\vspace{-6.5mm}
            \subcaptionbox*{$t=0.9$}{\includegraphics[trim={2cm 6cm 11cm 15cm},clip, width=\linewidth]{numerical_examples/non-separable_Hamiltonian_modif/gradient_phif_with_momentum_0.6.pdf}}
        \end{minipage}
    \end{minipage}
    \vspace{-0.3cm}
    \caption{{\textbf{Left:} Phase portraits of the Hamiltonian system associated with non-separable Hamiltonian \eqref{def: non separable Hamiltonian }. Here $0\leq t \leq 1$. We visualize the portraits by projecting the trajectories onto the first component of $x$ and $p$. Different colors separate the time intervals: green-$[0,0.2)$, blue-$[0.2,0.4)$, orange-$[0.4,0.6)$, red-$[0.6,8)$, pink-$[0.8,1.0)$. \textbf{Right:} (\textbf{Upper row}) Graphs of the numerical solution $\psi_\theta$ at different time stages on the {$4\text{th}-8\text{th}$} plane. (\textbf{Lower row}) Plots of $\nabla\psi_\theta(\cdot, t)$ (green) with the momentum of samples (red) at $t=0.3,0.6,0.9$.}}\label{non separable Hamiltonian : graph and 2d plot vector field }
\end{figure}

\begin{figure}[htb!]
    \centering
    \begin{subfigure}{0.24\textwidth}
        \includegraphics[width=\textwidth]{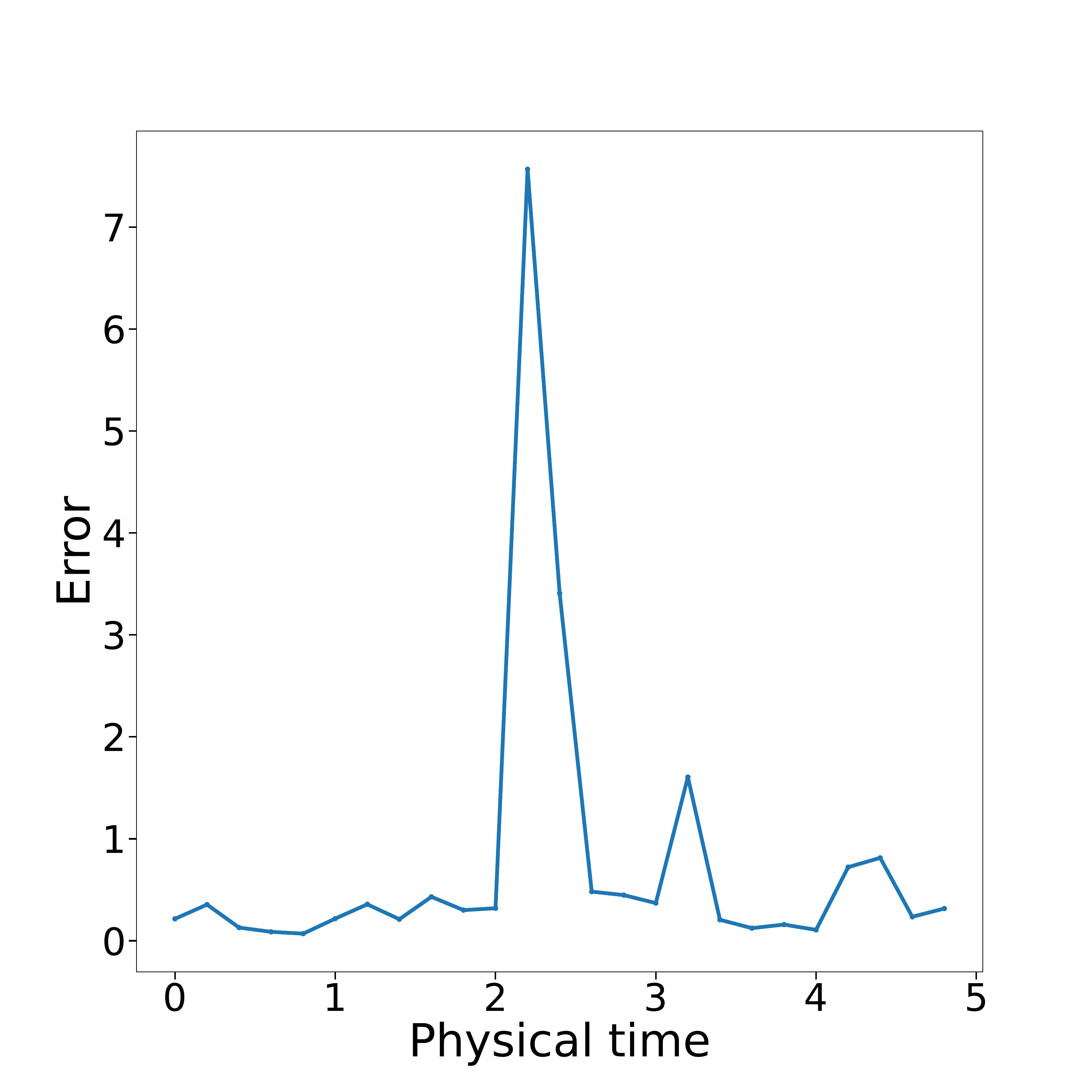}
        \subcaption{{Example} \ref{example: harmonic}}
    \end{subfigure}
    \begin{subfigure}{0.24\textwidth}
        \includegraphics[width=\textwidth]{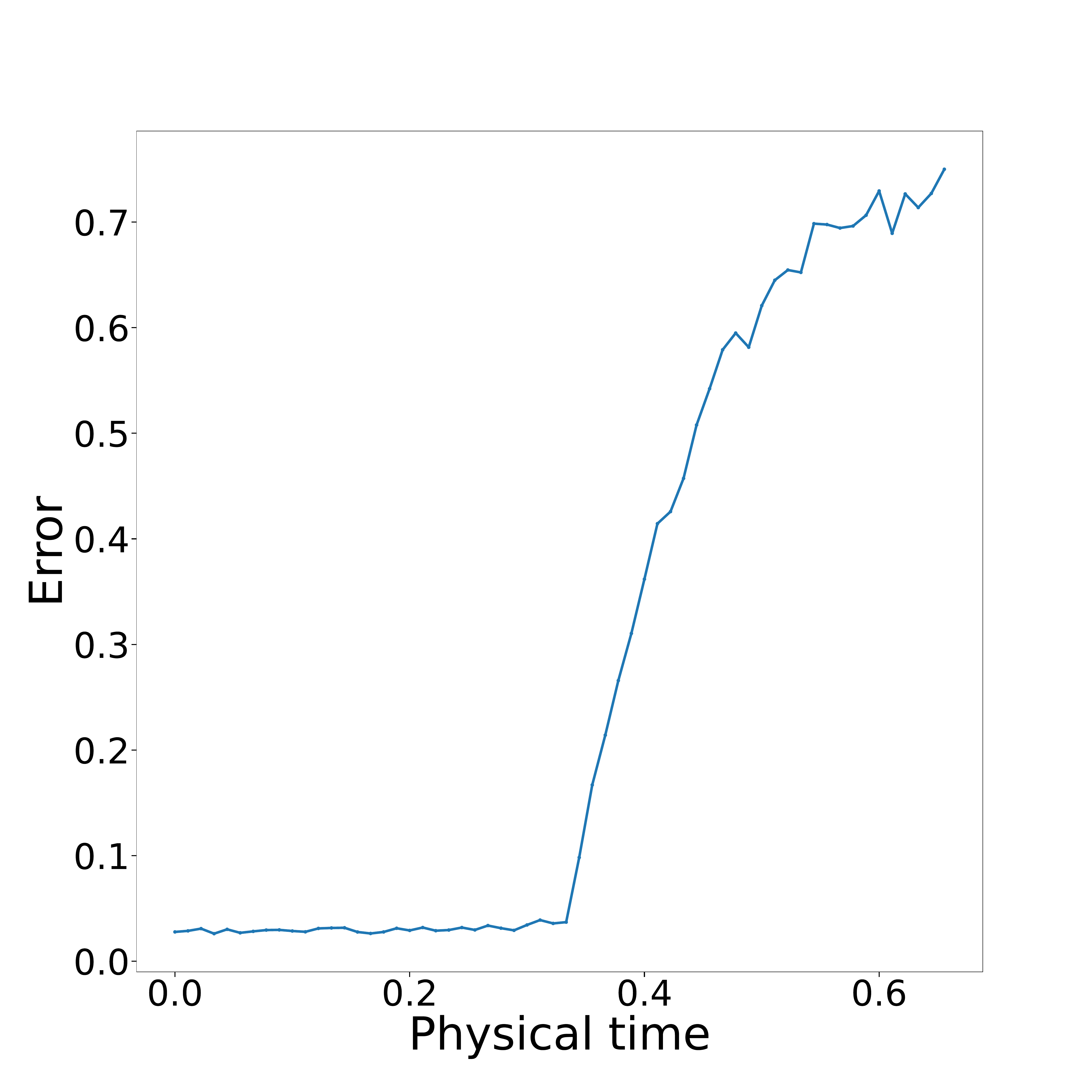}
        \subcaption{{Example} \ref{example: sinusoidal_initial}}
    \end{subfigure}
    \begin{subfigure}{0.24\textwidth}
        \includegraphics[width=\textwidth]{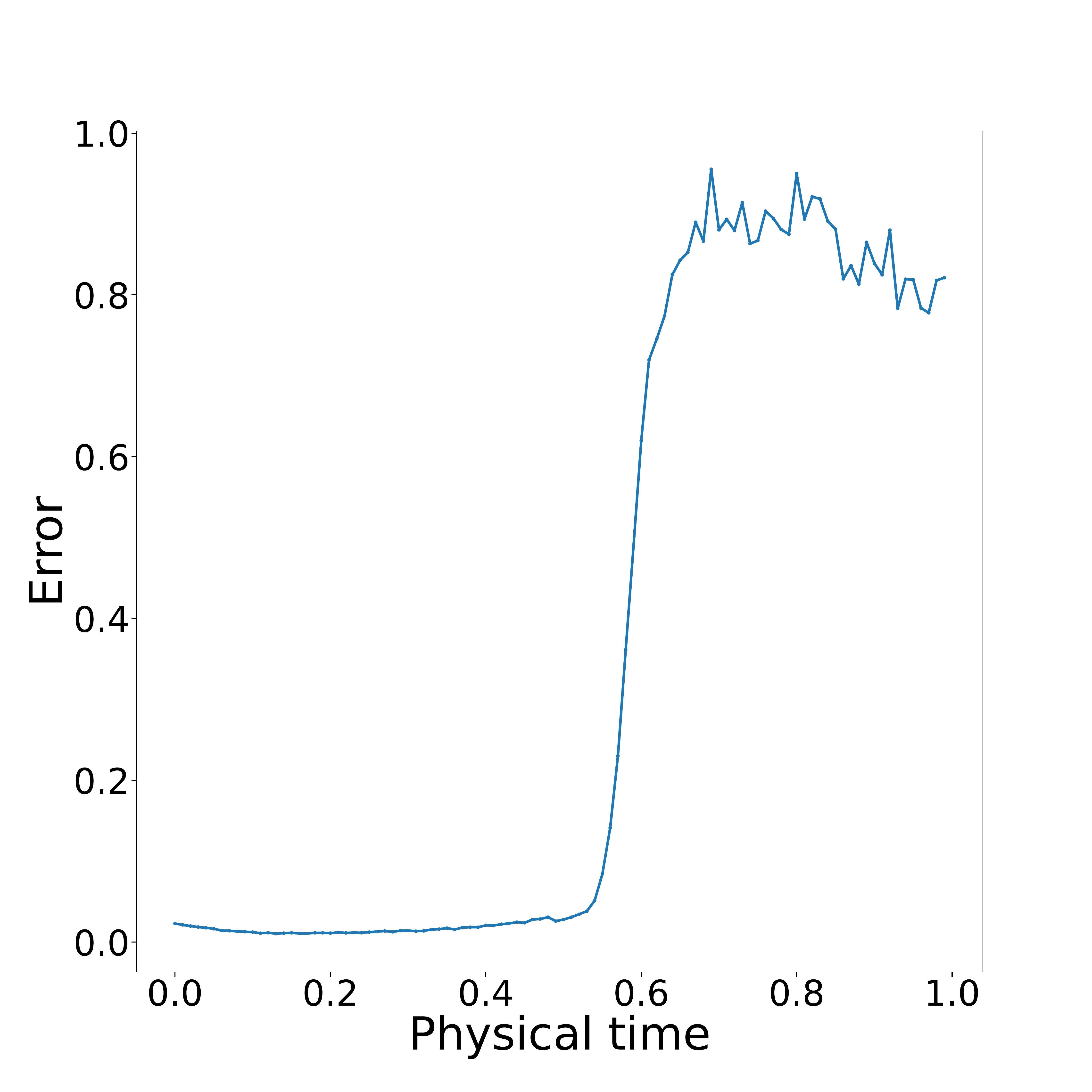}
        \subcaption{{Example} \ref{example: Shi Jin}}
    \end{subfigure}
        \begin{subfigure}{0.24\textwidth}
        \includegraphics[width=\textwidth]{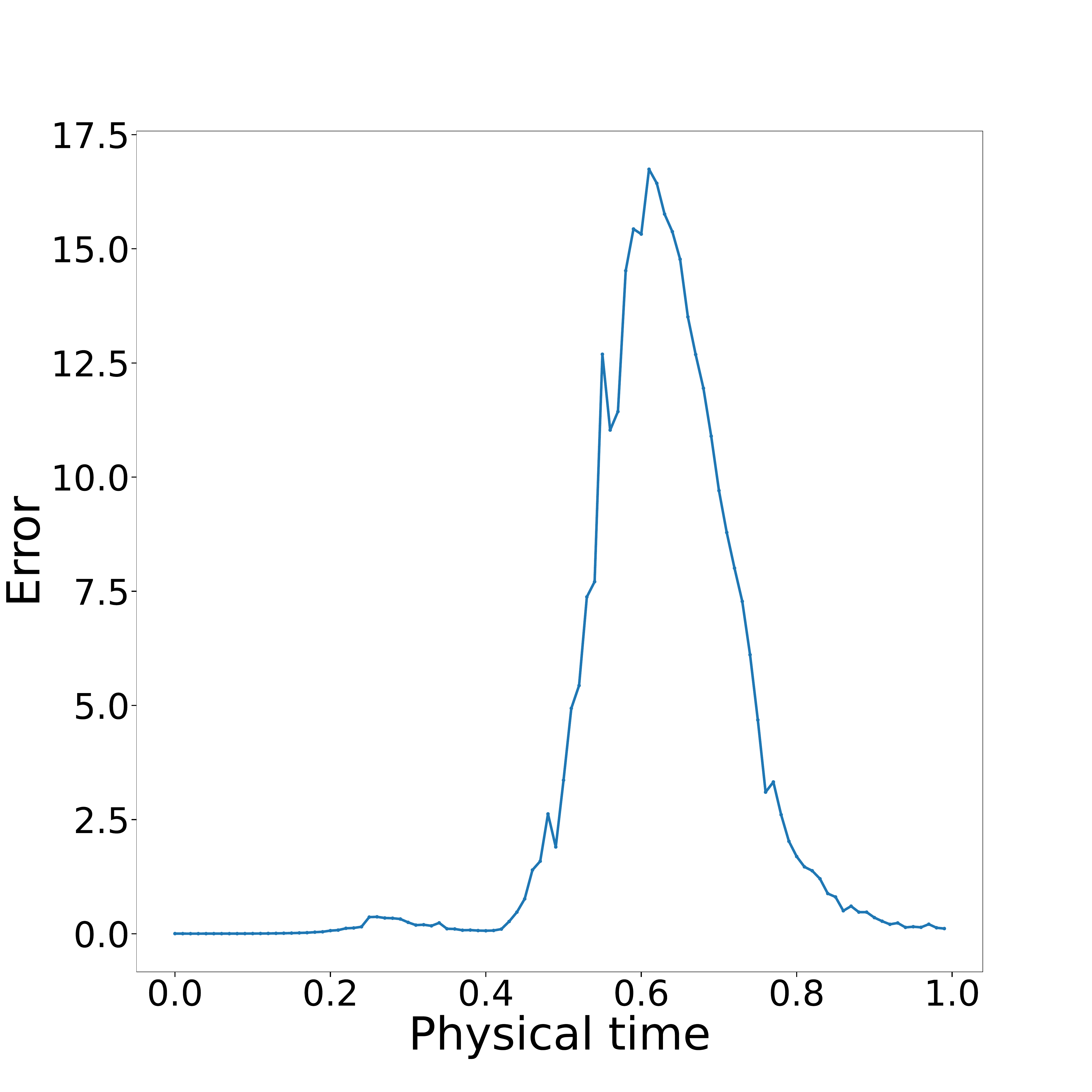}
        \subcaption{{Example} \ref{example: nonseparable hamiltonian }}
    \end{subfigure}
    \vspace{-0.3cm}
    \caption{ {Plots of loss $\frac{1}{N}\sum_{k=1}^N|e_{t_i}^{(k)}|^2$ versus $t_i$ for example \ref{example: harmonic}-\ref{example: nonseparable hamiltonian }.}
    }
    \label{error vs time plots}
\end{figure}
}

\subsection{The Linear Quadratic Control (LQC) problem and related HJ equation for inverted pendulum in subsection \ref{example: LQC inverted pendulum }}\label{example: LQC inverted pendulum 1}

This susbsection discusses more details in subbsection \ref{example: LQC inverted pendulum }.

The Linear Quadratic Control (LQC) problem in $\mathbb{R}^d$ \cite{lehtomaki1981robustness}\cite{underactuated} is usually posed as 
\begin{equation}
\begin{split}\label{LQC Control}
  & \min_{\{x_\tau\}_0^T, \{v_\tau\}_0^T }  ~  \int_0^T \frac12 v_\tau^\top R v_\tau + \frac12 x_\tau^\top Q x_\tau ~d\tau + \frac{1}{2} x_T^\top P_1 x_T, \\
  & \textrm{subject to  }  \dot x_\tau = A x_\tau + B v_\tau,  x_\tau|_{\tau=0} = x_0.
\end{split}
\end{equation}
Here we assume that $R, Q, P_1$ are symmetric matrices, $R$ is positive definite, $Q, P_1$ are semi-positive definite. The critical point of this LQC problem solves the following ODE system based on the Pontryagin's minimum principle,
\begin{equation}
\begin{split}\label{Pontryagin Principle of LCQ problem}
  \dot x_\tau = & Ax_\tau + B v_\tau, \quad v_\tau = R^{-1}B^\top \lambda_\tau, \quad x_\tau|_{\tau=0} = x_0,\\
  \dot \lambda_\tau = & - A^\top \lambda_\tau + Qx_\tau, \quad \lambda_T = -P_1 x_T.
\end{split} 
\end{equation}
Furthermore, if we define the value function
\begin{align*} 
  u(x,t) & = \min_{\{x_\tau\}_t^T, ~ \{v_\tau\}_t^T }  ~    \int_t^T  \frac12 v_\tau^\top R v_\tau + \frac12 x_\tau^\top Q x_\tau ~d\tau + \frac{1}{2} x_T^\top P_1 x_T  \\
   & \quad \textrm{subject to  } \dot x_\tau = A x_\tau + B v_\tau, ~ x_\tau|_{\tau=t} = x,
\end{align*}
then it can be verified that $u(\cdot, t)$ solves the following Hamilton-Jacobi equation with terminal condition
\begin{align}
    \frac{\partial u(x,t)}{\partial t} + \underbrace{\min_{v} \left\{\nabla u(x,t)^\top Bv + \frac{1}{2}v^\top R v + \nabla u(x,t)^\top Ax + \frac{1}{2} x^\top Q x \right\}}_{J(x, \nabla u(x))} = 0,  &  \label{HJ w T cond }  \\
    u(x, T) = \frac12 x^\top P_1 x.  &  \nonumber
\end{align}
The term $J(x, \nabla u(x,t))$ takes an explicit form
\begin{equation*}
  J(x, \nabla u(x, t) ) = -\frac12 (B^\top \nabla u(x,t))^\top R^{-1} (B^\top \nabla u(x,t)) + \nabla u(x,t)^\top Ax + \frac12 x^\top Qx.
\end{equation*}
The optimal control $v_\tau$ is given by 
\begin{equation}
  v_\tau = -R^{-1}B^\top \nabla u(x_\tau, \tau).  \label{optimal control of LQC equals grad u from HJ}
\end{equation}

\vspace{0.2cm}

Let we consider the LQC problem of a swarm of agents in which each of them minimizes its own control cost by resolving \eqref{LQC Control}, while their terminal distribution equals to a given probability distribution $\rho_T$. The proposed method readily handles this control problem with terminal density constraints. 

To be more specific, we consider the ``time-reversal'' of the Hamilton-Jacobi equation \eqref{HJ w T cond }, i.e., we denote $\widetilde{u}(x, t) = u(x, T-t)$. This yields $\partial_t \widetilde{u} = -\partial_t u$. Thus $\widetilde{u}$ solves the HJ equation with initial condition
\begin{align}
  \frac{\partial \widetilde{u}(x,t)}{\partial t} + \underbrace{\frac12 (B^\top \nabla \widetilde{u}(x,t))^\top R^{-1} (B^\top \nabla \widetilde{u}(x,t)) - \nabla \widetilde{u}(x,t)^\top Ax - \frac12 x^\top Qx}_{H(x, \nabla u(x))=-J(x, \nabla u(x))} = 0, & \label{HJ w init cond }\\
  \widetilde{u}(x,0) = \frac{1}{2} x^\top P_1x. &  \nonumber
\end{align}
Here we denote the Hamiltonian $H(x, p)$ as
\begin{equation*}
  H(x, p) = \frac12 (B^\top p)^\top R^{-1} (B^\top p) - p^\top Ax - \frac12 x^\top Qx.
\end{equation*}
We then apply the proposed method to \eqref{HJ w init cond } coupled with the initial probability distribution $\widetilde{\rho}_0 = \rho_T$. 

Notice that the associated Hamiltonian system is 
\begin{equation*}
          \dot q_t = \partial_p H(q_t, p_t) , \quad \dot p_t = - \partial_x H(q_t, p_t). \quad \textrm{with } q_0 \sim \widetilde{\rho}_0, ~ p_0 = P_1 q_0.
\end{equation*}
This yields the linear ODE system
\begin{equation*}
    {\left[\begin{array}{c}
        \dot  q_t \\
        \dot  p_t  
    \end{array}\right]} = \left[ \begin{array}{cc}
       -A    &   BR^{-1}B^\top  \\
        Q    &    A^\top
    \end{array} \right] \left[ \begin{array}{c}
         q_t    \\
         p_t    
    \end{array} \right], \quad  \begin{array}{c}
         q_0 \sim \widetilde{\rho}_0,  \\
         p_0 = P_1 q_0. 
    \end{array}
\end{equation*}
We denote $\widetilde{\rho}_T$ as the density of $\mathrm{Law}(q_T)$.

It is worth mentioning that this Hamiltonian system is equivalent to the ODE \eqref{Pontryagin Principle of LCQ problem} obtained from the Pontryagin's minimum Principle up to the transformation $q_t = x_{T-t}, p_t = -\lambda_{T - t}$. 

Recall \eqref{optimal control of LQC equals grad u from HJ} and the fact that $\widetilde{u}$ is the time-reversal of $u$, the optimal control is given by $v_\tau = -R^{-1}B^\top \nabla \widetilde{u}(x_{\tau} , T-\tau )$ for $0\leq \tau \leq T$. In our computation, we evaluate for the neural network-surrogate solution $\nabla \psi_\theta \approx \nabla \widetilde{u}$ of the HJ equation \eqref{HJ w T cond }. To verify the accuracy of $\nabla\psi_\theta$, we compare the trajectory $\{\widehat{x}_\tau\}$ under our learned control
\[ \dot{\widehat{x}}_\tau = -A \widehat{x}_\tau + B R^{-1} B^\top \nabla \psi_\theta(\widehat{x}_\tau, \tau), \quad \widehat{x}_0  \sim  {\rho}_0=\widetilde{\rho}_T, \]
with the dynamic computed from the Pontryagin's minimum principle \eqref{Pontryagin Principle of LCQ problem}.

\vspace{0.3cm}

\subsection{Hyperparameters for the algorithm}\label{append: hyperparams_alg}
{ We summarize the hyperparameters used in our algorithm for each numerical example in the following table. Recall that $L$ is the depth and $\widetilde d$ is the width of the neural network $\psi_\theta$; $M$ denotes the total number of time steps; $M_T$ denotes the number of subintervals used to divide the entire time interval $[0, T]$, which will be explained in details in example \ref{example: harmonic}; $N$ is the number of samples used in our computation; $lr$ is the learning rate for the Adam method; and $N_{\textrm{Iter}}$ denotes the total iteration number.
\begin{table}[htb!]
    \centering
    \begin{tabular}{c|ccccccc}
    \hline
        Example (dimension) & $L$ &  $\widetilde{d}$  & $M$ & $M_T$ & $N$ & $lr$ &  $N_{\textrm{Iter}}$   \\
         \hline\hline
       \ref{example: harmonic} $(d=30)$ & $6$ & $50$ & $200$ & $25$ & $8000$ & $10^{-4}$ & $30000$ \\
       \hline
       \ref{example: sinusoidal_initial} $(d=20)$ & $6$ & $50$ & $30$ & $1$  & $5000$ & $0.5\times 10^{-4}$ & $6000$ \\
       \hline
       \ref{example: Shi Jin} $(d=30)$ & $6$ & $80$ & $100$ &  $1$  & $5000$ & $0.5\times 10^{-4}$ & $6000$ \\
              \hline
       \ref{example: nonseparable hamiltonian } ${  (d=10)}$ & {$6$} & {$50$} & $100$ & $4$ &  $5000$ & {$0.5 \times 10^{-4}$} & $12000$ \\
       \hline
       {\ref{example: LQC inverted pendulum }  }  ${ (d=4)}$ & { $6$} & { $40$}  & { $100$} & { $1$} &  { $5000$} & { $0.5\cdot 10^{-4}$}  & {$20000$ } \\ 
       \hline  
    \end{tabular}
    \caption{Hyperparameters of our algorithm for examples \ref{example: harmonic}-{ \ref{example: LQC inverted pendulum }} and examples in section \ref{more-example}.}
    \label{tab : hyperparameter }
\end{table}
}

{
\section{Discussion on geometric integrator}\label{subsec-4.4}

As discussed in section \ref{sub : alg} and mentioned at the beginning of sections \ref{numerical example separable hamilton } and \ref{more-example}, for the numerical examples demonstrated above, we always apply the one-step St\"ormer-Verlet scheme to resolve the Hamiltonian system associated to the equation. 

In this subsection, we compare the numerical behaviors with and without the usage of geometric integrators in our algorithm based on the Kepler system. Consider the Hamiltonian of two-body dynamic (Kepler system) as $H(x, p) = \frac12|p|^2 - \frac{1}{|x|},$ $x \in \mathbb{R}^2\setminus\{0\}, p\in\mathbb{R}^2$. The associated Hamiltonian system is
\begin{align*}
  \dot x_t = p_t, ~ ~ \dot p_t = - \frac{x_t}{|x_t|^3}.
\end{align*}
We assume $x_0 \sim \rho_0$, and $\rho_0=\mathcal N(\mu_0, \sigma_0^2 I)$ with $\mu_0 = [-3, -3]^\top, \sigma_0=\frac12$. We set $g(x) = \vec{v}^\top x,$ with $\vec{v}=(\frac12, 0)^\top$. Then, $p_0 = \nabla g(x_0) = \vec{v}$ determines the initial momentum. The proposed algorithm is used to compute the numerical solution $\psi_\theta$ up to $T=9$ with time stepsize $h=0.03$. 

Figure \ref{fig: keplera} shows the resulting trajectories in configuration space obtained using different numerical integrators: forward Euler, St\"ormer–Verlet, and Runge–Kutta (RK4). As reflected in this figure, na\"ive schemes such as the forward Euler fail to maintain accuracy over the interval $[0, T]$, even for a moderate terminal time $T$.

We evaluate and verify the conservation of the following Hamiltonian on $t\in [0, T]$:
\[ 
    \bar{H}_\theta(t) := \mathbb{E} \left[ H(\boldsymbol{X}_t, \nabla \psi_\theta(\boldsymbol{X}_t, t)) \right], 
\]
where $\{\boldsymbol{X}_t\}_{0\leq t\leq T}$ denotes the numerical solution of \eqref{correspd Hamiltonian System} computed using a specific numerical integrator. As shown in Figure \ref{fig: keplerb}, the training loss drops to zero, suggesting that no caustics form throughout the dynamics, and that $\nabla\psi_\theta(\boldsymbol{X}_t, t)$ accurately approximates the momentum $\boldsymbol{P}_t$ on $t\in[0, T]$. It is known that the non-symplectic numerical schemes may fail to conserve the Hamiltonian along computed trajectories, particularly for large terminal times $T$. This degradation can adversely affect the quality of the learned solution $\nabla\psi_\theta$, as demonstrated in Figure \ref{fig: keplerc}. 

In contrast, we observe that employing the 4th-order Runge-Kutta(RK4) integrator achieves comparable performance with the 2nd-order symplectic St\"omer-Verlet scheme, likely due to the moderate terminal time $T = 9$. Nevertheless, identifying a clearly superior approach between symplectic and high-order non-symplectic schemes remains a nuanced and open question. 
\begin{figure}[htb!]
    \centering
    \begin{subfigure}{0.238\textwidth}
        \includegraphics[width=\textwidth]{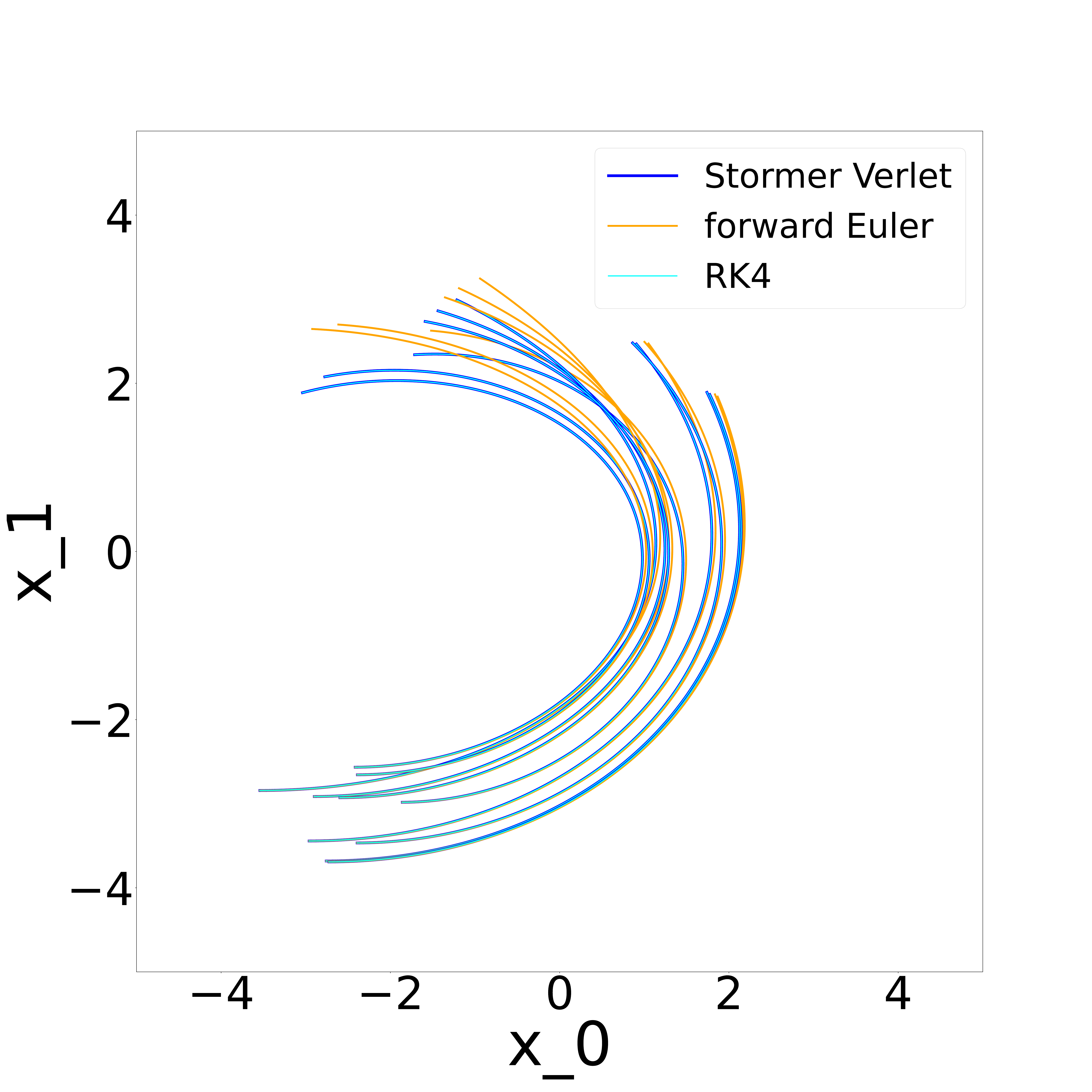}
        \subcaption{}\label{fig: keplera}
    \end{subfigure}
    \begin{subfigure}{0.31\textwidth}
        \includegraphics[width=\textwidth]{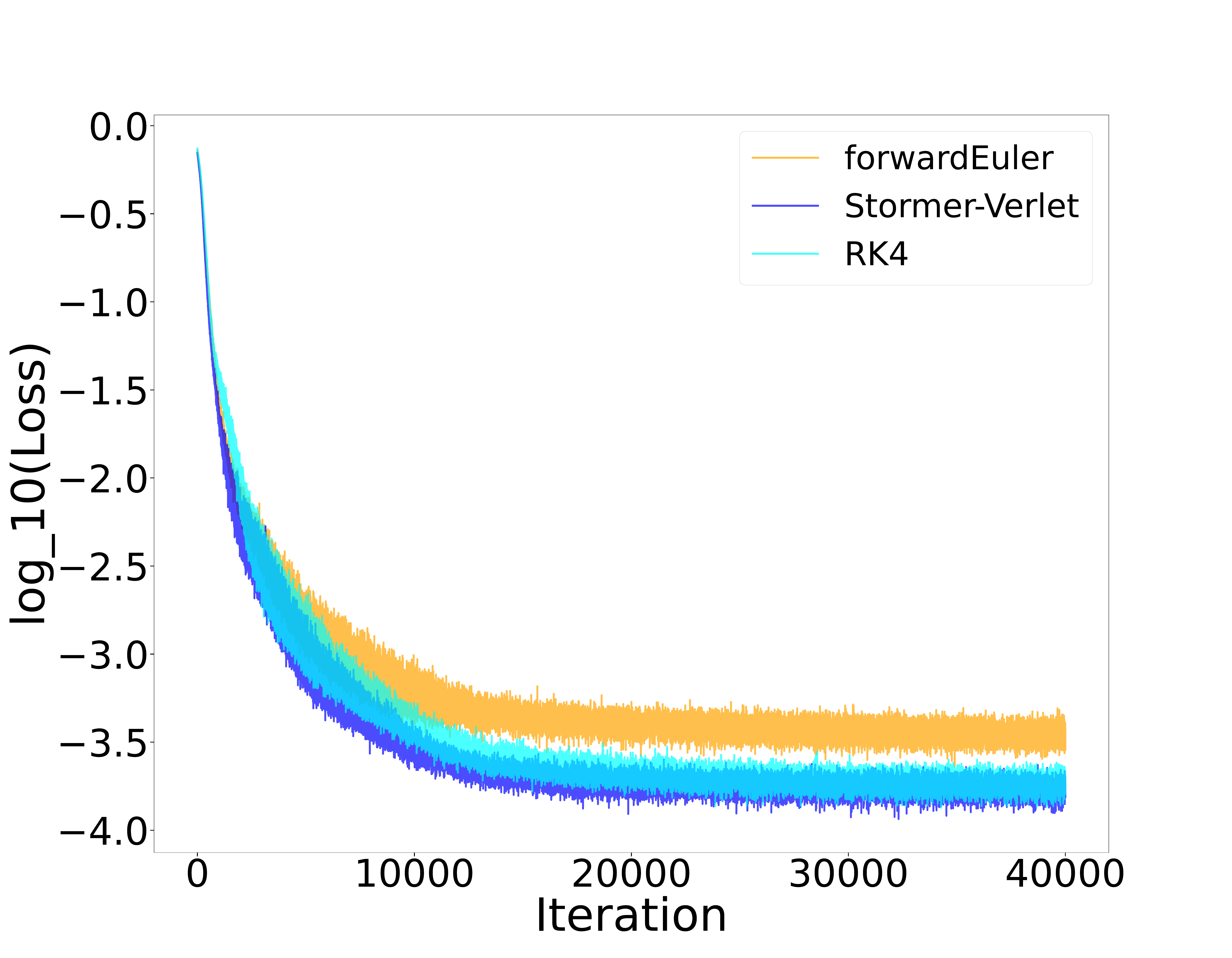}
        \subcaption{}\label{fig: keplerb}
    \end{subfigure}
    \begin{subfigure}{0.434\textwidth}
        \includegraphics[width=\linewidth]{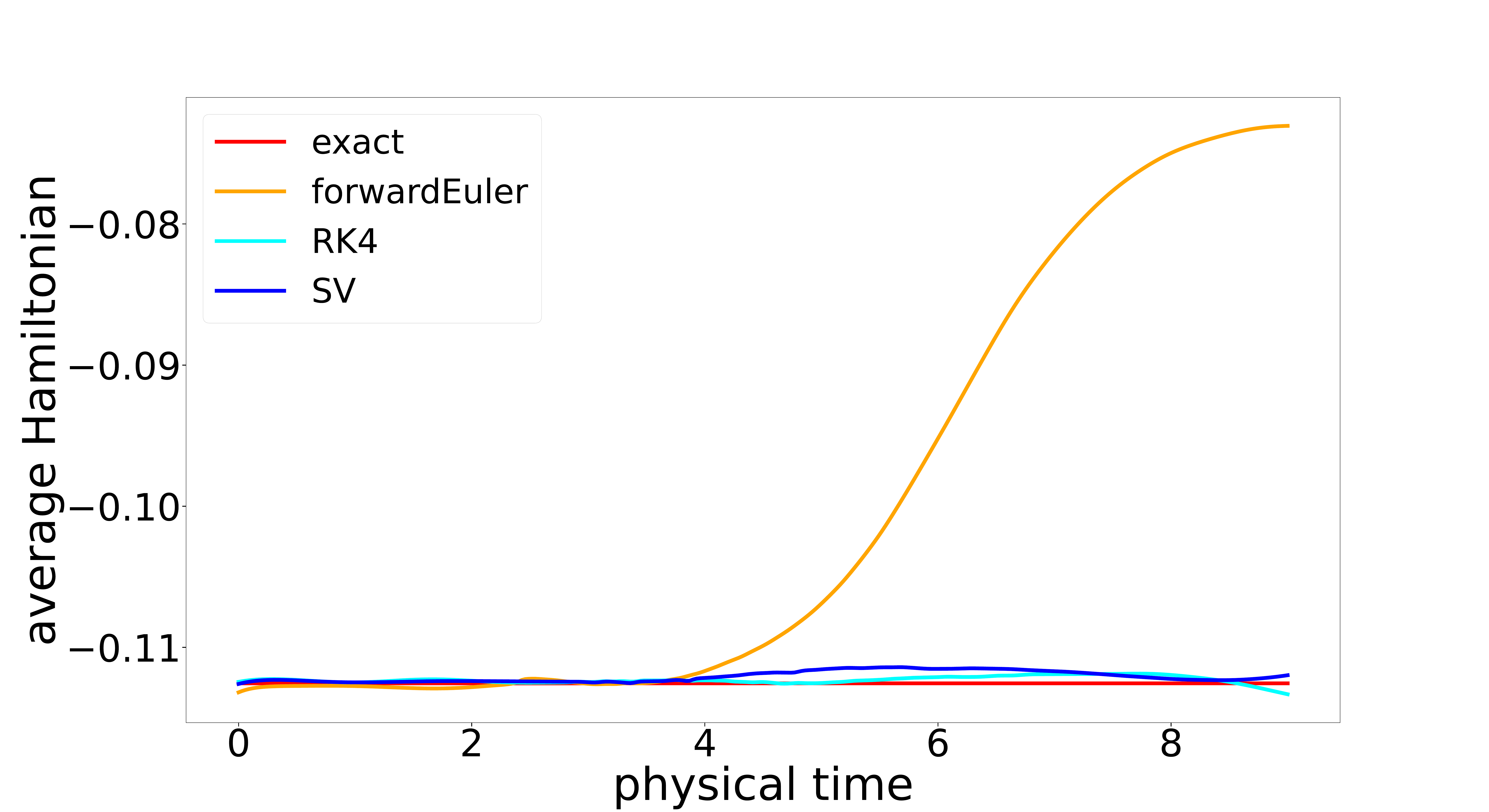}
        \subcaption{}\label{fig: keplerc}
    \end{subfigure}
    \vspace{-0.5cm}
    \caption{\textbf{Left:} Plot of trajectories of the Kepler system for $0\leq t \leq 9$ computed using different numerical integrators; \textbf{Middle:} Training loss (St\"ormer-Verlet scheme) vs. iteration; \textbf{Right:} Plot of average Hamiltonian $\bar{H}_\theta(t)$ vs $t$ for forward Euler, St\"ormer-Verlet, and Runge-Kutta(RK4) schemes.}\label{fig: kepler}
\end{figure}}

\end{document}

%% file: ex_shared.tex
\usepackage{lipsum}
\usepackage{amsfonts}
\usepackage{graphicx}
\usepackage{epstopdf}
\ifpdf
  \DeclareGraphicsExtensions{.eps,.pdf,.png,.jpg}
\else
  \DeclareGraphicsExtensions{.eps}
\fi

\headers{A supervised learning scheme for HJ equation via density coupling}{J. Cui, S. Liu, and H. Zhou}

\title{A supervised learning scheme for computing Hamilton--Jacobi equation via density coupling\thanks{\funding{The research is partially supported by research grants NSF DMS-2307465 and ONR N00014-21-1-2891. The research of the first author is partially supported by the Hong Kong Research Grant Council GRF grant 15302823, NSFC grant 12301526, NSFC/RGC Joint Research Scheme N$\_$PolyU5141/24, MOST National Key $R\&D$ Program No. 2024YFA1015900, the internal grants 
(P0045336, P0046811) from Hong Kong Polytechnic University and the CAS AMSS-PolyU Joint Laboratory of Applied Mathematics.}}}

\author{Jianbo Cui\thanks{Department of Applied Mathematics, The Hong Kong Polytechnic University, Hung Hom, Hong Kong.
  (\email{jianbo.cui@polyu.edu.hk}).}
\and Shu Liu\thanks{Department of Mathematics, University of California, Los Angeles, CA 90095, USA.
  (\email{shuliu@math.ucla.edu}).}
\and Haomin Zhou\thanks{School of Mathematics, Georgia Tech, Atlanta, GA 30332, USA
 (\email{hmzhou@math.gatech.edu}).}}

\usepackage{amsopn}
